\documentclass[a4paper,11pt, twoside]{amsart}

\usepackage[english]{babel}
\usepackage{siunitx}
\usepackage{amsmath}
\usepackage{amssymb}
\usepackage{amsthm}
\usepackage{todonotes}
\usepackage{mathtools}
\usepackage{enumitem}
\usepackage{multirow}
\usepackage[all]{xy}
\usepackage[normalem]{ulem}
\usepackage{tikz-cd}
\usepackage{color}
\usepackage[thicklines]{cancel}

\usepackage{mathtools}

\usepackage[colorlinks=true, pdfstartview=FitV, linkcolor=black, citecolor=black, urlcolor=black]{hyperref}

\usepackage{changepage}

\usepackage{tabularx} 
\usepackage{amsmath, amssymb}
\usepackage[
    paper=a4paper,
    inner=23mm,   
    outer=23mm,   
    top=20mm,     
    bottom=38mm   
]{geometry}

\DeclareMathOperator{\Pic}{Pic}
\DeclareMathOperator{\rank}{rank}
\DeclareMathOperator{\Tr}{Tr}
\DeclareMathOperator{\Aut}{Aut}
\DeclareMathOperator{\Bir}{Bir}
\DeclareMathOperator{\PGL}{PGL}
\DeclareMathOperator{\GL}{GL}
\DeclareMathOperator{\Cr}{Cr}
\DeclareMathOperator{\J}{J}
\DeclareMathOperator{\car}{char}
\DeclareMathOperator{\Ker}{Ker}
\DeclareMathOperator{\Img}{Im}

\newcommand{\et}{\textrm{\'et}}
\newcommand{\QQ}{\mathbb{Q}}
\newcommand{\PP}{\mathbb{P}}
\newcommand{\ZZ}{\mathbb{Z}}

\newcommand{\kk}{\textbf{k}}
\newcommand{\KK}{\mathbf{K}} 
\newcommand{\set}[2]{\left\{\,#1 \ | \ #2\,\right\}}
\newcommand{\Bigset}[2]{\left\{\,#1 \ \Big| \ #2\,\right\}}

\newcommand*{\DashedArrow}[1][]{\mathbin{\tikz [baseline=-0.25ex,-latex, dashed,#1] \draw [#1] (0pt,0.5ex) -- (1.3em,0.5ex);}}

\theoremstyle{theorem}
\newcounter{lettretheo}
\newcolumntype{M}[1]{>{\centering\arraybackslash}m{#1}}

\newtheorem{theolettre}[lettretheo]{Theorem}
\newtheorem{theo}{Theorem}[section]
\newtheorem{prop}[theo]{Proposition}
\newtheorem{lemme}[theo]{Lemma}
\newtheorem{corr}[theo]{Corollary}
\newtheorem{defi}[theo]{Definition}
\newtheorem{ex}[theo]{Example}
\newtheoremstyle{bolditalic}
  {3pt}{3pt}
  {\itshape} 
  {}
  {\bfseries\itshape} 
  {.} 
  {.5em} 
  {}

\theoremstyle{bolditalic}

\theoremstyle{remark}
\newtheorem{remark}[theo]{Remark}
\newtheorem{example}[theo]{Example}

\usepackage{pdfpages}
\usepackage{fancyhdr}
\makeatletter

\@addtoreset{section}{part}

\@addtoreset{subsection}{section}

\makeatother
\title{$p$-elementary non-cyclic subgroups of the Cremona group of the plane}
\author{Mani Esna Ashari}
\address{Mani Esna Ashari, Universit{\"a}t Basel,
Departement Mathematik und Informatik,
Spiegelgasse 1,
CH-4051, Basel,
Switzerland}
\email{esnamani@gmail.com}
\begin{document}

\newgeometry{top=28mm, bottom=28mm, left=31mm, right=31mm}
\maketitle

\noindent \textsc{Abstract.}
\noindent We classify, up to conjugacy, the subgroups of the Cremona group of the plane isomorphic to $(\mathbb{Z}/p\mathbb{Z})^r$, where $p$ is prime and $r \geq 2$, over an algebraically closed field $\mathbf{k}$ of characteristic not equal to $p$.
It is known (\cite{beauville}) that:
\begin{itemize}
    \item If $p \geq 5$, then $r \leq 2$.
    \item If $p = 3$, then $r \leq 3$.
    \item If $p = 2$, then $r \leq 4$.
\end{itemize}

\noindent The main contribution of this article is a concrete description of these groups. In fact, we give an explicit list of representatives via a set of 20 families consisting of subgroups of the de Jon\-quiè\-res group and subgroups of automorphisms of del Pezzo surfaces, and we study the possible conjugacy by birational map between these families.

\noindent Finally, we give some results on subgroups of the Cre\-mo\-na group of the plane isomorphic to $(\mathbb{Z}/p\mathbb{Z})^r$, where $p$ is prime and $r$ is an integer, over an al\-ge\-braically closed field $\mathbf{k}$ of cha\-ra\-cteristic equal to $p$.

\restoregeometry

\newpage

\tableofcontents
\normalsize
\newpage

\part{Introduction} \label{part1} $ $


\bigskip\bigskip

\section{Notations and conventions} $ $





\noindent We use the notation $\ZZ/n$ for the cyclic group of order $n$.
\\

\noindent We fix an algebraically closed field $\kk$ and a prime number $p$.
\\

\noindent A $p$-elementary group is a group that is isomorphic to $(\ZZ/p)^r$ for some integer $r \in \mathbb{N}_{\geq 0}$.
\\

\noindent A Klein group \textit{resp} subgroup is a group \textit{resp} subgroup isomorphic to $(\ZZ/2)^2$. \\

\noindent The dihedral group of order $2n$ is denoted $D_{2n}$. \\

\noindent All varieties and morphisms are defined over the field $\kk$.
\\

\noindent Curves and surfaces are always projective in this article.
\\

\noindent $\KK \supseteq \kk$ denotes a field extension.
\\

\noindent The symbols $i, j, \epsilon, \xi \in \kk$ denote primitive 4th, 3rd, 5th, and $p$-th roots of unity in $\kk$.
\\

\noindent The notation $[a_0 : \dots : a_n]$ will often be used for the following automorphism of $\mathbb{P}^n$: $$[x_0 : \dots : x_n] \mapsto [a_0 x_0 : \dots : a_n x_n]$$

\bigskip

\noindent The $p$-torsion subgroup of the diagonal torus of $\PGL(n,\KK)$ is the subgroup composed of diagonal elements of order $p$ and of identity:
$$\langle [\xi : 1 : \dots : 1],  \dots ,  [1 : \dots :  1 : \xi : 1 : \dots : 1] , \dots , [1 : \dots : 1 : \xi]\rangle \simeq (\ZZ/p)^{n-1}$$

\bigskip

\noindent The Cremona group, denoted $\Cr$, is the group of birational maps of the plane $\PP^2$. \\

\noindent Recall that the Fermat cubic is the hypersurface in $\PP^3$ given by the following equation: $$W^3 + X^3 + Y^3 + Z^3 = 0$$

\bigskip

\noindent The de Jonquières group is the subgroup $\J$ in $\Cr$ given by
conjugating the following group of birational maps of $\mathbb{A}^2$
\[
    \J = \Bigset{ (x,y) \DashedArrow 
    \left(\dfrac{ax+b}{cx+d} , \dfrac{\alpha(x)y + \beta(x)}{\gamma(x)y + \delta(x)} \right)}{ 
    \begin{array}{l}
    a,b,c,d \in \kk \, , \\ 
    \alpha, \beta, \gamma, \delta \in \kk(x) \, , \\ 
    ad-bc \neq 0 \, , \ \alpha \delta - \beta \gamma \neq 0
    \end{array}}
\]
via the embedding $\mathbb{A}^2 \to \PP^2$, $(x,y) \mapsto [y:x:1]$.
Observe that $\J$ is the semi-direct product $\PGL(2, \kk(x)) \rtimes \PGL(2, \kk)$.
\\

\noindent If $S, S'$ are two surfaces, then we call two subgroups 
$G \subseteq \Bir(S)$ and $G' \subseteq \Bir(S')$ conjugate if there is a birational
map $\phi \colon S \dasharrow S'$ such that $G' = \phi G \phi^{-1}$. \\

\newpage

\section{Context} $ $

\subsection{Previous results on $p$-elementary subgroups of the Cremona group} $ $

Among the recent results on $p$-elementary subgroups of the Cremona group, we can cite the thesis of Jérémy Blanc \cite{blancthesis}, in which, along with other findings, he determined the isomorphism classes of finite abelian subgroups of the Cremona group over the field of complex numbers. In particular, he showed that the isomorphism class of maximal $p$-elementary subgroups are: 
\[
\hspace{2.5cm} \hfill (\mathbb{Z}/2)^4, \quad  (\mathbb{Z}/3)^3, \quad  (\mathbb{Z}/p)^2 \hfill \makebox[2.5cm][r]{ \quad when $p > 3$.}
\]

Arnaud Beauville classified maximal $p$-elementary subgroups up to birational map (see \cite{beauville}).

Finally, Jérémy Blanc classified cyclic subgroups of the Cremona group over the field of complex numbers (see \cite{blancarticle}) up to birational conjugacy.

Hence the question of classifying up to conjugation by birational the "intermediate" $p$-elementary subgroups of the Cremona group, \textit{i.e.} not cyclic and not maximal, in practice: $$(\ZZ/2)^2, \quad (\ZZ/2)^3, \quad (\ZZ/3)^2,$$ was still open.






\subsection{Results} $ $

There are three theorems, labelled  \ref{theorem1},  \ref{theorem2} and  \ref{theorem3}.
\begin{itemize}
\item
In Theorem \ref{theorem1} we give an exhaustive list of the $p$-elementary non-cyclic subgroups of $\Cr$ up to conjugation, when $\car(\kk) \neq p$.
\item In Theorem \ref{theorem2} we give further information about the conjugation of $p$-elementary subgroups in $\Cr$ when $\car(\kk) \neq p$.
\item We give also some results about the case $\car(\kk) = p$ in Theorem \ref{theorem3}: We explore the $p$-elementary subgroups of the de Jonquières group, and determine the Klein subgroups of $\Aut(\mathbb{P}^2)$.
\end{itemize}

\subsection{Open problems and perspectives} $ $

A natural next step would be to extend these results to more general settings:

\begin{itemize}
\item Study of the case where $\car(\kk) = p$. While Theorem \ref{theorem3} provides some results for $p$-elementary subgroups, a comprehensive classification for the full Cremona group in characteristic $p$ remains an open challenge.
\item Generalization to non-algebraically closed fields. The results in this article are established over an algebraically closed field $\kk$. It would be of interest to extend this classification to more general fields.
\end{itemize}
 

\newpage

\section{Main results} $ $

\subsection{Theorem \ref{theorem1}: A list of representatives} $ $

\medskip


\begin{theolettre} \label{theorem1} $ $
We assume $\car(\kk) \neq p$.

\medskip

Let $G \simeq (\ZZ /p)^r$ be a $p$-elementary non-cyclic subgroup of the 
Cremona group $\Cr$.

Then $G$ is isomorphic to one of the following groups:
\begin{center}
$(\ZZ/2)^2$,  \quad $(\ZZ/2)^3$,  \quad $(\ZZ/2)^4$  \quad $(\ZZ/3)^2$,  \quad $(\ZZ/3)^3$,  \quad   $(\ZZ/p)^2$ \ with $p > 3$.
\end{center}

\medskip

Furthermore:  
\end{theolettre}
\medskip

\begin{enumerate}[label=\textbf{(\arabic*)}, leftmargin=2em, labelsep=1em]

\item \textbf{If} {\boldmath $G \simeq (\ZZ/2)^2$}, then $G$ is conjugate to one of the following subgroups:
\begin{enumerate}[label=(\Alph*), start=1]
\item  \label{I} The subgroup of the de Jonquières group generated by:
\[
    (x,y) \mapsto (-x,y) \, , \ 
    (x,y) \mapsto \left(x,\dfrac{g(x^2)}{y}\right) 
\]
where $g \in \kk[x]$ is a monic polynomial with simple roots with at least two roots not equal to $0$.

\item  \label{J} The subgroup of the de Jonquières group generated by:
\[
    (x,y) \mapsto \left(x,\dfrac{g(x)}{y}\right) \, , \
    (x,y) \mapsto \left(x,\dfrac{a(x)y-b(x)g(x)}{b(x)y-a(x)}\right) 
 \] 
 where $g \in \kk[x]$ is a monic polynomial with simple roots of degree at least three and 
 $a,b \in \kk(x)$ are rational functions such that $a^2 \neq g b^2$.
\item  \label{G} The subgroup:
\[ \{ [W:X:Y:Z] \mapsto [ \pm W: \pm X:Y:Z] \} \]
of the automorphism group of 
a del Pezzo surface of degree $2$ given by 
$$\left\{W^2 = X^4 + L_2(Y,Z) X^2 + L_4(Y,Z) \right\} \subset \mathbb{P}(2,1,1,1),$$ where $L_2$ \textit{resp} $L_4$ are forms of degree $2$ \textit{resp} $4$ 
such that the surface is smooth.
\item  \label{H} The subgroup:
\[ 
    \{ [W:X:Y:Z] \mapsto [ \pm W: \pm X:Y:Z]  \}
\]
of the automorphism group of a del Pezzo surface of degree $1$ given by
$$ \quad \quad \quad \left\{W^2 = Z^3 + L_1(X^2,Y^2) Z^2 + L_2 (X^2,Y^2)Z + L_3 (X^2,Y^2)\right\} \subset \mathbb{P}(2,1,1,3),$$ where $L_1$ \textit{resp} $L_2$ \textit{resp} $L_3$ is a form of degree $1$ \textit{resp} $2$ \textit{resp} $3$ such that the surface is smooth.

 \item \label{Iprime}  The subgroups:
 \[ \begin{array}{ll}
 & \left\{ (x,y) \mapsto (\pm x, \pm y) \right\} \\
  & \left\{ (x,y) \mapsto (\pm x^{\pm 1},y) \right\}  \\
\end{array}\]
of the automorphism group of $\mathbb{P}^1 \times \mathbb{P}^1$.

\end{enumerate}
\bigskip
\item
\textbf{If} {\boldmath $G \simeq (\ZZ/2)^3$}, then $G$ is conjugate to one of the following subgroups:
\begin{enumerate}[label=(\Alph*), start=6]
\item  \label{M} The subgroup of the de Jonquières group generated by:
\[  \quad \quad \quad \quad 
(x,y) \mapsto (-x,y) \ , \ 
(x,y) \mapsto \left(\frac{1}{x},y\right) \, , \
(x,y) \mapsto \left(x,\dfrac{g(x^2+\frac{1}{x^2})}{y}\right)
\]
where $g$ is a monic polynomial with simple roots with at least one root not equal to $\pm 2$.
\item  \label{N} The subgroup of the de Jonquières group generated by:
\[ \quad \quad \quad \quad 
(x,y) \mapsto (-x,y) \ , \ 
(x,y) \mapsto \left(x,\dfrac{g(x^2)}{y}\right) \, , \
(x,y) \mapsto  \left(x,\dfrac{a(x^2)y-b(x^2)g(x^2)}{b(x^2)y-a(x^2)}\right)
\]
where $g \in \kk[x]$ is a monic polynomial with simple roots with at least two roots not equal to $0$ and 
$a,b \in \kk(x)$ are rational functions such that $a^2 \neq g b^2$.
\item  \label{K} The subgroup:
\[
    \left\{ [X_0:X_1:X_2:X_3:X_4] \mapsto [ \pm X_0: \pm X_1: \pm X_2:X_3:X_4] \right\}
\]
of the automorphism group of a quartic del Pezzo surface given by \[ \left\{ \sum_{k=0}^4 X_k^2 = \sum_{k=0}^4 a_k X_k^2 = 0 \right\} \subset \PP^4 \] 
\item  \label{L} The subgroup: 
\[
    \{ [W:X:Y:Z] \mapsto [ \pm W: \pm X: \pm Y:Z] \}
\]
of the automorphism group of a del Pezzo surface of degree $2$ given by 
\[ \quad \left\{W^2 = X^4 + Y^4 + Z^4 + d X^2 Y^2 + e X^2 Z^2 + f Y^2 Z^2 \right\} \subset \mathbb{P}(2,1,1,1),\] where $d,e,f \in \kk \setminus \{-2,2\}$ such that $ 4 - f^2 - d^2 - e^2 + def \neq 0$.

Two such subgroups are conjugate by an isomorphism if and only if there exist $\epsilon_d, \epsilon_e \in \{ -1, 1\}$ such that $\{d',e',f'\} = \{\epsilon_d d, \epsilon_e e, \epsilon_d \epsilon_e f \}$.

\item \label{Mprime} The subgroups: \[\begin{array}{ll}
 & \langle (x,y) \mapsto (-x,-y),(x,y) \mapsto (\frac{1}{x},\frac{1}{y}), (x,y) \mapsto (y,x) \rangle  \\
 & \langle (x,y) \mapsto (-x,y), (x,y) \mapsto (\frac{1}{x},y), (x,y) \mapsto (x,-y) \rangle \\
& \langle (x,y) \mapsto (\frac{1}{x},-y), (x,y) \mapsto (x,\frac{1}{y}), (x,y) \mapsto (-x,y) \rangle \\
\end{array}\]
of the automorphism group of $\mathbb{P}^1 \times \mathbb{P}^1$.

\end{enumerate}


\item \textbf{If} {\boldmath $G \simeq (\ZZ/2)^4$}, then $G$ is conjugate to one of the following subgroups:
\begin{enumerate}[label=(\Alph*), start=11]
\item \label{Q} The subgroup of the de Jonquières group generated by:
\begin{eqnarray*}
(x,y) &\mapsto& (-x,y) \, , \\
(x,y) &\mapsto& \left(\frac{1}{x},y\right) \, , \\
(x,y) &\mapsto& \left(x,\dfrac{g(x^2+\frac{1}{x^2})}{y}\right) \, , \\
(x,y) &\mapsto& 
\left(x,\dfrac{a(x^2+\frac{1}{x^2})y-
b(x^2+\frac{1}{x^2})g(x^2+\frac{1}{x^2})}{b(x^2+\frac{1}{x^2})y-a(x^2+\frac{1}{x^2})}\right)
\end{eqnarray*}
where $g \in \kk[x]$ is a monic polynomial with simple roots with at least one root not equal to $\pm 2$. and 
$a,b \in \kk(x)$ are rational functions such that $a^2 \neq g b^2$.
\item \label{P} The subgroup:
\[
    \left\{ [X_0:X_1:X_2:X_3:X_4] \mapsto [ \pm X_0: \pm X_1: \pm X_2: \pm X_3:\pm X_4] \right\}
\]
of the automorphism group of a quartic del Pezzo  surface given by 
\[ \left\{ \sum_{k=0}^4 X_k^2 = \sum_{k=0}^4 a_k X_k^2 = 0 \right\} \subset \PP^4,\] 
where the $a_k \in \kk$  are pairwise distinct.
\item \label{Qprime}
The subgroup:
\[ \left\{ (x,y) \mapsto (\pm x^{\pm 1},\pm y^{\pm 1}) \right\} \]
of the automorphism group of $\mathbb{P}^1 \times \mathbb{P}^1$.

\end{enumerate}


\bigskip

\item \textbf{If} {\boldmath $G \simeq (\ZZ/3)^2$}, then $G$ is conjugate to one of the following subgroups:

\begin{enumerate}[label=(\Alph*), start=14]
\item \label{A} The $3$-torsion subgroup of the diagonal torus of $\PGL(3,\kk) \subseteq \Cr$.
\item \label{B} The subgroup of $\PGL(3, \kk) \subseteq \Cr$ generated by:
\[
    [X:Y:Z] \mapsto [X:j Y: j^2 Z] \, , \ [X:Y:Z] \mapsto [Z:X:Y] \, .
\]
\item \label{C} The subgroup of the automorphism group of the Fermat cubic
generated by:
\[ \quad \quad \quad
        {[W:X:Y:Z]}  \mapsto  {[jW:X:Y:Z]} \, , 
        {[W:X:Y:Z]}  \mapsto  {[W:jX:Y:Z]} \, . 
\]
\item  \label{D} The subgroup of the automorphism group of the cubic surface given by $$\left\{W^3 + X^3 + Y^3 + Z^3 + \mu X Y Z = 0\right\} \subset \PP^3$$ where $\mu \in \kk$ such that $\mu^3 \neq -27$, generated by:
 \[ \quad \quad \quad \quad
      {[W:X:Y:Z]} \mapsto {[jW:X:Y:Z]} \, ,  
      {[W:X:Y:Z]} \mapsto {[W:jX:j^2 Y:Z]} \, .
\]
Two such subgroups are conjugate by an isomorphism if and only if $\mu^3 = \mu'^3$.

\item  \label{E} If $\car(\kk) \neq 2$: The subgroup of the automorphism group of 
the del Pezzo surface of degree $1$ given by $$\left\{W^2 = Z^3 + X^6 + Y^6 +  c X^3 Y^3 = 0 \right\} \subset \mathbb{P}(3,1,1,2)$$
where $c \in \kk \setminus \{ -2,2\}$,
generated by:
\[ \quad  \quad  \quad
        {[W:X:Y:Z]}  \mapsto  {[W:jX:Y:Z]} \, , 
        {[W:X:Y:Z]}  \mapsto  {[W:X:jY:Z]} \, .
\] 
Two such subgroups are conjugate by an isomorphism if and only if $c = \pm c'$.

If $\car(\kk) = 2$: The subgroup of the automorphism group of 
the del Pezzo surface of degree $1$ given by:
$$\quad \quad \left\{ (u,v,x,y) | 0 = y^2 + uv(u+v) y + x^3 + (e + \sqrt{e}) (u^5 v + u v^5) + e u^3 v^3\right\},$$
inside $\mathbb{P}(1,1,2,3)$, where $e \in \kk \setminus \{0,1\}$,
generated by:
\[ 
    \begin{array}{rcl}
        {[u:v:x:y]} & \mapsto & {[u:v:j x : y]} \, , \\ 
        {[u:v:x:y]} & \mapsto & {[v:u+v:x:y+\sqrt{e}(u^2 v + u v^2 + v^3)]} \, .
    \end{array}
\]
\end{enumerate}
\bigskip
\item \textbf{If} {\boldmath $G \simeq (\ZZ/3)^3$}, then $G$ is conjugate to the $3$-torsion subgroup of the diagonal torus in the automorphism group of the Fermat cubic.
\bigskip\bigskip
\item \textbf{If} {\boldmath $G \simeq (\ZZ/p)^2$} \textbf{with} {\boldmath $p>3$}, then $G$ is conjugate to the $p$-torsion subgroup of the diagonal torus of $\PGL(3,\kk) \subseteq \Cr$.
\end{enumerate}

\bigskip

The subgroups in families \ref{I}, \ref{J}, \ref{M}, \ref{N}, \ref{Q} are subgroups of the de Jonquières group. We will denote the families in the case by $\mathbb{C} \mathbb{A}$.
Every subgroup from such a family has at least one non-trivial element fixing a non-rational curve.
\medskip

The subgroups in families  \ref{G}, \ref{H}, \ref{K}, \ref{L}, \ref{P}, \ref{A}, \ref{B}, \ref{C}, \ref{D}, 
\ref{E} have an invariant Picard rank equal to $1$. We will denote the families in this case by $\mathbb{D}$.
\medskip

The subgroups in families \ref{Iprime}, \ref{Mprime}, \ref{Qprime} are subgroups of the del Pezzo surface $\mathbb{P}^1 \times \mathbb{P}^1$. We will denote the families in this case by $\mathbb{C} \mathbb{P}$.
No non-trivial element in these subgroups fixes a non-rational curve.

\bigskip
\bigskip

\begin{remark} $ $
In the statement of Theorem \ref{theorem1}, the case $\mathbb{D}$ denotes the subgroups that act on a del Pezzo surface with Picard rank $1$. The case $\mathbb{C} \mathbb{A}$ denotes the subgroups of the de Jonquières group, we use this notation because such subgroups act on the \textbf{a}ffine space $\mathbb{A}^2$. The case $\mathbb{C} \mathbb{P}$ denotes the subgroups of $\Aut(\mathbb{P}^1 \times \mathbb{P}^1)$, we use this notation because such subgroups act on the \textbf{p}rojective space $\mathbb{P}^1 \times \mathbb{P}^1$.
\end{remark}

\newpage

\subsection{Theorem \ref{theorem2}: Conditions for possible conjugations between the representatives} \label{amazingandsinequanon} $ $

\bigskip
\bigskip
\bigskip

We start this subsection with an important introductory example, to which we will refer as \textbf{\textit{the amazing result}}:

\bigskip\bigskip

\noindent \textbf{\textit{\hypertarget{Amazing}{The amazing result.}}} \textit{
We assume $\car(\kk) \neq 2$.} \bigskip

\textit{
Let $t \in \kk \setminus\{0,1,-1\}$. Let $S_t$ be the quartic del Pezzo surface given by:
\[S_t=\left\{X_0^2+X_1^2+X_2^2+X_3^2+X_4^2=X_1^2+ t X_2^2 - t X_3^2+X_4^2=0\right\} \subset \mathbb{P}^4 \]
and let $G\subset \mathrm{\Aut}(S_t)$ be the Klein group generated by:
 \[ \begin{array}{ll}
 &[X_0:X_1:X_2:X_3:X_4]\mapsto [-X_0:X_1:X_2:X_3:X_4]\\
  & [X_0:X_1:X_2:X_3:X_4]\mapsto [X_0:X_4:X_3:X_2:X_1]. \\
\end{array}\]
Then $rk(\mathrm{Pic}(S_t)^G)=2$ and there are exactly two conic bundles $\pi_1,\pi_2\colon S_t\to \mathbb{P}^1$ that are $G$-invariant (Proposition \ref{2elemconj2}). This gives two subgroups of the de Jonquières group, which are not conjugate as one is acting trivially on the base (and thus belongs to $\PGL(2,\kk(x))$) whereas the other does not (again Proposition \ref{2elemconj2}). The two groups then are in families \ref{I} and \ref{J} respectively (see Theorem \ref{theorem1}).}


\bigskip
\bigskip
\bigskip

We now state the main result about conjugations between $p$-elementary non-cyclic subgroups of the Cremona group:

\bigskip\bigskip

\begin{theolettre} \label{theorem2} $ $
We assume $\car(\kk) \neq p$.
\bigskip

\begin{enumerate}[label=\textbf{(\arabic*)}, leftmargin=2em, labelsep=1em]
\item \label{theorem21} The only possible conjugation by a birational map between two subgroups in two different families is between a subgroup in family \ref{I} and a subgroup in family \ref{J}.
Such a conjugation exists if and only if the two groups are birationally conjugate to the group of 
\hyperlink{Amazing}{\textbf{\textit{the amazing result}}} 
(see above).
\item \label{theorem22} Two subgroups in the same family in the case $\mathbb{D}$ are conjugate by a birational map if and only if they are conjugate by an isomorphism.
\item \label{theorem23} Two subgroups in the same family in the case $\mathbb{C}\mathbb{A}$ are conjugate by a birational map if and only if they are conjugate by a de Jonquières map.
\item \label{theorem24} Two subgroups in the same family in the case $\mathbb{C}\mathbb{P}$ are conjugate by a birational map if and only if they are equal.
\item \label{theorem25} Two $2$-elementary subgroups of $\J$ of the form
$G= L \times R, G' = L' \times R'$, where $L,L' \subset \PGL(2,\kk(x)), R,R' \subset \PGL(2,\kk)$ from the list of Proposition \ref{theo2} are conjugate in $\J$ if and only if $R = R'$ and $L, L'$ have the same determinant in $\kk(x)^\times$ modulo the action of the normalizer of $R$ in $\PGL(2,\kk)$.
\end{enumerate}
\end{theolettre}

\bigskip
\bigskip

\begin{remark} $ $

The term {\textbf{\textit{amazing result}}} was first introduced by Jérémy Blanc (see \cite[Lemma 8.1.12]{blancthesis}). However, the uniqueness that truly makes this result 'amazing' is established here in Theorem \ref{theorem2} Assertion \ref{theorem21}, confirming its central role in our classification.
\end{remark}
\newpage

\subsection{Theorem \ref{theorem3}: An insight into the case $\car(\kk) = p$}
$ $

\bigskip
\bigskip
\bigskip

The study is different when $\car(\kk)=p$. We do not give the full classification but illustrate some aspects different from the case $\car(\kk)\neq p$ in Theorem \ref{theorem3}.

\bigskip
\bigskip

\begin{theolettre} \label{theorem3} $ $
We assume $\car(\kk) = p$.
\bigskip

\begin{enumerate}[label=\textbf{(\arabic*)}, leftmargin=2em, labelsep=1em]
\item \label{theorem31} Let $G$ be a $p$-elementary subgroup of the de Jonquières group. Then $G$ is conjugate in the de Jonquières group to a subgroup of the form $L \times R$
where:
\begin{itemize}
\item $R$ is a subgroup of $\{\left[\begin{smallmatrix}  1 & t   \\ 0 & 1  \end{smallmatrix} \right] | t \in \kk \} \subset \PGL(2,\kk)$.
\item $L$ is a subgroup of
 $\{\left[\begin{smallmatrix}  1 & t   \\ 0 & 1  \end{smallmatrix} \right] | t \in \kk(x)^R \} \subset \PGL(2,\kk(x))$

or of $\{\left[\begin{smallmatrix}  a & bg   \\ b & a  \end{smallmatrix} \right] | a,b \in \kk(x)^R, a^2 + b^2 g \neq 0 \} \subset \PGL(2,\kk(x))$ (if $p=2$), where $g \in \kk[X]^*$ is a monic polynomial with simple roots.
\end{itemize}
\bigskip

\item \label{theorem32} We assume $\car(\kk) = p = 2$. Every Klein subgroup of $\Aut(\mathbb{P}^2)$ is conjugate in $\Aut(\mathbb{P}^2)$ to one of the following subgroups: $$\quad \quad G_{1} = \langle \left[\begin{smallmatrix}  1 & 0 & 0   \\
 0 & 1 & 1  \\
 0 & 0 & 1
\end{smallmatrix} \right]  , \left[\begin{smallmatrix}  1 & 0 & 1   \\
 0 & 1 & 0  \\
 0 & 0 & 1
\end{smallmatrix} \right] \rangle  , \quad G_{2} = \langle \left[\begin{smallmatrix}  1 & 1 & 0   \\
 0 & 1 & 0 \\
 0 & 0 & 1
\end{smallmatrix} \right] , \left[\begin{smallmatrix}  1 & 0 & 1  \\
 0 & 1 & 0 \\
 0 & 0 & 1
\end{smallmatrix} \right] \rangle, \quad G_{3,t} = \langle \left[\begin{smallmatrix}  1 & 1 & 0  \\
 0 & 1 & 0 \\
 0 & 0 & 1
\end{smallmatrix} \right] , \left[\begin{smallmatrix}  1 & t & 0  \\
 0 & 1 & 0 \\
 0 & 0 & 1
\end{smallmatrix} \right] \rangle ,$$
Furthermore:
\begin{itemize} 
\item
The subgroups $G_1, G_2, G_{3,t}$ are pairwise not conjugate in $\Aut(\mathbb{P}^2)$.
\item
The subgroups $G_{3,t}, G_{3,t'}$ are conjugate in $\Aut(\mathbb{P}^2)$ if and only if:
$$t' \in \left\{t,t+1,\frac{1}{t},\frac{t+1}{t},\frac{1}{t+1},\frac{t}{t+1}\right\}$$
\item
For every $t \in \kk \setminus \{ 0,1\}$, the subgroups $G_1$ and $G_{3,t}$ are conjugate by a birational map.
\item
The subgroups $G_1$ and $G_2$ are not conjugate by a birational map.
\end{itemize}
\end{enumerate}
\end{theolettre}
\bigskip
\bigskip

\begin{remark} $ $

When the characteristic of the field is $p$, the de Jonquières group has subgroups isomorphic to $(\mathbb{Z}/p)^r$ for any $r \in \mathbb{N}$.
\end{remark}
\bigskip
\bigskip

\begin{remark} $ $

The two subgroups $G_1$ and $G_2$ defined in Theorem \ref{theorem3} Assertion \ref{theorem32} are both Klein subgroups of the Cremona group, conjugate to subgroups of de Jonquières group, both having no non-trivial element fixing a non-rational curve, and each of them acts with fixed points. The fixed points then do not determine the conjugacy class in that case, contrary to the case $\car(\kk) \neq p = 2$.
\end{remark}

\newpage

\section{Fundamental lemmas on subgroups of the Cremona group}
\label{appendix0} $ $

The study of finite subgroups of the Cremona group relies on the ability to realize these purely algebraic objects as groups of automorphisms of smooth rational surfaces.
In this section, we recall the fundamental definitions and results regarding minimal pairs and finite subgroups of the Cremona group. In particular we state the important Proposition \ref{prop1} which justifies the split of our study of finite subgroups between Part \ref{part2} and Part \ref{part3}.

\subsection{Minimality of finite subgroups with Picard rank $1$} $ $

\begin{defi} $ $
\noindent Let $S$ be a rational surface and let $G$ be a subgroup of the automorphism group $\Aut(S)$.
We say that the pair $(G,S)$ is minimal if any $G$-equivariant birational morphism 
$\phi \colon S \rightarrow S'$ such that $\phi G \phi^{-1} \subset \Aut(S')$ is an isomorphism.
\end{defi}

\begin{prop} $ $
Let $S$ be a del Pezzo surface and let $G \subset  \Aut(S)$ be a finite subgroup.
If $\rank(\Pic(S)^G) =1$, then $(G,S)$ is minimal.
\end{prop}

\begin{proof} $ $
By contradiction, let's assume that $(G,S)$ is not minimal.
Then there exists a $G$-equivariant morphism $\phi \colon S \to S'$
to a smooth surface $S'$ with 
$\phi G \phi^{-1} \subseteq \Aut(S')$. Let $E$ be a $(-1)$-curve contracted by $\phi$. Then 
for all $g \in G$ the $(-1)$-curve $gE$ is contracted by $\phi$ as well.
Now, for $g, g' \in G$ the $(-1)$-curves $gE$, $g'E$ are either disjoint or they coincide.
Hence, there exists a non-zero $G$-invariant divisor $D = \sum_{i=1}^r E_i \in \Pic(S)^G$ which is a sum of disjoint  $(-1)$-curves $E_i$. 
We write $D = \dfrac{a}{b} K_S$ where $a,b \neq 0$ are integers. Then we can produce the 
following contradiction
$0 \geq  - r b^2 = (b D)^2 = (a K_S)^2 = a^2 {K_S}^2 \geq 1$.
\end{proof}

\subsection{A dichotomy result for minimal pairs} $ $

\begin{defi} $ $
\noindent Let $S$ be a rational surface.
Let $\pi : S \rightarrow \mathbb{P}^1$ be a morphism. We say that $(S,\pi)$ is a conic bundle if a general fibre of $\pi$ is isomorphic to $\mathbb{P}^1$ and if there is a finite number of exceptions: these singular fibres are the union of rational curves $F$ and $F'$ such that $F^2 = {F'}^2 = -1$ and $F F' = 1$.
\end{defi}

\begin{prop} \label{prop1} $ $
\begin{enumerate}[left=0pt]
\item \label{Existence} Every finite subgroup of $\Cr$ is conjugate to a subgroup
of the automorphism group of a smooth rational surface.

\item \label{Uniqueness} Let $S$ be a smooth rational surface and let $G$ be a 
finite subgroup of $\Aut(S)$ such that 
$(G,S)$ is minimal.
Then one of the following holds:
\begin{itemize}[left=0pt]
\item $\rank(\Pic(S)^G) = 1$ and the surface $S$ is a del Pezzo surface. Furthermore, if $S$ is not isomorphic to $\mathbb{P}^2$ or $\mathbb{P}^1 \times \mathbb{P}^1$, then $\deg(S)$ divides the order of $G$.
\item $\rank(\Pic(S)^G) = 2$ and the surface $S$ has a conic bundle structure invariant by $G$.
\end{itemize}
\end{enumerate}
\end{prop}

\begin{proof} $ $

    \eqref{Existence}: This can be found in \cite[Theorem~1.4]{FeEi2002Resolution-of-inde}
    or \cite[Proposition~2.2.3]{blancthesis}.

    \eqref{Uniqueness}:
    This is a reformulation of \cite[Lemma 4.2.4]{blancthesis} and \cite[Proposition 2.3.1]{blancthesis}.
\end{proof}

\subsection{Conservation of non-rational curves by birational maps} $ $

The following Lemma \ref{ultimatelemma3} is used only in the proof of Theorem \ref{theorem2}.

\begin{lemme} \label{ultimatelemma3} $ $
Let $G, G' \subset \Cr$ be two subgroups. We assume they are conjugate by a birational map.
The subgroup $G$ contains a non-trivial element fixing a non-rational curve if and only if the subgroup $G'$ contains a non-trivial element fixing a non-rational curve.
\end{lemme}

\begin{proof} $ $
The image of a non-rational curve by a birational map is still a non-rational curve.
\end{proof}

\newpage
\textbf{Organization} 

\bigskip\bigskip\bigskip\bigskip

\begin{enumerate}[label=\textbf{(\arabic*)}, leftmargin=2em, labelsep=1em]
\item \hspace*{1.4em} Part \ref{part1} is dedicated to presenting the main results and providing some preliminary lemmas that we will use at various places in the text. In particular, we state Proposition \ref{prop1} that allows us to reduce the work to the study of subgroups of the de Jonquières group (Part \ref{part2}) and subgroups of automorphisms of del Pezzo surfaces (Part \ref{part3}). 
\bigskip \bigskip

\item \hspace*{1.4em} Part \ref{partnew} serves as a technical preamble, gathering essential results on subgroups of $\PGL(n,\KK)$. By isolating these preliminary results, we establish a robust framework that simplifies the classification of subgroups in the subsequent parts. This self-contained part acts as a reference for the proofs in futures parts. \bigskip \bigskip

\item \hspace*{1.4em} Part \ref{part2} is devoted to the study of $p$-elementary subgroups of the de Jonquières group up to conjugation in this group. This requires various tools including cohomology and linear algebra. A part of this study is also done in the case where the characteristic of the field is equal to $p$. \bigskip \bigskip

\item \hspace*{1.4em} Part \ref{part3} explores the $p$-elementary subgroups of automorphisms of some del Pezzo surfaces. The reader can observe that we don't study all del Pezzo surfaces, rather we focus on del Pezzo surfaces of degree lower or equal than $4$, and on the Hirzebruch surface $\mathbb{P}^1 \times \mathbb{P}^1$. That is because we don't need the results about the other del Pezzo surfaces for our study. Examples of tools used in this part are the Riemann-Hurwitz formula, the Picard rank formula, and group actions. \bigskip \bigskip

\item \hspace*{1.4em} Part \ref{part4} investigates potential conjugations between the representatives established in Parts \ref{part2} and \ref{part3}, primarily through the Iskovskikh-Sarkisov program. Our analysis demonstrates that this list is remarkably robust, 
as we identify only one instance of conjugation between distinct families; this corresponds to
\hyperlink{Amazing}{\textbf{\textit{the amazing result}}} (see Subsection \ref{amazingandsinequanon})

\bigskip \bigskip

\item \hspace*{1.4em} Part \ref{part5} is entirely dedicated to the case where $\car(\kk)=p$. More specifically, we study some $2$-elementary subgroups of $\Aut(\mathbb{P}^2)$ when the characteristic of the field is $2$. This allows us to have a first impression of what can happen when $\car(\kk)=p$, and observe some major differences with the case where $\car(\kk) \neq p$.
This part is independent of Parts \ref{part2}, \ref{part3} and \ref{part4}. \bigskip \bigskip

\item \hspace*{1.4em} In Part \ref{part6} we summarize all intermediate steps from previous parts in order to prove the three theorems. \bigskip \bigskip

\item \hspace*{1.4em} 
Finally, Part \ref{part7} consists of three appendices. 
In Appendix \ref{bratac0} we prove a group-theoretic result used in our classification.
Appendix \ref{appendixB} provides some results on automorphisms of del Pezzo surfaces; it notably covers the embedding of automorphism group for degrees $1$, $2$ and $3$, as well as some properties of Geiser and Bertini involutions, relying on cohomological methods. 
Finally Appendix \ref{appendixA} establishes a formula for the rank of the invariant Picard group, a recurring tool in our study. 
\end{enumerate}

\newpage

\begin{center}
    \begin{minipage}{0.85\textwidth}
        \centering
        \itshape

\vspace{30mm}

This article is the PhD thesis of the author, which was done under the supervision of Jérémy Blanc and Immanuel van Santen at the University of Basel.

        \vspace{24mm}
        
    The author is grateful to Jean-Pierre Serre for his careful reading of the introduction and for his valuable corrections.
        
    \end{minipage}
\end{center}


\newpage

\part{Subgroups of $\PGL(n, \KK)$} \label{partnew} $ $

\bigskip\bigskip

This part is dedicated to presenting and proving several results concerning subgroups of the group $\PGL(n, \KK)$ that will be used throughout this article. As suggested by Proposition  \ref{prop1}, the classification of finite subgroups of the Cremona group requires a preliminary study of finite subgroups of the de Jonquières group and of finite subgroups of automorphism groups of del Pezzo surfaces. Since these groups are closely related to subgroups of $\PGL$, their study is crucial. To ensure a focused presentation, we regroup all necessary preliminary results in this part. The reader may skip this tedious part on a first reading.
\\


We remind the reader that $\KK \supseteq \kk$ is a field extension. This part is organized as follows:

\begin{itemize}
\item We start this part with an easy case : Section \ref{21new} is dedicated to the classification of $3$-elementary non-cyclic subgroups of $\PGL(3,\kk)$ in characteristic not $3$.

\item In Section \ref{22new} we tackle a more difficult case : The study of $p$-elementary subgroups of $\PGL(2,\KK)$ in any characteristic. This will be used for the study of de Jonquières subgroups (Part \ref{part2}).

\item In Section \ref{23new} we give a thorough study of two specific cases used for the study of subgroups of automorphisms of quartic and cubic del Pezzo surfaces (Section \ref{quarticsection} and Section \ref{cubicsection}).

\item We also include Section \ref{24new}, which contains independent results regarding the case where $\car(\KK) = p$, although they are not used in this study.
\end{itemize}

\bigskip\bigskip

We start this part by presenting a general result, stated in Proposition \ref{prop3}:

\begin{prop} \label{prop3} $ $

Let $n \in \mathbb{N}_{\geq 1}$. We assume $p$ does not divide $n$ and $p \neq \car(\kk)$.
Let $G$ be a $p$-elementary subgroup of $\PGL(n,\kk)$.
Then $G$ is conjugate to a subgroup of the diagonal torus of $\PGL(n,\kk)$.
\end{prop}

\begin{proof} $ $
This follows from the proof of \cite[Lemma 3.1]{beauville}.
\end{proof}

\newpage

\section{The $3$-elementary subgroups of $\PGL(3,\kk)$ when $\car(\kk) \neq 3$} \label{21new} $ $

We assume $\car(\kk) \neq p = 3$ in this section.
The goal of this section is to classify the $3$-elementary non-cyclic subgroups of $\PGL(3,\kk)$ in the case of characteristic not equal to $3$. The result of the section is Proposition \ref{p21}. We first give the classification of subgroups isomorphic to $\mathbb{Z}/3$ in the following Lemma \ref{lemmep2}, which will be required for the classification of $3$-elementary subgroups.

\bigskip

\subsection{Subgroups of $\PGL(3,\kk)$ isomorphic to $\ZZ/3$}

\begin{lemme} \label{lemmep2} $ $
Let $G$ be a subgroup of $\PGL(3,\kk)$ isomorphic to $\mathbb{Z}/3$.

Then $G$ is conjugate to one of the two following groups: $$\langle \left[\begin{smallmatrix}  1 & 0 & 0   \\
 0  &    j & 0   \\
0  & 0 & j^2
\end{smallmatrix} \right] \rangle, \quad \langle \left[\begin{smallmatrix}  1 & 0&  0  \\
0  &    1 & 0  \\
 0 & 0 & j
\end{smallmatrix} \right] \rangle$$
\end{lemme}

\begin{proof} $ $
It follows directly from the study of the multiplicities of the eigenvalues of elements of order 3 of $\PGL(3,\kk)$.
\end{proof}

\bigskip

\subsection{Subgroups of $\PGL(3,\kk)$ isomorphic to $(\ZZ/3)^r$, $r > 1$}

\begin{prop} \label{p21} $ $
Let $G$ be a $3$-elementary non-cyclic subgroup of $\PGL(3,\kk)$. \newline
Then $G \simeq (\ZZ/3)^2$ and $G$ is conjugate in $\PGL(3,\kk)$ to one of the two following groups: $$\langle \left[\begin{smallmatrix}  1 & 0 & 0   \\
  0 &    j & 0 \\
  0 & 0 & j^2
\end{smallmatrix} \right]  ,\left[\begin{smallmatrix}  0 & 1 &  0 \\
 0 &  0   & 1 \\
1 &  0 & 0 
\end{smallmatrix} \right] \rangle , \quad \langle \left[\begin{smallmatrix}  1 & 0 & 0  \\
 0 &    j & 0 \\
 0 & 0 & j^2
\end{smallmatrix} \right]  , \left[\begin{smallmatrix}  1 &  0& 0   \\
 0 &    1 & 0 \\
 0 &  0 & j
\end{smallmatrix} \right] \rangle$$
Furthermore, these two groups are not conjugate by birational map.
\end{prop}


\begin{proof} $ $
Let $G$ be a 3-elementary non-cyclic subgroup of $\PGL(3,\kk)$.

According to Lemma \ref{lemmep2}, up to conjugation $G$ contains $\langle \left[\begin{smallmatrix}  1 & 0 & 0   \\
 0 &    j & 0 \\
 0 & 0 & j^2
\end{smallmatrix} \right] \rangle$ or $\langle \left[\begin{smallmatrix}  1 & 0 &  0  \\
 0 &    1 & 0 \\
 0 & 0 & j
\end{smallmatrix} \right] \rangle$.
\begin{itemize}
\item Case $ \left[\begin{smallmatrix}  1 & 0 & 0     \\
0 & j & 0   \\
0 & 0 & j^2   \\
\end{smallmatrix}\right] \in G$ :
The centralizer of $\left[\begin{smallmatrix}  1 & 0 & 0     \\
0 & j & 0   \\
0 & 0 & j^2   \\
\end{smallmatrix}\right]  $ is the following subgroup: $$\left\{ \left[\begin{smallmatrix}  a & 0 & 0     \\
0 & b & 0   \\
0 & 0 & c   \\
\end{smallmatrix}\right]  | a,b,c \in \kk^\times \right\} \rtimes \langle \left[\begin{smallmatrix}  0 & 1 & 0     \\
0 & 0 & 1   \\
1 & 0 & 0  \\
\end{smallmatrix}\right]  \rangle$$
The subgroup $G$ is non-cyclic so there exists $g \in G$, $g \notin \langle \left[\begin{smallmatrix}  1 & 0 & 0     \\
0 & j & 0   \\
0 & 0 & j^2   \\
\end{smallmatrix}\right] \rangle$ of order 3. Up to square, \textit{i.e.} replacing $g$ by $g^2$, we can assume $g = \left[\begin{smallmatrix}  a & 0 & 0     \\
0 & b & 0   \\
0 & 0 & c  \\
\end{smallmatrix}\right] $ or $g = \left[\begin{smallmatrix}  0 & a & 0     \\
0 & 0 & b   \\
c & 0 & 0  \\
\end{smallmatrix}\right] $.

If $g$ is diagonal, then since $g$ is of order $3$ and $g \notin \langle \left[\begin{smallmatrix}  1 & 0 & 0     \\
0 & j & 0   \\
0 & 0 & j^2   \\
\end{smallmatrix}\right] \rangle$, we get $g = \left[\begin{smallmatrix}  1 & 0 & 0     \\
0 & 1 & 0   \\
0 & 0 & j   \\
\end{smallmatrix}\right] $ up to conjugation leaving $\left[\begin{smallmatrix}  1 & 0 & 0     \\
0 & j & 0   \\
0 & 0 & j^2   \\
\end{smallmatrix}\right] $ invariant and up to square. So $\langle \left[\begin{smallmatrix}  1 & 0 & 0     \\
0 & j & 0   \\
0 & 0 & j^2   \\
\end{smallmatrix}\right]  , \left[\begin{smallmatrix}  1 & 0 & 0     \\
0 & 1 & 0   \\
0 & 0 & j   \\
\end{smallmatrix}\right] \rangle \subset G$. The centralizer of $\langle \left[\begin{smallmatrix}  1 & 0 & 0     \\
0 & j & 0   \\
0 & 0 & j^2   \\
\end{smallmatrix}\right]  , \left[\begin{smallmatrix}  1 & 0 & 0     \\
0 & 1 & 0   \\ 
0 & 0 & j   \\
\end{smallmatrix}\right] \rangle$ is the diagonal, therefore if an element of order $3$ commutes with $\langle \left[\begin{smallmatrix}  1 & 0 & 0     \\
0 & j & 0   \\
0 & 0 & j^2   \\
\end{smallmatrix}\right]  , \left[\begin{smallmatrix}  1 & 0 & 0     \\
0 & 1 & 0   \\
0 & 0 & j   \\
\end{smallmatrix}\right] \rangle$, it is contained in $\langle \left[\begin{smallmatrix}  1 & 0 & 0     \\
0 & j & 0   \\
0 & 0 & j^2   \\
\end{smallmatrix}\right]  , \left[\begin{smallmatrix}  1 & 0 & 0     \\
0 & 1 & 0   \\
0 & 0 & j   \\
\end{smallmatrix}\right] \rangle$. So $G = \langle \left[\begin{smallmatrix}  1 & 0 & 0     \\
0 & j & 0   \\
0 & 0 & j^2   \\
\end{smallmatrix}\right]  , \left[\begin{smallmatrix}  1 & 0 & 0     \\
0 & 1 & 0   \\
0 & 0 & j   \\
\end{smallmatrix}\right] \rangle$.

If $g$ is not diagonal, then since $g$ is of order $3$ and $g \notin \langle \left[\begin{smallmatrix}  1 & 0 & 0     \\
0 & j & 0   \\
0 & 0 & j^2   \\
\end{smallmatrix}\right] \rangle$, we get $g = \left[\begin{smallmatrix}  0 & 1 & 0     \\
0 & 0 & 1  \\
1 & 0 & 0   \\
\end{smallmatrix}\right] $ up to conjugation leaving $\left[\begin{smallmatrix}  1 & 0 & 0     \\
0 & j & 0   \\
0 & 0 & j^2   \\
\end{smallmatrix}\right] $ invariant and up to square.

So $\langle \left[\begin{smallmatrix}  1 & 0 & 0     \\
0 & j & 0   \\
0 & 0 & j^2   \\
\end{smallmatrix}\right]  , \left[\begin{smallmatrix}  0 & 1 & 0     \\
0 & 0 & 1   \\
1 & 0 & 0   \\
\end{smallmatrix}\right] \rangle \subset G$. The centralizer of $\langle \left[\begin{smallmatrix}  1 & 0 & 0     \\
0 & j & 0   \\
0 & 0 & j^2   \\
\end{smallmatrix}\right] , \left[\begin{smallmatrix}  0 & 1 & 0     \\
0 & 0 & 1  \\
1 & 0 & 0   \\
\end{smallmatrix}\right] \rangle$ is itself, therefore if an element of order $3$ commutes with $\langle \left[\begin{smallmatrix}  1 & 0 & 0     \\
0 & j & 0   \\
0 & 0 & j^2   \\
\end{smallmatrix}\right]  , \left[\begin{smallmatrix}  0 & 1 & 0     \\
0 & 0 & 1   \\
1 & 0 & 0  \\
\end{smallmatrix}\right] \rangle$, it is contained in $\langle \left[\begin{smallmatrix}  1 & 0 & 0     \\
0 & j & 0   \\
0 & 0 & j^2   \\
\end{smallmatrix}\right]  , \left[\begin{smallmatrix}  0 & 1 & 0     \\
0 & 0 & 1   \\
1 & 0 & 0   \\
\end{smallmatrix}\right] \rangle$. So $G = \langle \left[\begin{smallmatrix}  1 & 0 & 0     \\
0 & j & 0   \\
0 & 0 & j^2   \\
\end{smallmatrix}\right]  , \left[\begin{smallmatrix}  0 & 1 & 0     \\
0 & 0 & 1   \\
1 & 0 & 0   \\
\end{smallmatrix}\right] \rangle$.

\item Case $\left[\begin{smallmatrix}  1 & 0 & 0     \\
0 & j & 0   \\
0 & 0 & j^2   \\
\end{smallmatrix}\right]  \in G :$ The centralizer of $\left[\begin{smallmatrix}  1 & 0 & 0     \\
0 & j & 0   \\
0 & 0 & j^2   \\
\end{smallmatrix}\right] $ is the following subgroup:$$\left\{ \left[\begin{smallmatrix}  a & b & 0     \\
c & d & 0   \\
0 & 0 & 1   \\
\end{smallmatrix}\right]  | a,b,c,d,e \in \kk, ad - bc \neq 0 \right\}$$
As the subgroup $G$ is non-cyclic, there exists an element $g = \left[\begin{smallmatrix}  a & b & 0     \\
c & d & 0   \\
0 & 0 & 1   \\
\end{smallmatrix}\right]  \in G$, $g \notin \langle \left[\begin{smallmatrix}  1 & 0 & 0     \\
0 & 1 & 0   \\
0 & 0 & j   \\
\end{smallmatrix}\right] \rangle$. $g$ is of order $3$, so up to conjugation by an element of the form $\left[\begin{smallmatrix}  * & * & 0     \\
* & * & 0   \\
0 & 0 & 1   \\
\end{smallmatrix}\right] $, we get $g = \left[\begin{smallmatrix} j & 0 & 0     \\
0 & j^2 & 0   \\
0 & 0 & 1   \\
\end{smallmatrix}\right] $ or $g = \left[\begin{smallmatrix}  j^2 & 0 & 0     \\
0 & j & 0   \\
0 & 0 & 1  \\
\end{smallmatrix}\right] $. Hence $\langle \left[\begin{smallmatrix}  1 & 0 & 0     \\
0 & j & 0   \\
0 & 0 & j^2   \\
\end{smallmatrix}\right] , \left[\begin{smallmatrix}  1 & 0 & 0     \\
0 & 1 & 0   \\
0 & 0 & j   \\
\end{smallmatrix}\right] \rangle \subset G$. The centralizer of the subgroup $\langle \left[\begin{smallmatrix}  1 & 0 & 0     \\
0 & j & 0   \\
0 & 0 & j^2   \\
\end{smallmatrix}\right]  , \left[\begin{smallmatrix}  1 & 0 & 0     \\
0 & 1 & 0   \\
0 & 0 & j   \\
\end{smallmatrix}\right] \rangle$ is the diagonal, therefore if an element of order $3$ commutes with $\langle \left[\begin{smallmatrix}  1 & 0 & 0     \\
0 & j & 0   \\
0 & 0 & j^2   \\
\end{smallmatrix}\right]  , \left[\begin{smallmatrix}  1 & 0 & 0     \\
0 & 1 & 0   \\
0 & 0 & j   \\
\end{smallmatrix}\right] \rangle$, it is contained in $\langle \left[\begin{smallmatrix}  1 & 0 & 0     \\
0 & j & 0   \\
0 & 0 & j^2   \\
\end{smallmatrix}\right]  , \left[\begin{smallmatrix}  1 & 0 & 0     \\
0 & 1 & 0   \\
0 & 0 & j  \\
\end{smallmatrix}\right] \rangle$. So $G = \langle \left[\begin{smallmatrix}  1 & 0 & 0     \\
0 & j & 0   \\
0 & 0 & j^2   \\
\end{smallmatrix}\right]  , \left[\begin{smallmatrix}  1 & 0 & 0     \\
0 & 1 & 0   \\
0 & 0 & j   \\
\end{smallmatrix}\right] \rangle$.

\end{itemize}



These two groups are not conjugate by birational map because the subgroup $\langle \left[\begin{smallmatrix}  1 & 0 & 0     \\
0 & j & 0   \\
0 & 0 & j^2   \\
\end{smallmatrix}\right]  , \left[\begin{smallmatrix}  1 & 0 & 0     \\
0 & 1 & 0   \\
0 & 0 & j   \\
\end{smallmatrix}\right] \rangle$ has three fixed points whereas the subgroup $\langle \left[\begin{smallmatrix}  1 & 0 & 0     \\
0 & j & 0   \\
0 & 0 & j^2   \\
\end{smallmatrix}\right]  , \left[\begin{smallmatrix}  0 & 1 & 0     \\
0 & 0 & 1   \\
1 & 0 & 0   \\
\end{smallmatrix}\right] \rangle$ has none, \cite[page 1054-1056 Proposition A.2]{kollar}.
\end{proof}

\newpage

\section{Subgroups of $\PGL(2, \KK)$: A thorough study}  \label{22new} $ $

The result of this section is Proposition \ref{dublin1}.

\subsection{A description of centralizers}
\begin{defi} $ $
Let $g \in \KK^\times$.
We define $\mathcal{C}_g$ the centralizer of $\left[ \begin{smallmatrix} 0 & g \\ 1 & 0 \end{smallmatrix} \right]$ in $\PGL(2,\KK)$ and
$\mathcal{C}_g'$  the centralizer of $\left[ \begin{smallmatrix} 0 & g \\ 1 & 0 \end{smallmatrix} \right]$ in $\GL(2,\KK)$.

We also define
$\Delta_\KK = \set{\left[\begin{smallmatrix}  1 & t   \\
0 &    1
\end{smallmatrix} \right]}{t \in \KK} \subset \PGL(2,\KK)$.
\end{defi}

In the following Lemma \ref{commutateur}, we describe explicitely the centralizers 
$\mathcal{C}_g$, $\mathcal{C}'_g$ for $g \in \KK^\times$
and the centralizers of
$\sigma = \left[ \begin{smallmatrix} 1 & 0 \\ 0 & -1 \end{smallmatrix} \right]$,
$\tau = \left[ \begin{smallmatrix} 0 & 1 \\ 1 & 0 \end{smallmatrix} \right]$.

\begin{lemme} \label{commutateur} $ $
Let $g \in \KK^\times$. Then we have:

\begin{enumerate}[left=0pt]

\item \label{commutateurassertion1} 
$\mathcal{C}_g' = \set{ \left[ \begin{smallmatrix} a & bg \\ b & a \end{smallmatrix} \right] }
{a,b \in \KK \, , \ a^2 \neq g b^2}$.

\item \label{commutateurassertion2} $\mathcal{C}_g = 
\set{\left[ \begin{smallmatrix} a & bg \\ b & a \end{smallmatrix} \right]}
    {a,b \in \KK \, , \ a^2 \neq g b^2 } \cup 
    \set{ \left[ \begin{smallmatrix} a & -bg \\ b & -a \end{smallmatrix} \right]}
    {a,b \in \KK \, , \ a^2 \neq g b^2 }$.

\item \label{commutateurassertion3} The set of elements of order $2$ in $\mathcal{C}_{g}$
is equal to 
\[
    \left\{ \left[\begin{smallmatrix} 0 & g \\ 1 & 0 \end{smallmatrix}\right] \right\} 
    \cup \set{ \left[ \begin{smallmatrix} a & -bg \\ b & -a \end{smallmatrix} \right]}
    {a,b \in \KK \, , \ a^2 \neq g b^2 } \, .
\]

\item \label{commutateurassertion4} Two distinct elements in 
$\set{ \left[ \begin{smallmatrix} a & -bg \\ b & -a \end{smallmatrix} \right]}
{a,b \in \KK \, , \ a^2 \neq g b^2 }$ either don't commute or their product is
equal to $\left[\begin{smallmatrix} 0 & g \\ 1 & 0 \end{smallmatrix}\right]$.

\item \label{theo4} \label{commutateursigma}
Assume $\car(\KK) \neq 2$. The centralizer of $\sigma$ in $\PGL(2,\KK)$ is the following group:
\[
    \set{ \left[\begin{smallmatrix}
        1 & 0 \\
        0 & \lambda
        \end{smallmatrix}\right]}
    {\lambda \in \KK^\times} \rtimes \langle \tau \rangle \simeq 
\KK^{\times} \rtimes \mathbb{Z}/2 \, .
\]

\item \label{commutateursigmatau} 
Assume $\car(\KK) \neq 2$. The subgroup $\langle \sigma, \tau \rangle$ of $\PGL(2, \KK)$ is self-centralizing.
\end{enumerate}
\end{lemme}

\begin{remark} $ $
    If $\car(\KK) = 2$, then $\mathcal{C}_g$ is equal to 
    $\set{\left[ \begin{smallmatrix} a & bg \\ b & a \end{smallmatrix} \right]}
    {a,b \in \KK \, , \ a^2 \neq g b^2 }$.
\end{remark}

\begin{proof} $ $
\eqref{commutateurassertion1},~\eqref{commutateurassertion2},~\eqref{commutateurassertion3}: 
These are direct computations. \\
\eqref{commutateurassertion4}: Let 
$A =  
\left[\begin{smallmatrix} a & -bg \\ b & -a \end{smallmatrix}\right]
$ and let
$T =
\left[\begin{smallmatrix} s & -tg \\ t & -s \end{smallmatrix}\right]
$. Then we get for the products 
\[
    AA' = \left[\begin{smallmatrix}
        -b g t + a s & & b g s - a g t \\
        b s - a t & & -b g t + a s
        \end{smallmatrix}\right]
    \quad \textrm{and} \quad
    A'A = \left[\begin{smallmatrix}
        -b g t + a s & & -b g s + a g t \\
        -b s + a t & & -b g t + a s
        \end{smallmatrix}\right] \, .
\] 
Assume $A$, $A'$ commute. Then either $AA' = A'A = I$ or 
$AA' = A'A = \left[ \begin{smallmatrix} 0 & g \\ 1 & 0 \end{smallmatrix} \right]$.
\\
\eqref{commutateursigma}: Let 
$A = \left[\begin{smallmatrix}  a & b     \\
c & d   \\
\end{smallmatrix}\right] \in \PGL(2,\KK)$. Then, $A$ commutes with $\sigma$

$\Leftrightarrow \left[\begin{smallmatrix}  a & b     \\
c & d   \\
\end{smallmatrix}\right] \left[\begin{smallmatrix}  1 & 0     \\
0 & -1  \\
\end{smallmatrix}\right] = \left[\begin{smallmatrix}  1 & 0     \\
0 & -1  \\
\end{smallmatrix}\right] \left[\begin{smallmatrix}  a & b     \\
c & d   \\
\end{smallmatrix}\right]$

$\Leftrightarrow \ \left[\begin{smallmatrix}  a & -b     \\
c & -d   \\
\end{smallmatrix}\right] =  \left[\begin{smallmatrix}  a & b     \\
-c & -d   \\
\end{smallmatrix}\right]$

$\Leftrightarrow$ There exists $\lambda \in \KK^\times$ such that $a = \lambda a , - b = \lambda b, c = - \lambda c, - d = - \lambda d$

$\Leftrightarrow$ $a = d = 0$ or $b = c = 0$

$\Leftrightarrow A =  \left[\begin{smallmatrix}  a & 0     \\
0 & d   \\
\end{smallmatrix}\right]$ or $A =  \left[\begin{smallmatrix}  0 & b     \\
c & 0   \\
\end{smallmatrix}\right]$ \\
\eqref{commutateursigmatau}: We use~\eqref{commutateursigma}.
\end{proof}
\subsection{The determinant as a classification tool} $ $

We remind the reader that the determinant induces a group homomorphism
\[
    \det \colon \PGL(2, \KK) \to \KK^\times / (\KK^\times)^2 \, .
\]
We say that two elements or subgroups of $\PGL(2, \KK)$ have the same determinant, if they
have the same image in $\KK^\times / (\KK^\times)^2$ under $\det$.
Note, if two elements or two subgroups of $\PGL(2, \KK)$ are conjugate, 
then they have the same determinant.

\begin{lemme} \label{buda3} $ $
Let $U_\KK = \set{ \left[\begin{smallmatrix}  a & b   \\ 0 &    1 \end{smallmatrix} \right]}{a \in \KK^\times, b \in \KK} \subset \PGL(2,\KK)$.

\begin{enumerate}[left=0pt]
\item \label{bla1} Let $P = \left[\begin{smallmatrix}  t & x  \\
0 &    1
\end{smallmatrix} \right] \in U_\KK$, $A = \left[\begin{smallmatrix}  1 & y   \\
0 &    1
\end{smallmatrix} \right] \in \Delta_\KK$. Then $PAP^{-1} = \left[\begin{smallmatrix}  1 & ty   \\
0 &    1
\end{smallmatrix} \right]$.
\item \label{bla2} Let $A,B \in \Delta_\KK\setminus \{I\}$ and $P \in \PGL(2, \KK)$ such that $P AP^{-1} = B$. Then $P \in U_\KK$.
\item \label{bla3} Let $A \in \Delta_\KK \setminus \{I\}$. If $PAP^{-1} = A$
for some $P \in \PGL(2, \KK)$. Then $P \in \Delta_\KK$.
\end{enumerate}
\end{lemme}

\begin{proof} $ $
These are direct computations.
\end{proof}


\begin{lemme} 
    \label{LemCon} $ $
    Assume $A \in \PGL(2, \KK)$ is of order $2$. Then:
    \begin{enumerate}[left=0pt]
    \item \label{LemCon1} $A$ is conjugate to 
    $\left[\begin{smallmatrix}  0 & g  \\
    1 &    0
    \end{smallmatrix} \right]$ for some $g \in \KK^\times$.
    
    \item \label{LemCon2} For $g, h \in \KK^\times$,
    $\left[\begin{smallmatrix}  0 & g  \\
    1 &    0
    \end{smallmatrix} \right], \left[\begin{smallmatrix}  0 & h  \\
    1 &    0
    \end{smallmatrix} \right]$ are conjugate if and only if $\frac{g}{h}$ is a square.
    \end{enumerate}
\end{lemme}

\begin{proof} $ $

\eqref{LemCon1}: Let $B \in \GL(2,\KK)$ be a representative of 
$A$. There exists $g \in \KK^\times$ such that $B^2 = g I$. As $B$ 
is not a matrix of the form
$\left[\begin{smallmatrix}  t & 0  \\
    0 & t
\end{smallmatrix} \right]$ for some $t \in \KK^\times$, 
there exists $u_1 \in \KK^2, u_1 \neq 0$ not an eigenvector of $B$. 
For $u_2 = B u_1$ we get $B u_2 = g u_1$.  Then $(u_1, u_2)$ is a basis of $\KK^2$. 
This basis allows to conjugate $B$ to $\left[\begin{smallmatrix}  0 & g \\
1 &    0
\end{smallmatrix} \right]$.

\eqref{LemCon2}:
If 
$\left[\begin{smallmatrix}  0 & g  \\
1 &    0
\end{smallmatrix} \right]$,
$\left[\begin{smallmatrix}  0 & h  \\
1 &    0
\end{smallmatrix} \right]$
are conjugate, then they have the same determinant and hence, $g$, $h$ differ only
by a square in $\KK^\times$. On the other hand, if $r^2 = \frac{g}{h}$ for some $r \in \KK^\times$,
then
$\left[
    \begin{smallmatrix}  
        1 & 0  \\
        0 & r
    \end{smallmatrix} 
\right]
\left[
    \begin{smallmatrix}  
        0 & g  \\
        1 & 0
    \end{smallmatrix} 
\right]
\left[
    \begin{smallmatrix}  
        r & 0  \\
        0 & 1
    \end{smallmatrix} 
\right]
=
\left[
    \begin{smallmatrix}  
        0 & h  \\
        1 & 0
    \end{smallmatrix} 
\right]$.
\end{proof}

\begin{lemme} \label{buda10} $ $
    Assume  $\car(\KK) = p > 0$. Let $A \in \PGL(2,\KK)$ be of order $p$. Then:
    \begin{enumerate}[left=0pt]
    \item \label{Lemakro1} $A$ can't be diagonalized.
    \item \label{Lemakro1.5} If $A$ is triangular, then $A \in \Delta_\KK$.
    \item \label{Lemakro2} Assume $p > 2$. Then $A$ can be trigonalized.
    \item \label{Lemakro3} Assume $p = 2$. Then $A$ can be trigonalized if and only if $\det(A) = 1$.
    \end{enumerate}
\end{lemme}

\begin{proof} $ $
Let $B \in \GL(2,\KK)$ be a representative of $A$. Then $B^p = \lambda I$.
Now, all eigenvalues (inside an algebraic closure of $\KK$) are $p$-th roots of $\lambda$. 
Since $\car(\KK) = p$, we get that all eigenvalues of $B$ are the same.

\eqref{Lemakro1}: If $B$ can be diagonalized, then all eigenvalues of $B$
are contained in $\KK$ and they are the same. Thus, $A = I$.

\eqref{Lemakro1.5}: If $B$ is trigonal, then all eigenvalues of $B$ are contained in $\KK$
and they are the same, \textit{i.e.}~$A \in \Delta_\KK$. 

\eqref{Lemakro2}: From  $B^p = \lambda I$ we get 
$(\det B)^p = \lambda^2 $. Since $p > 2$, there exists $k$ with $p = 2k + 1$.
For $t = \lambda (\det B)^{-k} \in \KK^\times$ we get $t^p = \lambda^p \lambda^{-2k} = \lambda$. 
Then $B^p = t^p I$ and using that $\car(\KK) = p$ we get
$(B-tI)^p = 0$, \textit{i.e.}~$B-tI$ is nilpotent. 
Then, $B$ and hence $A$ can be trigonalized.

\eqref{Lemakro3}: By Lemma~\ref{LemCon}\eqref{LemCon1} we may assume
$A = 
    \left[\begin{smallmatrix}  
    0 & g  \\
    1 &    0
    \end{smallmatrix} \right]$.
Note that
$\left[\begin{smallmatrix}  1 & 1  \\
1 &    0
\end{smallmatrix} \right] \left[\begin{smallmatrix}  1 & 1  \\
0 &   1
\end{smallmatrix} \right] \left[\begin{smallmatrix}  1 & 1  \\
1 &    0
\end{smallmatrix} \right]^{-1}  =\left[\begin{smallmatrix}  0 & 1  \\
1 &    0
\end{smallmatrix} \right]$. Hence, the matrix $\left[\begin{smallmatrix}  0 & g  \\
1 &    0
\end{smallmatrix} \right]$ is conjugate to 
$\left[\begin{smallmatrix}  
    1 & 1  \\
    0 & 1
\end{smallmatrix} \right]$ if and only if $g$ is a square (use~Lemma~\ref{LemCon}\eqref{LemCon2}).
Now, the statement follows from the fact, that all matrices in $\Delta_\KK \setminus \{I\}$
are conjugate (use Lemma~\ref{buda3}\eqref{bla1}) and that $\Delta_\KK \setminus \{I\}$
is the set of all triangular matrices of order $2$ in $\PGL(2, \KK)$.
\end{proof}

\begin{lemme} \label{buda32} $ $
We assume $\car(\KK)=2$. Let $g \in \KK^\times$ be not a square. 
Then the restriction of $\det \colon \PGL(2, \KK) \to \KK^\times / (\KK^\times)^2$ 
to the centralizer $\mathcal{C}_g$ 
yields an injective
group homomorphism $\mathcal{C}_g \rightarrow \KK^\times / (\KK^\times)^2$.
\end{lemme}

\begin{proof} $ $
Let $\left[\begin{smallmatrix}  a & bg  \\
b &    a
\end{smallmatrix} \right] \in \mathcal{C}_g$ be an element of the kernel of this homomorphism.
Then there exists $c \in \KK^\times$ such that $a^2 + b^2 g = c^2$
and thus $b^2 g = (a+c)^2$ because $\car(\KK)=2$.
This yields $b=0$, as $g$ is not a square and thus
$\left[\begin{smallmatrix}  a & bg  \\
b &    a
\end{smallmatrix} \right] = I$.
\end{proof}

\begin{prop} \label{buda35} $ $
We assume $\car(\KK)=2$. Let $G, G'$ be $2$-elementary subgroups of $\PGL(2,\KK)$ 
such that $G$ contains an element that cannot be trigonalized.
Then, $\det(G) = \det(G')$ if and only if $G, G'$ are conjugate.
\end{prop}

\begin{proof} $ $
If $G, G'$ are conjugate, then we get $\det(G) = \det(G')$.

Assume $\det(G) = \det(G')$. 
Let $A \in G$ be 
an element that cannot be trigonalized. Then up to conjugation,
we may assume
$A = \left[\begin{smallmatrix}  
    0 & g  \\
    1 & 0
\end{smallmatrix} \right]$ 
for a non-square $g \in \KK^\times$ by Lemma~\ref{LemCon}\eqref{LemCon1} and 
Lemma~\ref{buda10}\eqref{Lemakro3}. 
As $\det(G) = \det(G')$, there exists
$A' \in G'$ such that $\det(A') = \det(A)$. By Lemma~\ref{LemCon}\eqref{LemCon1}
we may assume that 
$A' = \left[\begin{smallmatrix}
    0 & g' \\ 1 & 0
\end{smallmatrix}\right]$ for some $g' \in \KK^\times$. As $\det(A) = \det(A')$, 
the elements $g, g'$ differ only by a square in $\KK^\times$ and thus $A, A'$ are conjugate 
by Lemma~\ref{LemCon}\eqref{LemCon2}. Hence, we may assume that 
$A = A'$ and thus $G, G' \subseteq \mathcal{C}_g$. Thus 
$G = G'$ by Lemma~\ref{buda32}.
\end{proof}

\subsection{Classification of $p$-elementary subgroups of $\PGL(2,\KK)$} $ $

We now state the main result of this section:

\begin{prop} \label{dublin1} $ $
Let $G \simeq (\ZZ/p)^r$ be a subgroup of $\PGL(2,\KK)$ with $r \geq 1$.
\begin{enumerate}[left=0pt]
\item \label{orsaydublin11} If $\car(\KK) \neq p > 2$, then $G$ is conjugate to 
$\langle\left[ \begin{smallmatrix} 1 & 0 \\ 0 & \xi \end{smallmatrix}\right] \rangle$, 
where $\xi^p = 1, \xi \neq 1$.
\item \label{orsaydublin12} If $\car(\KK) \neq p = 2$, then $r \in \{1, 2\}$ and $G$ either conjugate to 
$\langle\left[ \begin{smallmatrix} 0 & g \\ 1 & 0 \end{smallmatrix} \right]\rangle$ or to 
$\langle\left[ \begin{smallmatrix} 0 & g \\ 1 & 0 \end{smallmatrix} \right] , \left[ \begin{smallmatrix} a & - b g \\ b & - a \end{smallmatrix} \right] \rangle$ for some 
$g, a, b \in \KK$ such that $g \neq 0, a^2 \neq b^2 g$.

\begin{enumerate}[label=(\alph*), left=0pt]
\item \label{chargr2inv1} For $g, h \in \KK^\times$ the subgroups $\langle\left[ \begin{smallmatrix} 0 & g \\ 1 & 0 \end{smallmatrix} \right]\rangle, \langle\left[ \begin{smallmatrix} 0 & h \\ 1 & 0 \end{smallmatrix} \right]\rangle$ are conjugate if and only if $\frac{g}{h}$ is a square, \textit{i.e.}~they have the same determinant.

\item \label{chargr2inv2} For $g \in \KK^\times$ the subgroup $\langle\left[ \begin{smallmatrix} 0 & g \\ 1 & 0 \end{smallmatrix} \right]\rangle$ is conjugate to $\langle\left[ \begin{smallmatrix} 1 & 0 \\ 0 & -1 \end{smallmatrix} \right]\rangle$ if and only if $g$ is a square.

\item \label{chargr2inv3} Assume that either $\KK$ is algebraically closed or $\KK = \kk(x)$. 
Then two sub\-groups \newline $\langle\left[ \begin{smallmatrix} 0 & g \\ 1 & 0 \end{smallmatrix} \right] , \left[ \begin{smallmatrix} a & - b g \\ b & - a \end{smallmatrix} \right] \rangle$ and $\langle\left[ \begin{smallmatrix} 0 & g' \\ 1 & 0 \end{smallmatrix} \right] , \left[ \begin{smallmatrix} a' & - b' g' \\ b' & - a' \end{smallmatrix} \right] \rangle$ are conjugate if and only if they have the same determinant.

\item \label{chargr2inv4} If $\KK$ is algebraically closed, then $G$ is conjugate to $\langle\left[ \begin{smallmatrix} 1 & 0 \\ 0 & -1 \end{smallmatrix} \right]\rangle$ or to $\langle\left[ \begin{smallmatrix} 1 & 0 \\ 0 & -1 \end{smallmatrix} \right] , \left[ \begin{smallmatrix} 0 & 1 \\ 1 & 0 \end{smallmatrix} \right] \rangle$.
\end{enumerate}

\item \label{hukfg100} If $\car(\KK) = p$, then exactly one of the following cases occurs:
Either $G$ is conjugate to a subgroup of 
$\Delta_\KK$ containing $\left[\begin{smallmatrix}  1 & 1  \\
0 &   1
\end{smallmatrix} \right]$, or $p = 2$ and $G$ is conjugate to a subgroup of $\mathcal{C}_g$ 
containing $\left[\begin{smallmatrix}  0 & g  \\
1 &    0
\end{smallmatrix} \right]$, where
$g \in \KK^\times$ is not a square.

\begin{enumerate}[label=(\alph*), left=0pt]
\item \label{hukfg1} Furthermore, two subgroups of $\Delta_\KK$ 
are conjugate in $\PGL(2, \KK)$ if and only if they are conjugate by a diagonal element of $\PGL(2, \KK)$.

\item \label{hukfg2} Assume that $G$ is not conjugate to a subgroup
of $\Delta_\KK$. If $G' \subset \PGL(2, \KK)$ is another $2$-elementary subgroup, then 
$G$, $G'$ are conjugate in $\PGL(2, \KK)$ if and only if they have the same determinant.

\item \label{hukfg3} If $\KK$ is algebraically closed, then $G$ is conjugate to a subgroup of 
$\Delta_\KK$.
\end{enumerate}

\end{enumerate}
\end{prop}


\begin{proof} $ $

\begin{enumerate}[left=0pt]
\item This can be found in \cite[Lemma~2.1(b)]{beauville}.
\item Let $A \in G \setminus \{ I \}$. 
By Lemma~\ref{LemCon}\eqref{LemCon2}, $A$ is conjugate to 
$\left[\begin{smallmatrix}  0 & g \\
1 &    0
\end{smallmatrix} \right]$, and thus we may assume 
$\left[\begin{smallmatrix}  0 & g \\
    1 &    0
    \end{smallmatrix} \right] \in G \subset \mathcal{C}_g$.
By Lemma~\ref{commutateur}\eqref{commutateurassertion2}\eqref{commutateurassertion3}\eqref{commutateurassertion4} it follows that $r \leq 2$. \\
\ref{chargr2inv1}: This follows form Lemma~\ref{LemCon}\eqref{LemCon2}. \\
\ref{chargr2inv2}: This follows from~\ref{chargr2inv1}, since 
$\left[\begin{smallmatrix}  0 & 1 \\
1 &   0
\end{smallmatrix} \right]$,
$\left[\begin{smallmatrix}  1 & 0 \\
0 &    -1
\end{smallmatrix} \right]$ are conjugate.
\\
\ref{chargr2inv3}: This follows from \cite[Lemma 2.1.(c)]{beauville}.  \\
\ref{chargr2inv4}: This is a consequence of~\ref{chargr2inv2} and~\ref{chargr2inv3}.

\item 
Assume first that every element of $G$ can be trigonalized. 
As $G$ is abelian, we get that $G$ can be trigonalized. Using Lemma~\ref{buda10}\eqref{Lemakro1.5}
and Lemma~\ref{buda3}\eqref{bla1}
we get that $G$ is conjugate to a subgroup of $\Delta_\KK$ that contains
$\left[\begin{smallmatrix}
    1 & 1 \\ 0 & 1
\end{smallmatrix}\right]$. If $G$ contains an element that cannot be
trigonalized, then $p = 2$ (see Lemma~\ref{buda10}\eqref{Lemakro2}) and 
we may assume that $G$ contains $\left[\begin{smallmatrix}
    0 & g \\ 1 & 0
\end{smallmatrix}\right]$ for some non-square $g \in \KK^\times$ 
(see Lemma~\ref{LemCon}\eqref{LemCon1} and Lemma~\ref{buda10}\eqref{Lemakro3}).
In particular, $G$ is contained in $\mathcal{C}_g$. \\
\ref{hukfg1}: If $H, H'$ are non-trivial subgroups of $\Delta_\KK$ with $PHP^{-1} = H'$
for some $P \in \PGL(2, \KK)$, then $P$ is a triangular matrix by Lemma~\ref{buda3}\eqref{bla2}.
Using Lemma~\ref{buda3}\eqref{bla1} we may replace $P$ by a diagonal matrix. \\
\ref{hukfg2}: We have $p = 2$ and we may assume that 
$\left[\begin{smallmatrix}
    0 & g \\ 1 & 0
\end{smallmatrix}\right] \in G$ for a non-square
$g \in \KK^\times$. In particular, $\left[\begin{smallmatrix}
    0 & g \\ 1 & 0
\end{smallmatrix}\right]$ cannot be trigonalized, see 
Lemma~\ref{buda10}\eqref{Lemakro3}.
Proposition~\ref{buda35} implies that $G$, $G'$ are conjugate. \\
\ref{hukfg3}: We use that every element in $\KK$ is a square
if $\KK$ is algebraically closed.
\end{enumerate}
\end{proof}

\newpage

\section{Some subgroups of $\PGL(4, \kk)$ and $\PGL(5, \kk)$: Two specific cases}  \label{23new} $ $

The goal of this section is to classify up to conjugacy:
\begin{itemize}
\item The $3$-elementary non-cyclic subgroups of $\PGL(4,\kk)$ in characteristic not $3$. This is done in Subsection \ref{subsection165} and will be used in Section \ref{cubicsection}.
\item The $2$-elementary non-cyclic subgroups of $\PGL(5,\kk)$ in characteristic not $2$. This is done in Subsection \ref{subsection166} and will be used in Section \ref{quarticsection}.
\end{itemize}

We start this section by introducing the notion of the eigenvalue multiplicity map $\nu_G^n$ (Definition \ref{definition1611}) in Subsection \ref{apowerfultool}, which will serve as a powerful tool in both Subsections \ref{subsection165} and \ref{subsection166}.

\subsection{The multiplicity of the eigenvalues: A powerful tool for the classification of $p$-elementary subgroups of $\PGL(n, \kk)$} \label{apowerfultool} $ $

\begin{defi} \label{definition1611} $ $

Let $\overline{A} \in \PGL(n,\kk)$ where $A \in \GL(n,\kk)$.
We define $\rho^n(\overline{A})$ to be the multiplicities of the eigenvalues of $A$ in decreasing order, which does not depend on the choice of $A$.

Let $G \subset \PGL(n,\kk)$ be a finite subgroup. We define the following map:
$$\footnotesize \begin{array}{ccccc}
\nu^n_G & : & \{ (a_k)_k \mbox{ finite sequence } | \sum a_k = n , a_k \geq a_{k+1} , a_k \in \mathbb{N}_{\geq 1} \} &    \to & \mathbb{N}_{\geq 0} \\

 & & (a_k)_k & \mapsto & |\{ g \in G | \rho^n(g) = (a_k) \}| \\
\end{array}$$
\end{defi}

\begin{ex} $ $
 Let $\overline{A} = [1:1:j:j^2] \in \PGL(4,\kk)$. Then
$\rho^4(A) = (2,1,1)$.
\end{ex}
 
 \begin{remark} \label{laremarque} $ $
 
The tuple $\rho^n(G)$ and the map $\nu^n_G$ are invariant by conjugation inside $\PGL(n,\kk)$.
\end{remark}

\label{multiplicitieseigenvalues}

\begin{remark} $ $

One can observe that the method employed in subsections \ref{subsection165} and \ref{subsection166} to classify $3$-elementary non-cyclic subgroups of $\PGL(4,\kk)$ and $2$-elementary non-cyclic subgroups of $\PGL(5,\kk)$ is more general, as it can be applied to various other subgroups of $\PGL(n,\kk)$. We limit ourselves to these two cases because we will not need more for our study.
\end{remark}

\bigskip

\subsection{The $3$-elementary non-cyclic subgroups of $\PGL(4,\kk)$} \label{subsection165}
$ $

\medskip

We assume $\car(\kk) \neq 3$.
The result of this subsection is Proposition \ref{3elempgl4}.

\begin{defi} $ $

Let $(a,b,c,d),(e,f,g,h) \in \kk^4 \setminus \{(0,0,0,0)\}$.
We denote by $\left[ \begin{smallmatrix} a & b & c & d \\ e & f & g & h \end{smallmatrix} \right]$ the subgroup of $\PGL(4,\kk)$ generated by 
$\left[ \begin{smallmatrix} a & 0 & 0 & 0 \\ 0 & b & 0 & 0 \\ 0 & 0 & c & 0  \\ 0 & 0 & 0 & d \end{smallmatrix} \right]$ and $\left[ \begin{smallmatrix} e & 0 & 0 & 0 \\ 0 & f & 0 & 0 \\ 0 & 0 & g & 0  \\ 0 & 0 & 0 & h \end{smallmatrix} \right]$.
\end{defi}

\begin{defi} \label{defi1620} $ $
We define the following $3$-elementary subgroups of $\PGL(4,\kk)$:
\begin{itemize}

\item $G_{31} = \left[ \begin{smallmatrix} j & j & 1 & 1 \\ 1 & j & j & 1 \end{smallmatrix} \right]  \simeq (\mathbb{Z}/3)^2$

\item $G_{32} = \left[ \begin{smallmatrix} j & 1 & 1 & 1 \\ 1 & j & 1 & 1 \end{smallmatrix} \right] \simeq (\mathbb{Z}/3)^2$

\item $G_{33} = \left[ \begin{smallmatrix} j & 1 & 1 & 1 \\ 1 & j & j^2 & 1 \end{smallmatrix} \right] \simeq (\mathbb{Z}/3)^2$

\item $G_{34} = \langle [j:1:1:1],[1:j:1:1],[1:1:j:1] \rangle \simeq (\mathbb{Z}/3)^3$
\end{itemize}

\end{defi}

\begin{lemme} \label{3elempgl4lemma} $ $
\begin{itemize}
\item The subgroups $\left[ \begin{smallmatrix} j & j & 1 & 1 \\ 1 & j & j & 1 \end{smallmatrix} \right]$, $\left[ \begin{smallmatrix} j & j & 1 & 1 \\ j & 1 & j & 1 \end{smallmatrix} \right]$ and $\left[ \begin{smallmatrix} 1 & j & j & 1 \\ j & 1 & j & 1 \end{smallmatrix} \right]$ are conjugate by permutation matrices of $\PGL(4,\kk)$ and pairwise distinct.

\item The subgroups $\left[ \begin{smallmatrix} j & 1 & 1 & 1 \\ 1 & j & 1 & 1 \end{smallmatrix} \right]$, $\left[ \begin{smallmatrix} j & 1 & 1 & 1 \\ 1 & 1 & j & 1 \end{smallmatrix} \right]$, $\left[ \begin{smallmatrix} j & 1 & 1 & 1 \\ 1 & 1 & 1 & j \end{smallmatrix} \right], \left[ \begin{smallmatrix} 1 & j & 1 & 1 \\ 1 & 1 & j & 1 \end{smallmatrix} \right], \left[ \begin{smallmatrix} 1 & j & 1 & 1 \\ 1 & 1 & 1 & j \end{smallmatrix} \right]$ and $\left[ \begin{smallmatrix} 1 & 1 & j & 1\\ 1 & 1& 1 & j \end{smallmatrix} \right]$ are conjugate by permutation matrices of $\PGL(4,\kk)$
and pairwise distinct.

\item The subgroups $\left[ \begin{smallmatrix} j & 1 & 1 & 1 \\ 1 & j & j^2 & 1 \end{smallmatrix} \right] , \left[ \begin{smallmatrix} 1& j & 1 & 1 \\ 1 & 1& j & j^2 \end{smallmatrix} \right] , \left[ \begin{smallmatrix} 1 & 1 & j & 1 \\ j & j^2 & 1 & 1 \end{smallmatrix} \right] , \left[ \begin{smallmatrix} 1 & 1 & 1 & j \\ 1 & j & j^2 & 1 \end{smallmatrix} \right]$ are conjugate by permutation matrices of $\PGL(4,\kk)$ and pairwise distinct.
\end{itemize}

Furthermore, $G_{31}, G_{32}$ and $G_{33}$ are pairwise non-conjugate in $\PGL(4,\kk)$.
\end{lemme}

\begin{proof} $ $
The conjugations by permutation matrices are trivial.

In order to prove that $G_{31}, G_{32}$ and $G_{33}$ are pairwise non-conjugate in $\PGL(4,\kk)$, we will use the maps $\nu^4_G$ from Definition \ref{definition1611}.
We observe that:
\begin{center}
    \begin{minipage}[t]{0.3\textwidth}
        \[
        \nu_{G_{31}}^4 : \begin{array}[t]{ccl}
            (4)     & \mapsto & 1 \\
            (2,2)   & \mapsto & 4 \\
            (2,1,1) & \mapsto & 4
        \end{array}
        \]
    \end{minipage}
    \hfill
    \begin{minipage}[t]{0.3\textwidth}
        \[
        \nu_{G_{32}}^4 : \begin{array}[t]{ccl}
            (4)     & \mapsto & 1 \\
            (3,1)   & \mapsto & 4 \\
            (2,2)   & \mapsto & 2 \\
            (2,1,1) & \mapsto & 2
        \end{array}
        \]
    \end{minipage}
    \hfill
    \begin{minipage}[t]{0.3\textwidth}
        \[
        \nu_{G_{33}}^4 : \begin{array}[t]{ccl}
            (4)     & \mapsto & 1 \\
            (3,1)   & \mapsto & 2 \\
            (2,1,1) & \mapsto & 6
        \end{array}
        \]
    \end{minipage}
\end{center}
Hence $G_{31}, G_{32}$ and $G_{33}$ are pairwise non-conjugate in $\PGL(4,\kk)$ using Remark \ref{laremarque}.
Finally we observe that the thirteen subgroups described in the statement are pairwise distinct.
\end{proof}

\begin{prop} \label{3elempgl4} $ $
A $3$-elementary non-cyclic subgroup of $\PGL(4,\kk)$ is conjugate to exactly one of the subgroups 
$G_{31}$, $G_{32}$, $G_{33}$, $G_{34}$ (see Definition \ref{defi1620}).
\end{prop}

\begin{proof} $ $
Let $G$ be a $3$-elementary  non-cyclic subgroup of $\PGL(4,\kk)$. According to Proposition \ref{prop3}, such a group is conjugate to a subgroup of the diagonal torus of $\PGL(4,\kk)$. So $G \simeq (\mathbb{Z}/p)^r$ where $r=2, 3$. If $r=3$, then $G = G_{34}$. Let's assume now that $r=2$.
By duality, we have a bijection between the set of $\mathbb{F}_3$-vector subspaces of $\mathbb{F}_3^3$ of dimension $1$, and the set of $\mathbb{F}_3$-vector subspaces of $\mathbb{F}_3^3$ of dimension $2$, where $\mathbb{F}_3$ is the field with three elements.
Hence $(\mathbb{Z}/3)^3$ has the same number of subgroups isomorphic to $(\mathbb{Z}/3)^2$ and isomorphic to $\mathbb{Z}/3$.
Hence
$(\mathbb{Z}/3)^3$ has thirteen subgroups isomorphic to $(\mathbb{Z}/3)^2$.
By identifying $(\mathbb{Z}/3)^3$ to the $3$-torsion subgroup of the diagonal torus of $\PGL(4,\kk)$, we get an action of the symmetric group $\mathfrak{S}_4$ on $F$, the set of subgroups of $(\mathbb{Z}/3)^3$ isomorphic to $(\mathbb{Z}/3)^2$. According to Lemma \ref{3elempgl4lemma}, the orbits of $G_{31}, G_{32}, G_{33}$ under this action are distinct and respectively of size at least $3, 6, 4$. Since $3+6+4 = 13$, it shows that the orbits are exactly of size $3, 6, 4$, and that every subgroup of $\PGL(4,\kk)$ isomorphic to $(\mathbb{Z}/3)^2$ is conjugate to $G_{31}, G_{32}$ or $G_{33}$. Furthermore, these three groups are not pairwise conjugate by Lemma~\ref{3elempgl4lemma}.
\end{proof}

\medskip

\subsection{The $2$-elementary non-cyclic subgroups of $\PGL(5,\kk)$} \label{subsection166}
$ $

\medskip

We assume $\car(\kk) \neq 2$.
The result of this subsection is Proposition \ref{quarticz}. 

\begin{defi} \label{defi1623} $ $
We define the following $2$-elementary subgroups of $\PGL(4,\kk)$:
\begin{itemize}[label=\textbullet]
\item $G_{41} = \langle [-1 : 1 : 1 : 1 : 1 ] , [ 1 : -1 : 1 : 1 : 1  ] \rangle \simeq (\mathbb{Z}/2)^2$.
\item $G_{42} = \langle [1 : -1 : -1 : 1 : 1 ] , [ 1 : 1 : 1 : -1 : -1  ] \rangle \simeq (\mathbb{Z}/2)^2$.
\item $G_{43} = \langle [-1 : -1 : 1 : 1 : 1 ] , [ -1 : 1 : -1 : 1 : 1] \rangle \simeq (\mathbb{Z}/2)^2$.
\item $G_{44} =  \langle [1 : -1 : -1 : 1 : 1 ] , [1 : -1 : 1 : -1 : 1 ] , [1 : -1 : 1 : 1 : -1] \rangle \simeq (\mathbb{Z}/2)^3$.
\item $G_{45} =  \langle [-1 : 1 : 1 : 1 : 1] , [1 : -1 : 1 : 1 :1] , [1:1:-1:1:1] \rangle \simeq (\mathbb{Z}/2)^3$.
\item $G_{46} = Z \simeq (\mathbb{Z}/2)^4$.
\end{itemize}
\end{defi}
\begin{prop} \label{quarticz} $ $

Every $2$-elementary non-cyclic subgroup of the diagonal torus of $\PGL(5,\kk)$ is conjugate by a permutation matrix to one of the six subgroups from Definition \ref{defi1623}.
Furthermore these six subgroups are pairwise non-conjugate in $\PGL(5,\kk)$.
\end{prop} 


\begin{proof} $ $
 Let $G$ be a $2$-elementary non-cyclic subgroup of the diagonal torus of $\PGL(5,\kk)$. Let $r \in \{2,3,4\}$ such that $G \simeq (\mathbb{Z}/2)^r$. 
\begin{itemize}
\item Case $r=2$:
Let $F_2$ be the set of subgroups of $(\mathbb{Z}/2)^4$ isomorphic to  $(\mathbb{Z}/2)^2$.
We define the following map:
$$\begin{array}{ccccc}
f_2  & : & \binom{(\mathbb{Z}/2)^4 \setminus \{0\} }{2} & \to & F_2 \\
 & & \{x,y\} & \mapsto & \langle x,y \rangle \\
\end{array}$$ where $\binom{(\mathbb{Z}/2)^4 \setminus \{0\} }{2}$ denotes the set of subsets of $(\mathbb{Z}/2)^4 \setminus \{0\}$ of size two.
The map $f_2$ is surjective and each subgroup in $F_2$ has three antecedents. So: 
$$|F_2| = \dfrac{|\binom{(\mathbb{Z}/2)^4 \setminus \{0\} }{2}|}{3} = 35 .$$
By identifying $(\mathbb{Z}/2)^4$ to the diagonal torus of $\PGL(5,\kk)$, we get an action of $\mathfrak{S}_5$ on $F_2$.
We see that the stabilizer of $G_{41}$ is: $$\langle (12) \rangle \times \{ \sigma | \sigma(1) = 1 , \sigma(2) = 2 \} \simeq \mathfrak{S}_2 \times \mathfrak{S}_3,$$ hence the orbit of $G_{41}$ is of size $\dfrac{|\mathfrak{S}_5|}{| \mathfrak{S}_2 \times \mathfrak{S}_3|} = \dfrac{120}{12} = 10$.
We see that the stabilizer of $G_{42}$ is $\langle (23) , (45) , (24)(35) \rangle \simeq D_8$, hence the orbit of $G_{42}$ is of size $\dfrac{|\mathfrak{S}_5|}{| D_8 |} = \dfrac{120}{8} = 15$.
We see that the stabilizer of $G_{43}$ is: $$ \{\sigma | \{ \sigma(4) ,\sigma(5) \} = \{ 4 , 5 \} \}  \simeq  \mathfrak{S}_3 \times \mathfrak{S}_2,$$ hence the orbit of $G_{43}$ is of size $\dfrac{|\mathfrak{S}_5|}{|\mathfrak{S}_3 \times \mathfrak{S}_2|} = \dfrac{120}{12} = 10$.
Since $10 + 15 + 10 = 35$, every Klein subgroup of $\PGL(5,\kk)$ is conjugate to one of these subgroups.

\item Case $r=3$:
%
%
By duality, we have a bijection between the set of $\mathbb{F}_2$-vector subspaces of $\mathbb{F}_2^4$ of dimension $1$, and the set of $\mathbb{F}_2$-vector subspaces of $\mathbb{F}_2^4$ of dimension $3$.
Hence $(\mathbb{Z}/2)^4$ has the same number of subgroups isomorphic to $(\mathbb{Z}/2)^3$ and isomorphic to $\mathbb{Z}/2$. Hence $(\mathbb{Z}/2)^4$ has fifteen subgroups isomorphic to $(\mathbb{Z}/2)^3$.
By identifying $(\mathbb{Z}/2)^4$ with the diagonal torus of $\PGL(5,\kk)$, we get an action of $\mathfrak{S}_5$ on $F_3$, the set of subgroups of $(\mathbb{Z}/2)^4$ isomorphic to $(\mathbb{Z}/2)^3$.
We see that the stabilizer of $G_{44}$ is $\{\sigma | \sigma(1) = 1 \} \simeq \mathfrak{S}_4$, hence the orbit of $G_{44}$ is of size $\dfrac{|\mathfrak{S}_5|}{|\mathfrak{S}_4|} = \dfrac{120}{24} = 5$.
We see that the stabilizer of $G_{45}$ is: $$\{\sigma | \{ \sigma(4) ,\sigma(5) \} = \{ 4 , 5 \} \} \simeq \mathfrak{S}_3 \times \mathfrak{S}_2,$$ hence the orbit of $G_{45}$ is of size $\dfrac{|\mathfrak{S}_5|}{|\mathfrak{S}_3 \times \mathfrak{S}_2|} = \dfrac{120}{12} = 10$.
Since $10 + 5 = 15$, we get that every subgroup of $\PGL(5,\kk)$ isomorphic to $(\mathbb{Z}/2)^3$ is conjugate to one of these subgroups.
\item Case $r=4$:
Then we have $G = G_{46}$. 

\end{itemize}

In order to prove that $G_{41}, G_{42}, G_{43}, G_{44},  G_{45}$ and $G_{46}$ are pairwise non-conjugate in $\PGL(5,\kk)$, we will use the maps $\nu^5_G$ 
from Definition \ref{definition1611}.
We observe that:
\begin{center}
    \begin{minipage}[t]{0.3\textwidth}
        \[
        \nu_{G_{41}}^5 : \begin{array}[t]{ccl}
            (5)   & \mapsto & 1 \\
            (4,1) & \mapsto & 2 \\
            (3,2) & \mapsto & 1
        \end{array}
        \]
    \end{minipage}
    \hfill
    \begin{minipage}[t]{0.3\textwidth}
        \[
        \nu_{G_{42}}^5 : \begin{array}[t]{ccl}
            (5)   & \mapsto & 1 \\
            (4,1) & \mapsto & 1 \\
            (3,2) & \mapsto & 2
        \end{array}
        \]
    \end{minipage}
    \hfill
    \begin{minipage}[t]{0.3\textwidth}
        \[
        \nu_{G_{43}}^5 : \begin{array}[t]{ccl}
            (5)   & \mapsto & 1 \\
            (3,2) & \mapsto & 3
        \end{array}
        \]
    \end{minipage}
\end{center}

\begin{center}
    \begin{minipage}[t]{0.45\textwidth}
        \[
        \nu_{G_{44}}^5 : \begin{array}[t]{ccl}
            (5)   & \mapsto & 1 \\
            (4,1) & \mapsto & 1 \\
            (3,2) & \mapsto & 6
        \end{array}
        \]
    \end{minipage}
    \hfill
    \begin{minipage}[t]{0.45\textwidth}
        \[
        \nu_{G_{45}}^5 : \begin{array}[t]{ccl}
            (5)   & \mapsto & 1 \\
            (4,1) & \mapsto & 3 \\
            (3,2) & \mapsto & 4
        \end{array}
        \]
    \end{minipage}
\end{center}
Hence our conclusion using Remark \ref{laremarque}.
\end{proof}

\newpage

\section{The $p$-elementary subgroups of $\PGL(3,\kk)$ when $\car(\kk) = p$}  \label{24new} $ $

We assume $\car(\kk) = p > 0$ in this section. 

We will not use this section for our study but we include it for completeness.

\begin{lemme} \label{lis0} 
Let $a,b,c, a', b', c' \in \KK$.

Let $A = \left[\begin{smallmatrix}  1 & a & b  \\
 & 1 & c \\
 & & 1
\end{smallmatrix} \right] \in \GL(3,\KK), A' = \left[\begin{smallmatrix}  1 & a' & b'  \\
 & 1 & c' \\
 & & 1
\end{smallmatrix} \right] \in \GL(3,\KK)$.
Then:
\begin{enumerate}
\item \label{lis1} Let $\alpha, \gamma \in \KK^\times$. Then, $\left[\begin{smallmatrix}  \alpha \gamma & &   \\
 & \gamma & \\
 & & 1
\end{smallmatrix} \right] \left[\begin{smallmatrix}  1 & a & b  \\
 & 1 & c \\
 & & 1
\end{smallmatrix} \right]  \left[\begin{smallmatrix}  \alpha \gamma & &   \\
 & \gamma & \\
 & & 1
\end{smallmatrix} \right]^{-1} = \left[\begin{smallmatrix}  1 & \alpha a & \alpha \gamma b  \\
 & 1 & \gamma c \\
 & & 1
\end{smallmatrix} \right]$.
\item \label{lis2} $A, A'$ commute if and only if $a c' = a' c$.
\item \label{lis3} If $\car(\KK) = p > 2$, then $A^p = I$.
\item \label{lis4} If $\car(\KK) = 2$, then $A^2 = I$ if and only if $a = 0$ or $c = 0$.
\end{enumerate}
\end{lemme}

\begin{proof} $ $
Assertions \ref{lis1}, \ref{lis2} and \ref{lis4} are direct computations

 By induction on $n$, for any $n \in \mathbb{N}_{\geq 0}$, $A^n = \left[\begin{smallmatrix}  1 & n a & n b + n(n-1) \frac{ac}{2}  \\
0 & 1 & n c \\
0 & 0 & 1
\end{smallmatrix} \right] $. Hence Assertion \ref{lis3} follows by taking $n=p$.

\end{proof}

\begin{prop} \label{prop242} 
Let $G$ be a $p$-elementary subgroup of $\PGL(3,\kk)$.
Then there exists a subgroup $\tilde{G} \subset \PGL(3,\kk)$ such that $G, \tilde{G}$ are conjugate and such that:
\begin{itemize}
\item If $p =2$, then $\tilde{G}$ is a subgroup of one of the two following groups: \[ \left\{ \left[\begin{smallmatrix}  1 & a & b   \\
 & 1 & 0  \\
 & & 1
\end{smallmatrix} \right] | a,b \in \kk \right\}, \left\{ \left[\begin{smallmatrix}  1 & 0 & b   \\
 & 1 & c  \\
 & & 1
\end{smallmatrix} \right] | b,c \in \kk \right\} \]
\item If $p > 2$, then $\tilde{G}$ is a subgroup of one of the three following groups: \[ \left\{ \left[\begin{smallmatrix}  1 & a & b   \\
 & 1 & 0  \\
 & & 1
\end{smallmatrix} \right] | a,b \in \kk \right\}, \left\{ \left[\begin{smallmatrix}  1 & 0 & b   \\
 & 1 & c  \\
 & & 1
\end{smallmatrix} \right] | b,c \in \kk \right\}, \left\{ \left[\begin{smallmatrix}  1 & a & b   \\
 & 1 & a  \\
 & & 1
\end{smallmatrix} \right] | a,b \in \kk \right\} \]
\end{itemize}
\end{prop}

\begin{proof} $ $
Assume $G$ is non-trivial.
Let $G'$ be the preimage of $G$ by the canonical projection $\GL(3,\kk) \twoheadrightarrow \PGL(3,\kk)$. Let $x \in G'$ be a non-trivial element. By construction, there exists $\lambda \in \kk^\times$ such that $x^p = \lambda I$. As $\kk$ is algebraically closed, $x$ can be trigonalized. Since $\car(\kk) = p$, there exists exactly one $p$-root of $\lambda$, we will denote it $\xi$. So the only eigenvalue of $x$ is $\xi$. Hence up to conjugation, we can assume that $x = \left[\begin{smallmatrix}  \xi & a & b   \\
 & \xi & c  \\
 & & \xi
\end{smallmatrix} \right]$, where $a,b,c \in \kk$.
Let $y = \left[\begin{smallmatrix}  d & e & f   \\
g & h & i  \\
j & k & l
\end{smallmatrix} \right] \in G'$. Since $G$ is abelian, it means there exists $\rho \in \kk^\times$ such that $x y = \rho y x$. So $\left[\begin{smallmatrix}  \xi d + ga + jb & * & *   \\
\xi g + jc & * & *  \\
\xi j & * & *
\end{smallmatrix} \right] = \left[\begin{smallmatrix}   \rho \xi d & * & *   \\
\rho \xi  g  & * & *  \\
\rho \xi   j & * & *
\end{smallmatrix} \right]$. So $j = \rho  j, \xi g + j c = \rho \xi g, \xi d + ga + jb  = \rho \xi d$.
By contradiction, we assume that $x$ and $y$ don't commute, so $\rho \neq 1$.
Hence $j = g = d = 0$. This is a contradiction because $y$ is invertible.

Hence it implies that $G'$ is abelian. So $G'$ can be trigonalized. So $G$ can be trigonalized. Also, every non-trivial element of $G$ is of order $p$ and the characteristic of $\KK$ is $p$, hence no non-trivial element of $G$ can be diagonalized. Hence, up to conjugation we can assume $G \subset \left\{ \left[\begin{smallmatrix}  1 & * & *   \\
 & 1 & *  \\
 & & 1
\end{smallmatrix} \right] \right\} $.
If $G$ is not a subgroup of $\left\{ \left[\begin{smallmatrix}  1 & 0 & *   \\
 & 1 & 0  \\
 & & 1
\end{smallmatrix} \right] \right\} $, then there exists an element $g = \left[\begin{smallmatrix}  1 & a & b   \\
 & 1 & c  \\
 & & 1
\end{smallmatrix} \right] \in G$ with $a \neq 0$ or $c \neq 0$. We conclude by distinguishing three cases:
\begin{enumerate}
\item If $a \neq 0, c = 0$, then using Lemma \ref{lis0} Assertion \ref{lis2} we get:  $$G \subset \left\{ \left[\begin{smallmatrix}  1 & a & b   \\
 & 1 & 0  \\
 & & 1
\end{smallmatrix} \right] | a,b \in \kk \right\} $$
\item
If $a = 0, c \neq 0$, then likewise, using Lemma \ref{lis0} Assertion \ref{lis2} we get: $$G \subset \left\{ \left[\begin{smallmatrix}  1 & 0 & b   \\
 & 1 & c  \\
 & & 1
\end{smallmatrix} \right] | b,c \in \kk \right\} $$
\item
If $a , c \neq 0$, then using Lemma \ref{lis0} Assertion \ref{lis4} we get $p > 2$. Using Lemma \ref{lis0} Assertion \ref{lis1}, we can conjugate $g$ to $ \left[\begin{smallmatrix}  1 & 1 & \frac{b}{ac}   \\
 & 1 & 1  \\
 & & 1
\end{smallmatrix} \right] $ using the diagonal element $ \left[\begin{smallmatrix}  \frac{1}{ac} &  &   \\
 & \frac{1}{c} &   \\
 & & 1
\end{smallmatrix} \right] $ . Using again Lemma \ref{lis0} Assertion \ref{lis2} we get: 
$$G \subset \left\{ \left[\begin{smallmatrix}  1 & a & b   \\
 & 1 & a  \\
 & & 1
\end{smallmatrix} \right] | a,b \in \kk \right\} $$
\end{enumerate}
\end{proof}


\begin{prop} \label{prop243} $ $
Let $G$ be a $p$-elementary non-trivial subgroup of $\PGL(3,\kk)$.
\begin{enumerate}
\item If $G$ is a subgroup of $\left\{ \left[\begin{smallmatrix}  1 & 0 & b   \\
 & 1 & c  \\
 & & 1
\end{smallmatrix} \right] | b,c \in \kk \right\} $, then the fixed points of $G$ are $[x:y:0]$, for $[x:y] \in \mathbb{P}^1$.

\item If $G$ is a subgroup of $\left\{ \left[\begin{smallmatrix}  1 & \lambda a & \lambda b   \\
 & 1 & 0  \\
 & & 1
\end{smallmatrix} \right] | \lambda \in \kk \right\}$ for some $a,b \in \kk$, then the fixed points of $G$ are $[x:y:z]$ where $x,y,z \in \kk$ such that $a y + b z = 0$.

\item If $G$ is a subgroup of $\left\{ \left[\begin{smallmatrix}  1 & a & b   \\
 & 1 & 0  \\
 & & 1
\end{smallmatrix} \right] | a,b \in \kk \right\} $ and not a subgroup of $\{ \left[\begin{smallmatrix}  1 & \lambda a & \lambda b   \\
 & 1 & 0  \\
 & & 1
\end{smallmatrix} \right] | \lambda \in \kk \} $ for any $a,b \in \kk$, then the only fixed point of $G$ is $[1:0:0]$.

\item If $p >2$, if $G$ is a subgroup of $\left\{ \left[\begin{smallmatrix}  1 & a & b   \\
 & 1 & a \\
 & & 1
\end{smallmatrix} \right] | a,b \in \kk \right\} $ such that there exists $\left[\begin{smallmatrix}  1 & a & b   \\
 & 1 & a \\
 & & 1
\end{smallmatrix} \right] \in G$ with $a \neq 0$, then the only fixed point of $G$ is $[1:0:0]$.
\end{enumerate}
\end{prop}

\begin{proof} $ $

\begin{enumerate}
\item 
Any point $[x:y:0]$ is a fixed point of $G$.
Furthermore, let $[x:y:z]$ be a fixed point of $G$. The group $G$ is non-trivial, hence there exists $g = \left[\begin{smallmatrix}  1 & 0 & b   \\
 & 1 & c  \\
 & & 1
\end{smallmatrix} \right]$ with $b \neq 0$ or $c \neq 0$.
Then there exists $\mu \in \kk^\times$ such that $\left\{
\begin{array}{l}
x + b z = \mu x,  \\
y + c z = \mu y \\
z = \mu z
\end{array}.\right.$
So $\mu = 1$, and we get $z = 0$.

\item
Any point $[x:y:0]$ such that $a y + b z = 0$ is a fixed point of $G$.
Furthermore, let $[x:y:z]$ be a fixed point of $G$. The group $G$ is non-trivial, hence there exists $g = \left[\begin{smallmatrix}  1 & a & b   \\
 & 1 & 0  \\
 & & 1
\end{smallmatrix} \right]$ with $a \neq 0$ or $b \neq 0$. 
Then there exists $\mu \in \kk^\times$ such that $\left\{
\begin{array}{l}
x + a y + b z = \mu x,  \\
y  = \mu y \\
z = \mu z
\end{array}.\right.$
So $\mu = 1$, and we get $a y + b z = 0$.

\item 
The point $[1:0:0]$ is a fixed point of $G$.
Furthermore, let $[x:y:z]$ be a fixed point of $G$. Because of the assumptions, there exists $g = \left[\begin{smallmatrix}  1 & a & b   \\
 & 1 & 0  \\
 & & 1
\end{smallmatrix} \right] \in G$, $g' = \left[\begin{smallmatrix}  1 & a' & b'   \\
 & 1 & 0  \\
 & & 1
\end{smallmatrix} \right] \in G$ such that $(a,b) , (a',b')\neq (0,0)$ and $a b' - a' b \neq 0$.
\newline
Then there exists $\mu \in \kk^\times$ such that $\left\{
\begin{array}{l}
x + a y + b z = \mu x,  \\
x + a' y + b' z = \mu x,  \\
y  = \mu y \\
z = \mu z
\end{array}.\right.$ \newline
So $\mu = 1$, and $\left\{
\begin{array}{l}
a y + b z =  0,  \\
a' y + b' z = 0
\end{array}.\right.$
Because $a b' - a' b \neq 0$, we get $y = z = 0$.

\item 
The point $[1:0:0]$ is a fixed point of $G$.
Furthermore, if $[x :y :z]$ is a fixed point of $G$, then there exists $\mu \in \kk^\times$ such that $\left\{
\begin{array}{l}
x + a y + b z = \mu x,  \\
y + a z = \mu y \\
z = \mu z
\end{array}.\right.$
So $\mu =1$ and
$\left\{
\begin{array}{l}
a y + b z = 0,  \\
a z =  0 
\end{array}.\right.$ Hence $y = z = 0$.
\end{enumerate} \end{proof}

\newpage

\part{Subgroups of the de Jonqui\`eres group} \label{part2} $ $

\bigskip\bigskip

 In this part we will give a classification of $p$-elementary subgroups of the de Jonqui\`eres group up to conjugation by elements of the de Jonqui\`eres group. The part is organized as follows:
 \begin{itemize}
 \item In Section \ref{jonquieressection31} we establish a list of representatives of $p$-elementary subgroups of the de Jonquières group up to conjugacy in the de Jonquières group, with no assumption on the characteristic of the field. The result is Proposition \ref{theo2}.
  \item In Section \ref{jonquieressection32} we study the possible conjugations by the de Jonquières map between such subgroups.
  \item Section \ref{jonquieressection33} may appear superfluous at first glance, as Section \ref{jonquieressection31} already provides a list of representatives and Section \ref{jonquieressection32} details their possible conjugations. However, this section is essential; indeed, abelian subgroups with no non-trivial elements fixing a non-rational curve can be conjugated to subgroups of $\operatorname{Aut}(\mathbb{P}^1 \times \mathbb{P}^1)$ (see \cite[Proposition 8.4]{blancarticle2}). Consequently, we primarily require a classification for subgroups of the de Jonquières group that do contain non-trivial elements fixing non-rational curves.
 \end{itemize}

\bigskip\bigskip

We remind of the following facts: \\

\noindent The de Jonquières group is the subgroup $\J$ in $\Cr$ given by
conjugating the following group of birational maps of $\mathbb{A}^2$
\[
    \J = \Bigset{ (x,y) \DashedArrow 
    \left(\dfrac{ax+b}{cx+d} , \dfrac{\alpha(x)y + \beta(x)}{\gamma(x)y + \delta(x)} \right)}{ 
    \begin{array}{l}
    a,b,c,d \in \kk \, , \\ 
    \alpha, \beta, \gamma, \delta \in \kk(x) \, , \\ 
    ad-bc \neq 0 \, , \ \alpha \delta - \beta \gamma \neq 0
    \end{array}}
\]
via the embedding $\mathbb{A}^2 \to \PP^2$, $(x,y) \mapsto [y:x:1]$.
Observe that $\J$ is the semi-direct product $\PGL(2, \kk(x)) \rtimes \PGL(2, \kk)$.
Hence we have the following split exact sequence:
\[
    1 \xhookrightarrow{} \PGL(2,\kk(x)) \xhookrightarrow{} 
    \J \stackrel{\pi}{\twoheadrightarrow} \PGL(2,\kk) \twoheadrightarrow 1 \, ,
\]
where $\pi$ is the projection onto the second coordinate:

\noindent If $j = (\left[\begin{smallmatrix}  \alpha & \beta   \\  \gamma & \delta  \end{smallmatrix} \right], \left[\begin{smallmatrix}  a & b   \\  c & d  \end{smallmatrix} \right]) \in \PGL(2,\kk(x)) \rtimes \PGL(2,\kk) = \J$ then $\pi(j) = \left[\begin{smallmatrix}  a & b   \\  c & d  \end{smallmatrix} \right] \in \PGL(2,\kk)$.

\noindent For a subgroup $G \subset \J$ we get by restriction the following exact sequence:
\[
    1 \xhookrightarrow{} \Ker(\pi|_G) \xhookrightarrow{} G 
    \stackrel{\pi|_G}{\twoheadrightarrow} \Img(\pi|_G) \twoheadrightarrow 1 \, .
\]
If $G$ is a $p$-elementary subgroup of $\J$, then $\Ker(\pi|_G)$ is isomorphic to 
$(\mathbb{Z}/p)^a$ for some $a \in \mathbb{N}_{\geq 0}$ and likewise, $\Img(\pi|_G)$ is isomorphic to 
$(\mathbb{Z}/p)^b$ for some $b \in \mathbb{N}_{\geq 0}$.

\newpage

\section{Explicit representatives of subgroups up to conjugation} \label{jonquieressection31} $ $

Here we do not make assumptions on the characteristic of the field.
The goal of this section is to prove Proposition \ref{theo2}.
To get a list of representatives, we require three consecutive conjugations on a given subgroup $G \subset \J$:
\begin{itemize}
\item The first one being a conjugation of $\pi(G) \subset \text{PGL}(2, \mathbf{k})$. This is done using Section \ref{22new}.
\item Then we perform a conjugation in $\text{PGL}(2, \mathbf{k}(x))$ using the following Subsection \ref{mitterand1}. 
This allows us to assume, up to conjugacy in $\PGL(2, \kk(x))$, that $\{I\} \times \pi(G) \leq G$.
\item Finally, we use Subsection \ref{mitterand2} in order to do a third and last conjugation inside \newline $\text{PGL}(2, \mathbf{k}(x)^{\pi(G)})$ in order to simplify the subgroup $\Ker(\pi|_G) \subset G$.
\end{itemize}
We summarize these three steps in Subsection \ref{mitterand3}.

\subsection{A first cohomology result} \label{mitterand1}

\begin{lemme} \label{cohomologie1} $ $
Let $G$ be a $p$-elementary subgroup of the de Jonquières group.
Then there exists a subgroup $G' \leq G$ such that $\pi|_{G'}  : G' \rightarrow \pi(G)$ is a group isomorphism.
\end{lemme}

\begin{proof} $ $
Let $r, a, b \in \mathbb{N}_{\geq 0}$ such that $G \simeq(\mathbb{Z}/p)^r$, $\pi(G) \simeq (\mathbb{Z}/p)^b$ and $\Ker (\pi|_G) \simeq(\mathbb{Z}/p)^a$.
Then $G$ and $\pi(G)$ are $\mathbb{F}_p$-vector spaces of dimension $r$ and $b$ respectively. $\Ker (\pi|_G)$ is a subvector space of $G$ of dimension $a$. Let $G'$ be a complement of $\Ker(\pi|_G)$. Then $\pi|_{G'} : G' \rightarrow \pi(G)$ is an isomorphism.
\end{proof}

The group $\PGL(2,\kk)$ acts on $\kk(x)$. This action induces an action of $\PGL(2,\kk)$ on $\kk(x)^\times$, $\GL(2,\kk(x))$ and $\PGL(2,\kk(x))$.

\begin{lemme} \label{cohomologie2} $ $
Let $G$ be a finite subgroup of the de Jonquières group.

Then the cohomology set  $H^1(\pi(G) , \PGL(2,\kk(x)))$ is trivial.
\end{lemme}

\begin{proof} $ $
We have the exact sequence:
\begin{center}
$1 \xhookrightarrow{} \kk(x)^\times \xhookrightarrow{} \GL(2,\kk(x)) \twoheadrightarrow \PGL(2,\kk(x)) \twoheadrightarrow 1$.
\end{center}
So according to \cite[Appendix VII, Proposition 2]{serre}, we get the following exact sequence of pointed sets:
\begin{center}
$ H^1(\pi(G) , \GL(2,\kk(x))) \to H^1(\pi(G) , \PGL(2,\kk(x))) \to H^2( \pi(G) , \kk(x)^\times )$
\end{center}


The cohomology set $H^1( \pi(G) , \GL(2,\kk(x)))$ is trivial thanks to \cite[X. \S.1 Proposition 3]{serre}  because the action of $\pi(G)$ on $\kk(x)$ is faithfull and is by field automorphisms.

The cohomology set $H^2( \pi(G) , \kk(x)^\times)$ is trivial thanks to \cite[ X. \S.7 Proposition 10]{serre} because the action of $\pi(G)$ on $\kk(x)$ is faithfull and is by field morphism, and $\kk(x)^{\pi(G)}$ has the $C_1$ property (we have $\kk \subset \kk(x)^{\pi(G)} \subset \kk(x)$, therefore $\kk(x)^{\pi(G)}$ is purely transcendental of degree $1$, by Lüroth's theorem).

Then the cohomology set $H^1(\pi(G) , \PGL(2,\kk(x)))$ is trivial.
\end{proof}

\begin{prop} \label{cohomologie0} $ $
Let $G$ be a $p$-elementary subgroup of $\J$.

Then there exists $\rho \in \PGL(2,\kk(x))$ such that $\{I\} \times \pi(\rho G \rho^{-1}) \leq \rho G \rho^{-1}$.
\end{prop}

\begin{proof} $ $
Using Lemma \ref{cohomologie1}, let $G'$ be a subgroup of $G$ such that $\pi|_{G'} : G' \rightarrow \pi(G)$ is an isomorphism.
Let:

\begin{center}
$\begin{array}{ccccc}
s & : &  \pi(G) & \to & \PGL(2,\kk(x)) \rtimes \pi(G) \\
& & \pi(g') & \mapsto & (I,\pi(g')) \qquad \mbox{for }g' \in G'
\end{array}$
\end{center}
\begin{center}
$\begin{array}{ccccc}
s' & : &   \pi(G) & \to & \PGL(2,\kk(x)) \rtimes \pi(G) \\
& & \pi(g') & \mapsto & g' \qquad \mbox{for }g' \in G'
\end{array}$
\end{center}
These two maps are sections of the following canonical projection:  $$\pi|_{\PGL(2,\kk(x)) \rtimes  \pi(G)} :\PGL(2,\kk(x)) \rtimes  \pi(G) \rightarrow \pi(G)$$
Using Lemma \ref{cohomologie2}, the cohomology set $H^1(\pi(G),\PGL(2,\kk(x)))$ is trivial. Hence these two sections are conjugate, \textit{i.e.} there exists $\rho \in \PGL(2,\kk(x))$ such that $\rho g' \rho^{-1} = (I,\pi(g'))$ for every $g' \in G'$.
Then for $g \in G$ we choose $g'$ such that $\pi(g') = \pi(g)$. Then we have:
$$(I,\pi(\rho g \rho^{-1}))  = (I , \pi(g))  = (I,\pi(g'))   = \rho g' \rho^{-1}  \in \rho G' \rho^{-1} \leq \rho G \rho^{-1} $$
Hence: $$\{I\} \times \pi(\rho G \rho^{-1}) \leq \rho G \rho^{-1}$$
\end{proof}

\subsection{A description of some subfields of $\kk(x)$ invariant by a subgroup of automorphisms} \label{mitterand2} $ $

\bigskip

\begin{lemme} \label{stabilisationbis} $ $
We assume $\car(\kk) \neq 2$.

\noindent We remind the reader that $\sigma$ \textit{resp} $\tau$ corresponds to the $\kk$-field automorphism of $\kk(x)$ given by $x \mapsto -x$ \textit{resp} $x \mapsto \frac{1}{x}$. We have the following:
\begin{enumerate}
\item The set of elements of  $\kk(x)$ fixed by $\sigma$ is the subfield $\kk(x^2)$.
\item The set of elements of  $\kk(x)$ fixed by $\tau$ is the subfield $\kk(x+\frac{1}{x})$.
\item The set of elements of  $\kk(x)$ fixed by $\sigma$ and $\tau$ is the subfield $\kk(x^2+\frac{1}{x^2})$.

\end{enumerate}
\end{lemme}

\begin{proof} $ $
\begin{enumerate}
\item
First, we observe that the elements of $\kk(x^2)$ are fixed by the action of $\left[\begin{smallmatrix}  1 & 0   \\ 0 & -1  \end{smallmatrix} \right] $. Since $x \mapsto - x$ is an element of order two, we get that the degree of $\kk(x)$ as an extension of its subfield of elements invariant by $x \mapsto - x$ is also two. Finally, we have that the degree of $\kk(x)$ as an extension of its subfield $\kk(x^2)$ is also two, hence the result.
\item
The elements of $\kk(x+\frac{1}{x})$ are fixed by the action of $\tau$. Since $x \mapsto \frac{1}{x}$ is an element of order two, the degree of $\kk(x)$ as an extension of its subfield of elements fixed by $x \mapsto \frac{1}{x}$ is also of degree two.
Finally, $x$ has a minimal polynomial of degree at most two over $\kk(x+\frac{1}{x})$ : $x^2 - (x+\frac{1}{x})x+1 = 0$, so the degree of $\kk(x)$ as an extension of its subfield $\kk(x+\frac{1}{x})$ is at most two, hence the result.
\item
By replacing $x$ by $x^2$ in the proof of the previous assertion, we get that the degree of $\kk(x^2)$ as an extension of its subfield $\kk(x^2+\frac{1}{x^2})$ is at most two. Furthermore, $\kk(x^2 + \frac{1}{x^2})$ is trivially a subfield of $\kk(x)^{\langle \sigma,\tau \rangle}$, and $\kk(x)^{\langle \sigma \rangle}$ is an extension of degree two of $\kk(x)^{\langle \sigma,\tau \rangle}$ because $\tau$ is of degree two. Hence the result.
\end{enumerate}
\end{proof}

\subsection{A table of representatives} \label{mitterand3} $ $


We now prove the main result of this section (which is stated on the next page):

\begin{proof}[Proof of Proposition \ref{theo2}] $ $

Let $G$ be a $p$-elementary non-trivial subgroup of the de Jonquières group. Using Proposition \ref{dublin1}, up to conjugation by an element of $\PGL(2,\kk)$, we can assume that $\pi(G)$ is of the form of one of the groups in the column $"R"$.
Let $R = \pi(G)$.
Using Proposition \ref{cohomologie0}, we can assume up to conjugation by an element of $\PGL(2,\kk(x))$ that $\{I \} \times R \subset G$.
Since $G$ is abelian, we get that $\{I \} \times R$ and $\Ker( {\pi}|_G)$ commute. Hence $\Ker ({\pi}|_G) \subset \PGL(2,\kk(x)^{R})$. Finally, using Proposition \ref{dublin1} and Lemma \ref{stabilisationbis}, up to conjugation by an element of $\PGL(2,\kk(x)^R)$, we can assume that $\Ker(\pi|_G)$ is of the form of one of the groups in the column $"L"$.
Let $L = \Ker(\pi|_G)$. Then we have $G = L \times R$.
\end{proof}

\newpage

\begin{prop} \label{theo2} \label{buda5} $ $
(Classification of $p$-elementary subgroups of $\J$) \newline
Let $G \simeq (\mathbb{Z}/p)^r$ be a $p$-elementary non-trivial subgroup of the de Jonquières group. \newline
Then $G$ is conjugate in the de Jonquières group to a subgroup $L \times R$ where:
\begin{itemize}
\item $L \simeq (\mathbb{Z}/p)^a$ is a subgroup of $\PGL(2,\kk(x))$.
\item $R \simeq (\mathbb{Z}/p)^b$ is a subgroup of $\PGL(2,\kk)$. 
\end{itemize}
Furthermore $L$ and $R$ are given in the following table:

\bigskip

\begin{center}
\rotatebox{90}{
\begin{tabular}{ | c | c | c | c | c | c| }
\hline
& $L$ & $a$ & $R$ & $b$ & $r$ \\

\hline

\multirow{3}{*}{$\car(\kk) \neq p \neq 2$} & $\{ I \}$ &0 & $\langle \left[\begin{smallmatrix}  1 & 0   \\ 0 &  \xi  \end{smallmatrix} \right] \rangle$ & 1&1 \\

 & $\langle \left[\begin{smallmatrix}  1 & 0   \\ 0 &  \xi  \end{smallmatrix} \right] \rangle$ & 1 &$\{ I \}$ & 0 & 1 \\

 & $\langle \left[\begin{smallmatrix}  1 & 0   \\ 0 &  \xi  \end{smallmatrix} \right] \rangle$ & 1 &$\langle \left[\begin{smallmatrix}  1 & 0   \\ 0 &  \xi  \end{smallmatrix} \right] \rangle$ & 1 & 2 \\
 
\hline

\multirow{8}{*}{$\car(\kk) \neq p = 2$}  & $\{ I \}$ &0 & $\langle \left[\begin{smallmatrix}  1 & 0   \\ 0 &  -1 \end{smallmatrix} \right] \rangle$ & 1&1 \\

& $\{ I \}$ &0 & $\langle \left[\begin{smallmatrix}  1 & 0   \\ 0 &  -1  \end{smallmatrix} \right] , \left[\begin{smallmatrix}  0 & 1   \\ 1 &  0  \end{smallmatrix} \right] \rangle$ & 2&2 \\

& $\langle \left[\begin{smallmatrix}  0 & g   \\ 1&  0 \end{smallmatrix} \right] \rangle$ &1& $\{I \}$ & 0&1 \\

& $\langle \left[\begin{smallmatrix}  0 & g(x^2)   \\ 1 &  0 \end{smallmatrix} \right] \rangle$ &1 & $\langle \left[\begin{smallmatrix}  1 & 0   \\ 0 &  -1 \end{smallmatrix} \right] \rangle$ & 1&2 \\

& $\langle \left[\begin{smallmatrix}  0 & g(x^2 + \frac{1}{x^2})   \\ 1&  0  \end{smallmatrix} \right] \rangle$ &1 & $\langle \left[\begin{smallmatrix}  1 & 0   \\ 0 &  -1  \end{smallmatrix} \right] , \left[\begin{smallmatrix}  0 & 1   \\ 1 &  0  \end{smallmatrix} \right] \rangle$ & 2&3 \\

& $\langle \left[\begin{smallmatrix}  0 & g   \\ 1&  0 \end{smallmatrix} \right] , \left[\begin{smallmatrix}  a & - bg   \\ b &  -a  \end{smallmatrix} \right] \rangle$ &2 & $\{I \}$ & 0&2 \\

& $\langle \left[\begin{smallmatrix}  0 & g(x^2)   \\ 1&  0 \end{smallmatrix} \right] , \left[\begin{smallmatrix}  a(x^2) & - b(x^2) g(x^2)   \\ b(x^2) &  -a(x^2)  \end{smallmatrix} \right] \rangle$ &2 & $\langle \left[\begin{smallmatrix}  1 & 0   \\ 0 &  -1 \end{smallmatrix} \right] \rangle$ & 1&3 \\

& $\langle \left[\begin{smallmatrix}  0 & g(x^2 + \frac{1}{x^2})   \\ 1&  0 \end{smallmatrix} \right] , \left[\begin{smallmatrix}  a(x^2 + \frac{1}{x^2}) & - b(x^2 + \frac{1}{x^2}) g(x^2 + \frac{1}{x^2})   \\ b(x^2 + \frac{1}{x^2}) &  -a(x^2 + \frac{1}{x^2})  \end{smallmatrix} \right] \rangle $ &2 & $\langle \left[\begin{smallmatrix}  1 & 0   \\ 0 &  -1  \end{smallmatrix} \right] , \left[\begin{smallmatrix}  0 & 1   \\ 1 &  0  \end{smallmatrix} \right] \rangle$ & 2&4 \\

\hline

\multirow{2}{*}{$\car(\kk) = p \neq 2$} & subgroup of  $\{\left[\begin{smallmatrix}  1 & t   \\ 0 & 1  \end{smallmatrix} \right] | t \in \kk(x)^R \}$   & $\in \mathbb{N}$ & subgroup of  $\{\left[\begin{smallmatrix}  1 & t   \\ 0 & 1  \end{smallmatrix} \right] | t \in \kk \}$  & $\in \mathbb{N}$ &$\in \mathbb{N}$ \\

&containing  $\left[\begin{smallmatrix}  1 & 1   \\ 0 & 1  \end{smallmatrix} \right] $ if non-trivial &&containing  $\left[\begin{smallmatrix}  1 & 1   \\ 0 & 1  \end{smallmatrix} \right] $ if non-trivial&& \\

\hline

\multirow{4}{*}{$\car(\kk) = p = 2$} & subgroup of  $\{\left[\begin{smallmatrix}  1 & t   \\ 0 & 1  \end{smallmatrix} \right] | t \in \kk(x)^R \}$   & $\in \mathbb{N}$ & subgroup of  $\{\left[\begin{smallmatrix}  1 & t   \\ 0 & 1  \end{smallmatrix} \right] | t \in \kk \}$  & $\in \mathbb{N}$ &$\in \mathbb{N}$ \\

&containing  $\left[\begin{smallmatrix}  1 & 1   \\ 0 & 1  \end{smallmatrix} \right] $ if non-trivial &&containing  $\left[\begin{smallmatrix}  1 & 1   \\ 0 & 1  \end{smallmatrix} \right] $ if non-trivial&& \\

&subgroup of  $\{\left[\begin{smallmatrix}  a & bg   \\ b & a  \end{smallmatrix} \right] | a,b \in \kk(x)^R, a^2 + b^2 g \neq 0 \}$ & $\in \mathbb{N}$ & subgroup of  $\{\left[\begin{smallmatrix}  1 & t   \\ 0 & 1  \end{smallmatrix} \right] | t \in \kk \}$ & $\in \mathbb{N}$ & $\in \mathbb{N}$  \\

&containing  $\left[\begin{smallmatrix}  0 & g   \\ 1 & 0  \end{smallmatrix} \right] $ if non-trivial &&containing  $\left[\begin{smallmatrix}  1 & 1   \\ 0 & 1  \end{smallmatrix} \right] $ if non-trivial  && \\
\hline

\end{tabular}}
\end{center}
\end{prop}




\newpage
\section{Conjugation by de Jonquières maps between the representatives} \label{jonquieressection32} $ $

The goal of this section is to describe when two $p$-elementary subgroups are conjugate in the de Jonqui\`eres group when the characteristic of the field is not equal to $p$.
The main result is Proposition \ref{resultatgeneral}.
We assume $\car(\kk) \neq p$ in this section.

First we observe that in the case $p > 2$, two different subgroups of the previous table of Proposition \ref{theo2} are never conjugate in the de Jonquières group (because their $(a,b)$ are different). Hence we now assume $p=2$.

\subsection{A general result for the conjugacy between $2$-elementary subgroups}

\begin{prop} \label{resultatgeneral} $ $
Let $G = L \times R$ and $G' = L' \times R$ be two 2-elementary subgroups of the de Jonquières group, where $L, L'$ are subgroups of $\PGL(2,\kk(x))$ of the same order, and $R$ is a subgroup of $\PGL(2,\kk)$.
Let $\mathcal{N}'$ be the normalizer of $R$ in $\PGL(2,\kk)$.

Then $G$ and $G'$ are conjugate if and only if there exists $\nu \in \mathcal{N}'$ such that $\nu L \nu^{-1}$ and $L'$ have the same determinant in $\kk(x)^{\times} / {(\kk(x)^{\times})}^2 $.
\end{prop}

\begin{proof} $ $
    The subgroups $L \times R $ and $L' \times R$ are conjugate in the de Jonquières group

    $\Leftrightarrow$ $\exists \nu \in \PGL(2,\kk)$ s.t. $\nu R \nu^{-1} = R$, and $\nu L \nu^{-1}$ and $L'$ are conjugate in $\PGL(2,\kk(x))$

    $\Leftrightarrow$ $\exists \nu \in \mathcal{N}'$ such that $det(\nu L \nu^{-1}) = det(L')$ (Proposition \ref{dublin1} Assertion \ref{orsaydublin12}). 
\end{proof}

\subsection{A description of the normalizer of the subgroup $R \subset \PGL(2,\kk)$} $ $

The normalizer of $\langle \sigma \rangle$ is described in Lemma \ref{commutateur} Assertion \ref{commutateursigma}. The normalizer of $\langle \sigma , \tau \rangle$ will be described in Proposition \ref{normalisateur2}.

\begin{remark} \label{remarkorsayautv} $ $ (Nature of the automorphism group of the Klein group)

We have $\Aut(V) \simeq \mathfrak{S}_3$ and $\Aut(V)$ induces all permutations of the set $V \setminus \{I\}$.
\end{remark}

\begin{defi} $ $
Let $V = \langle \sigma , \tau \rangle \subset \PGL(2,\kk)$ and denote by $\mathcal{N}$ 
its normalizer in $\PGL(2,\kk)$.
Let $\alpha = \left[\begin{smallmatrix}  1 & 1   \\  1 & -1  \end{smallmatrix} \right] \in \PGL(2,\kk)$ and $\beta = \left[\begin{smallmatrix}  1 & -i   \\  i & -1  \end{smallmatrix} \right] \in \PGL(2,\kk)$.
We define the following group homomorphism: 
$$\begin{array}{ccccc}
\theta & : & \mathcal{N} & \to & \Aut(V) \\
 & & n & \mapsto & (v \mapsto  nvn^{-1}) \\
\end{array}$$
\end{defi}

\begin{lemme} \label{conjugationJlemme} $ $
We have the following equalities:
$$\alpha^2 = \beta^2 =  id, \quad \alpha \sigma \alpha^{-1}  =  \tau, \quad \beta \sigma \beta^{-1} =  \sigma \tau, \quad \alpha \beta \alpha  =  \beta \alpha \beta $$
\end{lemme}

\begin{proof} $ $
Direct computations.
\end{proof}

\begin{lemme} \label{chineorsay57} $ $
\begin{enumerate}
\item  \label{chineorsay571} The kernel of the group homomorphism $\theta$ is the group $V$.
\item  \label{chineorsay572}
The group $\langle \alpha , \beta \rangle$ is a subgroup of $\mathcal{N}$ isomorphic to $\mathfrak{S}_3$. \newline Furthermore, the map $\theta_{\langle \alpha , \beta \rangle} : \langle \alpha , \beta \rangle \rightarrow \Aut(V)$ is a group isomorphism.
\end{enumerate}
\end{lemme}

\begin{proof} 
\begin{enumerate}
\item We observe that the kernel of $\theta$ is the centralizer of $V$, and using Lemma \ref{commutateur} Assertion \ref{commutateursigmatau}, this centralizer is $V$ itself.
\item
This is a direct consequence of Lemma \ref{conjugationJlemme} and Remark \ref{remarkorsayautv}.
\end{enumerate}
\end{proof}

%

\begin{prop} \label{normalisateur2} $ $
\begin{enumerate}
\item \label{normalisateur2old}
The normalizer of $V \subset \PGL(2,\kk)$ is: $$\langle \sigma , \tau , \alpha , \beta \rangle = \langle \sigma , \tau \rangle \rtimes \langle \alpha , \beta \rangle \simeq V \rtimes \mathfrak{S}_3 \simeq \mathfrak{S}_4$$
\item \label{chine2} 
The normalizer of $V \subset \PGL(2,\kk)$ induces all permutations of the set $V \setminus \{I\}$.
\end{enumerate}
\end{prop}

\begin{proof} 
\begin{enumerate}
\item 
 This is a direct consequence of Lemma \ref{chineorsay57}.
%
%
\item This is a direct consequence of Lemma 
\ref{chineorsay57} Assertion \ref{chineorsay572}.
\end{enumerate}
\end{proof}

\begin{remark} $ $
The group $V \rtimes \Aut(V)$ is called the holomorph of $V$.
\end{remark}

\newpage

\section{Subgroups with non-trivial elements fixing non-rational curve} \label{jonquieressection33} $ $

We assume $\car(\kk) \neq 2 = p$ in this section.
The goal of this section is to give a list of representatives of $2$-elementary subgroups containing at least one non-trivial element fixing a non-rational curve.
We remind the reader that $\sigma = \left[\begin{smallmatrix}  1 & 0   \\  0 & -1  \end{smallmatrix} \right] , \tau = \left[\begin{smallmatrix}  0 & 1   \\  1 & 0  \end{smallmatrix} \right]$.

\subsection{The de Jonquières elements fixing a non-rational curve} $ $

\begin{lemme} \label{ultimatelemma} $ $
Let $j$ be a non-trival de Jonquières involution fixing a non-rational curve.
Then $j \in \Ker(\pi)$. 
\end{lemme}

\begin{proof} $ $
We assume that $j \notin \Ker(\pi)$.
We write $j = (l,r)$ where $r \in \PGL(2,\kk)$ is a non-trivial involution. According to Proposition \ref{dublin1}, up to conjugation by an element of $\PGL(2,\kk)$ we can assume that $r = \sigma$. Hence the fixed locus of $j$ is a subset of $\{ x = 0 \} \cup \{x = \infty \}$. Hence it cannot contains a non-rational curve.
\end{proof}

\begin{lemme} \label{ultimatelemma20} $ $

Let $g \in \kk[X]^*$ be a monic polynomial with simple roots of degree $d \in \mathbb{N}_{\geq 0}$. \newline
The subset $\left\{y^2 = g(x)\right\} \subset \mathbb{A}^2$ contains a non-rational curve if and only if $d \geq 3$.
\end{lemme}

\begin{proof} $ $
This is a direct consequence of the Riemann-Hurwitz formula.
\end{proof}

\begin{lemme} \label{ultimatelemma2} $ $

Let $g \in \kk[X]^*$ be a monic polynomial with simple roots of degree $d \in \mathbb{N}_{\geq 0}$.
\begin{enumerate}
\item The element $j_1 = (\left[ \begin{smallmatrix} 0 & g \\ 1 & 0 \end{smallmatrix} \right],I) \in \J$ fixes a non-rational curve if and only if $d \geq 3$.
\item The element $j_2 = (\left[ \begin{smallmatrix} 0 & g(x^2) \\ 1 & 0 \end{smallmatrix} \right],I) \in \J$ fixes a non-rational curve if and only if $g$ has at least two distinct roots not equal to $0$.
\item The element $j_3 = (\left[ \begin{smallmatrix} 0 & g(x^2 + \frac{1}{x^2}) \\ 1 & 0 \end{smallmatrix} \right],I) \in \J$ fixes a non-rational curve if and only if $g$ has at least one root not equal to $\pm 2$.
\end{enumerate}
\end{lemme}

\begin{proof} $ $

\begin{enumerate}
\item The fixed locus of $j_1$ is $\left\{y^2 = g(x)\right\}$. We can then use Lemma \ref{ultimatelemma20}.
\item The fixed locus of $j_2$ is $\left\{y^2 = g(x^2)\right\}$. If $0$ is not a root of $g$ then $g(x^2)$ has no multiple roots and we can use Lemma \ref{ultimatelemma20}. If $0$ is a root of $g$ then we write $g(x^2) = x^2 h(x)$ where $h$ is a monic polynomial with simple roots of degree $2d-2$. The fixed locus of $j_2$ is $\left\{y^2 = x^2 h(x )\right\}$. Using the map $(x,y) \rightarrow (x,xy)$, this fixed locus is birationally equivalent to $\left\{y^2 = h(x) \right\}$. We can then use Lemma \ref{ultimatelemma20}.
\item Let $r_1, \dots, r_d \in \kk$ be the roots of $g$. The fixed locus of $j_3$ is: $$\left\{y^2 x^{2d}= \Pi_{i=1}^{d} (x^4 -  r_i x^2 + 1) \right\}$$
Using the map $(x,y) \rightarrow (x,x^{-d} y)$, the fixed locus is birationally equivalent to: $$\left\{y^2 = \Pi_{i=1}^{d} (x^4 - r_i x^2 + 1)\right\}$$
We observe that the right part $\Pi_{i=1}^{d} (x^4 -  r_i x^2 + 1)$ is with simple roots if and only if $r_{i} \neq \pm 2$ for all $i$. In that case we can use Lemma \ref{ultimatelemma20}.
If $g(2)=0$ and $g(-2) \neq 0$ \textit{resp} $g(2) \neq 0$ and $g(-2) = 0$ \textit{resp} $g(2)=g(-2)=0$, then using the map $(x,y) \rightarrow (x,(x^2-2)^{-1} y)$ \textit{resp} $(x,y) \rightarrow (x,(x^2+2)^{-1} y)$ \textit{resp} $(x,y) \rightarrow (x,(x^2-2)^{-1} (x^2+2)^{-1} y)$, this fixed locus is birationally equivalent to: $$\left\{y^2 = h(x) \right\},$$
where $h$ is a monic polynomial with simple roots of degree $4d-4$ \textit{resp} $4d-4$ \textit{resp} $4d-8$. We can then use Lemma \ref{ultimatelemma20}.
\end{enumerate}
\end{proof}


\subsection{A list of representatives of $2$-elementary subgroups of the de Jonquières group with non-trivial elements fixing non-rational curve} $ $

\begin{prop} \label{ultimateproposition} $ $
Let $G$ be a $2$-elementary subgroup of the de Jonquières group.
The subgroup $G$ has a non-trivial element fixing a non-rational curve if and only if $G$ is conjugate in the de Jonquières group to one of the following subgroups:
\begin{itemize}
\item $\langle \left[ \begin{smallmatrix} 0 & g \\ 1 & 0 \end{smallmatrix} \right] \rangle \times \langle I \rangle$ where $g \in \kk[X]^*$ is a monic polynomial with simple roots of degree at least three.
\item $\langle \left[ \begin{smallmatrix} 0 & g \\ 1 & 0 \end{smallmatrix} \right],  \left[ \begin{smallmatrix} a & - b g \\ b & - a \end{smallmatrix} \right] \rangle \times \langle I \rangle$ where $g \in \kk[X]^*$ is a monic polynomial with simple roots of degree at least three.
\item $\langle \left[ \begin{smallmatrix} 0 & g(x^2) \\ 1 & 0 \end{smallmatrix} \right] \rangle \times \langle \sigma \rangle$ where $g \in \kk[X]^*$ is a monic polynomial with simple roots with at least two roots not equal to $0$.
\item $\langle \left[ \begin{smallmatrix} 0 & g(x^2) \\ 1 & 0 \end{smallmatrix} \right],  \left[ \begin{smallmatrix} a(x^2) & - b(x^2) g(x^2) \\ b(x^2) & - a(x^2) \end{smallmatrix} \right] \rangle \times \langle \sigma \rangle$ where $g \in \kk[X]^*$ is a monic polynomial with simple roots with at least two roots not equal to $0$.
\item $\langle \left[ \begin{smallmatrix} 0 & g(x^2+ \frac{1}{x^2}) \\ 1 & 0 \end{smallmatrix} \right] \rangle \times \langle \sigma, \tau \rangle$ where $g \in \kk[X]^*$ is a monic polynomial with simple roots of degree with at least one root not equal to $\pm 2$.
\item $\langle \left[ \begin{smallmatrix} 0 & g(x^2+ \frac{1}{x^2}) \\ 1 & 0 \end{smallmatrix} \right],  \left[ \begin{smallmatrix} a(x^2+ \frac{1}{x^2}) & - b(x^2+ \frac{1}{x^2}) g(x^2+ \frac{1}{x^2}) \\ b(x^2+ \frac{1}{x^2}) & - a(x^2+ \frac{1}{x^2}) \end{smallmatrix} \right] \rangle \times \langle \sigma, \tau \rangle$ where $g \in \kk[X]^*$ is a monic polynomial with simple roots with at least one root not equal to $\pm 2$.
\end{itemize}
And where $a, b \in \kk(x)$ are such that $a^2 \neq b^2 g$.
\end{prop}

\begin{proof} $ $

$\Rightarrow:$
Let $G \simeq (\mathbb{Z}/2)^r$ be a $2$-elementary subgroup of the de Jonquières group containing a non-trivial element fixing a non-rational curve. We will show that $G$ is conjugate to one of the subgroups of the list.
Let $a, b \in \mathbb{N}_{\geq 0}$ such that $\Ker(\pi|_G) \simeq (\mathbb{Z}/2)^a$ and $\Img(\pi|_G) \simeq (\mathbb{Z}/2)^b$.
Using Proposition \ref{dublin1}, up to conjugation by an element of $\PGL(2,\kk)$, we can assume that $\pi(G)$ is of the following subgroups:
$$\{I\}, \langle \sigma \rangle, \langle \sigma, \tau \rangle$$
Let $R = \pi(G)$.
Using Proposition \ref{cohomologie0}, we can assume up to conjugation by an element of $\PGL(2,\kk(x))$ that $\{I \} \times R \subset G$.
Since $G$ is abelian, we get that $\{I \} \times R$ and $\Ker( {\pi}|_G)$ commute. Hence $\Ker ({\pi}|_G) \subset \PGL(2,\kk(x)^{R})$.
Now we distinguish three cases depending on the value of $b$.
\begin{itemize}
\item $b=0$. We have $G \subset \PGL(2,\kk(x))$. Let $j \in G$ be a non-trivial element fixing a non-rational curve.
According to Proposition \ref{dublin1} we can assume up to conjugation inside $\PGL(2,\kk(x))$ that $j = \left[\begin{smallmatrix}  0 & g   \\  1 & 0  \end{smallmatrix} \right]$ where $g \in \kk[X]^*$ is a monic polynomial with simple roots. Since $j$ fixes a non-rational curve, this implies that the degree of $g$ is at least three (Lemma \ref{ultimatelemma2}).
Hence according to Proposition \ref{dublin1} we have: \begin{center}
$G = \langle \left[ \begin{smallmatrix} 0 & g \\ 1 & 0 \end{smallmatrix} \right] \rangle$
or $G = \langle \left[ \begin{smallmatrix} 0 & g \\ 1 & 0 \end{smallmatrix} \right],  \left[ \begin{smallmatrix} a & - b g \\ b & - a \end{smallmatrix} \right] \rangle$ \end{center}
where $a, b \in \kk(x)$ are such that $a^2 \neq b^2 g$.
\item $b=1$. Using Lemma \ref{stabilisationbis}, we have $\PGL(2,\kk(x)^R) = \PGL(2,\kk(x^2))$. Let $j \in G$ be a non-trivial element fixing a non-rational curve. According to Lemma \ref{ultimatelemma} we have $j \in \Ker(\pi)$. Hence $j \in \PGL(2,\kk(x^2))$. Then according to Proposition \ref{dublin1}, we can assume up to a conjugation inside $\PGL(2,\kk(x^2))$ that $j = \left[\begin{smallmatrix}  0 & g(x^2)   \\  1 & 0  \end{smallmatrix} \right]$ where $g \in \kk[X]^*$ is a monic polynomial with simple roots. Since $j$ fixes a non-rational curve, this implies that the degree of $g$ is at least two (Lemma \ref{ultimatelemma2}).
Hence according to Proposition \ref{dublin1} we have: \begin{center}
$\Ker(\pi|_G) = \langle \left[ \begin{smallmatrix} 0 & g(x^2) \\ 1 & 0 \end{smallmatrix} \right] \rangle$
or $\Ker(\pi|_G) = \langle \left[ \begin{smallmatrix} 0 & g(x^2) \\ 1 & 0 \end{smallmatrix} \right],  \left[ \begin{smallmatrix} a(x^2) & - b(x^2) g(x^2) \\ b(x^2) & - a(x^2) \end{smallmatrix} \right] \rangle$ \end{center}
where $a, b \in \kk(x)$ are such that $a^2 \neq b^2 g$.
Hence: \begin{center} $G = \langle \left[ \begin{smallmatrix} 0 & g(x^2) \\ 1 & 0 \end{smallmatrix} \right] \rangle \times \langle \sigma \rangle$
or $G = \langle \left[ \begin{smallmatrix} 0 & g(x^2) \\ 1 & 0 \end{smallmatrix} \right],  \left[ \begin{smallmatrix} a(x^2) & - b(x^2) g(x^2) \\ b(x^2) & - a(x^2) \end{smallmatrix} \right] \rangle \times \langle \sigma \rangle$ \end{center}
\item $b=2$. Using Lemma \ref{stabilisationbis}, we have $\PGL(2,\kk(x)^R) = \PGL(2,\kk(x^2 + \frac{1}{x^2}))$. Let $j \in G$ be a non-trivial element fixing a non-rational curve. According to Lemma \ref{ultimatelemma} we have $j \in \Ker(\pi)$. Hence $j \in \PGL(2,\kk(x^2+ \frac{1}{x^2}))$. Then according to Proposition \ref{dublin1}, we can assume up to a conjugation inside $\PGL(2,\kk(x^2+ \frac{1}{x^2}))$ that $j = \left[\begin{smallmatrix}  0 & g(x^2+ \frac{1}{x^2})   \\  1 & 0  \end{smallmatrix} \right]$ where $g \in \kk[X]^*$ is a monic polynomial with simple roots. Since $j$ fixes a non-rational curve, this implies that the degree of $g$ is at least one (Lemma \ref{ultimatelemma2}).
Hence according to Proposition \ref{dublin1} we have:
\begin{center}
$\Ker(\pi|_G) = \langle \left[ \begin{smallmatrix} 0 & g(x^2+ \frac{1}{x^2}) \\ 1 & 0 \end{smallmatrix} \right] \rangle$
or $\Ker(\pi|_G) = \langle \left[ \begin{smallmatrix} 0 & g(x^2+ \frac{1}{x^2}) \\ 1 & 0 \end{smallmatrix} \right],  \left[ \begin{smallmatrix} a(x^2+ \frac{1}{x^2}) & - b(x^2+ \frac{1}{x^2}) g(x^2+ \frac{1}{x^2}) \\ b(x^2+ \frac{1}{x^2}) & - a(x^2+ \frac{1}{x^2}) \end{smallmatrix} \right] \rangle$ \end{center}
where $a, b \in \kk(x)$ are such that $a^2 \neq b^2 g$.
Hence: \begin{center} $G = \langle \left[ \begin{smallmatrix} 0 & g(x^2+ \frac{1}{x^2}) \\ 1 & 0 \end{smallmatrix} \right] \rangle \times \langle \sigma, \tau \rangle$ \end{center}

or $G = \langle \left[ \begin{smallmatrix} 0 & g(x^2+ \frac{1}{x^2}) \\ 1 & 0 \end{smallmatrix} \right],  \left[ \begin{smallmatrix} a(x^2+ \frac{1}{x^2}) & - b(x^2+ \frac{1}{x^2}) g(x^2+ \frac{1}{x^2}) \\ b(x^2+ \frac{1}{x^2}) & - a(x^2+ \frac{1}{x^2}) \end{smallmatrix} \right] \rangle \times \langle \sigma, \tau \rangle$. \\ \\
\end{itemize}

$\Leftarrow:$ The converse implication is trivial by construction of this list of subgroups.
\end{proof}

\newpage

\part{Subgroups of automorphisms of del Pezzo surfaces} \label{part3} $ $

 \bigskip\bigskip
 
Now that we have classified $p$-elementary subgroups of the de Jonquières group, we study $p$-elementary subgroups of automorphisms of del Pezzo surfaces. We only study subgroups of automorphisms of del Pezzo of degree lower or equal than $4$, and $\mathbb{P}^1 \times \mathbb{P}^1$. Furthermore, for each of these cases (except $\PP^1 \times \PP^1$), we study only the $p$-elementary subgroups when the degree of the surface $K_S^2$ is a power of $p$. We summarize in the following table the cases we study:
\\

\begin{center}
\begin{tabular}{ | c | c | c | c | }
 \hline 
   Del Pezzo surface & $p$ & $\car(\kk)$ & Corresponding section  \\
   \hline \hline
   $\PP^1 \times \PP^1$ & any & any & Section \ref{orsay41}  \\
   \hline
      Degree 4 (quartic) & 2 & not 2  & Section \ref{quarticsection} \\
   \hline
      Degree 3 (cubic) & 3 & not 3 & Section \ref{cubicsection} \\
   \hline
      Degree 2 & 2 & not 2 & Section \ref{orsay44} \\
   \hline
      Degree 1 & any & not $p$ and not $2$ & Sections \ref{orsay45}   \\
   \hline
   Degree 1 & not $2$ & $2$ & Sections  \ref{orsay46}  \\
   \hline
 \end{tabular} 
 \end{center}
 \bigskip
 
 The six sections of this part are mutually independent.
 
  \bigskip

\begin{remark} $ $

It is in fact not needed for our study to classify the $p$-elementary subgroups of other del Pezzo surfaces because of the important Proposition \ref{prop1}:
\begin{itemize}
\item Del Pezzo surfaces of degree $7$ or $8$ given by the blow-up of a point or two points of $\PP^2$: Their automorphism groups are conjugate to a subgroup of $\Aut(\PP^2)$.
\item Del Pezzo surface of degree $6$: The number $6$ is not the power of a prime number, hence this case can be excluded using Proposition \ref{prop1}.
\item Del Pezzo surface of degree $5$: Its automorphism group does not contain subgroups isomorphic to $(\ZZ/5)^2$. We conclude using Proposition \ref{prop1} again.
\end{itemize}
\end{remark}

\newpage

\section{Automorphisms of the Hirzebruch surface $\mathbb{P}^1 \times \mathbb{P}^1$} \label{orsay41} $ $

The goal of this section is to classify up to conjugation the $p$-elementary non-cyclic subgroups of the automorphism group of $\mathbb{P}^1 \times \mathbb{P}^1$.
This section is organized as follows:
\begin{itemize}
\item We first provide a preliminary lemma in Subsection \ref{orsay411} that we will use in both Subsections \ref{orsay412} and \ref{orsay413}.
\item In Subsection \ref{orsay412} we address the case where $\car(\kk)=p$.
\item In Subsection \ref{orsay413} we show that every $p$-elementary subgroup in characteristic not $p$ is conjugate to a de Jonquières subgroup.
\item Finally in Subsection \ref{orsay414} we investigate more deeply the case of $2$-elementary non-cyclic subgroups in characteristic not $2$ and give a complete list of representatives up to birational map.
\end{itemize}

\begin{defi} $ $
We define the following automorphism of $\mathbb{P}^1 \times \mathbb{P }^1$: $$\begin{array}{ccccc}
s& : & \mathbb{P}^1 \times \mathbb{P }^1& \rightarrow & \mathbb{P}^1 \times \mathbb{P }^1\\
 & & (x,y) & \mapsto & (y , x) \\
\end{array}$$
\end{defi}

\begin{remark} $ $ 
We remind that: $\Aut(\mathbb{P}^1 \times \mathbb{P}^1) \simeq (\Aut(\mathbb{P}^1) \times \Aut(\mathbb{P}^1 )) \rtimes \langle s \rangle$.
\end{remark}

\medskip

\subsection{A first lemma on $2$-elementary subgroups of $\Aut(\mathbb{P}^1 \times \mathbb{P}^1)$} \label{orsay411}
$ $


\begin{lemme} \label{autP1P12elem} $ $

If $G$ is a $2$-elementary subgroup of $\Aut(\mathbb{P}^1 \times \mathbb{P}^1)$ not contained in $\Aut(\mathbb{P}^1) \times \Aut(\mathbb{P}^1 )$.
Then up to conjugation, there exists $G'$ such that $G = \langle G',s \rangle$  where $G'$ is a $2$-elementary subgroup of the diagonal of $\PGL(2,\kk) \times \PGL(2,\kk)$.
\end{lemme}

\begin{proof} $ $

We know that $G$ contains an element of the form $((A,B),s)$ with $A, B \in \PGL(2,\kk)$. It is of order $2$ so $A = B^{-1}$. We have:  $(A^{-1},I) ((A,B),s) (A,I) = ((I,I),s)$. So we can assume $A = I$. Then $G$ contains $s$. The intersection of the centralizer of $s$ with $\PGL(2,\kk) \times \PGL(2,\kk)$ is the diagonal of $\PGL(2,\kk) \times \PGL(2,\kk)$.
Hence $G$ is of the form $\langle G' , s \rangle$ where $G'$ is a $2$-elementary subgroup of the diagonal of $\PGL(2,\kk) \times \PGL(2,\kk)$.
\end{proof}


\medskip

\subsection{Classification of $p$-elementary subgroups in characteristic $p$}  \label{orsay412}
$ $


\begin{prop} \label{moscou1} $ $ 
We assume $\car(\kk) = p > 0$.

Let $G$ be a non-trivial $p$-elementary subgroup of $\Aut(\mathbb{P}^1 \times \mathbb{P}^1)$.

\begin{itemize}
\item If $p > 2$, then $G$ is conjugate by an automorphism to a subgroup of the following subgroup:
$\left\{ (\left[\begin{smallmatrix}  1 & t   \\
0 &   1
\end{smallmatrix} \right],\left[\begin{smallmatrix}  1 & t'   \\
0 &   1
\end{smallmatrix} \right] ) | t,t' \in \kk \right\}$.

\item If $p = 2$, then $G$ is conjugate by an automorphism to a subgroup $G' = \langle H , s \rangle$ 
where $H \subset \left\{ (\left[\begin{smallmatrix}  1 & t   \\
0 &   1
\end{smallmatrix} \right],\left[\begin{smallmatrix}  1 & t   \\
0 &   1
\end{smallmatrix} \right] ) | t \in \kk \right\} \subset \PGL(2,\kk) \times \PGL(2,\kk)$ and contains $(\left[\begin{smallmatrix}  1 & 1   \\
0 &   1
\end{smallmatrix} \right],\left[\begin{smallmatrix}  1 & 1   \\
0 &   1
\end{smallmatrix} \right] )$ if non-trivial,
or to a subgroup of $\left\{ (\left[\begin{smallmatrix}  1 & t   \\
0 &   1
\end{smallmatrix} \right],\left[\begin{smallmatrix}  1 & t'   \\
0 &   1
\end{smallmatrix} \right] ) | t,t' \in \kk \right\}$.
\end{itemize}
\end{prop}

\begin{proof} $ $
We distinguish three cases:

If $p > 2$ then $G$ is a subgroup of $\Aut(\mathbb{P}^1) \times \Aut(\mathbb{P}^1)$, then we use Proposition \ref{dublin1} Assertion 3 \ref{hukfg3}.

If $p =2$ and $G$ is a subgroup of $\Aut(\mathbb{P}^1) \times \Aut(\mathbb{P}^1)$, then we use Proposition \ref{dublin1} Assertion 3 \ref{hukfg3}.

If $p =2$ and $G$ is not a subgroup of $\Aut(\mathbb{P}^1) \times \Aut(\mathbb{P}^1)$, 
then according to Lemma \ref{autP1P12elem}, $G$ is of the form $\langle G' , s \rangle$ where $G'$ is a $2$-elementary subgroup of the diagonal of $\PGL(2,\kk) \times \PGL(2,\kk)$.
We conclude using Proposition \ref{dublin1} Assertion 3 \ref{hukfg3}.
\end{proof}

\medskip


\subsection{Classification of $p$-elementary subgroups in characteristic not $p$: Conjugation to a de Jonquières subgroup}  \label{orsay413} $ $

\medskip

We assume $\car(\kk) \neq p$ in this subsection. The main result is Proposition \ref{prop27}. The notation: $$(a(x,y),b(x,y))$$ will be used for the following birational map of $\mathbb{P}^1 \times \mathbb{P}^1$: 
$$(x,y) \mapsto (a(x,y),b(x,y))$$

\medskip

\begin{lemme} \label{lemme27} $ $
We assume $\car(\kk) \neq p$.

\noindent Let $G$ be a $p$-elementary subgroup of $\Aut(\mathbb{P}^1~\times \mathbb{P}^1)$ not contained in $\Aut(\mathbb{P}^1)~\times~\Aut(\mathbb{P}^1 )$.
Then $p=2$ and up to conjugation $G$ is one of the following subgroups:
$$\langle (y,x) \rangle, \quad \langle (-x,-y) , (y,x) \rangle, \quad \langle (-x,-y) , (\frac{1}{x} , \frac{1}{y}) , (y,x) \rangle$$
\end{lemme}

\begin{proof} $ $
If $G$ is not a subgroup of $\Aut(\mathbb{P}^1) \times \Aut(\mathbb{P}^1 )$, then according to Lemma \ref{autP1P12elem}, $G$ is of the form $\langle G' , s \rangle$ where $G'$ is a $2$-elementary subgroup of the diagonal of $\PGL(2,\kk) \times \PGL(2,\kk)$.
According to Proposition \ref{dublin1} Assertion 2 \ref{chargr2inv4}, up to conjugation, $G' = \{ I \}$ or $G' = \langle (-x,-y) \rangle$ or $G' = \langle (-x,-y) , (\frac{1}{x} , \frac{1}{y}) \rangle$.
Then $G$ is conjugate to one of the three following subgroups: \[ \langle (y,x) \rangle, \langle (-x,-y),(y,x) \rangle, \langle (-x,-y),(\frac{1}{x},\frac{1}{y}),(y,x) \rangle\]
\end{proof}

\begin{prop} \label{prop27} \label{p11} $ $
We assume $\car(\kk) \neq p$.

\noindent    Every $p$-elementary non-cyclic subgroup of the automorphism group of $\mathbb{P}^1 \times \mathbb{P}^1$ is $\Cr$-conjugate to a subgroup of the de Jonqui\`eres group.
\end{prop}

\begin{proof} $ $
We distinguish two cases:

If $G \subset \Aut(\mathbb{P}^1) \times \Aut(\mathbb{P}^1)$ then both conic bundle structures of $\mathbb{P}^1 \times \mathbb{P}^1$ are preserved.

If $G$ is not a subgroup of $\Aut(\mathbb{P}^1) \times \Aut(\mathbb{P}^1 )$, then according to Lemma \ref{lemme27}, $G = \langle (-x,-y) , (y,x) \rangle$ or $G = \langle (-x,-y) , (\frac{1}{x} , \frac{1}{y}),(y,x)\rangle$ up to conjugation. We define the following birational map:
$$\begin{array}{ccccc}
\phi & : & \mathbb{P}^1 \times \mathbb{P }^1& \DashedArrow & \mathbb{P}^1 \times \mathbb{P }^1\\
 & & (x,y) & \DashedArrow & (\dfrac{x}{y} , x) \\
\end{array}$$
Then $\phi$ conjugates $(-x,-y), (\frac{1}{x}, \frac{1}{y}), (y,x)$ respectively on $(x,-y)$, $(\frac{1}{x} , \frac{1}{y})$, $(\frac{1}{x} , \frac{y}{x})$, which generate a de Jonqui\`eres subgroup.
\end{proof}


\medskip

\subsection{Detailed classification of $2$-elementary non-cyclic subgroups in characteristic not $2$: A list of representatives}  \label{orsay414} $ $

\medskip

We assume $\car(\kk) \neq p = 2$ in this subsection.
Like in the previous subsection we use the notation: $$(a(x,y),b(x,y))$$ for the following birational map of $\mathbb{P}^1 \times \mathbb{P}^1$: 
$$(x,y) \mapsto (a(x,y),b(x,y))$$

\medskip

\begin{prop} \label{detailedclassificationsubgsF0} $ $
Let $G$ be a $2$-elementary non-cyclic subgroup of $\Aut(\mathbb{P}^1 \times \mathbb{P}^1)$. Then $G$ is conjugate by a birational map to exactly one of the following subgroups of $\Aut(\mathbb{P}^1 \times \mathbb{P}^1)$:
$$\begin{array}{lll}
G_{81} & = & \langle (-x,y),(x,-y) \rangle \\
G_{82} & =  & \langle (-x,y),(\dfrac{1}{x},y) \rangle  \\
G_{83} & =  & \langle (-x,-y),(\dfrac{1}{x},\dfrac{1}{y}),(y,x) \rangle  \\
G_{84} & =  & \langle (-x,y),(\dfrac{1}{x},y),(x,-y) \rangle \\
G_{85} & = & \langle (\dfrac{1}{x},-y),(x,\dfrac{1}{y}),(-x,y) \rangle \\
G_{86} & =  & \langle (-x,y),(x,-y),(\dfrac{1}{x},y),(x,\dfrac{1}{y}) \rangle
\end{array}$$
No non-trivial element in these six subgroups fixes a non-rational curve.
Furthermore:

 Two subgroups of $\Aut(\mathbb{P}^1 \times \mathbb{P}^1)$ isomorphic to $(\mathbb{Z}/2)^3$ are conjugate by a birational map if and only if they have the same number of Klein subgroups with no fixed points.
 
Two Klein subgroups of $\Aut(\mathbb{P}^1 \times \mathbb{P}^1)$ are conjugate by a birational map if and only if they both have fixed points, or if they both have no fixed points.
\end{prop}



\begin{proof} $ $
This is an adaptation of the proofs of \cite[Proposition 6.2.3]{blancthesis} and \cite[Proposition 6.2.4]{blancthesis}.

If $G \not \subset \Aut(\mathbb{P}^1) \times \Aut(\mathbb{P}^1)$ then we use Lemma \ref{lemme27} and we get the following subgroups:
$\quad G_{87}  :=  \langle (-x,-y),(y,x) \rangle, \quad G_{83}  :=   \langle (-x,-y),(\frac{1}{x},\frac{1}{y}),(y,x) \rangle$.


If $G \subset \Aut(\mathbb{P}^1) \times \Aut(\mathbb{P}^1)$ then we define the projections $\pi_1, \pi_2 : \Aut(\mathbb{P}^1) \times \Aut(\mathbb{P}^1) \twoheadrightarrow \Aut(\mathbb{P}^1)$ onto the first and second coordinates. We have $G \subset \pi_1(G) \times \pi_2(G)$. According to Proposition \ref{dublin1} Assertion \ref{orsaydublin12}  \ref{chargr2inv4} we can assume up to conjugation: $$\pi_1(G), \pi_2(G) \in \{ \{ I \}, \langle \left[\begin{smallmatrix}  1 & 0   \\
0 &   -1
\end{smallmatrix} \right] \rangle , \langle \left[\begin{smallmatrix}  1 & 0   \\
0 &   -1
\end{smallmatrix} \right], \left[\begin{smallmatrix}  0 & 1   \\
1 &   0
\end{smallmatrix} \right] \rangle\}$$
We distinguish four cases depending on $\pi_1(G)$ and $\pi_2(G)$:
\begin{itemize}
\item Case $\pi_1(G) = \pi_2(G) = \langle \left[\begin{smallmatrix}  1 & 0   \\
0 &   -1
\end{smallmatrix} \right] \rangle$:

Then $G =  \langle (-x,y),(x,-y) \rangle =: G_{81}$.

\item Case $\pi_1(G) = \langle \left[\begin{smallmatrix}  1 & 0   \\
0 &   -1
\end{smallmatrix} \right], \left[\begin{smallmatrix}  0 & 1   \\
1 &   0
\end{smallmatrix} \right] \rangle , \pi_2(G) = \{ I \}$:

Then $G = \langle (-x,y),(\frac{1}{x},y) \rangle =: G_{82}$.

\item Case $\pi_1(G) = \langle \left[\begin{smallmatrix}  1 & 0   \\
0 &   -1
\end{smallmatrix} \right], \left[\begin{smallmatrix}  0 & 1   \\
1 &   0
\end{smallmatrix} \right] \rangle, \pi_2(G) = \langle \left[\begin{smallmatrix}  1 & 0   \\
0 &   -1
\end{smallmatrix} \right] \rangle$:

If $G \simeq (\mathbb{Z}/2)^3$ then $G = \langle (-x,y),(\frac{1}{x},y),(x,-y) \rangle =: G_{84}$.

If $G \simeq (\mathbb{Z}/2)^2$ then there exist $f,g \in \{ -x, \frac{1}{x}, - \frac{1}{x} \}, f \neq g$ such that $G = \langle (f,-y),(g,y) \rangle$. Using Proposition \ref{normalisateur2} Assertion \ref{chine2} we can assume $f = -x$ and $g = \frac{1}{x}$. Hence $G =   \langle (-x,-y),(\frac{1}{x},y) \rangle =: G_{88}$ 

\item Case $\pi_1(G) = \pi_2(G) = \langle \left[\begin{smallmatrix}  1 & 0   \\
0 &   -1
\end{smallmatrix} \right], \left[\begin{smallmatrix}  0 & 1   \\
1 &   0
\end{smallmatrix} \right] \rangle$:

If $G \simeq (\mathbb{Z}/2)^2$ then there exist $f,g \in \{ -x, \frac{1}{x}, - \frac{1}{x} \}, f \neq g$ such that $G = \langle (f,-y),(g,\frac{1}{y}) \rangle$. Using Proposition \ref{normalisateur2} Assertion \ref{chine2} we can assume $f = -x$ and $g = \frac{1}{x}$. Hence $G =   \langle (-x,-y),(\frac{1}{x},\frac{1}{y}) \rangle =: G_{89}$.

If $G \simeq (\mathbb{Z}/2)^4$ then $G = \langle (-x,y),(x,-y),(\frac{1}{x},y),(x,\frac{1}{y}) \rangle =: G_{86}$.

If $G \simeq (\mathbb{Z}/2)^3$ then there exist $f, g, h \in \{x, -x, \frac{1}{x}, - \frac{1}{x} \}$ such that $h \neq x$ and $G = \langle (f,-y),(g,\frac{1}{y}),(h,y) \rangle$. Using Proposition \ref{normalisateur2} Assertion \ref{chine2} we can assume $h = -x$. Because $\pi_1(G) = \langle \left[\begin{smallmatrix}  1 & 0   \\
0 &   -1
\end{smallmatrix} \right], \left[\begin{smallmatrix}  0 & 1   \\
1 &   0
\end{smallmatrix} \right] \rangle$, we need $f$ or $g$ to be equal to $\frac{1}{x}$ or to $- \frac{1}{x}$. Then using Proposition \ref{normalisateur2} Assertion \ref{chine2} and permuting $-y$ and $\frac{1}{y}$ if necessary, we can assume $f \in \{ \frac{1}{x} , - \frac{1}{x} \}$. Using Proposition \ref{normalisateur2} Assertion  \ref{chine2} and permuting $\frac{1}{x}$ and $- \frac{1}{x}$ we can assume $f = \frac{1}{x}$. Hence we have $G = \langle (\frac{1}{x},-y),(g,\frac{1}{y}),(-x,y) \rangle =: G_{85,g}$. We observe that $G_{85,x} = G_{85,-x}$ and $G_{85,\frac{1}{x}} = G_{85,-\frac{1}{x}}$. And finally, $G_{85,x}$ and $G_{85,\frac{1}{x}}$ are conjugate using the permutation of $\frac{1}{y}$ and $- \frac{1}{y}$ (Proposition \ref{normalisateur2} Assertion \ref{chine2}).
Hence we can assume up to conjugation by an automorphism of $\mathbb{P}^1 \times \mathbb{P}^1$:  $G =   \langle  (\frac{1}{x},-y),(x,\frac{1}{y}),(-x,y) \rangle =: G_{85}$.
\\
\end{itemize}


Now we study the conjugation by birational map between these subgroups.

We first show that $G_{82}, G_{88}, G_{89}$ are conjugate by birational map.
We define the following birational maps:
$$\begin{array}{ccccc}
\phi & : & \mathbb{P}^1 \times \mathbb{P }^1& \DashedArrow & \mathbb{P}^1 \times \mathbb{P }^1\\
 & & (x,y) & \DashedArrow & (x, \frac{y+x^2}{y+x^{-2}}) \\  \\
 \psi & : & \mathbb{P}^1 \times \mathbb{P }^1& \DashedArrow & \mathbb{P}^1 \times \mathbb{P }^1\\ 
 &   & (x,y) & \DashedArrow & (x, x \frac{y+x^2}{y+x^{-2}})
\end{array}$$
We observe that $\phi$ commutes with $(-x,y)$ and $\phi (\frac{1}{x},y) \phi^{-1} = (\frac{1}{x},\frac{1}{y})$. So $\phi G_{82} \phi^{-1} = \langle (-x,y),(\frac{1}{x},\frac{1}{y}) \rangle$. We also see that $(\alpha,\alpha) \langle (-x,y),(\frac{1}{x},\frac{1}{y}) \rangle (\alpha,\alpha)^{-1} = G_{88}$ where $\alpha = \left[\begin{smallmatrix}  1 & 1   \\
1 &   -1
\end{smallmatrix} \right] \in \PGL(2,\kk)$ (see Lemma \ref{conjugationJlemme}).
Finally we observe that $\psi G_{82} \psi^{-1} = G_{89}$.

We will now prove that $G_{87}$ and $G_{81}$ are conjugate by a birational map.
We define the following birational map: $$\begin{array}{ccccc}
\sigma & : & \mathbb{P}^1 \times \mathbb{P }^1& \to & \mathbb{P}^2 \\
 & & ([1:x],[1:y]) & \mapsto & [x:y:1] \\
\end{array}$$
We observe that $\sigma$ conjugates the subgroups $G_{87}$ and $G_{81}$ onto subgroups of automorphisms of $\mathbb{P}^2$. But according to Proposition \ref{prop3}, all Klein subgroups of automorphisms of $\mathbb{P}^2$ are conjugate. Hence $G_{87}$ and $G_{81}$ are conjugate by a birational map.
We also observe that $G_{81}$ has no fixed points whereas $G_{82}$ has fixed points. Hence according to \cite[page 1054-1056 Proposition A.2 ]{kollar} these two subgroups are not conjugate by a birational map.
In conclusion, every Klein subgroup of automorphisms of $\mathbb{P}^1 \times \mathbb{P}^1$ is conjugate by a birational map either to $G_{81}$ or to $G_{82}$, and these two subgroups are not conjugate by a birational map.

It remains to show that $G_{83}, G_{84}, G_{85}$ are pairwise not conjugate by a birational map. We will do that by counting the number of Klein subgroups with or without fixed points of these three groups. 
According to Lemma \ref{lemme27}, any Klein subgroup of $G_{83}$ not contained in $\Aut(\mathbb{P}^1) \times \Aut(\mathbb{P}^1)$ is conjugate to $G_{87}$, so has fixed points. This implies that the only Klein subgroup of $G_{83}$ with no fixed point is $\langle (-x,-y),(\frac{1}{x},\frac{1}{y}) \rangle$.
We observe that $G_{84}$ has at least two subgroups with no fixed points: $\langle (-x,y),(\frac{1}{x},y) \rangle, \langle (-x,y),(\frac{1}{x},-y) \rangle$, and at least two subgroups with fixed points: $\langle (-x,y),(x,-y) \rangle, \langle (\frac{1}{x},y),(x,-y) \rangle$.
If $K$ is a Klein subgroup of $G_{85}$ with fixed points, then $\pi_1(K) = \pi_2(K) \simeq \mathbb{Z}/2$, hence $K$ is of the form $\langle (x,f),(g,y) \rangle$ for some $f,g$ of order two. The only elements of this form are $(x,\frac{1}{y})$ and $(-x,y)$. Hence $K = \langle (x,\frac{1}{y}),(-x,y) \rangle$ is uniquely given, \textit{i.e.} $G_{85}$ has only one Klein subgroup with fixed points.
Since a group isomorphic to $(\mathbb{Z}/2)^3$ has seven Klein subgroups, we can use \cite[page 1054-1056 Proposition A.2 ]{kollar}
and we get that these three groups are pairwise not conjugate by a birational map.
\end{proof}

\newpage

\section{Automorphisms of quartic del Pezzo surfaces} \label{quarticsection} $ $

We assume $\car(\kk) \ne 2$ in this section.
The goal of this section is to prove Proposition \ref{quarticresult}. This proposition gives the classification of $2$-elementary non-cyclic subgroups of automorphisms of quartic del Pezzo surfaces. We will use \cite[Section 3]{dolcubic} many times in this study.

\medskip

\subsection{Notations}
$ $

We denote by $[a:b:c:d:e]$ the element $\left[ \begin{smallmatrix} a & 0 & 0 & 0 & 0 \\ 0 & b & 0 & 0 & 0 \\ 0 & 0 & c & 0 & 0 \\ 0 & 0 & 0 & d & 0 \\ 0 & 0 & 0 & 0 & e \end{smallmatrix} \right]$ of $\PGL(5,\kk)$.

\begin{defi} $ $

Let $Z = \{ [\pm 1 : \pm 1 : \pm 1 : \pm 1 : \pm 1 ]\} \subset \PGL(5,\kk)$ be the $2$-torsion subgroup of the diagonal torus of $\PGL(5,\kk)$.
\end{defi}

We know from \cite[Section 3]{dolcubic} that a quartic del Pezzo surface is isomorphic to the complete intersection of two quadrics in $\mathbb{P}^4$.

Let $S$ be a quartic del Pezzo surface.
According to  \cite[Theorem 3.1]{dolcubic}, after a coordinate change of $\mathbb{P}^4$, there exists a subgroup $P \subset \mathfrak{S}_5$ such that $\Aut(S) = Z \rtimes P$.

\medskip

\subsection{Results} $ $

\begin{prop} \label{quarticresult} $ $

Let $G$ be a $2$-elementary non-cyclic subgroup of automorphism of a quartic del Pezzo surface. We assume $rk(Pic(S)^G) = 1$. Then $G \simeq (\mathbb{Z}/2)^r$ where $r \in \{2,3,4\}$, and there exist $a_0, \dots , a_4 \in \kk$ pairwise distinct such that up to conjugation:
$$S = \left\{ \sum_{n=0}^4 X_n^2 = \sum_{n=0}^4 a_n X_n^2 = 0 \right\} \subset \mathbb{P}^4$$ and:

\begin{itemize}
\item If $r=2$, then $G = \{ [\pm 1 : \pm 1 : 1 : 1 : 1 ]\}$.
\item If $r=3$, then $G = \{ [\pm 1 : \pm 1 : \pm 1 : 1 : 1 ]\}$.
\item If $r=4$, then $G = \{ [\pm 1 : \pm 1 : \pm 1 : \pm 1 : \pm 1 ]\}$.
\end{itemize}
\end{prop}

\begin{proof} $ $
Up to conjugation, there exist $a_0, \dots , a_4 \in \kk$ pairwise distinct such that: $$S = \left\{X_0^2 + X_1^2 + X_2^2 + X_3^2 + X_4^2 = 0 = a_0 X_0^2 + a_1 X_1^2 + a_2 X_2^2 + a_3 X_3^2 + a_4 X_4^2 = 0 \right\}$$ and $\Aut(S) = Z \rtimes P$ where $P$ is a subgroup of $\mathfrak{S}_5$.

First we assume $G$ is not contained in $Z$. Hence $\Aut(S)$ is not reduced to $Z$. Using Proposition \ref{propquartic5points} Assertion \ref{propquartic5points1} we can assume up to conjugation for a given $t \in \kk \setminus \{ 0,1,-1 \}$: $$S = \left\{X_0^2 + X_1^2 + X_2^2 + X_3^2 + X_4^2 = 0 =  X_1^2 + t X_2^2 - t X_3^2 -  X_4^2 = 0 \right\}$$  Using Proposition \ref{subgroupsofSt} we get $G = G_{4n}$ with $n \in \{7,8,8',9 \}$ up to conjugation.
According to Proposition \ref{quarticpicard} we have $rk(Pic(S)^G) \in  \{2,3\}$, hence our contradiction.

Therefore $G \subset Z$.
According to Proposition \ref{quarticz}, we get up to conjugation by a permutation matrix (potentially permuting the scalars $a_i$): $G = G_{4n}$ with $n \in \{1,2,3,4,5,6 \}$.
We assumed  $rk(Pic(S)^G) = 1$, therefore according to Proposition \ref{quarticpicard} we have: $n \in \{1,5,6\}$. Hence the result.
\end{proof}

\medskip



\subsection{Subgroups of $\Aut(S)$ not contained in $Z$} \label{massy423} $ $

\medskip

\begin{defi} \label{defi323} $ $
Let $t \in \kk \setminus \left\{0,1,-1\right\}$.
\begin{enumerate}
\item \label{defi3231} We define the following quartic del Pezzo surface:
$$S_t = \left\{ X_0^2 + X_1^2 +  X_2^2  + X_3^2 + X_4^2 = X_1^2 + t X_2^2  - t X_3^2 - X_4^2= 0 \right\} \subset \mathbb{P}^4$$
\item \label{defi3232}
We define the following automorphism of $S_t$:
$$\begin{array}{ccccc}
s & : & S_t & \rightarrow& S_t \\
 & & [X_0 : X_1 : X_2 : X_3 : X_4] & \mapsto & [X_0 : X_4 : X_3 : X_2 : X_1]  \\
\end{array}$$
\item \label{defi3233}
We define the following $2$-elementary subgroups of automorphisms of $S_t$:

\begin{itemize}[label=\textbullet]

\item $G_{47} = \langle [1:-1:1:1:-1],[1:1:-1:-1:1],s \rangle \simeq (\mathbb{Z}/2)^3$.

\item $G_{48} =  \langle [1:-1:1:1:-1],s \rangle \simeq (\mathbb{Z}/2)^2$.

\item $G_{48'} =  \langle [1:1:-1:-1:1],s \rangle\simeq (\mathbb{Z}/2)^2$.

\item $G_{49} =  \langle [-1:1:1:1:1],s \rangle \simeq (\mathbb{Z}/2)^2$.
\end{itemize}
\end{enumerate}
\end{defi}

\begin{prop} \label{propquartic5points} $ $
Let $S$ be a quartic del Pezzo surface such that its automorphism group is not reduced to $Z$. Then:
\begin{enumerate}
\item \label{propquartic5points1} There exists $t \in \kk \setminus \left\{ 0,1,-1 \right\}$ such that $S \simeq S_t$.
 \item \label{propquartic5points2} All $2$-elementary non-trivial subgroups of $P$ are cyclic and conjugate in $P$.
\end{enumerate}
\end{prop}

\begin{proof} $ $

\begin{enumerate}
\item According to \cite[Theorem 3.1]{dolcubic} and \cite[Equation (3.1)]{dolcubic}, if $P$ is non-trivial, then $S$ is isomorphic to: $$\left\{X_0^2 + X_2^2 + X_3^2 + a X_4^2 = X_1^2 + X_2^2 + a X_3^2 + X_4^2 = 0 \right\} \subset \mathbb{P}^4$$ for some $a \in \kk \setminus \{0, -1,1\}$. We observe that $(X_0^2 + X_2^2 + X_3^2 + a X_4^2) - (X_1^2 + X_2^2 + a X_3^2 + X_4^2)  = X_0^2 - X_1^2 + (1-a) X_3^2 + (-1+a) X_4^2$ and $(X_0^2 + X_2^2 + X_3^2 + a X_4^2) + (X_1^2 + X_2^2 + a X_3^2 + X_4^2) = X_0^2 + X_1^2 + 2 X_2^2 + (a+1) X_3^2 + (a+1) X_4^2$, hence  $\{X_0^2 + X_2^2 + X_3^2 + a X_4^2 = 0 = X_1^2 + X_2^2 + a X_3^2 + X_4^2\} = \{X_0^2 - X_1^2 + (1-a) X_3^2 + (-1+a) X_4^2=0 = X_0^2 + X_1^2 + 2 X_2^2 + (a+1) X_3^2 + (a+1) X_4^2\}$. Finally, replacing $X_2, X_3$ and $X_4$ by $c X_2, bX_3$ and $bX_4$ respectively, where $b,c\in \kk^\times, c^2=2, b^2=a+1$, we obtain an isomorphism with the surface $ \{X_0^2 - X_1^2 + \frac{1-a}{1+a} X_3^2 + \frac{-1+a}{1+a} X_4^2=0 = X_0^2 + X_1^2 + X_2^2 +  X_3^2 +  X_4^2\}$, hence the result by taking $t = \frac{1-a}{1+a}$.

\item According to \cite[Theorem 3.1]{dolcubic}, $P$ is isomorphic to one of the following groups: $$\mathbb{Z}/2 , \quad \mathbb{Z}/4 ,  \quad \mathfrak{S}_3,  \quad D_{10},  \quad \mathbb{Z}/5 \rtimes (\mathbb{Z}/5)^{\times}.$$
The result is trivial for the cyclic groups. For the three non-cyclic groups we use Corollary \ref{pirouz3} 
\end{enumerate}
\end{proof}

\begin{prop} \label{subgroupsofSt} $ $
Let $t \in \kk \setminus \{0,1,-1\}$.
%
%
%
%
%

Every $2$-elementary non-cyclic subgroup of $\Aut(S_t)$ not contained in $Z$ is conjugate in $\Aut(S_t)$ to exactly one of the subgroups of Definition \ref{defi323} Assertion \ref{defi3233}. 
\end{prop}

\begin{proof} $ $
According to \cite[Theorem 3.1]{dolcubic}, let $P_{S_t}$ be a subgroup of $\Aut(S_t)$ isomorphic to a subgroup of $\mathfrak{S}_5$ such that $\Aut(S_t) = Z \rtimes P_{S_t}$. Let $\pi : \Aut(S_t) \twoheadrightarrow P_{S_t}$ be the canonical surjection.
Let $G$ be a $2$-elementary non-cyclic subgroup of $\Aut(S_t)$ not contained in $Z$.
Using Proposition \ref{propquartic5points} Assertion \ref{propquartic5points2}, we get that $\pi(G)$ is conjugate in $P_{S_t}$ to $\langle s \rangle$.
Hence up to conjugation, $G = \langle (z,s) , H \rangle$ where $z = [1 : z_1 : z_2 : z_3 : z_4] \in Z$ and $H$ is a $2$-elementary non-trivial subgroup of $Z$.
The element $(z,s)$ is of order $2$, hence $z_1 = z_4$ and $z_2 = z_3$. Let $z' = [1 : z_1 : z_2  :  1 :  1]$. Then $z' (z,s) z'^{-1} = (id,s)$.
Hence up to conjugation we have $G = \langle s , H \rangle$.
$H$ commutes with $s$, hence $H \subset \langle [1:-1:1:1:-1],[1:1:-1:-1:1] \rangle$.
Hence the result.
\end{proof}

\subsection{The Picard rank of the $2$-elementary non-cyclic subgroups} $ $

\medskip

We use the notations of Definitions \ref{defi1623} and \ref{defi323}.

\begin{remark} $ $
The subgroups $G_{48}$ and $G_{48'}$ are \textit{a priori} not conjugate in $S_t$ but the pairs $(G_{48},S_t)$ and $(G_{48'},S_{\frac{1}{t}})$ are conjugate by the following isomorphism:
$$\begin{array}{cccc}
& S_t & \rightarrow& S_{\frac{1}{t}} \\
& [X_0 : X_1 : X_2 : X_3 : X_4] & \mapsto & [X_0 : X_2 : X_1 : X_4 : X_3]  \\
\end{array}$$
 We will not need this observation, because we will be able to exclude these two subgroups using the Proposition \ref{quarticpicard}.
\end{remark}

\begin{prop} \label{quarticpicard} $ $
Let $S$ be the following quartic del Pezzo surface: $$S =\left\{ \sum_{n=0}^4 X_n^2 = \sum_{n=0}^4 a_n X_n^2 = 0 \right\} \subset \mathbb{P}^4$$ where the $a_k \in \kk$ are pairwise distinct.

Let $G \in \{G_{41} , G_{42} , G_{43} , G_{44} , G_{45} ,G_{46} ,G_{47} ,G_{48}, G_{48'} , G_{49} \}$.

If $G \in \{ G_{47} ,G_{48}, G_{48'} , G_{49} \}$, then we assume $S = S_t$ for some $t \in \kk \setminus \{0,1,-1\}$.

Then we have:
$rk(Pic(S)^{G_{4n}}) = \left\{
\begin{array}{l}
1 \mbox{ if } n \in \{1, 5 , 6\} \\
2 \mbox{ if } n \in \{2, 4 , 7, 9\} \\
3 \mbox{ if } n \in \{3, 8 , 8'\}

\end{array} 
\right.$
\end{prop}

\begin{proof} $ $
The following table gives us the Euler characteristic $\chi^g$ of the fixed point variety $X^g$ of each element $g$ of these subgroups:

\tiny{$$\begin{tabular}{ | c | c | c | c|  }
 \hline 
   element & fixed points & nature & $\chi$ \\
   \hline
   id &  &  & 8 \\
   \hline
   $[-1:1:1:1:1]$ & $\{ [0:X_1:X_2:X_3:X_4] | \sum_{k=1}^4 X_k^2 = \sum_{k=1}^4 a_k X_k^2 = 0 \}$ & elliptic curve & 0   \\
 \hline  
    $[-1:-1:1:1:1]$ & $\{ [0:0:X_2:X_3:X_4] | \sum_{k=2}^4 X_k^2 = \sum_{k=2}^4 a_k X_k^2 = 0 \}$ & 4 points & 4   \\
 \hline  
    $[X_0:b_4 X_4:b_3 X_3 : b_2 X_2 : b_1 X_1]$ & $\{ [0:1:c i : - c b_2 i : - b_1] | c=\pm 1\} \sqcup$ & 2 points &  \\
    where $b_4 = b_1 = \pm1$, $b_2 = b_3 = \pm 1$ & $  \{ [X_0 : X_1 : X_2 : b_2 X_2 : b_1 X_1 ] | X_0^2 + 2 X_1^2 + 2 X_2^2 = 0\}$  & and a $\mathbb{P}^1$ & 4   \\
 \hline  
 \end{tabular}$$}
 
 \normalsize

 We will prove the nature of the fixed points of these elements:
 \begin{itemize}
\item The curve $\{ [0:X_1:X_2:X_3:X_4] | \sum_{k=1}^4 X_k^2 = \sum_{k=1}^4 a_k X_k^2 = 0 \}$ is the complete intersection of two quadrics hypersurface of $\mathbb{P}^3$. Hence according to  \cite[Exercise I.7.2(d) page 54]{Hartshorne} (We use the adjunction formula and Bezout's Theorem), its genus is $1$.

\item The set $\{ [0:0:X_2:X_3:X_4] | \sum_{k=2}^4 X_k^2 = \sum_{k=2}^4 a_k X_k^2 = 0 \}$ is the set of solutions $[X_2 : X_3 : X_4]$ verifying the equation $ \left[\begin{smallmatrix}  1 & 1 & 1 \\
a_2 &  a_3 & a_4
\end{smallmatrix} \right]  \left[\begin{smallmatrix}  X_2^2  \\
X_3^2 \\
X_4^2
\end{smallmatrix} \right] =  \left[\begin{smallmatrix} 0   \\
0
\end{smallmatrix} \right]$. The matrix $\left[\begin{smallmatrix}  1 & 1 & 1 \\
a_2 &  a_3 & a_4
\end{smallmatrix} \right]$ is of rank $2$ because the $a_k$ are pairwise distinct, hence the set of solutions $(X_2^2, X_3^2, X_4^2)$ is a subvectorspace of dimension $1$. Hence there exist $x , y, z \in \kk^\times$ such that $[X_2 : X_3 : X_4]$ is a solution if and only if there exists $\lambda \in \kk^\times$ such that $X_2^2 = \lambda x, X_3^2 = \lambda y,  X_4^2 = \lambda z$. By taking the square roots of $x, y ,z$, we get our four solutions: $[\pm \sqrt{x} : \pm \sqrt{y} :\sqrt{z}]$.

\item The set $\{ [0:1:c i : - c b_2 i : - b_1] | c=\pm 1\}$ is obviously a set of two points.
The set $\{ [X_0 : X_1 : X_2 : b_2 X_2 : b_1 X_1 ] | X_0^2 + 2 X_1^2 + 2 X_2^2 = 0\}$ is isomorphic to $\mathbb{P}^1$ by the genus-degree formula.
\end{itemize}

Using Proposition \ref{rankformula}, we have the following formula for the Picard rank: $$\rank \Pic(X)^G + 2 = \frac{1}{|G|} \sum_{g \in G} \chi(X^g)$$ Hence we get the result for each subgroup by using this formula with the nature of each set of fixed points.
 \end{proof}

\newpage

\section{Automorphisms of cubic del Pezzo surfaces} \label{cubicsection} $ $

We assume $\car(\kk) \neq 3$ in this section.
A cubic del Pezzo surface is a smooth hypersurface of degree 3 in $\mathbb{P}^3$, \textit{i.e.} the zero set of a cubic polynomial in $\mathbb{P}^3$ that is also smooth. Its automorphism group is always a subgroup of $\PGL(4,\kk)$ (see Lemma \ref{2lemconj40}).
The goal of this section is to classify the 3-elementary non-cyclic subgroups of the automorphism group of cubic del Pezzo surfaces in characteristic not $3$.
The result of this section is the Proposition \ref{resultcubic}.

\medskip

\subsection{Notations}

\begin{defi} \label{defi331} $ $
Let $\mu \in \kk$. We define the following cubic surface:
$$S_{\mu} = \{W^3 + X^3 + Y^3 + Z^3 + \mu X Y Z = 0 \} \subset \mathbb{P}^3$$ 
\end{defi}


\begin{defi} $ $

We define the following $3$-elementary non-cyclic subgroups of $\PGL(4,\kk)$:
\begin{itemize}

\item $G_{31} = \langle [j:j:1:1] , [1:j:j:1] \rangle  \simeq (\mathbb{Z}/3)^2$

\item $G_{32} =  \langle [j:1:1:1] , [1:j:1:1] \rangle   \simeq (\mathbb{Z}/3)^2$

\item $G_{33} =  \langle [j:1:1:1] , [1:j:j^2:1] \rangle   \simeq (\mathbb{Z}/3)^2$

\item $G_{34} =  \langle [j:1:1:1] , [1:j:1:1] , [1:1:j:1] \rangle   \simeq (\mathbb{Z}/3)^3$
\end{itemize}
\end{defi}

\medskip

\subsection{Results}

\begin{prop} \label{resultcubic} $ $

    Let $S$ be a cubic del Pezzo surface, and $G$ a 3-elementary non-cyclic subgroup of automorphisms of $S$. We assume $rk(Pic(S)^G) = 1$.
    Then $G \simeq (\mathbb{Z}/3)^2$ or $G \simeq (\mathbb{Z}/3)^3$. Furthermore:
\begin{enumerate}
\item  
    If  $G \simeq (\mathbb{Z}/3)^3$ then up to conjugation we have $S = \{W^3 + X^3 + Y^3 + Z^3 = 0\}$ and $G$ is the 3-torsion subgroup $G_{34}$ of the diagonal torus of $\PGL(4,\kk)$.
\item
    If $G \simeq (\mathbb{Z}/3)^2$ then up to conjugation we have one of the following:
\begin{itemize}
\item $G= G_{32}$ and  $S = \{W^3 + X^3 + Y^3 + Z^3 = 0\}$.
\item $G=G_{33}$ and $S = S_{\mu}$ where $\mu$ is a scalar such that $\mu^3 \neq -27$.

Let $\mu, \mu' \in \kk$ such that $\mu^3 \neq - 27 \neq \mu'^3$. Then the two pairs $(G,S_{\mu})$ and $(G,S_{\mu'})$ are conjugate by isomorphism if and only if $\mu^3 =  \mu'^3$.
    \end{itemize}
Furthermore these two cases are birationaly distinct. 
\end{enumerate}
\end{prop}

\begin{proof} $ $

The list of subgroups up to conjugation comes from Proposition \ref{cubicprop111} and Proposition \ref{rkcubic}.

According to Proposition \ref{rkcubic}, $S_0$
has three fixed points under the action of $G_{32}$
and $S_{\mu}$ has no fixed point under the action of $G_{33}$. Hence, using \cite[page 1054-1056 Proposition A.2 ]{kollar}, $(G_{32},S_{0})$
and $(G_{33},S_{\mu})$ are not conjugate by a birational map.
Using Proposition \ref{cubicconj0}, we get the condition on conjugation of couples $(G_{33},S_{\mu})$.
\end{proof}

\medskip

\subsection{A list of representatives up to conjugation of $3$-elementary non-cyclic subgroups of automorphisms of cubic del Pezzo surfaces}
$ $

In this subsection we will give a list of representatives of $3$-elementary non-cyclic subgroups of automorphisms of cubic del Pezzo surfaces. The Proposition \ref{3elempgl4} serves as a foundation for this subsection, since the automorphism group of any cubic del Pezzo surface is a subgroup of $\PGL(4,\kk)$ (see Lemma \ref{2lemconj40}). The result of this subsection is the Proposition \ref{cubicprop111}. We first give a few preliminary lemmas about the smoothness of said surfaces.


\begin{lemme} \label{smoothcubic} $ $
Let $\mu \in \kk$.

The cubic surface $S_{\mu}$ (see Definition \ref{defi331}) is smooth if and only if $\mu^3 \neq -27$.
\end{lemme}

\begin{proof} $ $
Let $P(W,X,Y,Z) = W^3 + X^3 + Y^3 + Z^3 + \mu X Y Z$.
We have: 
$$\nabla P(W,X,Y,Z)=\left[\begin{smallmatrix}
3 W^2\\
3 X^2 + \mu YZ \\
3 Y^2 + \mu XZ \\
3 Z^2 + \mu XY
\end{smallmatrix}\right]
$$

Let's assume $
\nabla P(W,X,Y,Z) = \left[\begin{smallmatrix}
0\\
0 \\
0 \\
0
\end{smallmatrix}\right] $ for some $[W:X:Y:Z] \in S$. Then $W=0$ and by symmetry we can assume $X=1$. Then $ \left\{
    \begin{array}{lll}
        3 + \mu Y Z & = & 0 \\
        3 Y^2 + \mu Z & = & 0 \\
        3 Z^2 + \mu Y & = & 0
    \end{array}
\right.
$ then $ \left\{
    \begin{array}{lll}
         \mu Y Z & = & - 3 \\
        3 Y^2  - \dfrac{3}{Y} & = & 0 \\
        3 Z^2  - \dfrac{3}{Z} & = & 0
    \end{array}
\right.
$ then $ \left\{
    \begin{array}{lll}
         \mu Y Z & = & - 3 \\
        Y^3 & = & 1 \\
        Z^3  & = & 1
    \end{array}
\right.
$ then $\mu^3=-27$.

Conversely, let's assume $\mu = -3 j^n$ where $n \in \mathbb{Z}$. Then $[0:1:1:j^{-n}]$ is a singular point of the surface.
\end{proof}

\begin{lemme} \label{lemmacubic1} $ $
Let $g \in \GL(4,\kk)$ and $\overline{g}$ its class in $\PGL(4,\kk)$.
Let $P \in \kk[W,X,Y,Z]_3$ be a homogeneous cubic polynomial. Let $S = \{ P = 0\} \subset \mathbb{P}^3$.
Then:
$$\overline{g}|_{S} \in \Aut(S) \Leftrightarrow \exists \lambda \in \kk^\times, P \circ g = \lambda P$$

\end{lemme}

\begin{proof} $ $
Trivial.
\end{proof}

\begin{defi} $ $
We define the following elements of $\PGL(4,\kk)$:
$$g_1 = [j:1:1:1] , \quad g_2 =[1:j:j:1] , \quad  g_3 = [1:j:j^2:1]$$
\end{defi}

\begin{lemme} \label{lemmacubic2} $ $
Let $P = \sum_{\alpha+\beta+\gamma+\delta=3} p_{\alpha,\beta,\gamma,\delta} W^\alpha X^\beta Y^\gamma Z^\delta \in \kk[W,X,Y,Z]_3$. \newline
Let $S = \{ P = 0 \} \subset \mathbb{P}^3$. We assume $S$ is a cubic del Pezzo surface.
\begin{enumerate}
\item If $g_1 |_{S} \in \Aut(S)$, then $P = p_{3,0,0,0} W^3 + Q(X,Y,Z)$ where $Q$ is a cubic homogeneous polynomial.
\item If $g_2 |_{S} \in \Aut(S)$, then $P = Q_1(X,Y) + Q_2(W,Z)$ where $Q_1, Q_2$ are cubic homogeneous polynomials.
\item If $g_3 |_{S} \in \Aut(S)$, then $P = Q_1(W,Z) + p_{0,3,0,0} X^3 + XY Q_2(W,Z) + p_{0,0,3,0} Y^3$ or \newline $X Q_3(W,Z) + p_{0,2,1,0} X^2 Y + Y^2 Q_4 (W,Z)$ or $X^2 Q_5(W,Z) + Y Q_6(W,Z) + p_{0,1,2,0} X Y^2$ where $Q_1 \in \kk[X,Y]_3$, $Q_3, Q_6 \in \kk[W,Z]_2$, $Q_2, Q_4, Q_5 \in \kk[W,Z]_1$.
\end{enumerate}
\end{lemme}

\begin{proof} $ $
Let $g \in \Aut(S)$. The automorphism $g$ acts linearly on the space of homogeneous polynomials of degree three and the polynomial $P$ needs to be an eigenvector. Hence $P$ is contained in an eigenspace. The eigenspaces are the following:

$$\begin{tabular}{ | c | c |}
\hline
  $g$ & eigenspaces  \\
 \hline
     & $\{ \alpha W^3 + Q(X,Y,Z) | \alpha \in \kk, Q \in \kk [X,Y,Z]_3 \}$ \\
    $g_1$ & $\{ W^2 . Q(X,Y,Z) |  Q \in \kk [X,Y,Z]_1 \}$ \\
    & $\{ W . Q(X,Y,Z) | \ Q \in \kk [X,Y,Z]_2 \}$ \\
\hline    
& $\{ P(X,Y)+ Q(W,Z) | P, Q \in \kk [X,Y,]_3 \}$ \\
    $g_2$ & $\{ X . P(W,Z)+ Y . Q(W,Z) | P , Q \in \kk [W,Z]_2 \}$ \\
    & $\{ W . P(X,Y) + Z . Q(X,Y) | P , Q \in \kk [X,Y]_2 \}$ \\
\hline
     & $\{ P(W,Z) + \alpha X^3 + XY . Q(W,Z) + \beta Y^3 | \alpha, \beta \in \kk, P \in \kk [W,Z]_3 , Q \in \kk[W,Z]_1 \}$ \\
    $g_3$ & $\{ X . P(W,Z) + \alpha X^2 Y + Y^2 . Q(W,Z) |  \alpha \in \kk, P \in \kk [W,Z]_2 , Q \in \kk [W,Z]_1 \}$ \\
    & $\{ X^2 . P(W,Z) + Y . Q(W,Z) + \alpha X Y^2 | \alpha \in \kk, P \in \kk[W,Z]_1, Q \in \kk [W,Z]_2 \}$ \\
\hline
 \end{tabular}$$

We leave the calculations to the reader.
Furthermore, the smoothness of the surface implies that each variable comes with degree at least two, hence our result.
\end{proof}


\begin{prop} \label{propocubic1} $ $
Let $P = \sum_{\alpha+\beta+\gamma+\delta=3} p_{\alpha,\beta,\gamma,\delta} W^\alpha X^\beta Y^\gamma Z^\delta \in \kk[W,X,Y,Z]_3$. Let $S = \{ P = 0 \} \subset \mathbb{P}^3$. We assume $S$ is a cubic del Pezzo surface.
\begin{enumerate}
\item If $G_{31} |_{S} \subset \Aut(S)$, then $P = p_{3,0,0,0} W^3 + p_{0,3,0,0} X^3 + p_{0,0,3,0} Y^3 + p_{0,0,0,3} Z^3$.
\item If $G_{32} |_{S} \subset \Aut(S)$, then $P = p_{3,0,0,0} W^3 + p_{0,3,0,0} X^3 + Q(Y,Z)$ where $Q$ is a cubic homogeneous polynomial.
\item If $G_{33} |_{S} \subset \Aut(S)$, then $P = p_{3,0,0,0} W^3 + p_{0,3,0,0} X^3 + p_{0,0,3,0} Y^3 + p_{0,0,0,3} Z^3 + p_{0,1,1,1} X Y Z$.

\item If $G_{34} |_{S} \subset \Aut(S)$, then $P = p_{3,0,0,0} W^3 + p_{0,3,0,0} X^3 + p_{0,0,3,0} Y^3 + p_{0,0,0,3} Z^3$.
\end{enumerate}
\end{prop}

\begin{proof} $ $
This is a direct consequence of Lemma \ref{lemmacubic2}.
\end{proof}

\begin{prop} \label{propocubic2} $ $
Let  $S \subset \mathbb{P}^3$ be a cubic del Pezzo surface.
\begin{enumerate}
\item Let $n \in \{1,2,4\}$ such that $G_{3n} |_{S} \subset \Aut(S)$. Then there exists $\Phi \in \PGL(4,\kk)$ such that $\Phi G_{3n} \Phi^{-1} = G_{3n}$ and such that $\Phi(S) = \{ W^3 + X^3 + Y^3 + Z^3 = 0\} \subset \mathbb{P}^3$.
\item If $G_{33} |_{S} \subset \Aut(S)$, then  there exists $\Phi \in \PGL(4,\kk)$ such that $\Phi G_{33} \Phi^{-1} = G_{33}$ and such that $\Phi(S) = \{ W^3 + X^3 + Y^3 + Z^3 + \mu X Y Z= 0\} \subset \mathbb{P}^3$ where $\mu^3 \neq -27$ 
\end{enumerate}
\end{prop}

\begin{proof} $ $

\begin{itemize}
\item If $n=1,3,4$: 
Using Proposition \ref{propocubic1},
there exist $p_{3,0,0,0}$,  $p_{0,3,0,0}$, $p_{0,0,3,0}$, $p_{0,0,0,3}$, $p_{0,1,1,1}$ $\in \kk$ such that $S = \{ p_{3,0,0,0} W^3 + p_{0,3,0,0} X^3 + p_{0,0,3,0} Y^3 + p_{0,0,0,3} Z^3 + p_{0,1,1,1} X Y Z = 0 \}$, and with $p_{0,1,1,1} = 0$ if $n \neq 3$.
Since $\kk = \overline{\kk}$, we can define $q_1 , q_2, q_3 , q_4 \in \kk$ such that $q_1^3 = p_{3,0,0,0} , q_2^3 = p_{0,3,0,0} , q_3^3 = p_{0,0,3,0} , q_4^3 = p_{0,0,0,3} $.
Since $S$ is smooth, we get that $p_{3,0,0,0}$, $p_{0,3,0,0}$, $p_{0,0,3,0}$, $p_{0,0,0,3}$ $\neq 0$.
Hence the following map: $$\begin{array}{ccccc}
\Phi  & : & \mathbb{P}^3 & \to & \mathbb{P}^3 \\
 & & [W:X:Y:Z] & \mapsto & [q_1W: q_2 X: q_3 Y:q_4 Z] \\
\end{array}$$ is a well-defined automorphism.

We observe that $\Phi(S) = \{ W^3 + X^3 + Y^3 + Z^3 + \dfrac{p_{0,1,1,1}}{q_2 q_3 q_4} X Y Z= 0 \}$. If $n \neq 3$ we have our result.
If $n=3$, we define $\mu = \dfrac{p_{0,1,1,1}}{q_2 q_3 q_4} \in \kk$. According to Lemma \ref{smoothcubic}, we have $\mu^3 \neq -27$, hence our result.

\item If $n=2$:
Using Proposition \ref{propocubic1},
there exist $p_{3,0,0,0} ,  p_{0,3,0,0} \in \kk$ and $Q \in \kk[Y,Z]_3$ such that $S = \{ p_{3,0,0,0} W^3 + p_{0,3,0,0} X^3 + Q(Y,Z) = 0 \}$.
We observe that since $S$ is smooth, $Q$ has 3 distinct roots.
The action of $\PGL(2,\kk)$ on $\mathbb{P}^1$ is $3$-transitive,
hence there exists an automorphism $\Theta \in \PGL(4,\kk)$ acting only on $Y$ and $Z$, thus preserving $G_{32}$, such that
$\Theta (S) = \{ p_{3,0,0,0} W^3 + p_{0,3,0,0} X^3 + Y^3 + Z^3 = 0\}$.
Since $\kk = \overline{\kk}$, we can define $q_1 , q_2 \in \kk$ such that $q_1^3 = p_{3,0,0,0} , q_2^3 = p_{0,3,0,0}$.
Since $S$ is smooth, we get that $p_{3,0,0,0} ,  p_{0,3,0,0} \neq 0$.
Hence the following map:
$$\begin{array}{ccccc}
\Phi  & : & \mathbb{P}^3 & \to & \mathbb{P}^3 \\
 & & [W:X:Y:Z] & \mapsto & [q_1W: q_2 X:  Y: Z] \\
\end{array}$$ is a well-defined automorphism.

We observe that $\Phi( \Theta(S)) = \{ W^3 + X^3 + Y^3 + Z^3 = 0 \}$.
\end{itemize}
\end{proof}

\begin{prop} \label{cubicprop111} $ $
Let $S$ be a cubic del Pezzo surface.
Let $G$ be a $3$-elementary non-cyclic subgroup of automorphisms of $S$.
Then:
\begin{enumerate}
\item $G \simeq (\mathbb{Z}/3)^2$ or $G \simeq (\mathbb{Z}/3)^3$.
\item Up to isomorphism we have: $$(G,S) \in \left\{ (G_{31} , S_0) , (G_{32} , S_0) , (G_{33} , S_{\mu})  , (G_{34} , S_0) | \mu \in \kk , \mu^{3} \neq -27 \right\}.$$
\end{enumerate}
\end{prop}

\begin{proof} $ $
This is a combination of Proposition \ref{3elempgl4} and of Proposition \ref{propocubic2}.
\end{proof}


\subsection{The Picard rank of $3$-elementary non-cyclic subgroups of automorphisms of cubic del Pezzo surfaces}
$ $

The goal of this subsection is to give the rank of the Picard invariant of the subgroups given in the Proposition \ref{cubicprop111}, in order to exclude the subgroups with Picard invariant not equal to 1.
We will also give the fixed points of these subgroups, in order then to use \cite[page 1054-1056 Proposition A.2 ]{kollar} to study the conjugation between these subgroups.
Recall that $S_{\mu}$ is defined as in Definition \ref{defi331}.

\begin{prop} \label{rkcubic} $ $
Let $\mu \in \kk$ such that $\mu^3 \neq -27$.
\begin{enumerate}
\item We have: $rk(Pic(S_0)^{G_{32}}) = rk(Pic(S_0)^{G_{34}}) = rk(Pic(S_{\mu})^{G_{33}}) = 1$.
\item We have: $rk(Pic(S_0)^{G_{31}}) = 3$.
\item The surface $S_0$ has three fixed points under the action of $G_{32}$.
\item The surface $S_{\mu}$ has no fixed point under the action of $G_{33}$.
\end{enumerate}
\end{prop}

\begin{proof} $ $

We describe the fixed locus and Euler characteristic of each element of the fixed point set of the diagonal torus acting on $S_\mu$ in the following table:

\tiny$$\begin{tabular}{ | c | c | c | c| c | }
 \hline 
   element & fixed points & nature & $\chi$  \\
   \hline
   $id$ &  &  & 9  \\
   \hline
   $[j:1:1:1]$ & $\{[0:X:Y:Z] \in \mathbb{P}^3 | X^3 + Y^3 + Z^3 + \mu X Y Z = 0\}$ & elliptic curve & 0  \\
 \hline  
    $[j:j:1:1]$ & $\{ [1:-j^{n}:0:0] | n \in \{0,1,2\} \} \sqcup \{ [0:0:1:-j^n]| n \in \{0,1,2\} \} $ & 6 pts & 6   \\
 \hline  
    $[1:j^2 :j :1]$ & $\{ [1:0:0:-j^n] | n \in \{0,1,2\} \} $ & 3 pts & 3   \\
 \hline  
 \end{tabular}$$
\normalsize
%
%
 
 The only result that requires a proof in this table is the nature of the following subset: $$\left\{X^3 + Y^3 + Z^3 + \mu X Y Z = 0\right\} \subset \PP^2$$
 Using the adjunction formula, we get that a smooth irreducible cubic curve in $\mathbb{P}^2$ is an elliptic curve. Hence this subset is an elliptic curve.
 
To conclude our proof we use Proposition \ref{rankformula} which gives the following formula for the Picard rank: $$\rank \Pic(X)^G + 2 = \frac{1}{|G|} \sum_{g \in G} \chi(X^g)$$ Hence we get the result for each subgroup by using this formula with the nature of each set of fixed points.
\end{proof}


\medskip

\subsection{Isomorphic conjugation of $3$-elementary non-cyclic subgroups of automorphisms of cubic del Pezzo surfaces} \label{g33smuconjiso} $ $


In light of the results established in the previous subsections, the only case to address for the conjugacy by isomorphisms of these subgroups is the one involving the pairs $(G_{33}, S_{\mu})$. To this end, we first describe the normalizer of the subgroup $G_{33}$ (see Lemma~\ref{lemmeg33}), before stating the main result of this subsection: Proposition~\ref{cubicconj0}.


\begin{lemme} \label{lemmeg33} $ $

The normalizer $\mathcal{N}$ of $G_{33}$ in $\PGL(4,\kk)$ is the following subgroup of $\PGL(4,\kk)$:
 $$\left\{ \left[\begin{array}{c|c}
1 & 0  \\
\hline
0 & \sigma D
\end{array}\right] | \sigma \text{ is a permutation matrix, } D \text{ is a diagonal invertible matrix} \right\}$$
\end{lemme}

\begin{proof} $ $
Let $a = [j:1:1:1] , b = [1:1:j:j^2] \in G_{33}$.

We define $E$ the set of subgroups of $G_{33}$ isomorphic to $\mathbb{Z}/3$:
$$E = \left\{ \langle a \rangle  ,  \langle b \rangle  ,  \langle ab \rangle  ,  \langle a^2 b \rangle  \right\}$$

We define the following map:
$$\begin{array}{ccccc}
\rho & : & \mathcal{N} & \to & \mathfrak{S}_E \\
 & & \Phi& \mapsto & \left\{ \begin{array}{ccccc}
 &  & E & \to & E \\
 & & H & \mapsto & \Phi H \Phi^{-1}  \\
\end{array} \right\} \\ 
\end{array}$$
Because of the nature of the multiplicities of the eigenvalues of the elements of $G_{33}$, for every $\Phi \in \mathcal{N}$, we have $\rho(\Phi) (\langle a \rangle) = \langle a \rangle$.
Hence we have the inclusion: $$\rho(\mathcal{N}) \quad \subset \quad \mathfrak{S}_{\left\{\langle b \rangle ,  \langle ab \rangle,  \langle a^2 b \rangle \right\}}\simeq \mathfrak{S}_3$$
For $\sigma \in \mathfrak{S}_{\{ \langle b \rangle ,  \langle ab \rangle ,  \langle a^2 b \rangle\}}$, we define:
$$\Psi_{\sigma} = \left[\begin{array}{c|c}
1 & 0_{1 \times 3} \\
\hline
0_{3 \times 1} & \sigma'
\end{array}\right],$$
where $\sigma'$ is a $3 \times 3$ permutation matrix such that $\rho(\Psi_{\sigma}) = \sigma$.
\newline
We observe that $\Psi_{\sigma} \in \mathcal{N}$. Hence we have the equality:
$$\rho(\mathcal{N}) \quad = \quad \mathfrak{S}_{\{\langle b \rangle  ,  \langle ab \rangle ,  \langle a^2 b \rangle \}}\simeq \mathfrak{S}_3$$
We also observe that for a given $\Phi \in \mathcal{N}$, we have $\Psi_{{\rho(\Phi)}^{-1}} \Phi \in \Ker(\rho)$.
Hence we have the equality: $$\mathcal{N} \quad = \quad \Ker(\rho) \rtimes \left\{ \Psi_{\sigma} | \sigma \in \mathfrak{S}_{\{ \langle b \rangle , \langle ab \rangle , \langle a^2 b \rangle \}} \right\} \quad \simeq \quad \Ker(\rho) \rtimes \mathfrak{S}_3$$
Now we compute $\Ker(\rho)$: We observe that if $\psi \in \PGL(4,\kk)$ preserves the subgroup $ \langle a \rangle$, then it is of the form $\left[\begin{array}{c|ccc}
1 & 0 & 0 & 0  \\
\hline
0 & * & * & *  \\
0 & * & * & *  \\
0 & * & * & *  \\
\end{array}\right]$. If it preserves the subgroup $ \langle b \rangle$, then it is of the form $\left[\begin{array}{cc|cc}
* & * & 0 & 0  \\
* & * & 0 & 0  \\
\hline
0 & 0 & * & 0  \\
0 & 0 & 0 & *  \\
\end{array}\right]$ or $\left[\begin{array}{cc|cc}
* & * & 0 & 0  \\
* & * & 0 & 0  \\
\hline
0 & 0 & 0 & *  \\
0 & 0 & * & 0  \\
\end{array}\right]$. If it preserves the subgroup $\langle a b \rangle$, then it is of the form  $\left[ \begin{smallmatrix} * & 0 & * & 0 \\ 0 & * & 0 & 0 \\ * & 0 & * & 0 \\ 0 & 0 & 0 & *     \end{smallmatrix} \right]$
 or $\left[ \begin{smallmatrix} * & 0 & * & 0 \\ 0 & 0 & 0 & * \\ * & 0 & * & 0 \\ 0 & * & 0 & 0     \end{smallmatrix} \right]$.
Hence by taking the intersection of these three subsets, it appears that $\Ker(\rho)$ is the diagonal. Hence our result.
\end{proof}

\begin{remark} $ $
The group $\mathcal{N}$ is isomorphic to ${\kk^\times}^3 \rtimes \mathfrak{S}_3$.
\end{remark}

\begin{prop}  \label{cubicconj0} $ $
Let $\mu$, $\mu' \in \kk$ such that $\mu^3 \neq -27 \neq \mu'^{3}$. 

The following assertions are equivalent:
\begin{enumerate}
\item \label{cubicconj01} There exists $n \in \{0,1,2\}$ such that $\mu = j^n \mu'$.
\item \label{cubicconj02} There exists an isomorphism $S_\mu \rightarrow S_{\mu'}$ that conjugates $G_{33}$ on itself.
\end{enumerate} 
\end{prop}

\begin{proof} $ $

\ref{cubicconj01} $\Rightarrow$ \ref{cubicconj02}:
If $\mu = j^n \mu'$ for some $n \in \{0,1,2\}$, then the following map:
$$\begin{array}{cccc}
 & S_{\mu} & \to & S_{\mu'} \\
  & [W:X:Y:Z] & \mapsto & [W:X:Y:j^n Z] \\
\end{array}$$ is an isomorphism which leaves $G_{33}$ invariant.

\ref{cubicconj02} $\Rightarrow$ \ref{cubicconj01}:
Conversely let's assume there exists an isomorphism $\phi : S_{\mu'} \to S_\mu$ leaving $G_{33}$ invariant. Such an isomorphism is induced by an automorphism of $\mathbb{P}^3$. According to Lemma \ref{lemmeg33}, this automorphism will be of the form $\left[\begin{array}{c|c}
1 & 0_{1 \times 3}  \\
\hline
 0_{3 \times 1} & \sigma D
\end{array}\right]$ where $\sigma$ is a $3 \times 3$ permutation matrix and $D = \left[ \begin{smallmatrix} a & 0 & 0 \\ 0 & b & 0 \\ 0 & 0 & c    \end{smallmatrix} \right]$
is an invertible $3 \times 3$ diagonal matrix. The automorphism $\left[\begin{array}{c|c}
1 &  0_{1 \times 3}  \\
\hline
 0_{3 \times 1}  & \sigma^{-1}
\end{array}\right]$ 
sends $S_{\mu'}$ on itself.
Hence by composition, the automorphism
$\left[\begin{array}{c|c}
1 &  0_{1 \times 3}  \\
\hline
 0_{3 \times 1}  & D
\end{array}\right] = \left[ \begin{smallmatrix} 1 & 0 & 0 & 0 \\ 0 & a & 0 & 0 \\ 0 & 0 & b & 0 \\ 0 & 0 & 0 & c     \end{smallmatrix} \right] \in \PGL(4,\kk)$ sends $S_\mu$ on $S_{\mu'}$, which means:
$\forall [W:X:Y:Z] \in \mathbb{P}^3, W^3 + X^3 + Y^3 + Z^3 + \mu' X Y Z = 0 \Leftrightarrow W^3 + a^3 X^3 + b^3 Y^3 + c^3 Z^3 + \mu a b c X Y Z = 0$.
So $W^3 + X^3 + Y^3 + Z^3 + \mu' X Y Z$ is a multiple of $W^3 + a^3 X^3 + b^3 Y^3 + c^3 Z^3 + \mu a b c X Y Z$. In particular $a^3 = b^3 = c^3 = 1$ and $\mu' = \mu abc$. Thus $abc$ is a cubic root of unity and we get the result.
\end{proof}


\newpage

\section{Automorphisms of del Pezzo surfaces of degree $2$} \label{orsay44} $ $

We assume $\car(\kk) \neq 2$ and $p=2$ in this section.
The main result of this section is Proposition \ref{2results}.
The subsection \ref{subsection443} allows us to get a list of representatives of the $2$-elementary non-cyclic subgroups of automorphisms of such surfaces.
In subsection \ref{subsection444} we go further by studying the conjugation by isomorphism in the specific case of Klein subgroups.

\medskip

\subsection{Notations and context}
$ $

For $a,b,c,d \in \kk^\times$, we denote by $[a:b:c:d]$ the following automorphism of $\mathbb{P}(2,1,1,1)$: $[W:X:Y:Z] \mapsto [a W : b X : c Y : d Z]$. If this map can be restricted to $S$, the notation $[a:b:c:d]$ will also refer to the corresponding automorphism of $S$.
\\

According to \cite[Section 4]{martinodd} in odd characteristic and to
\cite[Theorem III.3.5]{Kol} and \cite[Corollary 3.54]{KoSmCo} in characteristic 0, up to isomorphism, a del Pezzo surface of degree $2$ is a surface of the form: $$S = \left\{ W^2 = F (X,Y,Z) \right\} \subset \mathbb{P}(2,1,1,1)$$
 where $F \in \kk[X,Y,Z]_4$ is a form of degree four.
 \\
 
This surface has a special automorphism called the Geiser involution. We refer to \cite[Section 3]{martin2} for a reference about this automorphism (in all characteristic). In our case of characteristic not $2$, this automorphism is the element $[-1:1:1:1]$.

 \begin{defi} $ $
 We define the following quartic curve: $\Gamma = \left\{ F = 0 \right\} \subset \mathbb{P}^2$.
 \end{defi}
 
\begin{defi} $ $
Let $d, e, f \in \kk$. We define the following surface of degree $2$:
$$S_{def} = \left\{W^2 = X^4 + Y^4 + Z^4 + d X^2 Y^2 + e X^2 Z^2 + f Y^2 Z^2 \right\} \subset \mathbb{P}(2,1,1,1)$$ 
\end{defi}

\medskip

\subsection{Results}

\begin{prop} \label{2results} $ $

    Let $G$ be a 2-elementary non-cyclic subgroup of automorphism of a del Pezzo surface of degree $2$. We assume $rk(Pic(S)^G) = 1$. Then $G \simeq(\mathbb{Z}/2)^2$ or  $G \simeq (\mathbb{Z}/2)^3$.
\begin{itemize}
\item    If $G \simeq (\mathbb{Z}/2)^3$, then up to applying an automorphism of
$\mathbb{P}(2,1,1,1)$, the surface and the group are given by: $$G = G_{21} = \left\{ [ \pm 1: \pm 1: \pm 1: \pm 1] \right\}$$  
     $$S = S_{def} = \left\{W^2 = X^4 + Y^4 + Z^4 + d X^2 Y^2 + e X^2 Z^2 + f Y^2 Z^2 \right\} \subset \mathbb{P}(2,1,1,1)$$ where $d,e,f \in \kk \setminus \{-2,2\}$ such that $ 4 - f^2 - d^2 - e^2 + def \neq 0$. Furthermore, $(G_{21}, S_{def})$ and $(G_{21}, S_{d'e'f'})$ are conjugate if and only if there exist $\epsilon_d, \epsilon_e \in \{ -1, 1\}$ such that $\{d',e',f'\} = \{\epsilon_d d, \epsilon_e e, \epsilon_d \epsilon_e f \}$.

\item    If $G \simeq (\mathbb{Z}/2)^2$, then up to applying an automorphism of
$\mathbb{P}(2,1,1,1)$, the surface and the group are given by: $$G= G_{22} = \left\{ [ \pm 1: \pm 1:1:1] \right\}$$
$$S= \left\{W^2 = X^4 + L_2 (Y,Z) X^2 + L_4 (Y,Z) \right\} \subset \mathbb{P}(2,1,1,1)$$
 where $L_2, L_4$ are forms of degree two and four, respectively.
\end{itemize}
\end{prop}

\begin{proof} $ $

    Using Proposition \ref{degree2reduction}, we get $G=G_{2n}$, where $n \in \{1,2,3,4\}$ and the corresponding surfaces. Using Proposition \ref{picard2}, we can exclude $n=3$ and $n=4$.
    Now we assume $n=1$. Using Proposition \ref{21surface}, we get that up to conjugation, there exist $d,e,f \in \kk$ such that $S = S_{def}$.
    Using Proposition \ref{21conjugation}, we get that $(G_{21}, S_{def})$ and $(G_{21}, S_{d'e'f'})$ are conjugate if and only if there exist $\epsilon_d, \epsilon_e \in \{ -1, 1\}$ such that $\{d',e',f'\} = \{\epsilon_d d, \epsilon_e e, \epsilon_d \epsilon_e f \}$.
    \end{proof}
    

\subsection{Classification of $2$-elementary non-cyclic subgroups of automorphisms of del Pezzo surfaces of degree $2$} \label{subsection443}
$ $

In this subsection we will give the list of possible subgroups in the Propositions \ref{degree2reduction} and \ref{picard2}.
 
 \begin{lemme} \label{smoothnessgamma} $ $
The surface $S$ is smooth if and only if the curve $\Gamma$ is smooth.
 \end{lemme}
 
 \begin{proof} $ $
We define the following open subsets: 
$$U_X = \left\{ [W:X:Y:Z] \in \mathbb{P}(2,1,1,1) | X \neq 0 \right\}$$
$$U_Y = \left\{ [W:X:Y:Z] \in \mathbb{P}(2,1,1,1) | Y \neq 0 \right\}$$
$$U_Z = \left\{ [W:X:Y:Z] \in \mathbb{P}(2,1,1,1) | Z \neq 0 \right\}$$
$$V_X = \left\{ [X:Y:Z] \in \mathbb{P}^2 | X \neq 0 \right\}$$
$$V_Y = \left\{ [X:Y:Z] \in \mathbb{P}^2 | Y \neq 0 \right\}$$
$$V_Z = \left\{ [X:Y:Z] \in \mathbb{P}^2 | Z \neq 0 \right\}$$


%

Let's study the smoothness of $S \cap U_X$.
We define the following isomorphism:
$$\begin{array}{cccc}
\Phi_X : & U_X & \stackrel{\sim}{\longrightarrow} & \mathbb{A}^3 \\
  & [W:1:Y:Z] & \mapsto & (W,Y,Z)\\
\end{array}$$


We restrict this isomorphism to $S \cap U_X$:
$$\begin{array}{cccc}
\Phi_X|_{S \cap U_X} : & S \cap U_X & \stackrel{\sim}{\longrightarrow} &  \{ (w,y,z) \in \mathbb{A}^3 | F(1,y,z) = w^2 \} \\
  & [W:1:Y:Z] & \mapsto & (W,Y,Z)\\
\end{array}$$

Because $\car(\kk) \neq 2$, by writing the derivatives of $w^2 - F(1,y,z)$ in $\mathbb{A}^3$, we observe that $[W:1:Y:Z]$ is a singular point of $S$ if and only if $W = 0$ and if $[1:Y:Z]$ is a singular point of $\Gamma$.
It implies that $S \cap U_X$ is smooth if and only if $\Gamma \cap V_X$ is smooth.
Likewise, $S \cap U_Y$ \textit{resp} $S \cap U_Z$ is smooth if and only if $\Gamma \cap V_Y$ \textit{resp} $\Gamma \cap V_Z$ is smooth.
Since $S = (S \cap U_X) \cup (S \cap U_Y) \cup (S \cap U_Z)$ and $\Gamma = (\Gamma \cap V_X) \cup (\Gamma \cap V_Y) \cup (\Gamma \cap V_Z)$, we get our result.
\end{proof}

\begin{lemme} \label{structureautdp2} $ $
We have: $\Aut(S) = \Aut(\Gamma) \times \langle [-1:1:1:1] \rangle$. 
\end{lemme}

 \begin{proof} $ $
 
If $\car(\kk)$ is odd, we use \cite[Section 4]{martinodd}.

If $\car(\kk) = 0$, we use \cite[Section 8.7.3]{dolgachevbook}.
\end{proof}

%
%
%

\begin{prop} \label{degree2reduction} $ $
Let $G$ be a $2$-elementary non-cyclic subgroup of automorphisms of a del Pezzo surface of degree $2$. Then up to isomorphism, we can assume that $G$ is one of the following:

\begin{itemize}
\item $G_{21} = \left\{ [ \pm 1: \pm 1: \pm 1: \pm 1] \right\} \simeq (\mathbb{Z}/2)^3$. And $F$ is of the form $L_2(X^2,Y^2,Z^2)$.
\item
$G_{22} = \left\{ [ \pm 1: \pm 1:1:1] \right\} \simeq (\mathbb{Z}/2)^2$. And $F$ is of the form $F(X,Y,Z) =  X^4 + L_2(Y,Z) X^2 + L_4(Y,Z)$.
\item
$G_{23} = \left\{ [1: \pm 1: \pm 1: \pm 1] \right\} \simeq (\mathbb{Z}/2)^2$. And $F$ is of the form $L_2(X^2,Y^2,Z^2)$.
\item
$G_{24} = \left\{  [ 1:  1: 1:1] , [- 1: - 1: 1:1] , [- 1: 1: - 1:1], [ 1:  1: 1: - 1] \right\} \simeq (\mathbb{Z}/2)^2$. 
And $F$ is of the form $L_2(X^2,Y^2,Z^2)$.
\end{itemize}

\end{prop}

\begin{proof} $ $
We define the map $\pi : \Aut(S) \rightarrow \Aut(\Gamma) \subset \PGL(3,\kk)$. The kernel is $\langle [-1:1:1:1] \rangle \simeq \mathbb{Z}/2$. Furthermore, according to Lemma \ref{structureautdp2} we have $\Aut(S) \simeq \Aut(\Gamma) \times  \langle [-1:1:1:1] \rangle$.
Let $G$ be a $2$-elementary non-cyclic subgroup of automorphism of $S$.
There exist $a, b \in \mathbb{N}_{\geq 0}$ such that $\Ker (\pi|_G) \simeq (\mathbb{Z}/2)^a$ and $\pi(G) \simeq (\mathbb{Z}/2)^b$. We have $\Ker(\pi) \simeq \mathbb{Z}/2$, hence $a \leq 1$. We have $\Img(\pi) \subset \PGL(3,\kk)$, hence $b \leq 2$ according to Proposition \ref{prop3}.

If $(a,b) = (1,2)$ then we can conjugate $G$ to $G_{21}$ using Proposition \ref{prop3}.

If $(a,b) = (1,1)$ then we can conjugate $G$ to $G_{22}$ using Proposition \ref{prop3}.

If $(a,b) = (0,2)$ then
using Proposition \ref{prop3}, $\pi(G)$ is $\PGL(3,\kk)$-conjugate to $\left\{ \left[\begin{smallmatrix}  \pm 1 &   &    \\
 & \pm 1  &   \\
 &   & \pm 1   \\
\end{smallmatrix}\right] \right\}$. Hence up to conjugation inside $\Aut(\mathbb{P}(2,1,1,1))$ we can assume that $\pi(G) = \left\{ \left[\begin{smallmatrix}  \pm 1 &   &    \\
 & \pm 1  &   \\
 &   & \pm 1   \\
\end{smallmatrix}\right] \right\}$.
All elements of $G$ can be written of the form $[\pm 1 :\pm 1 :\pm 1 :\pm 1]$. 
Let's assume that $G \neq G_{23}$. Then there exists $g \in G$ of the form $[-1:\alpha:\beta:\gamma]$ where $\alpha, \beta, \gamma \in \left\{-1,1\right\}$.
Since $[-1:1:1:1] \notin G$, it means that we don't have "$\alpha = \beta = \gamma$".
Up to multiplying by $\lambda = -1$, we can assume that exactly one the scalars among $\alpha, \beta, \gamma$ is equal to $-1$. Up to a conjugation by a permutation matrix, we can assume that $\alpha =-1$. We have $\pi(G) \simeq (\mathbb{Z}/2)^2$, hence there exists $h \in G$ such that $\pi(h) = [1:-1:1]$. Hence $h = [\epsilon:1:-1:1]$ where $\epsilon \in \{-1,1\}$. If $\epsilon = -1$, then $G = G_{24}$. Now we assume that $\epsilon = 1$. We observe that by permuting the last two variables, we can conjugate $G$ and $G_{24}$, hence our result.

Since $G \subset \Aut(S)$, we get that $S$ has to be of the corresponding form in the cases $G_{21}, G_{23}, G_{24}$. For the subgroup $G_{22}$, we get that $F$ is of the form $F(X,Y,Z) =  \tau X^4 + L_2(Y,Z) X^2 + L_4(Y,Z)$ where $\tau \in \kk$. If $\tau = 0$, then $[1:0:0]$ is a singular point of $\Gamma$. According to Lemma \ref{smoothnessgamma}, we will get that $S$ is not smooth, hence a contradiction. Hence $\tau \neq 0$. Hence up to an isomorphism we can assume that $\tau = 1$.
\end{proof}

%
%


\begin{lemme} \label{2fixed} $ $
Let $F$ be of the form $F(X,Y,Z) =  X^4 + L_2(Y,Z) X^2 + L_4(Y,Z)$  where $L_2$ is a second degree homogeneous polynomial and $L_4$ is a fourth degree homogeneous polynomial.
The following table gives us the Euler characteristic $\chi^g$ of the fixed point variety $X^g$ of some automorphisms $g$ of $S$:
\tiny
$$\begin{tabular}{ | c | c | c | c| }
 \hline 
   element & fixed points & nature & $\chi$  \\
   \hline
   id &  &  & $10$ \\
   \hline
   $[1 : -1 : 1 : 1]$ & $\{[\pm 1 : 1 : 0 : 0 ]\} \sqcup \{ [W:0:Y:Z] | W^2 = F(0,Y,Z) \}$ & $2$ pts and an elliptic curve & $2$   \\
 \hline  
    $[-1 : -1 : 1 : 1]$ & $\{ [0:0:Y:Z] | F(0,Y,Z) = 0 \} $ & $4$ pts & $4$  \\
 \hline  
    $[-1 : 1 : 1 : 1]$ & $\Gamma$ & quartic curve of genus $3$ & $-4$  \\
 \hline  
 \end{tabular}$$
\end{lemme}

\begin{proof} $ $
We observe that $\Gamma$ is invariant under the action of the involution $[-1:1:1] \in \PGL(3,\kk)$. Hence according to \cite[Proposition A.8.10]{reductive}, $\{ [0:0:Y:Z] | F(0,Y,Z) = 0 \} $ is smooth.
We also know that the curve $\Gamma$ is smooth (Lemma \ref{smoothnessgamma}), hence using Bézout's Theorem we find that the following set: \[ \left\{ [0:0:Y:Z] | F(0,Y,Z) = 0 \right\} \] is composed of exactly four points.

We now prove the nature of the set of fixed points $\left\{ [W:0:Y:Z] | W^2 = F(0,Y,Z) \right\}$:
First we observe that the action of $\langle [1:-1:1:1] \rangle$ is a linearly reductive group action (since $\car(\kk) \neq 2$). Hence according to \cite[Proposition A.8.10]{reductive} applied to the subgroup $\langle [1:-1:1:1] \rangle$, we get that the curve $\{ [W:0:Y:Z] | W^2 = F(0,Y,Z) \}$
is smooth.
Now we will compute its genus.

The map:
$$\begin{array}{ccccc}
& & \{ [W:0:Y:Z] | W^2 = F(0,Y,Z) \} & \to & \mathbb{P}^1 \\
& & [W:X:Y:Z] & \mapsto & [Y:Z] \\
\end{array}$$ is a 2:1 morphism. The branching points of this morphism are the points of the set $\{ [0:0:Y:Z] | F(0,Y,Z) = 0 \} $. Hence it has four branching points. Using Riemann-Hurwitz formula, we get that the curve $\left\{ [W:0:Y:Z] | W^2 = F(0,Y,Z) \right\}$ is elliptic.
\end{proof}

We will now compute the invariant Picard rank of these four subgroups.

\begin{prop} \label{picard2} $ $
We have:

$rk(Pic(S)^{G_{21}}) = rk(Pic(S)^{G_{22}}) = 1$

$rk(Pic(S)^{G_{23}})  = 2$

$rk(Pic(S)^{G_{24}})  = 3$
\end{prop}

\begin{proof} $ $
    Using Proposition \ref{rankformula}, we have the following formula for the Picard rank: $$\rank \Pic(X)^G + 2 = \frac{1}{|G|} \sum_{g \in G} \chi(X^g)$$ Hence we get the result for each subgroup by using this formula with Lemma \ref{2fixed}.
\end{proof}


\subsection{A detailed study of the case $G \simeq (\mathbb{Z}/2)^3$: Simplification of the surface and conjugations by isomorphism} \label{subsection444}
$ $

\medskip

In this subsection we will study the case: $$G = G_{21}  = \left\{ [ \pm 1: \pm 1: \pm 1: \pm 1] \right\} \subset \Aut(S)$$ where: $$S = \left\{ W^2 = a X^4 + b Y^4 + c Z^4 + d X^2 Y^2 + e X^2 Z^2 + f Y^2 Z^2  \right\} \subset \mathbb{P}(2,1,1,1)$$ for some $a,b,c,d,e,f \in \kk$, described in Proposition \ref{degree2reduction}.

We will first show that we can simplify the equation of the surface using the symmetry of the subgroup $G$ (Proposition \ref{21surface}).
Then we will study the possible conjugations by isomorphisms between such couples $(G_{21},S)$ (Proposition \ref{21conjugation}).


\begin{defi} $ $
Let $d,e,f \in \kk$.
We define the del Pezzo surface of degree $1$: $$S_{def} = \left\{W^2 = X^4 + Y^ 4 + Z^4 + d X^2 Y^2 + e X^2 Z^2 + f Y^2 Z^2 \right\} \subset \mathbb{P}(2,1,1,1)$$
\end{defi}

\begin{lemme} \label{21smooth}  $ $
Let $a,b,c,d,e,f \in \kk$.

Let $F(X,Y,Z) = a X^4 + b Y^4 + c Z^4 + d X^2 Y^2 + e X^2 Z^2 + f Y^2 Z^2$.

Then the two following assertions are equivalent:
\begin{enumerate}
\item The curve $\Gamma = \{ F = 0\} \subset \mathbb{P}^2$ is smooth.
\item $abc (d^2 - 4 ab)(e^2 - 4 ac)(f^2 - 4 cb)(4 abc -  a f^2 -  d^2 c -  b e^2 +  d e f ) \neq 0$ 
\end{enumerate}
\end{lemme}

\begin{proof} $ $
The gradient of $F$ is: $$\nabla F (X,Y,Z)=\left[\begin{smallmatrix}
4 a X^3 + 2 d X Y^2 + 2 e X Z^2 \\
4 b Y^3 + 2 d X^2 Y + 2 f Z^2 Y \\
4 c Z^3 + 2 e X^2 Z + 2 f Y^2 Z
\end{smallmatrix}\right]
$$

Since $\car(\kk) \neq 2$, we can use Euler's formula for homogeneous polynomials:
\newline
A point is a singular point of $\Gamma$ if and only if the gradient $\nabla F$ vanishes at this point.

We observe that $\nabla F (X,Y,Z) = 0$ if and only if:

$\left\{
    \begin{array}{r}
2 a X^2 +  d  Y^2 +  e  Z^2 = 0 \mbox{ or } X = 0 \\
2 b Y^2 + d X^2  +  f Z^2 = 0  \mbox{ or } Y = 0 \\
2 c Z^2 +  e X^2  +  f Y^2 = 0 \mbox{ or } Z = 0
    \end{array}
\right.$
\\

If $X = Y = 0$, then $[0:0:1]$ is a singular point of $\Gamma$ if and only if $c=0$. Likewise for $[0:1:0]$ and $[1:0:0]$.

Now we assume $a,b,c \neq 0$. If $X = 0, Y \neq 0, Z \neq 0$, then $[X:Y:Z]$ is a singular point of $\Gamma$ if and only if: $$\left\{
    \begin{array}{r}
2 b Y^2  +  f Z^2 = 0  \\
2 c Z^2   +  f Y^2 = 0
    \end{array}
\right.$$
This system has a non-trivial solution if and only if $f^2 = 4 bc$. Likewise if $X \neq 0, Y =  0, Z \neq 0$ and if $X \neq 0, Y \neq 0, Z = 0$.
\\

Finally, we assume $X, Y, Z \neq 0$:
$$\left\{
    \begin{array}{r}
2 a X^2 +  d  Y^2 +  e  Z^2 = 0 \\
2 b Y^2 + d X^2  +  f Z^2 = 0   \\
2 c Z^2 +  e X^2  +  f Y^2 = 0 
    \end{array}
\right.$$

By viewing this system as a linear system, we can see that it has a non-trivial solution if and only if: $$det \left[\begin{smallmatrix}  2a & d  & e   \\
d & 2b  & f  \\
e & f  & 2c   \\
\end{smallmatrix}\right] = 0,$$ \textit{i.e.} if and only if: $$4 abc -  a f^2 -  d^2 c -  b e^2 +  d e f  =  0$$ \end{proof}

\begin{prop} \label{21surface} $ $

We assume $S$ is smooth.
Up to an isomorphism preserving $G_{21}$, we can assume that $S = S_{def}$ where $d,e,f \in \kk \setminus \{-2,2\}$ such that $ 4 - f^2 - d^2 - e^2 + def \neq 0$.
\end{prop}

\begin{proof} $ $
Using Lemma \ref{smoothnessgamma}, we get that $\Gamma$ is smooth. Hence using Lemma \ref{21smooth}, we get $a , b , c \neq 0$. We can then do a conjugation by an isomorphism given by a diagonal element (hence not changing the subgroup $G_{21}$) to get $a = b = c = 1$. Then using again Lemma \ref{smoothnessgamma} and Lemma \ref{21smooth}, we get the conditions on $d, e , f$.
\end{proof}


We will now describe the normalizer of the subgroup $G_{21}$ in Lemma \ref{lemmenormalizerg21}.

\begin{lemme} \label{lemmenormalizerg21lemme} $ $

Let $\phi \in \PGL(3,\kk)$ be an element of the normalizer of $\{ [\pm 1 : \pm 1 : \pm 1]\} \subset \PGL(3,\kk)$.
Then there exists $\sigma$ a permutation matrix and $D$ a diagonal invertible matrix such that $\phi = \sigma . D$.
\end{lemme}

\begin{proof} $ $
The element $\phi$ is a permutation of the following subset: $$\left\{ [-1:1:1], [1:-1:1] , [1:1:-1] \right\} \subset \PGL(3,\kk)$$ Let $\sigma \in \PGL(3,\kk)$ be the permutation matrix inducing the same permutation as $\phi$ on this set. Then by construction, the element $\sigma^{-1} \phi$ belongs to the centralizer of $\left\{ [\pm 1: \pm 1: \pm 1] \right\} \subset \PGL(3,\kk)$. We observe that the centralizer of $[-1:1:1]$ \textit{resp}  $[1:-1:1]$ \textit{resp}  $[1:1:-1]$ is $ \left\{  \left[\begin{smallmatrix}  * & 0  & 0   \\
0 & *  & *  \\
0 & *  & *   \\
\end{smallmatrix}\right]  \right\}$ \textit{resp} $ \left\{  \left[\begin{smallmatrix}  * & 0  & *   \\
0 & *  & 0  \\
* & 0  & *   \\
\end{smallmatrix}\right]  \right\}$ \textit{resp} $ \left\{  \left[\begin{smallmatrix}  * & *  & 0   \\
* & *  & 0  \\
0 & 0  & *   \\
\end{smallmatrix}\right]  \right\}$. Hence the centralizer of $\left\{ [\pm 1: \pm 1: \pm 1] \right\}$ is the set of diagonal invertible matrices, which means there exists $D$ a diagonal invertible matrix such that $\sigma^{-1} \phi = D$, hence our result.
\end{proof}

\begin{defi} $ $
We define the subgroup $\mathcal{P}$ of automorphisms of $\Aut(\mathbb{P}(2,1,1,1))$ as the set of maps of the form: 
$$\begin{array}{ccccc}
& & \mathbb{P}(2,1,1,1) & \to & \mathbb{P}(2,1,1,1) \\
& & [W:X:Y:Z] & \mapsto & [W : bX+cY+dZ : eX+fY + gZ : hX+iY+jZ] \\
\end{array}$$
where
$\left[\begin{smallmatrix}  b & c  & d   \\
e & f  & g  \\
h & i  & j   \\
\end{smallmatrix}\right]\in \GL(3,\kk)$.
\end{defi}

\begin{lemme} \label{lemmenormalizerg21} $ $
The normalizer of $G_{21}$ in $\mathcal{P}$ is:

$ \left\{  \left[\begin{array}{c|c}
1 & 0_{1 \times 3}  \\
\hline
0_{3 \times 1}  & \sigma . D
\end{array}\right]  | \sigma \text{ is a permutation matrix} , D \text{ is a diagonal invertible matrix} \right\}$

\end{lemme}

\begin{proof} $ $
We define $H = \left\{ [1:\pm 1 :\pm 1 : \pm 1] \right\} \subset G_{21}$, Klein subgroup of $G_{21}$.

Let $\Phi = [1:\phi]$ an element of the normalizer of $G_{21}$, where $\phi \in \GL(3,\kk)$. We know that for every $g \in G_{21}$, the nature of the fixed point of $g$ and $\Phi g \Phi^{-1}$ are the same. We also observe that for a given $h \in H$ and $g \in G_{21}$, if the nature of the fixed points of $h$ and $g$ are the same, then $g \in H$ (Lemma \ref{2fixed}).
Hence we have $\Phi H \Phi^{-1} \subset H$, so $\Phi H \Phi^{-1} = H$ by cardinality.
Hence $\phi \left\{ [\pm 1 :\pm 1 : \pm 1] \right\} \phi^{-1} = \left\{ [\pm 1 :\pm 1 : \pm 1] \right\}$ inside $\PGL(3,\kk)$.
Hence according to Lemma \ref{lemmenormalizerg21lemme} there exist  $\sigma$ a permutation matrix and $D$ a diagonal invertible matrix such that $\phi = \sigma D$. Hence the result.

Conversely, if $\Phi =  \left[\begin{array}{c|c}
1 & 0_{1 \times 3}   \\
\hline
0_{3 \times 1} & \sigma . D
\end{array}\right] $, then $\Phi$ belongs to the normalizer of $G_{21}$.
\end{proof}

\bigskip

We will now use the description of the normalizer of the subgroup $G_{21}$ in order to get a necessary and sufficient condition for the conjugation of two such pairs:

\bigskip

\begin{prop} \label{21conjugation} $ $
Let $d,e,f,d',e',f' \in \kk$.
The following are equivalent:
\begin{enumerate}
\item \label{21conjugation1} There exists an isomorphism $S_{def} \rightarrow S_{d'e'f'}$ that conjugates $G_{21}$ to itself.
\item \label{21conjugation2} There exist $\epsilon_d, \epsilon_e \in \{ -1, 1\}$ such that $\{d',e',f'\} = \{\epsilon_d d, \epsilon_e e, \epsilon_d \epsilon_e f \}$.
\end{enumerate}
\end{prop}

\begin{proof} $ $

\ref{21conjugation2} $\Rightarrow$ \ref{21conjugation1}:
Let's assume there exist $\epsilon_d, \epsilon_e \in \{ -1, 1\}$ such that: $$\{d',e',f'\} = \left\{\epsilon_d d, \epsilon_e e, \epsilon_d \epsilon_e f \right\}$$
Applying a permutation of $X,Y,Z$ conjugates $G_{21}$ to itself and permutes $d,e,f$. We may thus assume that $d'=\epsilon_d d$, $e'=\epsilon_e e$ and $f'=\epsilon_d\epsilon_e f$. We then use $[W:X:Y:Z]\mapsto [W:\sqrt{\epsilon_d} X: \sqrt{\epsilon_e} Y: \sqrt{\epsilon_d}\sqrt{\epsilon_e} Z]$.
\\

\ref{21conjugation1} $\Rightarrow$ \ref{21conjugation2}:
Conversely, let's assume there exists an isomorphism $\Phi : S_{def} \rightarrow S_{d'e'f'}$ leaving $G_{21}$ invariant.
Such an isomorphism is induced by an element of the group $\Aut(\mathbb{P}(2,1,1,1))$ (see Lemma \ref{autodp2}). Since it sends $S_{def}$ onto $S_{d'e'f'}$, it means that it is of the form $\left[\begin{array}{c|c}
* & 0_{1 \times 3}  \\
\hline
0_{3 \times 1} & *
\end{array}\right]$, hence it is an element of $\mathcal{P}$. Since it conjugates $G_{21}$ to itself, using Lemma \ref{lemmenormalizerg21}, this automorphism is of the form $\left[\begin{array}{c|c}
1 & 0_{1 \times 3} \\
\hline
0_{3 \times 1} & \sigma D
\end{array}\right]$ where $\sigma$ is a $3 \times 3$ permutation matrix and $D$ is a $3 \times 3$ invertible diagonal matrix.
\newline
Let $\Phi_D = \left[\begin{array}{c|c}
1 & 0_{1 \times 3}  \\
\hline
0_{3 \times 1} &  D
\end{array}\right]$ and $\Phi_\sigma = \left[\begin{array}{c|c}
1 & 0_{1 \times 3} \\
\hline
0_{3 \times 1} &  \sigma
\end{array}\right]$.
\newline
Then we have $\Phi_\sigma^{-1} (S_{d'e'f'}) = \Phi_{\sigma^{-1}} (S_{d'e'f'}) = S_{d''e''f''}$ where $d'',e'',f''$ are scalars such that $\{d'',e'',f''\} = \{d',e',f'\}$. So we have $\Phi_D : S_{def} \rightarrow S_{d''e''f''}$. Let $\alpha, \beta, \gamma \in \kk^\times$ such that $D = \left[\begin{smallmatrix}  1 & 0  & 0 & 0  \\
0 & \alpha  & 0 & 0  \\
0 & 0  & \beta & 0  \\
0 & 0  & 0 & \gamma
\end{smallmatrix}\right]$.
Then the following map: $$\begin{array}{ccccc}
\Phi_D  & : & S_{def} & \to & S_{d''e''f''} \\
 & & [W:X:Y:Z] & \mapsto & [W:\alpha X:\beta Y:\gamma Z] \\
\end{array}$$ is an isomorphism.
\newline
So $W^2 = X^4 + Y^4 + Z^4 + d X^2 Y^2 + e X^2 Z^2 + f Y^2 Z^2 \Leftrightarrow W^2 = \alpha^4 X^4 + \beta^4 Y^4 + \gamma^4 Z^4 + d'' \alpha^2 \beta^2 X^2 Y^2 + e'' \alpha^2 \gamma^2 X^2 Z^2 + f'' \beta^2 \gamma^2 Y^2 Z^2$.
We observe that these two equations are the same up to a multiple. Hence we get:
$$\alpha^4 = \beta^4 = \gamma^4 = 1, d'' \alpha^2 \beta^2 = d, e'' \alpha^2 \gamma^2 = e, f'' \beta^2 \gamma^2 = f$$
Hence our result with $\epsilon_d = \alpha^2 \beta^2 = \pm 1, \epsilon_e = \alpha^2 \gamma^2 = \pm 1$.
\end{proof}

\newpage

\section{Automorphisms of del Pezzo surfaces of degree $1$ when $\car(\kk) \neq 2$} \label{orsay45} $ $

We assume $\car(\kk) \notin \{ 2,p\}$.
In this section we classify up to isomorphism the $p$-elementary non-cyclic subgroups of automorphism of del Pezzo surfaces of degree $1$, when the characteristic of the field is not equal to $p$ or to $2$. The results are given in Propositions \ref{propdegree1pgeq5}, \ref{12result}, \ref{1result}. 
This study is organized into three mutually independent subsections based on the value of $p$:
\begin{itemize}
\item Subsection \ref{subsection452samarcande} covers the case $p > 3$, which is particularly straightforward.
\item Subsection \ref{subsection453samarcande} focuses on the relatively simple case $p = 2$.
\item Subsection \ref{subsection454samarcande} deals with the more complex and extensive case $p = 3$.
\end{itemize}

We first establish some general results in the next Subsection \ref{451tashkent} which will be used across the three following subsections.

\bigskip

\subsection{General results about del Pezzo surfaces of degree $1$} \label{451tashkent}
$ $

\bigskip

According to \cite[5.1]{martinodd} in odd characteristic, and to \cite[Theorem III.3.5]{Kol} and \cite[Corollary 3.54]{KoSmCo} in characteristic 0, up to isomorphism, a del Pezzo surface of degree $1$ is a surface of the form: 
$$\left \{ [W:X:Y:Z] | W^2 = Z^3 + F_2(X,Y) Z^2 + F_4 (X,Y)Z + F_6 (X,Y) \right \} \subset \mathbb{P}(3,1,1,2)$$

where $F_2$, $F_4$ and $F_6$ are forms of degree two, four and six. \\

\begin{defi} $ $
We define the following subgroup of automorphisms of $\mathbb{P}(3,1,1,2)$:
 \[ \mathcal{P} = \left\{
\begin{array}{cccc}
   &  \mathbb{P}(3,1,1,2) & \to & \mathbb{P}(3,1,1,2) \\
 &  [W:X:Y:Z]  & \mapsto & [aW : bX + c Y : dX + e Y : f Z]    \\
\end{array} | \begin{array}{c} a,b,c,d,e,f \in \kk \\ af(be-cd) \neq 0 \end{array}
\right\} \]
\end{defi}

\begin{defi} $ $
We also define the following surjective group homomorphism:
$$\begin{array}{ccccc}
\pi & : & \mathcal{P} & \to & \PGL(2,\kk) \\
 & & ([W:X:Y:Z] \mapsto [aW : bX + c Y : dX + e Y : f Z]) & \mapsto & \left[\begin{smallmatrix}
b & c  \\
d & e  \\
 \end{smallmatrix}\right] \\
\end{array}$$ \\
\end{defi}


The following Lemma \ref{1lemme0} will serve as a foundation for the further study of $p$-elementary subgroups.

\begin{lemme} \label{1lemme0} $ $

\begin{enumerate}

\item \label{1lemme01} $\Aut(S)$ is a subgroup of $\mathcal{P}$.

\item \label{1lemme02} 
$
\Ker(\pi|_{\Aut(S)}) = 
\begin{cases} 
\langle [-1 : 1 : 1 : 1] \rangle & \text{if } F_2 \neq 0 \text{ or } F_4 \neq 0, \\
\langle [-1 : 1 : 1 : 1], [1 : 1 : 1 : j] \rangle & \text{if } F_2 = F_4 = 0.
\end{cases}
$
\end{enumerate}
\end{lemme}

\begin{proof} $ $

    \begin{enumerate}

        \item We use Lemma \ref{autodp1}.
        \item The kernel of $\pi$ is composed of the maps of the form $[a  : 1  :  1 :  b ]$ where $a,b \in \kk^\times$.  If $F_2 \neq 0$ or $F_4 \neq 0$ then such a map belongs to $\Aut(S)$ if and only if it belongs to $\langle [-1 : 1  :  1 :  1] \rangle$. If $F_2 = F_4 = 0$ then such a map belongs to $\Aut(S)$ if and only if it belongs to $\langle [-1 : 1  :  1 :  1] ,  [1 : 1  :  1 :  j ] \rangle$.
    \end{enumerate}
\end{proof}


\subsection{The study of $p$-elementary non-cyclic subgroups of automorphisms of del Pezzo surfaces of degree $1$ when $p > 3$} \label{subsection452samarcande}
$ $

\medskip

\begin{prop} \label{propdegree1pgeq5} 
 $ $
We assume $p > 3$. Let $S$ be a del Pezzo surface of degree $1$.
Then $\Aut(S)$ has no $p$-elementary non-cyclic subgroup.
\end{prop}

\begin{proof} $ $
Let $G$ be a $p$-elementary subgroup of $\Aut(S)$.
According to Lemma \ref{1lemme0}, $\Ker(\pi)$ is isomorphic to $\mathbb{Z}/2$ or $\mathbb{Z}/6$.
$\Ker(\pi|_{G})$ is a $p$-elementary subgroup of $\Ker(\pi)$, so it is trivial.
Hence $G \simeq \pi(G)$.
$\pi(G)$ is a subgroup of $\PGL(2,\kk)$, hence according to Proposition \ref{dublin1}, it is cyclic. So $G$ is cyclic.
\end{proof}

\medskip

\subsection{The study of $2$-elementary non-cyclic subgroups of automorphisms of del Pezzo surfaces of degree $1$} \label{subsection453samarcande} $ $

\medskip


\begin{lemme} \label{lemmefixedpointsg12} $ $
We define the following del Pezzo surface of degree $1$: $$S = \left\{ W^2  = Z^3 + L_1(X^2,Y^2) Z^2 + L_2 (X^2,Y^2)Z + L_3 (X^2,Y^2) \right\} \subset \mathbb{P}(3,1,1,2)$$ 
Where $L_n$ is a form of degree $n$.
Let $a =[1:-1:1:1], b = [1:1:-1:1] \in G_{12} \subset \Aut(S)$.
The following table gives us the Euler characteristic $\chi^g$ of the fixed point variety $X^g$ of the elements $g \in G_{12}$ acting on $S$:

\tiny
$$\begin{tabular}{ | c | c | c | c| }
 \hline 
   element & fixed points & nature & $\chi$  \\
   \hline
   id &  &  & $11$ \\
   \hline
   $a$ & $\{[W:0:Y:Z] \in S  \} \sqcup \{ [0:X:0:Z] \in S \}$  & $3$ pts and a curve of genus $1$& $3$   \\
 \hline  
    $b$ & $\{[W:X:0:Z \in S]\} \sqcup \{ [0:0:Y:Z] \in S \}$ & $3$ pts and a curve of genus $1$ & $3$   \\
 \hline  
    $a b $ & $\{[0:X:Y:Z]  \in S \} \sqcup \{ [1:0:0:1]\}$ & $1$ point and a curve of genus $4$ & $- 5$  \\
 \hline  

 \end{tabular}$$
\end{lemme}

\begin{proof} $ $
By adjunction formula, the following curves: $$\left\{ [W:0:Y:Z] | W^2 = Z^3 +  L_1(0,Y^2) Z^2 + L_2(0,Y^2) Z + L_3(0,Y^2) \right\} \subset \mathbb{P}(3,1,1,2)$$ $$\left\{ [W:X:0:Z] | W^2 = Z^3 + L_1(X^2,0) Z^2 + L_2(X^2,0) Z + L_3(X^2,0) \right\} \subset \mathbb{P}(3,1,1,2)$$ are of genus $1$ because they are equivalent to $-K_S$ (using Lemma \ref{autodp1}),
and the curve: $$\left\{ [0:X:Y:Z] | 0 = Z^3 + L_1(X^2,Y^2) Z^2 + L_2(X^2,Y^2) Z + L_3(X^2,Y^2) \right\} \subset \mathbb{P}(3,1,1,2)$$ is of genus $4$ because it is equivalent to $-3 K_S$ (using Lemma \ref{autodp1}).
\end{proof}

\begin{prop} \label{12result} $ $
Let $G$ be a $2$-elementary non-cyclic subgroup of automorphisms of a del Pezzo surface $S$ of degree $1$. Then up to conjugation:
$$G = G_{12} := \left\{  [\pm 1 : \pm 1 : \pm 1 : 1] \right\} \simeq (\mathbb{Z}/2)^2$$
$$S = \left\{ W^2  = Z^3 + L_1(X^2,Y^2) Z^2 + L_2 (X^2,Y^2)Z + L_3 (X^2,Y^2) \right\} \subset \mathbb{P}(3,1,1,2)$$

Where $L_n$ is a form of degree $n$.
Furthermore, $rk(Pic(S)^{G_{12}}) = 1$.
\end{prop}

\begin{proof} $ $
Let $G$ be a $2$-elementary non-cyclic subgroup of automorphisms of the following surface: $$S = \{ [W:X:Y:Z] | W^2 = Z^3 + F_2(X,Y) Z^2 + F_4 (X,Y)Z + F_6 (X,Y) \}.$$
\indent Using Lemma \ref{1lemme0}, and doing a linear transformation, we get: $$G = G_{12} = \langle [-1:1:1:1], [1:-1:1:1] \rangle.$$

Since $G$ acts on $S$, we get the conditions on $F_n$.

  The last assertion is a direct consequence of Lemma \ref{lemmefixedpointsg12} and of Proposition \ref{rankformula}.
\end{proof}

\subsection{The study of $3$-elementary non-cyclic subgroups of automorphisms of del Pezzo surfaces of degree $1$} \label{subsection454samarcande}
$ $

\medskip

We assume $\car(\kk) \notin \{2,3\}$ in this subsection.
The goal of this subsection is to give a classification of $3$-elementary non-cyclic subgroups of automorphisms of del Pezzo surfaces of degree $1$ up to isomorphism. The results are given in Proposition \ref{1result}.

\begin{defi} \label{defi353} $ $
Let $c \in \kk$. We define the following surface of degree $1$:
$$S_c = \{W^2 = Z^3 + X^6 + Y^6 + c X^3 Y^3 \} \subset \mathbb{P}(3,1,1,2)$$
We also define the following subgroup of automorphism of $S_c$:
$$G_{11} = \langle  [1:j:1:1] ,  [1:1:j:1] \rangle$$
\end{defi}

\begin{lemme} \label{1smooth} $ $
Let $a,b,c \in \kk$. The two following assertions are equivalent:


\begin{enumerate}
\item
The surface $\{W^2 = Z^3 + a X^6 + b Y^6 + c X^3 Y^3 \} \subset \mathbb{P}(3,1,1,2)$ is smooth.
\item $a b (c^2 - 4 ab) \neq 0$.
\end{enumerate}
\end{lemme}

\begin{proof} $ $

We define $S = \{W^2 = Z^3 + a X^6 + b Y^6 + c X^3 Y^3 \} \subset \mathbb{P}(3,1,1,2)$.

We define the following open subsets: 
$$U_W = \left\{ [W:X:Y:Z] \in \mathbb{P}(3,1,1,2) | W \neq 0 \right\}$$
$$U_X = \left\{ [W:X:Y:Z] \in \mathbb{P}(3,1,1,2) | X \neq 0 \right\}$$
$$U_Y = \left\{ [W:X:Y:Z] \in \mathbb{P}(3,1,1,2) | Y \neq 0 \right\}$$
$$U_Z = \left\{ [W:X:Y:Z] \in \mathbb{P}(3,1,1,2) | Z \neq 0 \right\}$$

Let's study the smoothness of $S \cap U_X$.
We define the following isomorphism:
$$\begin{array}{cccc}
\Phi_X : & U_X & \stackrel{\sim}{\longrightarrow} & \mathbb{A}^3 \\
  & [W:1:Y:Z] & \mapsto & (W,Y,Z)\\
\end{array}$$


We restrict this isomorphism to $S \cap U_X$:
$$\begin{array}{cccc}
\Phi_X|_{S \cap U_X} : & S \cap U_X & \stackrel{\sim}{\longrightarrow} &  \{ (w,y,z) \in \mathbb{A}^3 | z^3 + a  + b y^6 + c  y^3 = w^2 \} \\
  & [W:1:Y:Z] & \mapsto & (W,Y,Z)\\
\end{array}$$

By writing the partial derivatives (and using that $\car(\kk) \notin \{2,3\}$), we observe that $\{ (w,y,z) \in \mathbb{A}^3 | z^3 + a  + b y^6 + c  y^3 = w^2 \}$ is smooth if and only if $a \neq 0$ and $c^2 \neq 4 a b$.
Hence $S \cap U_X$ is smooth if and only if $a \neq 0$ and $c^2 \neq 4 a b$.

Likewise, $S \cap U_Y$ is smooth if and only if $b \neq 0$ and $c^2 \neq 4 a b$.

We observe that $S = (S \cap U_X) \cup (S \cap U_Y) \cup \{ [1:0:0:1]\}$, hence it only remains to study the smoothness of the point $[1:0:0:1] \in S$ to conclude our proof.
We define the following morphism:
$$\begin{array}{cccc}
\Psi_{W} : & \mathbb{A}^3 & \rightarrow & U_W  \\
  & (x,y,z) & \mapsto & [1:x:y:z]\\
\end{array}$$
This is a $3:1$ map. We observe that $U_W$ is then the quotient of $\mathbb{A}^3$ by the action of a cyclic group of order $3$. The action sends $(x,y,z)$ onto $(j x,j y,j^2z)$. The preimage of $[1:0:0:1]$ is $\{ (0,0,1) , (0,0,j), (0,0,j^2)\}$. 
We observe that the only fixed point by the action is $(0,0,0)$.
Let $\pi = \Psi_W |_{\mathbb{A}^3 \setminus \{(0,0,0)\}} : \mathbb{A}^3 \setminus \{(0,0,0)\} \rightarrow U_W \setminus \{ [1:0:0:0]\}$.
Now we observe that $\pi|_{\pi^{-1}(S \cap U_W)} : \pi^{-1}(S \cap U_W) \rightarrow S \cap U_W$ is also étale because $\pi$ is étale (see Lemma \ref{1smoothbis}) and because $(S \cap U_W) \subset (U_W \setminus [1:0:0:0])$ is locally closed.
Hence we have that $(0,0,1) \in \pi^{-1} (S \cap U_W)$ is a smooth point if and only if $\pi((0,0,1)) \in S \cap U_W$ is a smooth point.
But we observe that $\pi^{-1} (S \cap U_W) = \{ (x,y,z) \in \mathbb{A}^3 | [1:x:y:z] \in S \} = \{ (x,y,z) \in \mathbb{A}^3 | 1 = z^3 + a x^6 + b y^6 + c x^3 y^3\}$, hence the point $(0,0,1)$ is smooth because the derivate with respect to $z$ is $3 z^2 = 3 \neq 0$ since $\car(\kk) \neq 3$.
So $\pi((0,0,1)) \in S \cap U_W$ is a smooth point, hence our result since $\pi(0,0,1)= [1:0:0:1]$.
\end{proof}

\begin{lemme} \label{1smoothbis} $ $
With the notations of the proof of Lemma \ref{1smooth}, the morphism $\pi$ is étale.
\end{lemme}

\begin{proof} $ $
Let $X = \mathbb{A}^3 \setminus \{ (0,0,0) \}, Y = U_W, G = \langle (x,y,z) \mapsto (jx,jy,j^2 z) \rangle \simeq \mathbb{Z}/3$.
Since we excluded the point $(0,0,0)$ from $\mathbb{A}^3$, we get that the action of $G$ on $X$ is free and that any orbit $\{ (x,y,z) , (jx,jy,j^2 z), (j^2 x,j^2 y, jz)\}$ is contained in an affine open subset of $X$. Furthermore, we observe that $\pi : X \rightarrow Y$ satisfies the two assertions written in \cite[Chapter 2, Section 7, Theorem]{mumford}.
Hence according to this theorem, the morphism $\pi$ is étale.
\end{proof}

\begin{lemme} \label{fixedpointsg11} $ $
Let $c \in \kk \setminus \{-2,2\}$.

Let $a = [1:1:j:1] , b = [1:j:1:1] \in G_{11} \subset \Aut(S_c)$.
The following table gives us the Euler characteristic $\chi^g$ of the fixed point variety $X^g$ of the elements $g \in G_{11}$ acting on $S_c$:

\tiny
$$\begin{tabular}{ | c | c | c | c|}
 \hline 
   element & fixed points & nature & $\chi$  \\
   \hline
   id &  &  & $11$ \\
   \hline
   $a , a^2$ & $\{[W:X:0:Z] | W^2 = Z^3 + X^6\} \sqcup \{ [\pm 1:0:1:0] \}$  & $2$ pts and a curve of genus $\alpha$& $4 - 2 \alpha$  \\
 \hline  
    $b, b^2$ & $\{[W:0:Y:Z] | W^2 = Z^3 + Y^6\} \sqcup \{ [\pm 1:1:0:0] \}$ & $2$ pts and a curve of genus $\beta$ & $4 - 2 \beta$   \\
 \hline  
    $a b , (a b)^2$ & $\{[W:X:Y:0] | W^2 = X^6 + Y^6 + c X^3 Y^3 \} \sqcup \{ [1:0:0:1]\}$ & $1$ point and a curve of genus $\gamma$ & $3 - 2 \gamma$  \\
 \hline  
     $a^2 b , a b^2$ & $\{ [\pm 1 :0 :1:0 ] , [\pm 1 :1 :0:0 ]  , [ 1 :0 :0:1 ]\}$ & $5$ points & $5$ \\
 \hline  
 \end{tabular}$$
 \normalsize
 Where $\alpha, \beta, \gamma \in \mathbb{N}_{\geq 0}$. 
\end{lemme}

\begin{proof} $ $
Direct computations.
\end{proof}

\begin{lemme} \label{13norm} $ $
The normalizer of $G_{11}$ in $\mathcal{P}$ is the following set:

$\begin{array}{ll} & \left\{[W:X:Y:Z]   \mapsto  [aW : bX :  c Y : d Z]    | a,b,c,d, \in \kk^\times
\right\} \\
\sqcup & \left\{
 [W:X:Y:Z]   \mapsto  [aW : bY :  c X : d Z]    | a,b,c,d, \in \kk^\times
\right\} \end{array}$

\end{lemme}

\begin{proof} $ $
Let $\mathcal{N}$ be the normalizer of $G_{11}$ in $\mathcal{P}$.

Let $a = [1:1:j:1] , b = [1:j:1:1] \in G_{11} \subset \Aut(S_c)$.

Let $E = \{ \langle a \rangle ,  \langle b \rangle , \langle ab \rangle , \langle a^2 b \rangle \}$ be the set of subgroups of $G_{11}$ isomorphic to $\mathbb{Z}/3$.
We define the following map:
$$\begin{array}{ccccc}
\rho & : & \mathcal{N} & \to & \mathfrak{S}_E \\
 & & \Phi& \mapsto & \left\{ \begin{array}{ccccc}
 &  & E & \to & E \\
 & & H & \mapsto & \Phi H \Phi^{-1}  \\
\end{array} \right\} \\ 
\end{array}$$

Because of the nature of the fixed points of the elements of $G_{11}$ given in Lemma \ref{fixedpointsg11}, we observe that for a given $\Phi \in \mathcal{N}$, we have $\rho(\Phi) (\langle a \rangle) \in \{ \langle a \rangle , \langle b \rangle \}$ , $\rho(\Phi) (\langle b \rangle) \in \{ \langle a \rangle , \langle b \rangle \}$ , $\rho(\Phi) (\langle ab \rangle) = \langle ab \rangle$ , $\rho(\Phi) (\langle a^2b \rangle) = \langle a^2b \rangle$.
Hence: $$\rho(\mathcal{N}) \quad \subset \quad \langle \begin{array}{ccccc}
 &  & E & \to & E \\
 & & \langle a \rangle & \mapsto & \langle b \rangle  \\
  & & \langle b \rangle & \mapsto & \langle a \rangle  \\
   & & \langle ab \rangle & \mapsto & \langle ab \rangle  \\
    & & \langle a^2b \rangle & \mapsto & \langle a^2b \rangle  \\
\end{array} \rangle \quad \simeq \quad \mathbb{Z}/2$$

Let $\Psi = \left[\begin{smallmatrix}  1 & 0  & 0 & 0  \\
0 & 0  & 1 & 0  \\
0 & 1  & 0 & 0  \\
0 & 0  & 0 & 1  
\end{smallmatrix}\right] \in \mathcal{P}$.
We observe that $\Psi \in \mathcal{N}$ and $\rho(\Psi)$ is the transposition of $\langle a \rangle$ and $\langle b \rangle$. Hence we have the equality:
$$\rho(\mathcal{N}) \quad = \quad \langle \begin{array}{ccccc}
 &  & E & \to & E \\
 & & \langle a \rangle & \mapsto & \langle b \rangle  \\
  & & \langle b \rangle & \mapsto & \langle a \rangle  \\
   & & \langle ab \rangle & \mapsto & \langle ab \rangle  \\
    & & \langle a^2b \rangle & \mapsto & \langle a^2b \rangle  \\
\end{array} \rangle \quad \simeq \quad \mathbb{Z}/2$$

We also observe that for a given $\Phi \in \mathcal{N} \setminus \Ker(\rho)$, we have $\Psi \Phi \in \Ker(\rho)$.
Hence: $$\mathcal{N} \quad = \quad \Ker(\rho) \rtimes \langle \Psi \rangle \quad \simeq \quad \Ker(\rho) \rtimes \mathbb{Z}/2$$

By computations we get that $\Ker(\rho)$ is the diagonal. Hence the result.
\end{proof}

%
%
\begin{prop} \label{1result} $ $
\begin{enumerate}
\item \label{1result1} Let $G$ be a $3$-elementary non-cyclic subgroup of automorphisms of a del Pezzo surface $S$ of degree $1$. Then
     there exists $c \in \kk \setminus \{-2,2\}$ such that up to conjugation by an isomorphism, $S = S_c$ and $G = G_{11}$.
\item \label{1result2} Let $c \in \kk \setminus \{-2,2\}$. The surface $S_c$ is smooth and $rk(Pic(S_c)^{G_{11}}) = 1$.
\item \label{1resultbis} Let $c, c' \in \kk \setminus \{-2,2\}$. $(G_{11}, S_c)$ and $(G_{11}, S_{c'})$ are conjugate by isomorphism if and only if $c = \pm c'$.
\end{enumerate}
 \end{prop}


\begin{proof} $ $
   \begin{enumerate}
   \item
By contradiction, let's assume $F_2 \neq 0$ or $F_4 \neq 0$. According to Lemma \ref{1lemme0}, we have $\Ker (\pi|_{\Aut(S)}) = \langle [-1 : 1 : 1 : 1] \rangle \simeq \mathbb{Z}/2 $. Hence $G \simeq \pi(G)$. But $\pi(G)$ is a subgroup of $\PGL(2,\kk)$. According to Proposition \ref{dublin1} , any $3$-elementary subgroup of $\PGL(2,\kk)$ is cyclic. Hence $G$ is cyclic. Hence our contradiction.

So $F_2 = F_4 = 0$. According to Lemma \ref{1lemme0}, we have $\Ker(\pi|_{\Aut(S)}) = \langle [-1 : 1  :  1 :  1] , [1 : 1  :  1 :  j ] \rangle \simeq \mathbb{Z}/6$. We also observe that $\pi(G)$ is a subgroup of $\PGL(2,\kk)$ which does not contains non-cyclic $3$-elementary subgroup according to Proposition \ref{dublin1}.
Since $G$ is not cyclic, we get:
$$\pi(G) \simeq \Ker(\pi|_G) \simeq \mathbb{Z}/3$$
Since $\pi(G)$ is a subgroup of $\PGL(2,\kk)$ isomorphic to $\mathbb{Z}/3$, up to conjugation by an element of $\PGL(2,\kk)$ we can assume $\pi(G) = \langle \left[ \begin{smallmatrix} 1 & 0 \\ 0 & j \end{smallmatrix}\right] \rangle$ (Proposition \ref{dublin1}).
Hence there exists an element $g = [a:1:j:b] \in G$ with $a,b \in \kk^\times$.
Since $g^3 = id$ we get $a^3 = b^3 = 1$. Let $\alpha, \beta \in \{0,1,2\}$ such that $g = [j^\alpha:1:j:j^\beta]$.
Since $\Ker(\pi|_{\Aut(S)}) = \langle [-1 : 1  :  1 :  1] , [1 : 1  :  1 :  j ] \rangle \simeq \mathbb{Z}/6$ and $\Ker(\pi|_G) \simeq \mathbb{Z}/3$ we get $g' := [1:1:1:j] \in G$. Hence by composition we also have $[j^\alpha:1:j:1] = g g'^{- \beta}  \in G$. Since $[j^\alpha:1:j:1]$ preserves $S$, we get $\alpha = 0$.
Hence finally, $G  = \langle [1:1:1:j],[1:j:1:1] \rangle = G_{11}$.
 
 The elements of $G$ are maps from $S$ to $S$, so $F_6$ is of the form: $F_6 = a X^6 + b Y^6 + c X^3 Y^3$ where $a , b ,c \in \kk$. Since $S$ is smooth, using Lemma \ref{1smooth}, we get $a, b \neq 0$. Hence up to a linear conjugation preserving the subgroup $G$ we can assume $a = b = 1$. Using again Lemma \ref{1smooth} and the smoothness of the surface $S$, we get $c \neq \pm 2$. 
 \item The smoothness comes from Lemma \ref{1smooth}.
Using  Lemma \ref{fixedpointsg11} and Proposition \ref{rankformula}, we get that:

$\begin{array}{rcl} rk(Pic(S_c)^{G_{11}}) & = & \dfrac{1}{9}(9 + 2 \times (2 - 2 \alpha + 2 - 2 \beta + 1 - 2 \gamma + 3)) \\
& = & \dfrac{1}{9} (25 - 4 (\alpha + \beta + \gamma)), \end{array}$

where $\alpha, \beta , \gamma \in \mathbb{N}$ are the genera of some curves. Since $rk(Pic(S_c)^{G_{11}}) \in \mathbb{N}_{\geq 1}$, the only possibility is hence $\alpha + \beta + \gamma = 4$ and $rk(Pic(S_c)^{G_{11}}) = 1$.
\item Let $c \in \kk \setminus \{-2,2\}$.
We define the following map:
$$\begin{array}{ccccc}
\Phi & : & \mathbb{P}(3,1,1,2) & \to & \mathbb{P}(3,1,1,2) \\
 & & [W:X:Y:Z] & \mapsto & [W:X:-Y:Z]  \\
\end{array}$$

We observe that:
\begin{itemize}
\item $\Phi$ is a well-defined automorphism of $\mathbb{P}(3,1,1,2)$.
\item $\Phi(S_c) = S_{-c}$.
\item $\Phi G_{11} \Phi^{-1} = G_{11}$.
\end{itemize}
Hence $(G_{11}, S_c)$ and $(G_{11}, S_{-c})$ are conjugate by isomorphism.

Conversely, let $c ,c' \in \kk \setminus \{-2,2\}$ such that $(G_{11}, S_c)$ and $(G_{11}, S_{c'})$ are conjugate by an isomorphism $\Phi$, \textit{i.e.}
$\Phi G_{11} \Phi^{-1} = G_{11}$ and $\Phi(S_c) = S_{c'}$.
According to Lemma \ref{13norm} we have: $\Phi = 
\left[\begin{smallmatrix}  \alpha & 0 & 0 & 0   \\
 0 & \beta & 0 & 0  \\
 0 & 0 & \gamma & 0  \\
  0 & 0 & 0 & \delta  
\end{smallmatrix} \right] $ or $\Phi = \
\left[\begin{smallmatrix}  \alpha & 0 & 0 & 0   \\
 0 & 0 & \beta & 0  \\
 0 & \gamma & 0  & 0  \\
  0 & 0 & 0 & \delta  
\end{smallmatrix} \right] $
where $\alpha, \beta, \gamma, \delta \in \kk^\times$.
Let: $$\begin{array}{ccccc}
\Psi & : & \mathbb{P}(3,1,1,2) & \to & \mathbb{P}(3,1,1,2) \\
 & & [W:X:Y:Z] & \mapsto & [W:Y:X:Z]  \\
\end{array}$$
We have:
\begin{itemize}
\item $\Psi$ is a well-defined automorphism of $\mathbb{P}(3,1,1,2)$.
\item $\Psi(S_c) = S_{c}$.
\item $\Psi G_{11} \Psi^{-1} = G_{11}$.
\end{itemize}
Hence, up to composition by $\Psi$, we can assume that $\Phi = 
\left[\begin{smallmatrix}  \alpha & 0 & 0 & 0   \\
 0 & \beta & 0 & 0  \\
 0 & 0 & \gamma & 0  \\
  0 & 0 & 0 & \delta  
\end{smallmatrix} \right] $.
We can also assume $\gamma = 1$.
We then have:
$$\Phi(S_c) = \left\{ \frac{1}{\alpha^2} W^2  = \frac{1}{\delta^3} Z^3+ \frac{1}{\beta^6} X^6 + Y^6 +  \frac{c}{\beta^3} X^3 Y^3   \right\}$$

Hence: $$\quad \quad \left\{ W^2 = Z^3 + X^6 + Y^6 + c' X^3 Y^3 \right\} = \left\{ \frac{1}{\alpha^2} W^2  = \frac{1}{\delta^3} Z^3+ \frac{1}{\beta^6} X^6 + Y^6 +  \frac{c}{\beta^3} X^3 Y^3   \right\}$$ 

Now we take the intersection with $\{ W = Z = 0 \}$:
$$\left\{ X^6 + Y^6 + c' X^3 Y^3 = 0 \right\} = \left\{  \frac{1}{\beta^6} X^6 + Y^6+  \frac{c}{\beta^3} X^3 Y^3  = 0  \right\} \subset \mathbb{P}^1$$ 

Hence $1 = \dfrac{1}{\beta^6}$ and $c' = \dfrac{c}{\beta^3}$.
Hence $c'^2 = c^2$.

 \end{enumerate}
\end{proof}

\newpage

\section{Automorphisms of del Pezzo surfaces of degree $1$ when $\car(\kk) = 2$} \label{orsay46} $ $


We assume $\car(\kk)=2$ and $p \neq 2$ in all this section.

As we will see, the study of del Pezzo surfaces of degree $1$ in characteristic $2$ differs 
from the case in characteristic not $2$, this is why it deserves a separate section.
We will mainly use the article of Gebhard Martin and Igor Dolgachev: \cite{martin2} section $5$ and $6$. The main result of this section is Proposition \ref{resultdp1char2}.

\subsection{Notations} $ $

\begin{defi} \label{defi361} 

Let $e \in \kk$. \newline
We define the following del Pezzo surface of degree $1$, subset of $\mathbb{P}(1,1,2,3)$:
$$S_e = \left\{ [u:v:x:y] | 0 = y^2 + uv(u+v) y + x^3 + (e + \sqrt{e}) (u^5 v + u v^5) + e u^3 v^3\right\}$$
We also define the following maps:
$$\begin{array}{ccccc}
a & : & \mathbb{P}(1,1,2,3)& \to & \mathbb{P}(1,1,2,3) \\
& & [u:v:x:y] & \mapsto & [u:v:j x : y] \\ 
&&&&\\
b & : & \mathbb{P}(1,1,2,3) & \to & \mathbb{P}(1,1,2,3) \\
& & [u:v:x:y] & \mapsto & [v:u+v:x:y+\sqrt{e}(u^2 v + u v^2 + v^3)] \\
\end{array}$$
We define the following group:
$$G_{12}' = \langle a,b \rangle$$
\end{defi}

\subsection{Results} $ $

\begin{prop} \label{resultdp1char2} $ $

Let $S$ be a del Pezzo surface $S$ of degree $1$.
Let $G$ be a $p$-elementary non-cyclic subgroup of automorphisms of the surface $S$.


Then $p=3$, $G \simeq (\mathbb{Z}/3)^2$, there exists $e \in \kk \setminus \left\{0,1\right\}$ such that up to isomorphism, $S = S_e$
and $G = G_{12}' = \langle a,b \rangle$.
Furthermore, we have $rk(Pic(S_e)^G) = 1$.
\end{prop}

\begin{proof} $ $

According to \cite[Theorem 6.8]{martin2}, since $p \neq 2$, if $S$ has a $p$-elementary non-cyclic subgroup of automorphisms, then $p=3$ and up to isomorphism we have $S = S_e$ for a certain $e \in \kk$.
According to Proposition \ref{subgroupse}, we have $G = \langle a,b \rangle$.
We assumed the surface to be smooth, hence according to Lemma \ref{smoothnessse}, we have $e \in \kk \setminus \left\{0,1\right\}$.
According to Lemma \ref{rkg12prime} we have $rk(Pic(S_e)^{G_{12}'})=1$.
\end{proof}

\subsection{A unique $3$-elementary non-cyclic subgroup of automorphisms} \label{massy463} $ $

Let $e \in \kk$.

\begin{lemme} \label{finitegroup1bis}

The group $\ZZ/6 \times \mathfrak{S}_3$ has exactly one $3$-elementary non-cyclic subgroup. \end{lemme}

\begin{proof} $ $
Denote by $L$ \textit{resp} $R$ the unique subgroup of $\ZZ/6$ \textit{resp} $ \mathfrak{S}_3$ isomorphic to $\ZZ/3$.

Hence $L$ \textit{resp} $R$ is a normal subgroup of $\ZZ/6$ \textit{resp} $ \mathfrak{S}_3$.

Hence $L \times R$ is a normal subgroup of $\ZZ/6 \times \mathfrak{S}_3$.

Observe that $L \times R$ is also a $3$-Sylow subgroup of $\ZZ/6 \times \mathfrak{S}_3$.

Hence by Sylow's theorem, $L \times R$ is the unique $3$-Sylow subgroup of $\ZZ/6 \times \mathfrak{S}_3$.

It is then the unique $3$-elementary non-cyclic subgroup of $\ZZ/6 \times \mathfrak{S}_3$.
\end{proof}

%
%

\begin{prop} \label{subgroupse} $ $
The only $3$-elementary non-cyclic subgroup of $\Aut(S_e)$ is $\langle a,b \rangle$.
\end{prop}


%
%
%

\begin{proof} $ $
Direct computations show that $a$ and $b$ define automorphisms of $S_e$ that generate a subgroup isomorphic to $(\ZZ/3)^2$. 

According to \cite[Theorem 6.8]{martin2}, the automorphism group of $S_e$ is given by $\Aut(S_e) \simeq \mathbb{Z}/6 \times \mathfrak{S}_3$. By Lemma \ref{finitegroup1bis}, $\Aut(S_e)$ contains a unique $3$-elementary non-cyclic subgroup. Since $\langle a,b \rangle$ is such a subgroup, we conclude that it is the only one.
\end{proof}

\subsection{The smoothness of the surface $S_e$} $ $

Let $e \in \kk$.

\begin{lemme} \label{smoothnessse} $ $
The surface $S_e$ is smooth if and only if $e \in \kk \setminus \left\{0,1\right\}$.

\end{lemme}

\begin{proof} $ $
According to \cite[Proposition 5.2]{martin2}, the surface $S_e$ is not smooth if and only if there exists $(u,v,x,y) \in \kk \setminus \left\{(0,0,0,0)\right\}$ such that:

$$\left\{
    \begin{array}{rll}
         uv(u+v)& = & 0 \\
        x^2 & = & 0 \\
v^2 y + (e+\sqrt{e})(u^4 v + v^5) + e u^2 v^3 & =& 0 \\
u^2 y +  (e+\sqrt{e})(v^4 u + u^5) + e v^2 u^3 & = & 0 \\
y^2 + uv(u+v) y + x^3 + (e + \sqrt{e}) (u^5 v + u v^5) + e u^3 v^3  & =& 0
    \end{array}
\right.$$

We distinguish three cases depending on the solutions of the first equation. 
\begin{enumerate}
\item First case: $u = 0$:

$$\left\{
    \begin{array}{rll}
      x & = & 0 \\
v^2 y + (e+\sqrt{e})v^5  & =& 0 \\
0 & = & 0 \\
y^2  & =& 0
    \end{array}
\right.$$

Hence $x = y = 0$ and $(e+\sqrt{e}) v = 0$.
Observe that $e+\sqrt{e} = 0$ if and only if $e \in \{0,1 \}$. Hence the result in this case.

\item Second case: $v = 0$:
Similar to the first case by symmetry between the variables $u$ and $v$.

\item Third case: $u = v$:

$$\left\{
    \begin{array}{rll}
        x & = & 0 \\
u^2 y  + e u^5  & =& 0 \\
u^2 y  + e u^5 & = & 0 \\
y^2  + e u^6 & =& 0
    \end{array}
\right.$$

We distinguish two subcases depending on the solutions of the second equation $u^2 y  + e u^5   = 0$.
\begin{itemize}
\item If $u=0$:

$$\left\{
    \begin{array}{rll}
        x & = & 0 \\
0 & =& 0 \\
0 & =& 0 \\
y^2  & =& 0
    \end{array}
\right.$$
Hence $x=y=u=v=0$, so there is no non-trivial solution.

\item If $y = e u^3$:

$$\left\{
    \begin{array}{rll}
        x & = & 0 \\
0 & =& 0 \\
0 & = & 0 \\
(e^2  + e) u^6 & =& 0
    \end{array}
\right.$$
Hence we get a non-trivial solution in this subcase if and only if we have $e^2 + e = 0$, \textit{i.e.} $e \in \{0,1\}$.
\end{itemize}
\end{enumerate}
\end{proof}

\subsection{The Picard rank of the subgroup $G_{12}'$ acting on $S_e$}  $ $

Let $e \in \kk \setminus \left\{0,1\right\}$.

 \begin{lemme} \label{rkg12prime} $ $
We have: $rk(Pic(S_e)^{G_{12}'})=1$.
 \end{lemme}
 
 \begin{proof} $ $
 The following table gives us the Euler characteristic $\chi^g$ of the fixed point variety $X^g$ of the elements $g \in G_{12}'$ acting on $S_e$:

 \newpage
 

\tiny
\begin{table}[h]
\centering
\renewcommand{\arraystretch}{2.0} 
\begin{tabular}{| c | M{5cm}| c | c |} 
 \hline 
  element & fixed points & nature & $\chi$  \\
  \hline
  id &  &  & 11   \\
  \hline
  $a , a^2$ & $[0:0:1:1] \sqcup (S_e \cap \{x=0\})$  & 1 pts and a curve of genus $\alpha$ & $3 - 2 \alpha $ \\
 \hline  
   $b, b^2$ & $[0:0:1:1] \sqcup$ \newline $\{ [1:j^n : 0 :y] \mid n \in \{1,2\} , 0 = y^2 + y +\sqrt{e} \}$ & 5 pts & 7 \\
 \hline  
   $ab , a^2 b^2 $ & $\{ [1:j : 0 :y] \mid 0 = y^2 + y +\sqrt{e} \} \sqcup$ \newline $\{ [u:j^2 u :x :y] \mid 0 = y^2 + y u^3 + x^3 + \sqrt{e} u^6 \}$ & 2 pts and a curve of genus $\beta$ & $4-2 \beta$ \\
 \hline  
    $a^2 b , ab^2 $ & $\{ [1:j^2 : 0 :y] \mid 0 = y^2 + y +\sqrt{e} \} \sqcup$ \newline $\{ [u:j u :x :y] \mid 0 = y^2 + y u^3 + x^3 + \sqrt{e} u^6 \}$& 2 pts and a curve of genus $\beta$ & $4-2 \beta$ \\
 \hline  
 \end{tabular}
\end{table}
\normalsize
Where $\alpha, \beta, \gamma \in \mathbb{N}_{\geq 0}$.

According to Proposition \ref{rankformula}, we have the following formula for the Picard rank: $$\rank \Pic(X)^G + 2 = \frac{1}{|G|} \sum_{g \in G} \chi(X^g)$$

Hence for $G_{12}'$ we get:
$$rk(Pic(S_e)^{G_{12}'}) = \dfrac{1}{9}(9+2 \cdot (1-2\alpha)+2\cdot 5 + 4 \cdot(2-2\beta)) = 1 + \dfrac{20 - 4 \cdot (\alpha + 2 \beta)}{9}$$
 Since $\alpha+2\beta \geq 0$ and $rk(Pic(S_e)^{G_{12}'}) \geq 1$ we get: $\alpha + 2 \beta \leq 5$.
 Then using that  $\alpha+2\beta \in \mathbb{N}_{\geq 0}$ and $rk(Pic(S_e)^{G_{12}'}) \in \mathbb{N}_{\geq 1}$, we get that the only possibility is $\alpha + 2 \beta = 5$ and $rk(Pic(S_e)^{G_{12}'}) = 1$.
 \end{proof}

\newpage

\part{Conjugation by birational maps of $p$-elementary subgroups}  
\label{part4}
$ $

\bigskip\bigskip

The work in the previous parts has allowed us to give a list of representatives of $p$-elementary non-cyclic subgroups of the Cremona group. Now we would like to study the possible conjugations by birational maps between these representatives.
The goal of this part is to give the necessary results to prove Theorem \ref{theorem2}.
The four sections of this part are almost independent:

\begin{itemize}
\item

In Section \ref{section41} we study a first easy case: The possible conjugations of a subgroup of automorphisms of a del Pezzo surface of degree $1, 2$ or $3$ with Picard rank equal to $1$. We show that a conjugacy to a minimal pair can always be replaced by an isomorphism.
\item

In Section \ref{section42} and \ref{section43} we do a similar study as in Section \ref{section41} but for quartic surfaces. This requires more work, and the conclusion differs: Klein subgroups of automorphisms of quartic del Pezzo surface with Picard rank equal to $1$ are conjugate to de Jonquières subgroups (Proposition \ref{conjcubickleinblue}).
\item

Finally in Section \ref{section44} we study the conjugation between de Jonquières subgroups. This is used to show the 
\hyperlink{Amazing}{\textbf{\textit{amazing result}}} (see Subsection \ref{amazingandsinequanon})
\end{itemize}

\bigskip\bigskip

We start our study by defining links of type $(II,\mathbb{D})$ and $(II,\mathbb{C})$:
\begin{defi} \label{sublinks} $ $

\begin{enumerate}
\item A link of type $(II,\mathbb{D})$ is a link of type $II$ between two del Pezzo surfaces.
\item A link of type $(II,\mathbb{C})$ is a link of type $II$ between two rational surfaces with conic bundle structures.
\end{enumerate}
\end{defi}

The reader can refer to \cite[Theorem 2.6]{isk} for more details about these links.

\newpage

\section{Subgroups of automorphisms of del Pezzo surfaces of low degree} \label{section41}
$ $

In this first section we study the case of a subgroup of automorphisms of a del Pezzo surface of degree $1, 2$ or $3$ with Picard rank equal to $1$.
The result of this section, Proposition \ref{conjdplowdeg}, provides a key step in the proof of Proposition \ref{lemmetheorem2}.
\bigskip

\begin{prop} \label{conjdplowdeg} $ $ 

Let $S$ be a del Pezzo surface of degree 1, 2 or 3. Let $G$ be a subgroup of automorphism of $S$. We assume $rk(Pic(S)^G) = 1$. Let $(G',S')$ be a minimal pair.
We assume $(G,S), (G',S')$ represent the same conjugacy class in the Cremona group.
 Then there exists an isomorphism $\Phi : S \rightarrow S'$ such that $G' = \Phi G \Phi^{-1}$.
%
\end{prop}

\bigskip

\begin{proof} $ $

Let $\sigma : S \DashedArrow S'$ be a birational map such that $G' = \sigma \circ G \circ \sigma^{-1}$.
Using  \cite[Theorem 2.5]{isk}, we decompose $\sigma = \sigma_n \circ \dots \circ \sigma_1$ where $\sigma_k : S_{k-1} \DashedArrow  S_k$ are elementary links, and $S_0 =S$, $S_{k} = S'$. Let $G_k = \sigma_k \circ \dots \circ \sigma_1 G \sigma_1^{-1} \circ \dots \circ \sigma_k^{-1} \subset \Aut(S_k)$. We distinguish two cases:
\\

If $K_S^2 = 1$, according to \cite[Theorem 2.6]{isk}, there is no possible Sarkisov link for $\sigma_1$.

Now we assume $K_S^2 \in \left\{ 2,3 \right\}$, according to \cite[Theorem 2.6]{isk}, the only possible Sarkisov link for $\sigma_1$ is a link of type $II$ of the following form: $$\sigma_1 =  \phi \circ \pi \circ \tau \circ \pi^{-1}  : S_0 \DashedArrow S_1,$$ where: 

$\left\{
    \begin{array}{l}
     \tau : X \xrightarrow{\sim} X \mbox{ is a Bertini or Geiser involution} \\
X \mbox{ is a del Pezzo surface of degree 1 or 2} \\
\pi : X \twoheadrightarrow S_0 \mbox{ is a blowdown} \\
\phi : S_0 \rightarrow S_1 \mbox{ is an isomorphism}
    \end{array}
\right.$

Then:

$\begin{array}{llr}
G_1 & = \sigma_1 \circ G_0 \circ \sigma_1^{-1} \\ & = \phi \circ \pi \circ \tau \circ \pi^{-1} G_0 \circ \pi \circ \tau \circ \pi^{-1} \phi^{-1}& \\
& = \phi \circ \pi \circ   \pi^{-1} G_0 \circ \pi \circ \tau \circ \tau \circ \pi^{-1} \circ \phi^{-1} & \mbox{ \quad because $\pi^{-1} \circ G_0 \circ \pi \subset \Aut(X)$ }\\
&&  \mbox{      and $\tau$ commutes with $\Aut(X)$} \\
& & \mbox{because of Lemma \ref{bertinigeiser}} \\
& = \phi \circ \pi \circ   \pi^{-1} G_0 \circ \pi \circ \pi^{-1} \circ \phi^{-1} & \\
& = \phi G_0 \phi^{-1}&
\end{array}
$

So  $\phi : S_0 \rightarrow S_1$ is an isomorphism such that $\phi G_0 \phi^{-1} = G_1$.
\\

Hence, we can replace $S_0$ and $G_0$ by $S_1$ and $G_1$.
Assuming $n \geq 2$, we can apply the previous step to the next link $\sigma_2$. Hence by a finite induction we get that we can replace $\sigma$ by an isomorphism, hence the result.
\end{proof}

\newpage



\section{Non-Klein subgroups of automorphisms of quartic surfaces} \label{section42} $ $


This section proves in Proposition \ref{conjofbig2elemquartic} that any conjugacy of a $2$-elementary non-cyclic non-Klein subgroup (\textit{i.e.} isomorphic to $(\mathbb{Z}/2)^r$ where $r > 2$) of automorphisms of a quartic del Pezzo surface with Picard rank equal to $1$, to a minimal pair, can be replaced by an isomorphism. Note that Lemma \ref{replacinglinkbyiso} is invoked in both the present section and Section \ref{section44}.
We assume $\car(\kk) \neq 2, p=2, r > 2$.

\begin{lemme} \label{nofixedpoint2elemquartic} $ $
Let $G \simeq (\mathbb{Z}/2)^r$ be a subgroup of automorphisms a quartic del Pezzo surface $S$.
We assume $rk(Pic(S)^G) = 1$ and $r > 2$.
Then $G$ has no fixed point.
\end{lemme}

\begin{proof} $ $
Using Proposition \ref{quarticresult} we can assume: $$S = \left\{ \sum_{n=0}^4 X_n^2 = \sum_{n=0}^4 a_n X_n^2 = 0 \right\} \subset \mathbb{P}^4$$ where $a_0, \dots , a_4 \in \kk$ are pairwise distinct
and $G = \{ [\pm 1 : \pm 1 : \pm 1 : 1 : 1 ]\}$ or $G = \{ [\pm 1 : \pm 1 : \pm 1 : \pm 1 : \pm 1 ]\}$.
Since the first group is a subgroup of the second, it is enough to show that $\{ [\pm 1 : \pm 1 : \pm 1 : 1 : 1 ]\}$ has no fixed point.
We observe that the set of fixed points of $[-1:1:1:1:1]$ is $\{[0:X_1:X_2:X_3:X_4] \in S \}$.
It implies that if $X$ is a fixed point of $\{ [\pm 1 : \pm 1 : \pm 1 : 1 : 1 ]\}$ then it is of the form $[0:0:0:X_3:X_4]$. Since $a_3 \neq a_4$, this is not a point of the surface.
\end{proof}

\begin{lemme} \label{replacinglinkbyiso} $ $
Let $S$ be a quartic del Pezzo surface. Let $G$ be a $2$-elementary subgroup of automorphisms of $S$. We assume $rk(Pic(S)^G) = 1$.
Let $S'$ be a del Pezzo surface and $G'$ be a $2$-elementary subgroup of automorphisms of $S'$.
We assume there exists $\sigma$ a link of type $(II,\mathbb{D})$ (see Definition \ref{sublinks}) from $(G,S)$ to $(G',S')$.
Then there exists $\phi : S \rightarrow S'$ an isomorphism such that $\phi G \phi^{-1} = G'$.
\end{lemme}

\begin{proof} $ $
According to \cite[Theorem 2.6]{isk}
we have $\sigma = \phi \circ \pi \circ \tau \circ \pi^{-1}  : S \DashedArrow S'$ where $\tau : X \DashedArrow X$ is a Bertini or Geiser involution, $X$ is a del Pezzo surface of degree $1$ or $2$, $\pi : X \twoheadrightarrow S$ is a blowdown, and $\phi : S \rightarrow S'$ is an isomorphism.
Then:

$
\begin{array}{llr}
G' & = \sigma \circ G \circ \sigma^{-1} \\
& = \phi \circ \pi \circ \tau \circ \pi^{-1} G \circ \pi \circ \tau \circ \pi^{-1} \phi^{-1}& \\
& = \phi \circ \pi \circ   \pi^{-1} G \circ \pi  \circ \tau \circ \tau \circ \pi^{-1} \circ \phi^{-1}& \mbox{ \quad because $\pi^{-1} \circ G \circ \pi \subset \Aut(X)$}\\
& & \mbox{ and $\tau$ commutes with $\Aut(X)$ (Lemma \ref{bertinigeiser})} \\
& = \phi \circ \pi \circ   \pi^{-1} G \circ \pi \circ \pi^{-1} \circ \phi^{-1} & \\
& = \phi G \phi^{-1}&
\end{array}
$

\end{proof}

\begin{prop} \label{conjofbig2elemquartic} $ $
Let $(G,S)$ and $(G',S')$ be two minimal pairs such that $G \simeq G' \simeq (\mathbb{Z}/2)^r$, $rk(Pic(S)^G) = 1$, $r \in \mathbb{N}$, and $S$ is a quartic del Pezzo surface. We assume they represent the same conjugacy class in the Cremona group.
If $r > 2$,
then there exists an isomorphism $\Phi : S \rightarrow S'$ such that $\Phi G \Phi^{-1} = G'$.
\end{prop}

\begin{proof} $ $
Let $\sigma : S \DashedArrow S'$ be a birational map such that $G' = \sigma \circ G \circ \sigma^{-1}$.
Using  \cite[Theorem 2.5]{isk}, we decompose $\sigma = \sigma_n \circ \dots \circ \sigma_1$ where $\sigma_k : S_{k-1} \DashedArrow  S_k$ are elementary links, and $S_0 =S$, $S_{n} = S'$. Let $G_k = \sigma_k \circ \dots \circ \sigma_1 G \sigma_1^{-1} \circ \dots \circ \sigma_k^{-1} \subset \Aut(S_k)$.
Using \cite[Theorem 2.6]{isk}, there are two possibilities for $\sigma_1$:

\begin{itemize}[label=\textbullet]
\item $\sigma_1$ is a link of type $I$, $S_1$ is a cubic surface and $rk(Pic(S_1)^{G_1} )=2$. In particular $\sigma_1$ is the blowup of a fixed point of $G$. But according to Lemma \ref{nofixedpoint2elemquartic} ($r > 2$), $G$ has no fixed point. Hence this case is not possible.
\item $\sigma_1$ is a link of type $(II,\mathbb{D})$ (see Definition \ref{sublinks}).
According to Lemma \ref{replacinglinkbyiso}, $\sigma_1$ can be replaced by an isomorphism, and the previous step applies again for the next Sarkisov link $\sigma_2$ (assuming $n \geq 2)$.
\end{itemize}
In conclusion, $\sigma$ can be replaced by an isomorphism and we get our result.
\end{proof}

\newpage


\section{Klein subgroups of automorphisms of quartic surfaces} \label{section43} $ $

We assume $\car(\kk) \neq 2, p=2, r = 2$ in this section.
The study will differ from the previous section because in this specific case of Klein subgroups, we have fixed points, which create the possibility for a new Sarkisov link: The blow-up of a fixed point in a quartic surface.
The result of this section is given in Proposition \ref{conjcubickleinblue}.
The Lemma \ref{conjcubickleinredbis} will not be used in the current section but is used in the next Section \ref{section44}.

\subsection{Invariant conic bundles on del Pezzo surfaces} \label{loremsubsectiontocbechanged}

\begin{lemme} \label{433adapted} $ $

Let $S$ be a del Pezzo surface and $G \subset \Aut(S)$ a finite subgroup.
We assume $rk(Pic(S)^G) = 2$ and $S$ admits at least two $G$-invariant conic bundles.
Then $S \simeq \mathbb{P}^1 \times \mathbb{P}^1$ or $K_S^2 \in \{1,2,4\}$ and $S$ admits exactly two $G$-invariant conic bundles.
\end{lemme}

\begin{proof} $ $

Let $\pi, \pi' : S \rightarrow \mathbb{P}^1$ be two distinct $G$-invariant conic bundles. Denote by $f,f'$ the divisor class of a fibre of $\pi$ and $\pi'$ respectively. Then $Pic(S)^G$ is generated by $-K_S$ and $f$ over $\mathbb{Q}$. Now we proceed as in the proof of \cite[Proposition 4.3.3]{blancthesis}:

We write $f' = -a K_S + b f$, with $a,b \in \mathbb{Q}, a \neq 0$. We know that $f'^2 = 0$, whence $f' . K_S = -2$ (by the adjunction formula). These two relations imply that:
\begin{center}
$
\begin{array}{rcl}
a^2 K_S^2 + 4 ab = a (a K_S^2 + 4 b) & = & 0 \\
- a K_S^2 - 2 b & = & -2 
\end{array}
$
\end{center}
As $a \neq 0$, the first equation implies that $a K_S^2 = -4b$. This and the second equation imply that $2b = -2$, so $b = -1$. Thus, we find that $f' = -a K_S -f$ and $a K_S^2 = 4$.
If $S \neq \mathbb{P}^1 \times \mathbb{P}^1$ and $S \neq \mathbb{P}^2$, the canonical divisor is not a multiple in $Pic(S)$, so $a \in \mathbb{Z}$ and the degree $K_S^2$ of the del Pezzo surface is equal to $1, 2$ or $4$.
\end{proof}

\begin{remark} $ $
Let $(G, T)$ be a pair, where $T$ is a cubic del Pezzo surface with $rk(Pic(T)^G) = 2$. Then $T$ has exactly one $G$-invariant conic bundle by Lemma \ref{433adapted}. Hence we may speak of \textit{the} $G$-invariant conic bundle on $T$. This will be used in the following Lemma \ref{conjcubickleinredbis}.
\end{remark}

\begin{lemme} \label{conjcubickleinredbis} $ $

Let $G \simeq (\mathbb{Z}/2)^2$ be a Klein subgroup of automorphisms of a quartic del Pezzo surface $S$.
We assume $rk(Pic(S)^G) = 1$.
We assume there exist two Sarkisov links:
\begin{itemize}
\item $\sigma_1$ a link of type $III$ from $(G_1,S_1)$ to $(G,S)$ where $S_1$ is a cubic del Pezzo surface.
\item $\sigma_2$ a link of type $I$ from $(G,S)$ to $(G_2,S_2)$ where $S_2$ is a cubic del Pezzo surface.
\end{itemize}

Then the composition $\sigma_2 \circ \sigma_1$ can be replaced by an isomorphism, \textit{i.e.} there exists an isomorphism $\phi: S_1 \rightarrow S_2$ that conjugates $G_1$ to $G_2$ and preserves the conic bundle structures of $S_1$ and $S_2$.
\end{lemme}

\begin{proof} $ $

We observe that
$\sigma_1^{-1}$ and $\sigma_2$ are both blowup of a fixed point of $G$ in $S$.
According to Lemma \ref{433adapted}, the conic bundle is unique on $S_i$ for $i \in \{1,2\}$, as $K_{S_i}^2=3$.
According to Lemma \ref{conjcubickleinred} Assertion \ref{conjcubickleinred2}, up to isomorphism we can assume that these two fixed points are the same. Hence the composition $\sigma_2 \circ \sigma_1$ can be replaced by an isomorphism preserving the $G$-invariant conic bundles.
\end{proof}

\begin{remark} $ $
While Lemma~\ref{433adapted} is employed in this section, it is not strictly necessary for the proof of the main theorems. Its role here is limited to establishing Assertion~\ref{conjcubickleinblue4} of Proposition~\ref{conjcubickleinblue}, whereas only Assertion~\ref{conjcubickleinblue1} is required for our primary results. The core utility of Lemma~\ref{433adapted} lies in the following Section \ref{section44}, specifically through its consequence in Lemma~\ref{conjcubickleinredbis}.
\end{remark}

\newpage

\subsection{Characterization of the smoothness of quartic surfaces}

\begin{lemme} \label{section11smooth} $ $
Let $a_0, \dots , a_4 \in \kk$.
The following quartic surface: \[ \left\{ \sum_{i=0}^4 X_i^2 = \sum_{i=0}^4 a_i X_i^2 = 0\right\} \subset \mathbb{P}^4 \] is smooth if and only if the scalars $a_0, \dots , a_4$ are pairwise distinct.
\end{lemme}

\begin{proof} $ $

$\Rightarrow:$ If the scalars $a_0, \dots , a_4$ are not pairwise distinct, we can assume $a_0 = a_1$ without loss of generality. Then $[1:i:0:0:0]$ is a singular point of the surface.
\\

$\Leftarrow:$ Now we assume that the scalars $a_0, \dots , a_4$ are pairwise distinct and we will prove that the surface is smooth.
We define $f = \sum_{i=0}^4 X_i^2$ and $g = \sum_{i=0}^4 a_i X_i^2$.
Let $p = [x_0: … : x_4] \in \left\{ \sum_{i=0}^4 X_i^2 = \sum_{i=0}^4 a_i X_i^2 = 0\right\}$.
Then there exist $i \neq j$ such that $x_i$ and $x_j$ are both non-zero. The gradient of $f$ is equal to $(x_0, …, x_4)$ and that one of $g$ is $(a_0 x_0, …, a_4 x_4)$ (up to scalars (recall that $\car(\kk) \neq 2$) ). Since $x_i, x_j$ are both non-zero and $a_i \neq a_j$ for $i \neq j$, we get that both gradients are linearly independent. This implies that the tangent spaces $T_p {f=0}$ and $T_p {g=0}$ are distinct, \textit{i.e.} ${f=0}$ and ${g=0}$ intersect transversally at $p$. This shows that the complete intersection given by $f=g=0$ is smooth at $p$.
\end{proof}

\subsection{Fixed points of Klein subgroups of automorphism of quartic del Pezzo surfaces}

\begin{lemme} \label{conjcubickleinred} $ $
Let $G \simeq (\mathbb{Z}/2)^2$ be a Klein subgroup of automorphisms of a quartic del Pezzo surface $S$.
We assume $rk(Pic(S)^G) = 1$.
Then:
\begin{enumerate}
\item \label{conjcubickleinred1} The subgroup $G$ has exactly four fixed points in $S$.
\item \label{conjcubickleinred2} For any two given fixed points, there exists an automorphism of $S$ preserving the subgroup $G$ permuting these two fixed points.
\item \label{conjcubickleinred3} None of the fixed point of $G$ belong in one of the sixteen lines of $S$.
\end{enumerate}
\end{lemme}

\begin{proof} $ $
Using Proposition \ref{quarticresult} we can assume: $$S = \left\{ \sum_{n=0}^4 X_n^2 = \sum_{n=0}^4 a_n X_n^2 = 0 \right\} \subset \mathbb{P}^4$$ where $a_0, \dots , a_4 \in \kk$ are pairwise distinct
and $G = \{ [\pm 1 : \pm 1 : 1 : 1 : 1 ]\}$.

\begin{enumerate}
\item 
The fixed points of this group in $S$ are the points of the form $[0:0:X_2:X_3:X_4]$ where $\sum_{n=2}^4 X_n^2 = \sum_{n=2}^4 a_n X_n^2 = 0$. Given that the surface is smooth, we have that $a_2, a_3, a_4$ are pairwise distinct (Lemma \ref{section11smooth}). 
Hence there exists $(x,y,z) \in {\kk^\times}^3 $ such that the fixed points are exactly the following:
$$[0:0:\pm x : \pm y : \pm z]$$
\item 
Now for the result about the permutation of fixed points by automorphisms preserving the subgroup, we observe that for any $\epsilon, \mu, \rho, \nu = \pm 1$, the map  $[X_0:X_1:X_2:X_3:X_4]  \mapsto [X_0 : X_1 : \epsilon \rho X_2 : \mu \nu X_3 : X_4]$ is an automorphism of $S$ preserving the subgroup $G$ and permuting the two fixed points  $[0:0:\epsilon x:\mu  y :z]$ and $[0:0:\rho x:\nu y:z]$.
\item In order to prove this result, we will conjugate our group and surface in order that the fixed point is $[0:0:0:0:1]$. We define the following automorphism of $\mathbb{P}^4$: $$\quad\quad \Phi : [X_0:X_1:X_2:X_3:X_4] \mapsto [X_0:X_1:X_2 + x X_4:X_3+yX_4:zX_4]$$
The inverse of $\Phi$ is of the form:
$$\quad\quad \Phi^{-1} : [X_0:X_1:X_2:X_3:X_4] \mapsto [X_0:X_1:X_2 + r X_4:X_3+sX_4:tX_4],$$
where $r = - \frac{x}{z},s = - \frac{y}{z},t = \frac{1}{z} \in \kk^\times$.
By construction $\Phi^{-1}$ sends the fixed point $[0:0:x:y:z]$ on $[0:0:0:0:1]$.
We have:  $\Phi^{-1}(S) = \{  \sum_{n=0}^3 X_n^2  + (x^2 + y^2 + z^2) X_4^2 + 2 x X_2 X_4 + 2 y  X_3 X_4 = \sum_{n=0}^3 a_n X_n^2 + (x^2 a_2 + y^2 a_3 + z^2 a_4) X_4^2 + 2 x a_2 X_2 X_4 + 2 y a_3 X_3 X_4  = 0 \}$.
But we have: $$x^2 + y^2 + z^2 = a_2 x^2 + a_3 y^2 + a_4 z^2 = 0.$$ Hence the equations can be simplified and we have:
$$\Phi^{-1}(S) = \left\{  \sum_{n=0}^3 X_n^2  + 2 x X_2 X_4 + 2 y  X_3 X_4 =\sum_{n=0}^3 a_n X_n^2 + 2 x a_2 X_2 X_4 + 2 y a_3 X_3 X_4  = 0 \right\}$$
By contradiction we assume there exists a line in $\Phi^{-1}(S)$ through the fixed point $p = [0:0:0:0:1]$.
So there exists a point $q = [q_0:q_1:q_2:q_3:q_4] \in S \setminus \{ p \}$ such that for every $u,v \in \kk, [v q_0:v q_1:v q_2:v q_3:u + v q_4] \in S$, \textit{i.e.}
\begin{equation*}
\begin{array}{rcl}
(q_0^2 + q_1^2 + q_2^2 + q_3^2 + 2 x q_2 q_4 + 2 y q_3 q_4)v^2 + (2 x q_2 + 2 y q_3) uv  &= & 0 \\  (a_0 q_0^2 + a_1 q_1^2 + a_2 q_2^2 + a_3 q_3^2 + 2 x a_2 q_2 q_4 + 2 y a_3 q_3 q_4)v^2 + (2 x a_2 q_2 + 2 y a_3 q_3) uv  & = &  0
\end{array}
\end{equation*}
By isolating the term in $uv$ we get: 
\begin{equation*}
\begin{array}{rcl}
    2xq_2 + 2yq_3 & = & 0 \\
    2xa_2q_2 + 2ya_3q_3 & = & 0
\end{array}
\end{equation*}
This is a $2\times2$ linear system with determinant $4 x y (a_3 - a_2) \neq 0$ ($\car(\kk) \neq 2$).
Hence $q_2 = q_3 = 0$.
Since $q \in S$ we get: 
\begin{equation*}
\begin{array}{rcl}q_0^2 + q_1^2 & =& 0 \\ a_0 q_0^2 + a_1 q_1^2 & =& 0
\end{array}
\end{equation*}
This is $2 \times 2$ linear system with unknowns $(q_0^2,q_1^2)$ and determinant equal to $a_1 - a_0 \neq 0$.
Hence $q_0 = q_1 = 0$.
Then $q = [0:0:0:0:1] = p$, we get our contradiction.
\end{enumerate}
\end{proof}

\subsection{Conjugation of Klein subgroups of automorphisms of quartic del Pezzo surfaces of Picard rank $1$ to de Jonquières subgroups} $ $

We will only use the first assertion of the following Proposition \ref{conjcubickleinblue}.

\begin{prop} \label{conjcubickleinblue} $ $
Let $a_0, \dots , a_4 \in \kk$ be pairwise distinct. Let:
$$S = \left\{ \sum_{n=0}^4 X_n^2 = \sum_{n=0}^4 a_n X_n^2 = 0 \right\} \subset \mathbb{P}^4$$
$$G = \{ [\pm 1 : \pm 1 : 1 : 1 : 1 ]\} \subset \Aut(S)$$

Then there exists $(G', S’)$ such that:
\begin{enumerate}
\item \label{conjcubickleinblue1} $(G,S)$ is conjugate to $(G', S’)$, $S’$ is a cubic del Pezzo surface, and we have: $$rk(Pic(S’)^{G'})=2$$
\item \label{conjcubickleinblue4} $S'$ has a unique $G'$-invariant conic bundle structure.
\end{enumerate}
\end{prop}

\begin{proof} $ $

Assertion \ref{conjcubickleinblue1}: 
We first observe that such a surface is smooth (Lemma \ref{section11smooth}).
The points $[0:0:X_2:X_3:X_4]$ where $\sum_{n=2}^4 X_n^2 = \sum_{n=2}^4 a_n X_n^2 = 0$ are fixed points of $G$. Since none of the fixed point of $G$ belongs to a line of $S$ (Lemma \ref{conjcubickleinred} Assertion \ref{conjcubickleinred3}) and since $rk(Pic(S)^G)=1$ (Proposition \ref{quarticpicard}), when blowing-up such a fixed point we get a subgroup of automorphisms of a smooth cubic surface with Picard rank equal to $2$.

Assertion \ref{conjcubickleinblue4}:
Furthermore, the blow-up of a fixed point is a $G'$-invariant line of the cubic surface and the projection from this line gives a conic bundle, and this one is unique by Lemma \ref{433adapted}.
\end{proof}

%
%
%

\newpage

\section{Conjugation of subgroups of the de Jonqui\`eres group} \label{section44} $ $

We assume $\car(\kk) \neq 2$ in this section.
The main result of this section is Proposition \ref{2elemconj2}.
We will use Lemma \ref{conjcubickleinredbis} from the previous section.

\bigskip

\subsection{A bound of the surface's degree depending on the nature of the fixed locus of non-trivial involutions}

\begin{defi} $ $
Let $(S,\pi)$ be a conic bundle. We define the subgroup of automorphisms of $S$ that leave the conic bundle structure invariant:
$$\Aut(S,\pi) = \left\{ \phi \in \Aut(S) | \pi \circ \phi = \pi \right\} \subset \Aut(S)$$
\end{defi}

\begin{prop} \label{preli} $ $
Let $(S,\pi)$ be a conic bundle.

Let $h \in \Aut(S,\pi)$ be a non-trivial involution and assume that $h$ fixes point wise an irreducible smooth curve of genus $g>0$.
Then $K_S^2 \leq 6 - 2 g$.
\end{prop}

\begin{proof} $ $
Let $\Gamma$ be the fixed curve.
$\pi(\Gamma)$ is a closed subset of $\mathbb{P}^1$. By contradiction, if $\pi(\Gamma) \neq \mathbb{P}^1$ then we would have $\Gamma \subset F$ where $F$ is a fibre. Then we would have that $\Gamma$ contains a rational curve, hence our contradiction. So $\pi(\Gamma) = \mathbb{P}^1$.
So for a general fibre $F \subset S$, $F \cap \Gamma$ is non empty and is finite. As $h \in \Aut(S,\pi)$, the image of a fibre is a fibre, hence $h(F)$ is a fibre, but $F$ intersects the curve $\Gamma$, so the points of $F \cap \Gamma$ are fixed, and then $h(F)$ is equal to $F$.
We get a non-trivial involution on $F \subset \mathbb{P}^1$ , which then has two fixed points (because $\car(\kk) \neq 2$). This proves that $\Gamma$ intersects a general fibre into one or two points. The first case is impossible as $g>0$, so $F \cap \Gamma$ consists of two points for a general fibre $F$.
\newline
Locally, $\pi : S \rightarrow \mathbb{P}^1$ at a smooth fibre looks like $B \times \mathbb{P}^1 \rightarrow B$. We have:

$\begin{array}{ccccc}
h & : & B \times \mathbb{P}^1 & \to & B \times \mathbb{P}^1  \\
 & & (b,x) & \mapsto & (b,A(b)x) = (b,[a_{11}(b)x_0 + a_{12}(b)x_1 :a_{21}(b)x_0 + a_{22}(b)x_1]) \\
\end{array}$

\noindent
The fixed locus of $h$ in $B \times \mathbb{P}^1$ is given by the following equation:
$$(a_{11}(b)x_0 + a_{12}(b)x_1) x_1 - (a_{21}(b)x_0 + a_{22}(b)x_1) x_0 = 0$$
This is a hypersurface in $B \times \mathbb{P}^1$, hence each irreducible component is of dimension $1$. 
In particular $\Gamma’ := (B \times \PP^1) \cap \Gamma$ is of dimension $1$.
Because of \cite[Proposition A.8.10]{reductive}, the fixed point locus is smooth, hence no fibre of $\pi$ is fixed point wise.
Since $\pi |_{\Gamma} : \Gamma \rightarrow \PP^1$ has exactly two points in all but finitely many fibres and since every non-trivial involution of $\PP^1$ has exactly two fixed points ($\car(\kk) \neq 2$), we conclude that $\Gamma’$ is the fixed point locus of the action of $h$ on $B \times \PP^1$ and $\pi |_{\Gamma} : \Gamma \rightarrow \PP^1$ has no branch point in $B$.

\noindent Hence $\pi|_{\Gamma} : \Gamma \rightarrow \mathbb{P}^1$ is a surjective 2:1 morphism, and the only possible branching points correspond to singular fibres. According to the Riemann-Hurwitz formula, we have: $2g-2 = -4 + b$ where $b$ is the number of branching points.
$b$ is lower or equal than the number of singular fibres, which is $8-K_S^2$. Hence, $b \leq 8 - K_S^2$. Hence, $2g-2 \leq -4 + 8 - K_S^2$, hence the result.
\end{proof}

\bigskip\bigskip

\subsection{A change of coordinates for quartic del Pezzo surfaces}

\begin{defi} \label{defi443} $ $
Let $\xi \in \kk \setminus \{0,-1,1\}$.

We define the following quartic del Pezzo surface:
$$S^{\xi} =\left\{ (1+\xi)x_1^2 - x_3^2 - \xi x_4^2 + \xi (1+\xi) x_5^2= (1+\xi)x_2^2 - \xi x_3^2 -  x_4^2 + \xi (1+\xi) x_5^2 = 0 \right\}$$
\end{defi}
In the following Lemma \ref{conjuisoblanc} we use the notations from Definitions \ref{defi323} and \ref{defi443}.

\begin{lemme} \label{conjuisoblanc} $ $
Let $t \in \kk \setminus \{0,-1,1\}$.

Then there exists an isomorphism $\Phi : S^{\frac{1-t}{1+t}} \rightarrow S_t$ such that: $$\Phi^{-1} \langle [-1:1:1:1:1],s \rangle \Phi = \langle (x_1:x_2:x_3:x_4:-x_5),(x_2:x_1:x_4:x_3:-x_5) \rangle$$
\end{lemme}

\begin{proof} $ $
We define the following automorphism of $\mathbb{P}^4$:
$$\begin{array}{ccccc}
\Phi & : & \mathbb{P}^4 & \to & \mathbb{P}^4 \\
 & & [x_1 : x_2 : x_3 : x_4 : x_5] & \mapsto & [\sqrt{2 \xi} x_5 : x_1 : i x_3 : i x_4 : x_2] \\
\end{array}$$
By taking the sum and the difference of the equations of $S^{\xi}$, we observe that $S^{\xi}$ is the complete intersection of the following two equations:
$$\begin{array}{rrrlrrr}
(1+\xi) x_1^2 & + (1+\xi) x_2^2 & - (\xi +1) x_3^2 & - (\xi + 1) x_4 ^2  + 2 \xi (1+ \xi) x_5^2  & = &  0 & \\
(1+\xi) x_1^2 & - (1+\xi) x_2^2 & + (\xi - 1) x_3^2 & + (1 - \xi ) x_4 ^2   & =  & 0 &
\end{array}$$

Let $X = \Phi (x) = [X_0 : X_1 : X_2 : X_3 : X_4] = [\sqrt{2 \xi} x_5 : x_1 : i x_3 : i x_4 :x_2]$ for $x \in S^\xi$. Then:

$\begin{array}{ll} & X_0^2 + X_1^2 + X_2^2 + X_3^2 + X_4^2 \\
= & \xi x_5^2 + x_1^2 - x_3^2 - x_4^2 +x_2^2 \\
 = & \frac{1}{1+\xi} ((1+\xi) x_1^2 + (1+\xi) x_2^2 - (\xi +1) x_3^2 - (\xi + 1) x_4 ^2 + 2 \xi (1+ \xi) x_5^2) \\
 = & 0
\end{array}$ \newline
And: 

$\begin{array}{ll} & X_1^2 + \frac{1-\xi}{1+\xi} X_2^2 - \frac{1-\xi}{1+\xi} X_3^2 - X_4^2 \\
= & x_1^2 - \frac{1-\xi}{1+\xi} x_3^2 +\frac{1-\xi}{1+\xi} x_4^2 - x_2^2 \\
= & \frac{1}{1+\xi}((1+\xi) x_1^2 - (1+\xi) x_2^2 + (\xi - 1) x_3^2 + (1 - \xi ) x_4 ^2)  \\
= & 0
\end{array}$
\newline
So $X \in S_{\frac{1-\xi}{1+\xi}}$.
Hence we have: $\Phi(S^\xi) = S_{\frac{1-\xi}{1+\xi}}$.
The map $\xi \mapsto \frac{1-\xi}{1+\xi}$ is an involution of $\kk \setminus \{0,1,-1\}$, hence we have: $\Phi(S^\frac{1-t}{1+t}) = S_t$.

Finally we observe that we have the following equality:

 $\begin{array}{cl} & \langle [x_1:x_2:x_3:x_4:-x_5], [x_2:x_1:x_4:x_3:x_5] \rangle \\ = & \langle [x_1:x_2:x_3:x_4:-x_5], [x_2:x_1:x_4:x_3:-x_5] \rangle \end{array}$
\end{proof}

\bigskip

\subsection{Classification of $2$-elementary subgroups of automorphism of quartic del Pezzo surfaces of Picard rank $2$ and their actions on conic bundle structures}

\begin{lemme} \label{propquarticlink} $ $
We use the notations of Definitions \ref{defi323} and \ref{defi1623} .

Let $S$ be a quartic del Pezzo surface. Let $G$ be a $2$-elementary non-cyclic subgroup of $\Aut(S)$. We assume $rk(Pic(S)^G)=2$. Then:
\begin{enumerate} 
\item  \label{propquarticlink1} Up to isomorphism, there exist $a_0, \dots, a_4 \in \kk$ pairwise distinct such that: $$S = \left\{ X_0^2 + X_1^2 + X_2^2 + X_3^2 + X_4^2 = a_0 X_0^2 + a_1 X_1^2 + a_2 X_2^2 + a_3 X_3^2 + a_4 X_4^2 = 0 \right\}$$ and $G=G_{4n}$ where $n \in \left\{2 , 4 , 7 , 9 \right\}$
\item  \label{propquarticlink2} We can assume $S = S_t$ for some $t \in \kk \setminus \{ 0,1,-1\}$ if $n = 7,9$.

\item  \label{propquarticlink3} The subgroup $G$ fixes exactly two conic bundles structure of $S$.

\item  \label{propquarticlink4} If $n=2, 4, 7$, then there exists an automorphism of $S$ permuting the two conic bundles and belonging to the normalizer of $G$.

\item  \label{propquarticlink5} If $n = 9$, then $G$ acts
trivially on the base of exactly one of the two conic bundles fixed.
\end{enumerate}
\end{lemme}

\begin{proof} $ $

Assertion \ref{propquarticlink1}  and \ref{propquarticlink2}: This is a direct consequence of Proposition \ref{quarticpicard}.

Assertion \ref{propquarticlink3}: A quartic del Pezzo surface has exactly ten conic bundle structures, that can be separated into five pairs (
\cite[Lemma 9.11 Assertions 1 and 2]{blancarticle2}. Although the article is stated over the field of complex numbers, the proof works over any algebraically closed field of characteristic not $2$.). Furthermore, each element of the automorphism group sends a conic bundle on an other conic bundle. The elements that preserve the pairs are the elements of $Z$. Let $z = [z_0 X_0 : z_1 X_1 : z_2 X_2 : z_3 X_3 : z_4 X_4 ] \in Z$ where $z_n \in \{-1 , 1\}$ and $\prod_{i=0}^4 z_i=1$. Then $z$ exchanges the conic bundles of the $n$-th pair if and only if $z_n = -1$ (\cite[Lemma 9.11 Assertions 1 and 2]{blancarticle2}).
Because of Assertion \ref{propquarticlink1} of this Proposition, we can assume $G \in \{G_{42}, G_{44}, G_{47}, G_{49}\}$. In particular $G$ contains the automorphism $g = [1:-1:-1:-1:-1]$. This element fixes only the conic bundles of the first pair. Hence $G$ fixes at most two conic bundles.
Furthermore we observe that $(\langle g \rangle,S)$ is minimal for the automorphism $g = [1:-1:-1:-1:-1]$, hence $(G,S)$ is also minimal. Since $rk(Pic(S)^G)=2$, it implies that $G$ fixes at least one conic bundle. 
Finally, we observe that all elements of $G$ are involutions, hence $G$ fixes an even number of conic bundles. Hence $G$ fixes exactly two conic bundles: the two conic bundles of the first pair.

Assertion \ref{propquarticlink4}: The element $z = [-1:1:-1:-1:-1] \in Z$ belongs to the normalizer of $G$ and permutes the two conic bundles fixed by $G$ (\cite[Lemma 9.11 Assertions 1 and 2]{blancarticle2}).

Assertion \ref{propquarticlink5}: Using Lemma \ref{conjuisoblanc}, we can assume that up to isomorphism: $$S =\left\{ (1+\xi)x_1^2 - x_3^2 - \xi x_4^2 + \xi (1+\xi) x_5^2= (1+\xi)x_2^2 - \xi x_3^2 -  x_4^2 + \xi (1+\xi) x_5^2 = 0 \right\}$$
$$G= \langle (x_1:x_2:x_3:x_4:-x_5),(x_2:x_1:x_4:x_3:-x_5) \rangle$$
For some $\xi \in \kk \setminus \{ 0,1,-1\}$. Then we use \cite[Lemma 8.1.12]{blancthesis} (The thesis is stated over the field of complex numbers, but the proof works over any algebraically closed field of characteristic not $2$).
\end{proof}


\subsection{Conjugation between $2$-elementary subgroups of automorphism of quartic del Pezzo surfaces of Picard rank $2$}

\begin{lemme} \label{propquarticlinkbis} $ $
Let $S$ be a quartic del Pezzo surface.

Let $G$ be a $2$-elementary non-cyclic subgroup of $\Aut(S)$ such that $rk(Pic(S)^G)= 2$. If the two $G$-invariant conic bundles on $S$ (from Lemma \ref{propquarticlink} Assertion \ref{propquarticlink3}) are non-equivalent, then up to conjugation by an isomorphism we have:
$$S = \left\{ X_0^2 + X_1^2 + X_2^2 + X_3^2 + X_4^2 =   X_1^2 + t X_2^2 - t X_3^2 - X_4^2 = 0 \right\}$$
$$G = \langle [-1:1:1:1:1],[X_0:X_4:X_3:X_2:X_1] \rangle$$
where $t \in \kk \setminus \{-1,0,1\}$.
\end{lemme}

\begin{proof} $ $

We use the notations of Proposition \ref{subgroupsofSt} and Proposition \ref{quarticz} for the proof.
According to Lemma \ref{propquarticlink} Assertion \ref{propquarticlink1}
we can assume up to isomorphism:
\newline $G \in \{ G_{42}, G_{44}, G_{47}, G_{49} \}$, and if $G \in \{G_{47}, G_{49}  \}$: $$S = \left\{ X_0^2 + X_1^2 + X_2^2 + X_3^2 + X_4^2 =   X_1^2 + t X_2^2 - t X_3^2 - X_4^2 = 0 \right\}$$ where $t \in \kk \setminus \{-1,0,1\}$.
According to Lemma \ref{propquarticlink} Assertion \ref{propquarticlink4}, since the two $G$-invariant conic bundles on $S$ are non-equivalent, we have $G = G_{49} \simeq (\mathbb{Z}/2)^2$.
\end{proof}



\subsection{The \textit{sine qua non} condition for the conjugation of two de Jonquières subgroups outside the de Jonquières group: An \textbf{\textit{Amazing result}}}

\begin{defi} $ $

Let $(S,\pi)$ be a conic bundle and $G \subset \Aut(S,\pi)$ a group of automorphisms.
We say that the triple $(G,S,\pi)$ is minimal if any $G$-equivariant birational morphism 
of conic bundles $\phi \colon S \rightarrow S'$ is an isomorphism.
\end{defi}


\begin{prop} \label{2elemconj2} $ $

Let $(G,S,\pi)$ and $(G',S',\pi')$ be two minimal triples such that $G \simeq G'$ are two $2$-elementary non-cyclic subgroups, $rk(Pic(S)^G) = rk(Pic(S')^{G'}) = 2$, and $G$ contains a non-trivial element fixing a non-rational curve. We assume they represent the same conjugacy class in the Cremona group. The following assertions are equivalent:
\begin{enumerate}
\item \label{2elemconj21} The two triplets $(G,S,\pi)$ and $(G',S',\pi')$ are not conjugate by a de Jonqui\`eres map.
\item \label{2elemconj22} The subgroups $G, G'$ are Klein subgroups and up to 
a de Jonquières map: $$S = S' = \left\{X_0^2 + X_1^2 + X_2^2 + X_3^2 + X_4^2 = X_1^2 + t X_2^2 - t X_3^2 - X_4^2 = 0 \right\} \subset \mathbb{P}^4$$
$$G = \langle [-1:1:1:1:1] , [X_0 : X_4 : X_3 : X_2 : X_1] \rangle \simeq (\mathbb{Z}/2)^2$$ And $\pi, \pi'$ are the two conic bundle structures on $S$ invariant by $G$, $G$ acts trivially on the base of exactly one of the two conic bundles.
\end{enumerate}
\end{prop}

\begin{proof} $ $

\ref{2elemconj22} $\Rightarrow$ \ref{2elemconj21}: A subgroup with a trivial action on the base can not be conjugate by a de Jonquières map to a subgroup with a non-trivial action on the base. \\

\ref{2elemconj21} $\Rightarrow$ \ref{2elemconj22}:
Let $\sigma : S \DashedArrow S'$ be a birational map such that $G' = \sigma \circ G \circ \sigma^{-1}$.
Using  \cite[Theorem 2.5]{isk}, we decompose $\sigma = \sigma_n \circ \dots \circ \sigma_1$ where $\sigma_k : S_{k-1} \DashedArrow  S_k$ are elementary Sarkisov links, and $S_0 =S$, $S_{n} = S'$. Let $G_k = \sigma_k \circ \dots \circ \sigma_1 G \sigma_1^{-1} \circ \dots \circ \sigma_k^{-1} \subset \Aut(S_k)$.
According to Proposition \ref{preli}, $K_S^2 \leq 4$. Hence using \cite[Theorem 2.6]{isk}, there are three possibilities for $\sigma_1$:

\begin{itemize}[label=\textbullet]
\item $\sigma_1$ is a link of type $IV$, $K_S^2 \in \{1,2,4 \}$.
\begin{itemize}
\item If $K_S^2 \in \{1,2 \}$ then according to \cite[Theorem 2.6]{isk}, the link is represented by a Bertini or Geiser involution \textit{i.e.} there exist two isomorphisms $\phi : S_0 \rightarrow S_1, \tau : S_0 \rightarrow S_0$ such that $\sigma_1 = \phi \circ \tau$, $\tau$ is the Bertini or Geiser involution. Then:

$
\begin{array}{llr}
G_1 & = \sigma_1 \circ G \circ \sigma_1^{-1} & \\
& = \phi \circ \tau  G  \tau^{-1} \circ  \phi^{-1} & \\
& = \phi G \phi^{-1} & \mbox{ \quad because $\tau  G  \tau^{-1} \subset \Aut(X)$} \\
&  & \mbox{and $\tau$ commutes with $\Aut(X)$ (Lemma \ref{bertinigeiser})} \\
\end{array}
$

Hence it can be replaced by a de Jonqui\`eres conjugation.
\item
If $K_S^2 = 4$, then we have a conjugation permuting two conic bundles. If the conjugation can't be replaced by a de Jonqui\`eres element, then according to Lemma \ref{propquarticlinkbis}, we have only one possibility for the surface and the group up to isomorphism: $$S_0 = S_1 = \left\{X_0^2 + X_1^2 + X_2^2 + X_3^2 + X_4^2 = X_1^2 + t X_2^2 - t X_3^2 - X_4^2 = 0 \right\}$$
$$G = \langle [-1:1:1:1:1] , [X_0 : X_4 : X_3 : X_2 : X_1] \rangle \simeq (\mathbb{Z}/2)^2$$
\end{itemize}
\item $\sigma_1$ is a link of type $III$, $K_S^2=3$, $S_1$ is a quartic del Pezzo surface and $rk(Pic(S_1)^{G_1} )=1$. Assuming $n > 1$, according to \cite[Theorem 2.6]{isk}, there are two possibilities for the next Sarkisov link $\sigma_2$:
\begin{itemize} 
\item $\sigma_2$ is a link of type $I$, $S_2$ is a cubic surface and $rk(Pic(S_2)^{G_2} )=2$. According to Lemma \ref{conjcubickleinredbis} the composition $\sigma_2 \circ \sigma_1$ can be replaced by an isomorphism. 
\item $\sigma_2$ is a link of type $(II,\mathbb{D})$.
According to Lemma \ref{replacinglinkbyiso}, $\sigma_2$ can be replaced by an isomorphism.
\end{itemize}
\item $\sigma_1$ is a link of type $(II,\mathbb{C})$. Then it corresponds to a de Jonqui\`eres conjugation.
\end{itemize}
Hence we can replace $(G_0,S_0)$ by $(G_1,S_1)$ or $(G_2,S_2)$. We can then apply the previous step to the next Sarkisov link ($\sigma_2$ or $\sigma_3$) assuming it exists. Hence by a finite induction we get the result.

The last assertion is a direct consequence of Lemma \ref{propquarticlink} Asssertion \ref{propquarticlink5}.
\end{proof}

\newpage

\part{Klein subgroups of automorphisms of $\mathbb{P}^2$ in characteristic $2$}  \label{part5} $ $

 \bigskip\bigskip

We assume $\car(\kk) = 2$ in all this part. The goal of this part is to determine the Klein subgroups of $\Aut(\PP^2)$ when $\car(\kk) = 2$. This corresponds to Theorem \ref{theorem3} Assertion \ref{theorem32}. For the convenience of the reader, the results are summarized in Proposition \ref{mitterandtheoremC}.

 \bigskip 

\begin{defi} \label{definitionkleinP2} $ $
Let $t \in \kk \setminus \left\{ 0,1\right\}$.
We define the following Klein subgroups of $\Aut(\PP^2)$:

$G_{1} = \langle \left[\begin{smallmatrix}  1 & 0 & 0   \\
 0 & 1 & 1  \\
 0 & 0 & 1
\end{smallmatrix} \right]  , \left[\begin{smallmatrix}  1 & 0 & 1   \\
 0 & 1 & 0  \\
 0 & 0 & 1
\end{smallmatrix} \right] \rangle  , \quad G_{2} = \langle \left[\begin{smallmatrix}  1 & 1 & 0   \\
 0 & 1 & 0 \\
 0 & 0 & 1
\end{smallmatrix} \right] , \left[\begin{smallmatrix}  1 & 0 & 1  \\
 0 & 1 & 0 \\
 0 & 0 & 1
\end{smallmatrix} \right] \rangle, \quad G_{3,t} = \langle \left[\begin{smallmatrix}  1 & 1 & 0  \\
 0 & 1 & 0 \\
 0 & 0 & 1
\end{smallmatrix} \right] , \left[\begin{smallmatrix}  1 & t & 0  \\
 0 & 1 & 0 \\
 0 & 0 & 1
\end{smallmatrix} \right] \rangle$
\\
\end{defi}

\begin{prop} \label{mitterandtheoremC} $ $ Let $t,t' \in \kk \setminus \{0,1\}$.

\begin{enumerate} 
\item Every Klein subgroup of $\Aut(\mathbb{P}^2)$ is conjugate in $\Aut(\mathbb{P}^2)$ to one of the subgroups of Definition \ref{definitionkleinP2}.
\item
The three subgroups $G_1, G_2, G_{3,t}$ are pairwise not conjugate in $\Aut(\mathbb{P}^2)$.
\item
The two subgroups $G_{3,t}$ and $G_{3,t'}$ are conjugate in $\Aut(\mathbb{P}^2)$ if and only if:
$$t' \in \left\{t,t+1,\frac{1}{t},\frac{t+1}{t},\frac{1}{t+1},\frac{t}{t+1}\right\}$$
\item
The subgroups $G_1$ and $G_{3,t}$ are conjugate by a birational map.
\item
The subgroups $G_1$ and $G_2$ are not conjugate by a birational map.
\end{enumerate}
\end{prop}

\begin{proof} $ $

\begin{enumerate}
\item This is proven in Proposition \ref{kleininchar2} Assertion \ref{kleininchar21}.
\item This is proven in Proposition \ref{klein100}.
\item This is proven in Proposition  \ref{klein100bis}.
\item This is proven in Corollary \ref{g13conj}.
\item This is proven in Corollary \ref{klein1000corr}.
\end{enumerate}
\end{proof}


 \bigskip 
This part is organized as follows:
\begin{itemize}

\item We show in  
Section \ref{section52} that the subgroups from Definition \ref{definitionkleinP2} are representatives up to conjugation by automorphisms of Klein subgroups of $\Aut(\PP^2)$ and we study the possible conjugations by automorphism.

\item In Sections \ref{section53} and \ref{section54} we study the possible conjugations by a birational map between these Klein subgroups.
\end{itemize}

\newpage

\section{Classification of Klein subgroups up to conjugation in $\Aut(\PP^2)$} \label{section52} \label{section51} $ $

The goal of this section is to get a list of representatives of Klein subgroups of $\Aut(\mathbb{P}^2)$ up to conjugacy in $\Aut(\mathbb{P}^2)$. The main result is Proposition \ref{kleininchar2} which gives a list of representatives of Klein subgroups. We also give the results about the conjugations by automorphism between these subgroups in Propositions \ref{klein100} and \ref{klein100bis}.

Recall that four points of $\PP^2$ are in general position if no three of them are collinear. This is equivalent to ask that we can change coordinates and put the four points on $[1:0:0]$, $[0:1:0]$, $[0:0:1]$ and $[1:1:1]$.

\subsection{Action of subgroups of $\Aut(\mathbb{P}^2)$} $ $

We start this section with two independent preliminary results about subgroups of $\Aut(\mathbb{P}^2)$ in characteristic $2$.

\begin{lemme} \label{moscou3} $ $
Every subgroup of $\Aut(\mathbb{P}^2)$ of order greater or equal than three has an orbit of size at least three.
\end{lemme}

\begin{proof} $ $
By contradiction, we assume there exists a subgroup $G \subset \Aut(\mathbb{P}^2)$ with only orbits of size one and two and with at least three elements. Then there exist $g,h \in G$ such that $id \neq g \neq h \neq id$. For every $x \in \mathbb{P}^2$, we define $X_x = \{ f(x) | f \in G \}$, it is a subset of $\mathbb{P}^2$ of cardinal at most two.
Let $U = \mathbb{P}^2 \setminus (Fix(g) \cup  Fix(h))$. For every $x \in U$, we have $\left\{ x,g(x),h(x) \right\} \subset X_x, g(x) \neq x \neq h(x)$ and $|X_x| \leq 2$, hence $g(x) = h(x)$.
Since $g$ and $h$ are not the identity, $U$ is a dense subset of $\mathbb{P}^2$, hence $g =h$. It is a contradiction.
\end{proof}

\begin{lemme} \label{moscou111} $ $
Let $p_1 = [1 :0: 0], p_2 = [0:1:0] , p_3 = [0:0:1], p_4 =[1:1:1]$.
Let $X = \left\{  p_1 , p_2 , p_3 , p_4 \right\} \subset \mathbb{P}^2$.
Let $G_X =\left\{ g \in \Aut(\mathbb{P}^2) | g(X) =X \right\} \subset \Aut(\mathbb{P}^2)$.

\begin{enumerate}
\item \label{moscou1111}
The subgroup $G_X$  is isomorphic to $\mathfrak{S}_4$ and its transpositions are the following elements:
$$(12) = \left[\begin{smallmatrix}  0 & 1 & 0   \\
 1 & 0 & 0  \\
  0&0 & 1
\end{smallmatrix} \right], (13) = \left[\begin{smallmatrix}  0 & 0 & 1   \\
 0 & 1 & 0  \\
  1&0 & 0
\end{smallmatrix} \right], (23) =  \left[\begin{smallmatrix}  1 & 0 & 0   \\
 0 & 0 & 1  \\
  0&1 & 0
\end{smallmatrix} \right]$$
$$(14) = \left[\begin{smallmatrix}  1 & 0 & 0   \\
 1 & 1 & 0  \\
  1&0 & 1
\end{smallmatrix} \right], (24) = \left[\begin{smallmatrix}  1 & 1 & 0   \\
 0 & 1 & 0  \\
  0&1& 1
\end{smallmatrix} \right], (34) = \left[\begin{smallmatrix}  1 & 0 & 1   \\
 0& 1 & 1  \\
  0&0 & 1
\end{smallmatrix} \right]$$

\item \label{moscou1112} Let $H_{1} = \langle \left[\begin{smallmatrix}  0 & 1 & 1   \\
 1 & 0 & 1  \\
  0&0 & 1
\end{smallmatrix} \right]  , \left[\begin{smallmatrix}  0 & 1 & 1   \\
 0 & 1 & 0  \\
 1 &1 & 0
\end{smallmatrix} \right]  \rangle \subset \Aut(\mathbb{P}^2), H_{2} = \langle \left[\begin{smallmatrix}  0 & 1 & 0   \\
 1 & 0 & 0  \\
  0&0 & 1
\end{smallmatrix} \right]  , \left[\begin{smallmatrix}  1 & 0 & 1   \\
 0 & 1 & 1  \\
  0&0 & 1
\end{smallmatrix} \right]  \rangle \subset \Aut(\mathbb{P}^2)$.
$H_1$ and $H_2$ are Klein subgroups of $G_X$. Any $2$-elementary non-cyclic subgroup of $G_X$ is conjugate in $G_X$ to $H_1$ or $H_2$. $H_1$ and $H_2$ are not conjugate in $G_X$. $H_1$ contains the three double transpositions of $G_X$, $H_2$ is generated by two transpositions.
\end{enumerate}
\end{lemme}

\begin{proof} $ $

\begin{enumerate}
\item
The subset $X$ contains four points and no three of these four points are collinear. Hence an element of $G_X$ is determined by the image of the points of $X$, hence $G_X \subset \mathfrak{S}_4$.
Furthermore, $G_X$ contains all the transpositions of $X$:
$$\left[\begin{smallmatrix}  0 & 1 & 0   \\
 1 & 0 & 0  \\
  0&0 & 1
\end{smallmatrix} \right] , \left[\begin{smallmatrix}  0 & 0 & 1   \\
 0 & 1 & 0  \\
  1&0 & 0
\end{smallmatrix} \right] , \left[\begin{smallmatrix}  1 & 0 & 0   \\
 0 & 0 & 1  \\
  0&1 & 0
\end{smallmatrix} \right], \left[\begin{smallmatrix}  1 & 0 & 0   \\
 1 & 1 & 0  \\
  1&0 & 1
\end{smallmatrix} \right] , \left[\begin{smallmatrix}  1 & 1 & 0   \\
 0 & 1 & 0  \\
  0&1& 1
\end{smallmatrix} \right] , \left[\begin{smallmatrix}  1 & 0 & 1   \\
 0& 1 & 1  \\
  0&0 & 1
\end{smallmatrix} \right]$$
Hence $G_X \simeq \mathfrak{S}_4$.

\item 
We first observe that $H_1$ corresponds to $\langle (12)(34), (13)(24) \rangle$ and $H_2$ corresponds to $\langle (12), (34) \rangle$.
Hence $H_1$ and $H_2$ are Klein subgroups of $G_X$, $H_1$ contains the three double transpositions of $G_X$, and $H_2$ is generated by two transpositions. The rest comes from Lemma \ref{moscou111} Assertion \ref{moscou1111} and the group structure of $\mathfrak{S}_4$.
\end{enumerate}
\end{proof}


%
%
%
%

\subsection{A list of representatives} $ $

\medskip

We begin this subsection by stating Lemma~\ref{lis0bis} concerning the Heisenberg group, which corresponds to two of the four assertions of Lemma~\ref{lis0}.

\begin{lemme} \label{lis0bis} $ $
Let $a,b,c, a', b', c' \in \kk$.

Let $A = \left[\begin{smallmatrix}  1 & a & b  \\
 0 & 1 & c \\
 0 & 0 & 1
\end{smallmatrix} \right] \in \GL(3,\kk), A' = \left[\begin{smallmatrix}  1 & a' & b'  \\
 0 & 1 & c' \\
0 & 0 & 1
\end{smallmatrix} \right] \in \GL(3,\kk)$.
Then:
\begin{enumerate}
\item \label{lis2bis} $A, A'$ commute if and only if $a c' = a' c$.
\item \label{lis4bis} $A^2 = I$ if and only if $a = 0$ or $c = 0$.
\end{enumerate}
\end{lemme}

\begin{proof} $ $
Direct computations
\end{proof}

We now state the major result of this subsection:

\begin{prop} \label{kleininchar2} $ $
\begin{enumerate} 
\item \label{kleininchar21} Let $G$ be a Klein subgroup of $\Aut(\mathbb{P}^2)$.
Then $G$ is conjugate in $\Aut(\mathbb{P}^2)$ to one of the groups of Definition \ref{definitionkleinP2}.

\item \label{kleininchar22} The subgroup $G_1$ has an orbit of size four whose points are in general position and every orbit of $G_2$ and $G_{3,t}$ is contained in a line through $[1:0:0]$.
\item \label{kleininchar23} The subgroups $G_1$ and $G_{3,t}$ don't have an orbit of size two.
\item \label{kleininchar24}  The set of fixed points of $G_1$ is $\{[x:y:0] | [x:y] \in \mathbb{P}^1 \}$. $[1:0:0]$ is the only fixed point of $G_2$. The set of fixed points of $G_{3,t}$ is $\{[x:0:z] | [x:z] \in \mathbb{P}^1 \}$. 
\end{enumerate}
\end{prop}

\begin{proof} $ $

\begin{enumerate}
\item
Let $G$ be a Klein subgroup of $\Aut(\mathbb{P}^2)$.
According to Lemma \ref{moscou3}, $G$ has an orbit of size at least three. The order of $G$ is four, so $G$ has no orbit of size three. Hence $G$ has an orbit of size four.
\begin{itemize}
\item
If $G$ has an orbit $X$ of size four where no three of these four points are collinear, then up to conjugation in $\Aut(\mathbb{P}^2)$, $X =\{ [1 :0: 0], [0:1:0], [0:0:1], [1:1:1] \}$.
According to Lemma \ref{moscou111}, up to conjugation we have $G  \in \{ H_1 , H_2 \}$. We observe that $X$ is not an orbit for the subgroup $H_2$ (it is in fact the union of two orbits of size two). Hence, $G = H_1$.

Let $A =
\left[\begin{smallmatrix}  1 & 0 & 0  \\
 0 & 1 & 0  \\
 1 &1 & 1
\end{smallmatrix} \right] $.
We observe that $A^{-1} = \left[\begin{smallmatrix}  1 & 0 & 0  \\
 0 & 1 & 0  \\
 1 &1 & 1
\end{smallmatrix} \right] $.

Then, $A^{-1}\left[\begin{smallmatrix}  0 & 1 & 1   \\
 1 & 0 & 1  \\
  0&0 & 1
\end{smallmatrix} \right] A = \left[\begin{smallmatrix}  1 & 0 & 1   \\
 0 & 1 & 1  \\
 0 & 0 & 1
\end{smallmatrix} \right]$ and $A^{-1} \left[\begin{smallmatrix}  0 & 1 & 1   \\
 0 & 1 & 0  \\
 1 &1 & 0
\end{smallmatrix} \right] A = \left[\begin{smallmatrix}  1 & 0 & 1   \\
 0 & 1 & 0  \\
 0 & 0 & 1
\end{smallmatrix} \right]$.

Then, $A^{-1} H_1 A = \langle \left[\begin{smallmatrix}  1 & 0 & 1   \\
 0 & 1 & 1  \\
 0 & 0 & 1
\end{smallmatrix} \right]  , \left[\begin{smallmatrix}  1 & 0 & 1   \\
 0 & 1 & 0  \\
 0 & 0 & 1
\end{smallmatrix} \right] \rangle = \langle \left[\begin{smallmatrix}  1 & 0 & 0   \\
 0 & 1 & 1  \\
 0 & 0 & 1
\end{smallmatrix} \right]  , \left[\begin{smallmatrix}  1 & 0 & 1   \\
 0 & 1 & 0  \\
 0 & 0 & 1
\end{smallmatrix} \right] \rangle  = G_1 $.

\item
Now we assume that $G$ has an orbit of size four where three of these four points are collinear. Let $X$ be that orbit. Let $x \in X$ such that $X \setminus \left\{x\right\}$ lies on a line $L$.
$x$ is not fixed by $G$ so there exists $g \in G$ such that $g(x) \in L$.
Let $y,z \in L$ such that $X = \{x, g(x),y,z\}$.
We have $\{x,y,z\} = \{x,g(y),g(z)\} = g(\{ g(x) , y , z \}) \subset g(L)$.
By linearity, $g(L)$ is a line. As $\{y,z\} \subset L \cap g(L)$, we get $g(L) = L$ and $x \in L$.
Hence $X \subset L$, and up to conjugation we can assume $X \subset \left\{ [x:y:0] | [x:y] \in \mathbb{P}^1 \right\}$. This is the unique line containing the orbit $X$, hence it is preserved by the subgroup $G$.
Hence we have $G \subset \left\{ \left[\begin{smallmatrix}  * & * & *  \\
 * & * & * \\
 0 & 0 & *
\end{smallmatrix} \right]  \right\} \simeq \GL(2,\kk) \rtimes \kk^2$.

We have the following group homomorphism:
 $$\begin{array}{ccccc}
\pi & : & \left\{ \left[\begin{smallmatrix}  a & b & e  \\
 c & d & f \\
 0 & 0 & g
\end{smallmatrix}\right] \in \PGL(3,\kk) | ad-bc \neq 0 , g \neq 0, e ,f \in \kk   \right\} & \twoheadrightarrow & \PGL(2,\kk) \\
 & &  \left[\begin{smallmatrix}  a & b & *  \\
 c & d & * \\
 0 & 0 & *
\end{smallmatrix} \right]  & \mapsto &  \left[\begin{smallmatrix}  a & b   \\
 c & d 
\end{smallmatrix} \right]  \\
\end{array}$$

Since $G$ acts non-trivially on $L$, $\pi(G)$ is non-trivial.
Hence according to Proposition \ref{dublin1}, $\pi(G)$ is conjugate to a subgroup of $\left\{\left[\begin{smallmatrix}  1 & t   \\
 0 & 1
\end{smallmatrix} \right] | t \in \kk \right\}$ containing $\left[\begin{smallmatrix}  1 & 1   \\
 0 & 1 
\end{smallmatrix} \right]$.
Hence up to conjugation, we can assume that $G$ is a subgroup of $\left\{ \left[\begin{smallmatrix}  1 & * & *   \\
 0 & 1 & * \\
 0 & 0 & *
\end{smallmatrix} \right]\right\}$ containing an element of the form $\left[\begin{smallmatrix}  1 & 1 & *   \\
 0 & 1 & * \\
 0 & 0 & *
\end{smallmatrix} \right]$.
Every non-trivial element of $G$ is of order two, hence $G$ is a subgroup of $\left\{ \left[\begin{smallmatrix}  1 & * & *   \\
 0 & 1 & * \\
 0 & 0 & 1
\end{smallmatrix} \right]\right\}$ containing an element of the form $\left[\begin{smallmatrix}  1 & 1 & *   \\
 0 & 1 & * \\
 0 & 0 & 1
\end{smallmatrix} \right]$.
Since every element of $G$ is of order two, $G$ contains an element of the form $\left[\begin{smallmatrix}  1 & 1 & *   \\
 0 & 1 & 0 \\
 0 & 0 & 1
\end{smallmatrix} \right]$ according to Lemma \ref{lis0bis} Assertion \ref{lis4bis}.
$G$ is abelian, hence $G$ is a subgroup of $\left\{ \left[\begin{smallmatrix}  1 & * & *   \\
 0 & 1 & 0 \\
 0 & 0 & 1
\end{smallmatrix} \right] \right\}$ containing an element of the form $\left[\begin{smallmatrix}  1 & 1 & *   \\
 0 & 1 & 0 \\
 0 & 0 & 1
\end{smallmatrix} \right]$ according to Lemma \ref{lis0bis} Assertion \ref{lis2bis}.
We define the following group isomorphism:
 $$\begin{array}{ccccc}
\rho & : & \left\{ \left[\begin{smallmatrix}  1 & x & y  \\
 0 & 1 & 0 \\
 0 & 0 & 1
\end{smallmatrix} \right] | x,y \in \kk \right\} & \rightarrow & \kk^2\\
 & &   \left[\begin{smallmatrix}  1 & x & y  \\
 0 & 1 & 0 \\
 0 & 0 & 1
\end{smallmatrix} \right] & \mapsto &  (x,y) \\
\end{array}$$

And we define the following subgroup: $$\mathcal{P} = \left\{ \left[\begin{smallmatrix}  1 & 0 & 0  \\
 0 & a & b \\
 0 & c & d
\end{smallmatrix} \right] \in \PGL(3,\kk) | ad-bc \neq 0 \right\} \simeq \GL(2,\kk)$$
We observe that: $\left[\begin{smallmatrix}  1 & 0 & 0  \\
 0 & a & b \\
 0 & c & d
\end{smallmatrix} \right]^{-1}  \left[\begin{smallmatrix}  1 & x & y  \\
 0 & 1 & 0 \\
 0 & 0 & 1
\end{smallmatrix} \right] \left[\begin{smallmatrix}  1 & 0 & 0  \\
 0 & a & b \\
 0 & c & d
\end{smallmatrix} \right]= \left[\begin{smallmatrix}  1 & x' & y' \\
 0 & 1 & 0 \\
 0 & 0 & 1
\end{smallmatrix} \right]$, where $(x',y')^T = P^T (x,y)^T$, and $P = \left[\begin{smallmatrix}  a & b  \\
 c & d
\end{smallmatrix} \right]$.

\begin{itemize}
\item
If $\rho(G)$ spans $\kk^2$ as a $\kk$-vector space, then up to conjugation by an element of $\mathcal{P}$, $G$ contains $\left[\begin{smallmatrix}  1 & 1 & 0  \\
 0 & 1 & 0 \\
 0 & 0 & 1
\end{smallmatrix} \right]$ and $\left[\begin{smallmatrix}  1 & 0 & 1  \\
 0 & 1 & 0 \\
 0 & 0 & 1
\end{smallmatrix} \right]$, hence $G_2 \subset G$.
Since $G$ is a Klein subgroup, we get $G = G_2$.
\item
If $\rho(G)$ does not span $\kk^2$ as a $\kk$-vector space, since $G$ is not trivial, then $span(\rho(G))$ is a line of $\kk^2$.
Up to conjugation by an element of $\mathcal{P}$, we get that $G \subset \left\{ \left[\begin{smallmatrix}  1 & t  & 0  \\
 0 & 1 & 0 \\
 0 & 0 & 1
\end{smallmatrix} \right] | t \in \kk \right\}$ and contains $\left[\begin{smallmatrix}  1 & 1 & 0 \\
 0 & 1 & 0 \\
 0 & 0 & 1
\end{smallmatrix} \right]$.
Since $G$ is a Klein subgroup, we get $G = G_{3,t}$ for a certain $t \in \kk \setminus \{0,1\}$.
\end{itemize}
 \end{itemize}

\item By construction $G_1$ has an orbit of size four whose points are in general position.
Furthermore, if $X$ is an orbit of $G_2$ or $G_{3,t}$ different than $\{[1:0:0]\}$, then there exists $(b,c) \neq (0,0)$ such that every point $[x:y:z] \in X$ satisfies the equation $b y + c z =0$.

\item For the last assertion about the absence of orbits of size two for $G_1$ and $G_{3,t}$, we observe that all non-trivial elements of $G_1$ have the same set of fixed points: $\{ z = 0\}$; likewise, all non-trivial elements of $G_{3,t}$ have the same set of fixed points: $\{ y = 0\}$. 

\item All non-trivial elements of $G_1$ have the same set of fixed points: $\{ [x:y:0] | [x:y] \in \mathbb{P}^1 \}$, hence the result for $G_1$.

The set of fixed points of  $\left[\begin{smallmatrix}  1 & 1 & 0   \\
 0 & 1 & 0 \\
 0 & 0 & 1
\end{smallmatrix} \right]$ is $\{ y = 0 \}$. The set of fixed points of $\left[\begin{smallmatrix}  1 & 0 & 1  \\
 0 & 1 & 0 \\
 0 & 0 & 1
\end{smallmatrix} \right]$ is $\{z = 0 \}$. The intersection of these two subsets is $\{[1:0:0]\}$, hence the result for $G_2$.

All non-trivial elements of $G_{3,t}$ have the same set of fixed points: $\{ [x:0:z] | [x:z] \in \mathbb{P}^1 \}$, hence the result for $G_{3,t}$.
\end{enumerate}
\end{proof}



\subsection{Conjugation by automorphisms between the subgroups $G_1, G_2, G_{3,t}$} $ $



Let $t \in \kk \setminus \left\{0,1\right\}$.

\begin{prop} \label{klein100} $ $

 The three subgroups $G_1, G_2, G_{3,t}$ are pairwise not conjugate in $\Aut(\mathbb{P}^2)$.
\end{prop}

\begin{proof} $ $
Using Proposition \ref{kleininchar2} Assertion \ref{kleininchar24} we know that 
$G_1$ has an infinite number of fixed points: $[x:y:0]$,
$G_2$ has only one fixed point: $[1:0:0]$,
$G_{3,t}$ has an infinite number of fixed points: $[x:0:z]$.
Hence $G_1$ and $G_2$ are not conjugate, and $G_2$ and $G_{3,t}$ are not conjugate.

Furthermore, $G_1$ has an orbit of size four whose points are in general position whereas $G_{3,t}$ has none (Proposition \ref{kleininchar2} Assertion \ref{kleininchar22}). Hence $G_1$ and $G_{3,t}$ are not conjugate.
\end{proof}

\subsection{Conjugation by automorphisms between $G_{3,t}$ and $G_{3,t'}$} $ $

Let $t,t' \in \kk \setminus \left\{0,1\right\}$.

\begin{lemme}  \label{klein100bislemme} $ $
The following are equivalent:
\begin{enumerate}
\item
There exists $\lambda\in \kk^\times$ such that $\{1,t',t'+1\}=\{\lambda,\lambda t, \lambda(t+1)\}$.
\item
$t'\in \{t,t+1,\frac{1}{t},\frac{t+1}{t},\frac{1}{t+1},\frac{t}{t+1}\}$.
\end{enumerate}
\end{lemme}

\begin{proof} $ $
The six possible permutations for the equality: $$\{1,t',t'+1\}=\{\lambda,\lambda t, \lambda(t+1)\}$$ are the following:

\[\begin{array}{lllll}
(1,t',t'+1)=(\lambda,\lambda t,\lambda(t+1))&\Leftrightarrow& (t',\lambda)=(t,1)\vphantom{\Big)}\\
(1,t',t'+1)=(\lambda,\lambda(t+1),\lambda t)&\Leftrightarrow& (t',\lambda)=(t+1,1)\vphantom{\Big)}\\

(1,t',t'+1)=(\lambda t,\lambda ,\lambda(t+1))&\Leftrightarrow& (t',\lambda)=(\frac{1}{t},\frac{1}{t})\vphantom{\Big)}\\
(1,t',t'+1)=(\lambda t,\lambda(t+1),\lambda )&\Leftrightarrow& (t',\lambda)=(\frac{t+1}{t},\frac{1}{t})\vphantom{\Big)}\\

(1,t',t'+1)=(\lambda (t+1),\lambda ,\lambda t)&\Leftrightarrow& (t',\lambda)=(\frac{1}{t+1},\frac{1}{t+1})\vphantom{\Big)}\\
(1,t',t'+1)=(\lambda (t+1),\lambda t,\lambda )&\Leftrightarrow& (t',\lambda)=(\frac{t}{t+1},\frac{1}{t+1}).\vphantom{\Big)}
\end{array}\]
\end{proof}

\begin{prop} \label{klein100bis}  $ $
The following are equivalent:
\begin{enumerate}
\item
\label{G3tConjG3tprim}
The subgroups $G_{3,t}$, $G_{3,t'}$ are conjugate in $\PGL(3,\kk)$.
\item\label{tprimeis}
$t'\in \{t,t+1,\frac{1}{t},\frac{t+1}{t},\frac{1}{t+1},\frac{t}{t+1}\}$.
\end{enumerate}
\end{prop}

\begin{proof} $ $

$\ref{G3tConjG3tprim}\Rightarrow \ref{tprimeis}$:
For each $s\in \kk$, we write $M_s=\begin{bsmallmatrix} 0 & s& 0\\ 0 & 0 & 0\\ 0& 0 & 0\end{bsmallmatrix}$.
There is $A\in \GL(3,\kk)$ whose class conjugates $G_{3,t}$ to $G_{3,t'}$. For each $s\in \{1,t,t+1\}$, there is $s'\in \{1,t',t'+1\}$ and $\mu\in \kk^\times$ such that $\mu (I_3+M_{s'})=A(I_3+M_s) A^{-1}=I_3 +AM_sA^{-1}$. As $1$ is the only eigenvalue of $I_3+M_s$, we find $\mu=1$ and thus: \[\begin{bsmallmatrix} s'a_{21} & s'a_{22}& s'a_{23}\\ 0 & 0 & 0\\ 0& 0 & 0\end{bsmallmatrix}=M_{s'}A=AM_s=\begin{bsmallmatrix} 0 & s a_{11}& 0\\ 0 & sa_{21} & 0\\ 0& sa_{31} & 0\end{bsmallmatrix}.\]
so $A=\begin{bsmallmatrix} a_{11} & a_{12}& a_{13}\\ 0 & a_{22}& 0\\ 0 & a_{32}& a_{33}\end{bsmallmatrix}$, with $a_{11}a_{22}\not=0$ and $\lambda=\frac{s'}{s}=\frac{a_{22}}{a_{11}}$ does not depend on $s$. Hence we have $\{1,t',t'+1\}=\{\lambda,\lambda t, \lambda(t+1)\}$ with $\lambda \in \kk^\times$.
We conclude using Lemma \ref{klein100bislemme}.

$\ref{tprimeis}\Rightarrow\ref{G3tConjG3tprim}$: Using Lemma \ref{klein100bislemme}, we have $\lambda\in \kk^\times$ such that $\{1,t',t'+1\}=\{\lambda,\lambda t, \lambda(t+1)\}$ for some $\lambda \in \kk^\times$. Then the element $\begin{bsmallmatrix} \lambda & 0& 0\\ 0 & 1& 0\\ 0 & 0& 1\end{bsmallmatrix} \in \PGL(3,\kk)$ conjugates $G_{3,t}$ to $G_{3,t'}$.
\end{proof}

\newpage

 \section{Conjugation of the Klein subgroups $G_1$ and $G_{3,t}$ by birational map} \label{section53} $ $

We fix a scalar $t \in \kk \setminus \left\{0,1\right\}$.

In this section, we show that the subgroups $G_1$ and $G_{3,t}$ (from Definition \ref{definitionkleinP2}) are conjugate by a birational map via an \textit{ex machina} proof. The reader will find a discussion on the geometric intuition behind this approach in Remark \ref{ex machina} below.

\bigskip

\begin{defi} $ $

\begin{enumerate}
\item Let $\kappa_t = \left[\begin{smallmatrix}  t & 1 & 0  \\
0 &  0 & 1 \\
1 & 0 & 0
\end{smallmatrix} \right] \in \Aut(\mathbb{P}^2)$ with inverse given by $\kappa_t^{-1} = \left[\begin{smallmatrix}  0 & 0 & 1  \\
1 &  0 & t \\
0 & 1 & 0
\end{smallmatrix} \right]$.

\item Let $\omega_t \in \Bir(\mathbb{P}^2)$ given by:

 $$\begin{array}{ccccc}
\omega_t & : & \mathbb{P}^2 & \DashedArrow & \mathbb{P}^2 \\
 & & [x:y:z] & \DashedArrow & [xy:y^2:t(t+1)zy + x(x+y)]  \\
\end{array}$$
With inverse given by: 
 $$\begin{array}{ccccc}
\omega_t^{-1} & : & \mathbb{P}^2 & \DashedArrow & \mathbb{P}^2 \\
 & & [x:y:z] & \DashedArrow & [xy:y^2:\dfrac{zy+x(x+y)}{t(t+1)}]  \\
\end{array}$$
\end{enumerate}
\end{defi}

%
%

\begin{lemme} \label{explicitconjg1g3t} $ $
Let $s \in \{1,t\}$. Then we have:
$$\begin{array}{rcl} \kappa_t \circ  [x:y+z:z] &=&  [x+y:y:z] \circ  \kappa_t \\
\kappa_t \circ  [x+z:y:z] & = & [x+ty:y:y+z] \circ  \kappa_t \\
\omega_t \circ  [x+sy:y:y+z] & = & [x+sy:y:z]  \circ  \omega_t \end{array}$$
\end{lemme}

\begin{proof} $ $
Direct computations.
\end{proof}

\begin{corr} \label{g13conj} $ $
The subgroups $G_1$ and $G_{3,t}$ are conjugate by a birational map.
\end{corr}

\begin{proof} $ $
By taking $s = 1$ and $s = t$ in Lemma \ref{explicitconjg1g3t}, we get $\omega_t \circ \kappa_t \circ G_1 = G_{3,t} \circ \omega_t \circ \kappa_t$.
\end{proof}

\bigskip

\begin{remark} \label{ex machina} $ $
The reader may find this proof somewhat unmotivated, as it merely exhibits an explicit conjugacy without detailing its derivation. To construct this conjugation from scratch, one should observe that these two Klein subgroups are birational to subgroups of the de Jonquières group, and then either applying already proven results regarding the conjugacy of subgroups within the de Jonquières group to get the existence of such a conjugation 
(Proposition~\ref{cohomologie0}), or find an explicit de Jonquières conjugation by hand.
\end{remark} %

\bigskip

\begin{remark} $ $
It is also instructive to decompose the conjugation $\omega_t$ into Sarkisov links.
 We obtain the following decomposition:


 \[
\begin{tikzcd}
\mathbb{P}^2 \ar[r, dashed, "\beta"'] & \mathbb{F}_1 \ar[r, dashed, "\alpha"'] & \mathbb{F}_2 \ar[r, "\Omega_t"', "\sim"] & \mathbb{F}_2 \ar[r, dashed, "\alpha^{-1}"'] & \mathbb{F}_1 \ar[r, dashed, "\beta^{-1}"'] & \mathbb{P}^2
\end{tikzcd}
\]

Where:



$$\begin{array}{ccccc}
\alpha  & : & \mathbb{F}_1 & \DashedArrow & \mathbb{F}_2 \\
 & & [x_0:x_1;y_0:y_1] & \DashedArrow & [x_0 : x_1 y_0 ; y_1 : y_0] \\ \\
\beta  & : & \mathbb{P}^2 & \DashedArrow & \mathbb{F}_1 \\
 & & [x:y:z] & \DashedArrow & [1: z ; y : x] \\ \\
\Omega_t  & : & \mathbb{F}_2 & \xrightarrow{\sim} & \mathbb{F}_2 \\
 & & [x_0: x_1 ; y_0 : y_1] & \mapsto & [x_0:t(t+1)x_1 + x_0 y_0 (y_0 + y_1);y_0:y_1]\\
 \end{array}$$

\end{remark}

\newpage

\section{The non-conjugation of the Klein subgroups $G_1$ and $G_2$ by birational map} \label{section54} $ $

This section is organized into four subsections. Subsections \ref{631tashkent},  \ref{632tashkent}, and  \ref{633tashkent} are mutually independent and serve to establish several preliminary lemmas. These results will subsequently be used to prove Proposition  \ref{klein1000}, the major result of Subsection  \ref{634tashkent}, which leads directly to Corollary \ref{klein1000corr}.


\subsection{Results} $ $

\begin{corr} \label{klein1000corr} $ $
The subgroups $G_1$ and $G_2$ are not conjugate by a birational map.
\end{corr}

\begin{proof} $ $
According to Proposition \ref{klein100}, these two subgroups are not conjugate by an automorphism of $\mathbb{P}^2$. Hence according to Proposition \ref{klein1000}, they are not conjugate by a birational map. \\
\end{proof}

\subsection{Notations} $ $


We we will need the two following Klein subgroups of automorphisms of Hirzebruch surfaces for our study:

\begin{defi} $ $
We define the following Klein subgroups of birational maps:
\[ K_2 = \langle \begin{array}{cccc}
& [x_0 : x_1 ;y_0 : y_1] & \mapsto & [x_0 : x_1 + x_0 y_0 ;y_0:y_1] , \\
& [x_0 : x_1 ;y_0 : y_1] & \mapsto & [x_0 : x_1 + x_0 y_1 ;y_0:y_1] 
 \end{array}
 \rangle
 \subset \Aut(\mathbb{F}_1)  \]
 \[ R_2 = \langle \begin{array}{cccc}
& ( [x_0  : x_1] , [y_0  : y_1]) & \mapsto & ([x_0 + x_1 : x_1] , [y_0 + y_1 : y_1]) , \\
& ( [x_0  : x_1] , [y_0  : y_1]) & \mapsto & ([y_0  : y_1] , [x_0  : x_1]) \
 \end{array}
 \rangle
  \subset  \Aut(\mathbb{F}_0)  \]
\end{defi}

We refer to Definition \ref{sublinks} for the notions of links of type $(II,\mathbb{D})$ and $(II,\mathbb{C})$.

\subsection{Sarkisov links of type $I$ for $G_2$} \label{631tashkent} $ $

\begin{lemme} \label{klein10001} $ $

Let $\beta : \mathbb{P}^2 \DashedArrow \mathbb{F}_1$ be a link of type $I$ that is the blow-up of a fixed point of $G_2$. 

Then $\beta G_2 \beta^{-1}$ and $K_2$ are conjugate by an automorphism of $\mathbb{F}_1$.
\end{lemme}

\begin{proof} $ $

The only fixed point of $G_2$ is $[1:0:0]$. 
Hence, up to composition by an automorphism of $\mathbb{F}_1$, we can assume that:
$$\begin{array}{ccccc}
\beta  & : & \mathbb{P}^2 & \DashedArrow & \mathbb{F}_1 \\
 & & [x:y:z] & \DashedArrow & [1: x ; y : z] \\
 \end{array}$$
And:
$$\begin{array}{ccccc}
\beta^{-1}  & : & \mathbb{F}_1 & \DashedArrow &\mathbb{P}^2  \\
 & &  [x_0:x_1;y_0:y_1] & \DashedArrow & [x_1: x_0 y_0 : x_0 y_1 ] \\
 \end{array}$$

Then $\beta G_2 \beta^{-1} = \langle \,
\begin{aligned}
    [x_0 : x_1 ; y_0 : y_1] &\mapsto [x_0 : x_1 + y_0 x_0 ; y_0 : y_1], \\
    [x_0 : x_1 ; y_0 : y_1] &\mapsto [x_0 : x_1 + y_1 x_0 ; y_0 : y_1]
\end{aligned}
\, \rangle = K_2. $

\end{proof}


\subsection{Sarkisov links of type $(II,\mathbb{C})$ from $K_2$} \label{632tashkent} $ $

\begin{defi} $ $

We define the subgroup of elements in $\Aut(\mathbb{F}_1)$ that fix the base: $$F = \{[x_0 : x_1 ; y_0 : y_1] \mapsto [x_0 : x_1 + x_0 p(y_0,y_1) ; y_0 : y_1] | p \in \kk[y_0,y_1]_1 \} \subset \Aut(\mathbb{F}_1)$$

We define the following map: $$\begin{array}{ccccc}
\iota  & : & \mathbb{A}^2 & \xhookrightarrow{} &\mathbb{F}_1  \\
 & &  (x,y) & \mapsto & [1: x ; 1 : y ] \\
 \end{array}$$
 
 Finally we define $F' = \iota^{-1} \circ F \circ \iota$.
\end{defi}

\begin{remark} $ $
We have: $$F' = \{\left[\begin{smallmatrix} 1 & a x + b   \\
0 &  1
\end{smallmatrix} \right]  | a,b \in \kk \} \subset \PGL(2,\kk(x)) \subset \J$$
\end{remark}

\begin{lemme} \label{klein10000lemme} $ $

Let $f_1' = \left[\begin{smallmatrix}  1 & a_1 x + b_1  \\
0 &  1
\end{smallmatrix} \right] , f_2' = \left[\begin{smallmatrix}  1 & a_2 x + b_2  \\
0 &  1
\end{smallmatrix} \right] \in F' \setminus \{ \left[\begin{smallmatrix}  1& 0 \\
0 &  1
\end{smallmatrix} \right]  \}$ where $a_1, b_1, a_2, b_2 \in \kk^\times$. Let $\left[\begin{smallmatrix}  \alpha & \beta   \\
\gamma &  \delta
\end{smallmatrix} \right] \in \PGL(2,\kk(x))$ such that $f_2' = \left[\begin{smallmatrix}  \alpha & \beta   \\
\gamma &  \delta
\end{smallmatrix} \right] f_1' \left[\begin{smallmatrix}  \alpha & \beta   \\
\gamma &  \delta
\end{smallmatrix} \right]^{-1}$.
Then $\alpha \delta \neq 0$ and $a_2 x + b_2 = \dfrac{\alpha}{\delta} (a_1 x + b_1) $.
\end{lemme}

\begin{proof} $ $
We write $f_2' = \left[\begin{smallmatrix}  \alpha & \beta   \\
\gamma &  \delta
\end{smallmatrix} \right] \left[\begin{smallmatrix}  1 & a_1 x + b_1   \\
0 &  1
\end{smallmatrix} \right] \left[\begin{smallmatrix}  \delta & - \beta   \\
- \gamma &  \alpha
\end{smallmatrix} \right] = \left[\begin{smallmatrix}  * & *   \\
-\gamma^2 (a_1 x + b_1) &  *
\end{smallmatrix} \right]$, hence $\gamma = 0$.
\newline Hence $f_2' = \left[\begin{smallmatrix}  \alpha & \beta   \\
0 &  \delta
\end{smallmatrix} \right] \left[\begin{smallmatrix}  1 & a_1 x + b_1   \\
0 &  1
\end{smallmatrix} \right] \left[\begin{smallmatrix}  \delta & - \beta   \\
0 &  \alpha
\end{smallmatrix} \right] = \left[\begin{smallmatrix}  \alpha \delta & \alpha^2 (a_1 x + b_1)   \\
0 &  \alpha \delta
\end{smallmatrix} \right]$, hence the result.
\end{proof}

\begin{lemme} \label{klein10000} $ $

If $\sigma : \mathbb{F}_1 \DashedArrow \mathbb{F}_1$ is a birational map that induces the identity on the base and $\sigma K_2 \sigma^{-1} \subset \Aut(\mathbb{F}_1)$, then $\sigma K_2 \sigma^{-1}$ and $K_2$ are conjugate by an automorphism of $\mathbb{F}_1$.
\end{lemme}

\begin{proof} $ $
If we identify the elements of $F$ with $\kk[x]_{\leq 1}$, then $\sigma$ acts by multiplication with an element of $\kk(x)^\times$ on $\kk[x]_{\leq 1}$ by Lemma \ref{klein10000lemme}. Furthermore, if $K_2$ is preserved, then because $K_2$ corresponds to $\{0,1,x,x+1\}$, $\sigma$ has to acts by multiplication by a scalar. Hence $\sigma K_2 \sigma^{-1} = \phi_{\lambda} K_2 \phi_{\lambda}^{-1}$, where $\phi_{\lambda} \in \Aut(\mathbb{F}_1)$ is given by $[x_0 : x_1 ; y_0 : y_1] \mapsto [\lambda x_0 : x_1 ;  y_0 : y_1]$ for some $\lambda \in \kk^\times$.
\end{proof}

\subsection{Sarkisov links of type $(II,\mathbb{D})$ from $(\mathbb{F}_0, \mathbb{D})$} \label{633tashkent} $ $

\begin{lemme} \label{klein10007} $ $ 
\begin{enumerate}
\item \label{klein100071}
Let $G$ be a Klein subgroup of automorphism of $\Aut(\mathbb{F}_0)$. If $rk(Pic(\mathbb{F}_0)^G)=1$ then $G$ is conjugate by an automorphism of $\mathbb{F}_0$ to $R_2$.
\item \label{klein100072}
Let $\beta : \mathbb{F}_0 \DashedArrow \mathbb{P}^2$ be a link of type $(II,\mathbb{D})$ for $R_2$. Then $\beta \circ R_2 \circ \beta^{-1}$ and $G_2$ are conjugate by an automorphism of $\mathbb{P}^2$.
\end{enumerate}
\end{lemme}

\begin{proof} $ $

\begin{enumerate}
\item
Since $rk(Pic(\mathbb{F}_0)^G)=1$, it implies that $G$ is not included in $\Aut(\mathbb{P}^1) \times \Aut(\mathbb{P}^1)$. Hence we can use Proposition \ref{moscou1} to get that $G = R_2$ up to an automorphism.
\item Let $G = \beta \circ R_2 \circ \beta^{-1} \subset \Aut(\mathbb{P}^2)$. According to Proposition \ref{kleininchar2} Assertion \ref{kleininchar21}, up to automorphism we can assume that $G$ is $G_1$ or $G_2$ or $G_{3,t}$ for some $t \in \kk \setminus \{0,1\}$.  According to \cite[Theorem 2.6]{isk}, such a link can be decomposed as $\beta = \pi_2^{-1} \circ \pi_1$ where $\pi_1^{-1}$ is a blowup of $\mathbb{F}_0$ at a fixed point of $R_2$ and $\pi_2^{-1}$ is a blowup of $\mathbb{P}^2$ at an orbit of size two of $G$. Since $G_1$ and $G_{3,t}$ don't have an orbit of size two (Proposition \ref{kleininchar2} Assertion \ref{kleininchar23}), we get $G = G_2$.
\end{enumerate}
\end{proof}

\subsection{Birational conjugacy of the Klein subgroup $G_2$} \label{634tashkent} $ $



\begin{defi} $ $

We denote $dP_5$ \textit{resp} $dP_6$ the \textit{unique} del Pezzo surface of degree five \textit{resp} six.
\end{defi}

\begin{prop} \label{klein1000} $ $
Let $G$ be a Klein subgroup of $\Aut(\mathbb{P}^2)$.

\noindent If $G_2$ and $G$ are conjugate by a birational map, then they are conjugate in $\Aut(\mathbb{P}^2)$.
\end{prop}

\begin{proof} $ $
Let $\sigma : \mathbb{P}^2 \DashedArrow \mathbb{P}^2$ be a birational map such that $\sigma G_2 \sigma^{-1} = G$.
According to \cite[Theorem 2.5]{isk}, let $\sigma = \sigma_n \circ \dots \circ \sigma_1$ be a decomposition of $\sigma$ into Sarkisov links.
According to \cite[Theorem 2.6]{isk}, the link $\sigma_1$ is one of the following:
\begin{itemize}
\item $\sigma_1$ is a link of type $I$ from $(\mathbb{P}^2,\mathbb{D})$ to $(dP_5,\mathbb{C})$. According to Proposition \ref{kleininchar2} Assertion \ref{kleininchar22}, the subgroup $G_2$ has no orbit of size four in general position. Hence this link is not possible.
\item $\sigma_1$ is a link of type $I$ from $(\mathbb{P}^2,\mathbb{D})$ to $(\mathbb{F}_1,\mathbb{C})$. According to Lemma \ref{klein10001}, we have $\sigma_1 G_2 \sigma_1^{-1} = K_2$.
According to \cite[Theorem 2.6]{isk}, the next link $\sigma_2$ is either a link of type $III$ to $\mathbb{P}^2$, or a link of type $(II,\mathbb{C})$ (see Definition \ref{sublinks}). Let's assume this is a link of type $(II,\mathbb{C})$ to some Hirzebruch surface $\mathbb{F}_n, n \in \mathbb{N}_{\geq 0}$. Since the only fixed points of $K_2$ are on the $(-1)$-curve of $\mathbb{F}_1$, we cannot have $n = 0$. Since the action on the base is trivial, we can blow-up only orbits of size $1$, hence $n = 2$. The same logic applies to the next link $\sigma_3$: According to \cite[Theorem 2.6]{isk} it has to be a link of type $(II,\mathbb{C})$ to $(\mathbb{F}_m,\mathbb{C}), m = n \pm 1$. After finitely many steps we have to go back to $\mathbb{F}_1$, and according to Lemma \ref{klein10000} this composition can be replaced by an automorphism of $\mathbb{F}_1$.
When going back to $\mathbb{P}^2$ we then get a subgroup conjugate by an automorphism of $\mathbb{P}^2$ to $G_2$.
\item $\sigma_1$ is a link of type $(II,\mathbb{D})$ from $(\mathbb{P}^2,\mathbb{D})$ to $(\mathbb{F}_0,\mathbb{D})$. Assuming we don't go back to $(\mathbb{P}^2,\mathbb{D})$ directly, according to \cite[Theorem 2.6]{isk}, $\sigma_2$ is one of the following:
\begin{itemize}
\item A link of type $(II,\mathbb{D})$ from $(\mathbb{F}_0,\mathbb{D})$ to $(dP_5,\mathbb{D})$ or to $(dP_6,\mathbb{D})$. This is not possible for a Klein subgroup because it requires to blowup an orbit of size not a power of $2$ according to \cite[Theorem 2.6]{isk}.
\item A link of type $(II,\mathbb{D})$ from $(\mathbb{F}_0,\mathbb{D})$ to itself.
\item A link of type $I$ from $(\mathbb{F}_0,\mathbb{D})$ to $(dP_6,\mathbb{C})$. In this case, according to \cite[Theorem 2.6]{isk}, we could have a composition of links of type $(II,\mathbb{C})$ before being forced to go back to $(\mathbb{F}_0,\mathbb{D})$: there exists $k \geq 3$ such that $\sigma_k$ is a link of type $III$ from $(dP_6,\mathbb{C})$ to $(\mathbb{F}_0,\mathbb{D})$ and for all $3 \geq m \geq k-1, \sigma_{m}$ is a link of type $(II,\mathbb{C})$.
\end{itemize}
Hence in all the cases, we go back to $(\mathbb{F}_0,\mathbb{D})$. According to Lemma \ref{klein10007} Assertion \ref{klein100071}, up to automorphism there exists exactly one Klein subgroup in $(\mathbb{F}_0,\mathbb{D})$. Hence the link $\sigma_2$ in the second case, or the composition $\sigma_k \circ \dots \circ \sigma_2$ in the third case, can be replaced by an automorphism and the subgroup has not changed.
After a finite number of these sequences of links, we are then forced to go back to $(\mathbb{P}^2,\mathbb{D})$. According to Lemma \ref{klein10007} Assertion \ref{klein100072}, we get again the subgroup $G_2$ up to automorphism.
\end{itemize}
\end{proof}

\newpage

\part{Proofs of the main results}  \label{part6} $ $


\bigskip\bigskip

\section{Proof of Theorem \ref{theorem1}} $ $


Let $G$ be a $p$-elementary non-cyclic subgroup of the Cremona group.
\\

First we assume $p > 3$. According to Proposition \ref{prop1}, we can conjugate $G$ to a subgroup of the de Jonqui\`eres group, to a subgroup of the automorphism group of a del Pezzo surface of degree $1$, $5$ or $7$ with Picard rank equal to $1$, to a subgroup of $\Aut(\mathbb{P}^1 \times \mathbb{P}^1)$ or to a subgroup of $\Aut(\mathbb{P}^2)$.
Using Proposition \ref{prop27}, if $G$ is a subgroup of $\Aut(\mathbb{P}^1 \times \mathbb{P}^1)$ then it is conjugate to a subgroup of the de Jonqui\`eres subgroup.
If $G$ is conjugate to a subgroup of the de Jonqui\`eres subgroup, we can conjugate $G$ to $\langle (\left[\begin{smallmatrix}  1 & 0  \\
0 &   \xi
\end{smallmatrix} \right] ,I),(I,\left[\begin{smallmatrix}  1 & 0  \\
0 &    \xi
\end{smallmatrix} \right] ) \rangle$ using Proposition \ref{buda5}. This last group is conjugate to the $p$-torsion subgroup of the diagonal torus of $\PGL(3,\kk)$.
If $G$ is conjugate to a subgroup of the automorphism group of the del Pezzo surface of degree $7$, then $G$ is conjugate to a subgroup of $\Aut(\mathbb{P}^2)$.
Since the automorphism group of a del Pezzo surface of degree $5$ is isomorphic to $\mathfrak{S}_5$ when the field is algebraically closed (This result is proven for example in \cite[Proposition 5.1.]{blancarticle2} when $\kk$ is the field of complex numbers, but the proof holds for any algebraically closed field), it implies that $G$ is not a subgroup of the automorphism group of a del Pezzo surface of degree $5$.
According to Proposition \ref{propdegree1pgeq5} (when $\car(\kk) \neq 2$) or Proposition \ref{resultdp1char2} (when $\car(\kk)=2$), $G$ is not a subgroup of the automorphism group of a del Pezzo surface of degree $1$.
If $G$ is conjugate to a subgroup of $\Aut(\mathbb{P}^2)$, using Proposition \ref{prop3} we can conjugate $G$ to the $p$-torsion subgroup of the diagonal torus of $\PGL(3,\kk)$.
\\

Then we assume $p=3$. According to Proposition \ref{prop1}, we can conjugate $G$ to a subgroup of the de Jonqui\`eres group, to a subgroup of automorphism of a del Pezzo surface of degree $1$ or $3$ with Picard rank equal to $1$, to a subgroup of $\Aut(\mathbb{P}^1 \times \mathbb{P}^1)$ or to a subgroup of $\Aut(\mathbb{P}^2)$.
Using Proposition \ref{prop27}, if $G$ is a subgroup of $\Aut(\mathbb{P}^1 \times \mathbb{P}^1)$ then it is conjugate to a subgroup of the de Jonqui\`eres subgroup.
If $G$ is conjugate to a subgroup of $\Aut(\mathbb{P}^2)$, using Proposition \ref{p21} we can conjugate $G$ to the $3$-torsion subgroup of the diagonal torus of $\PGL(3,\kk)$ (family \ref{A}) or to: \begin{center} $\langle [X:Y:Z] \mapsto [X:j Y: j^2 Z] , [X:Y:Z] \mapsto [Z:X:Y] \rangle$

(family \ref{B}) \end{center}
If $G$ is conjugate to a subgroup of automorphism of a del Pezzo surface of degree $3$ with Picard rank equal to $1$, using Proposition \ref{resultcubic}, $G$ is conjugate to the $3$-torsion subgroup of the diagonal torus of $\PGL(4,\kk)$ acting on the Fermat cubic surface given by the equation $W^3 + X^3 + Y^3 + Z^3 = 0$ (isomorphic to $(\mathbb{Z}/3)^3$), or to the subgroup: \begin{center}$\langle [W:X:Y:Z] \mapsto [jW:X:Y:Z] , [W:X:Y:Z] \mapsto [W:jX:Y:Z] \rangle$

(family \ref{C}) \end{center} acting on the Fermat cubic surface:
\[ \left\{ W^3 + X^3  + Y^3 + Z^3 = 0 \right\} \subset \mathbb{P}^3 \] or to the subgroup: \begin{center} $\langle [W:X:Y:Z] \mapsto [jW:X:Y:Z] , [W:X:Y:Z] \mapsto [W:jX:j^2 Y:Z] \rangle$

(family \ref{D}) \end{center} acting on the cubic surface: \[ \left\{ W^3 + X^3 + Y^3 + Z^3 + \mu X Y Z = 0 \right\} \subset \mathbb{P}^3 \] where $\mu$ is a scalar such that $\mu^3 \neq -27$. Using Proposition \ref{resultcubic}, we get the possible conjugations by isomorphisms inside the family \ref{D}.
If $G$ is conjugate to a subgroup of the automorphism group of a del Pezzo surface of degree $1$ and $\car(\kk) \neq 2$, using Proposition \ref{1result}, $G$ is conjugate to the subgroup: \begin{center} $\langle [W:X:Y:Z] \mapsto [W:jX:Y:Z] , [W:X:Y:Z] \mapsto [W:X:jY:Z] \rangle$

(family \ref{E}) \end{center} acting on the surface: \[ \left\{W^2 = Z^3 + X^6 + Y^6 +  c X^3 Y^3 \right\} \subset \mathbb{P}(3,1,1,2) \] where $c$ is a scalar not equal to $\pm 2$.
Using Proposition  \ref{1result} Assertion \ref{1resultbis}, we get the possible conjugations by isomorphisms inside the family \ref{E}.
If $G$ is conjugate to a subgroup of the automorphism group of a del Pezzo surface of degree $1$ and $\car(\kk) = 2$, using Proposition \ref{resultdp1char2}, $G$ is conjugate to the subgroup: \begin{center} $\langle [u:v:x:y] \mapsto [u:v:j x : y] , [u:v:x:y] \mapsto [v:u+v:x:y+\sqrt{e}(u^2 v + u v^2 + v^3)] \rangle$

(family \ref{E}) \end{center} acting on the surface:  \[ \left\{ (u,v,x,y) | 0 = y^2 + uv(u+v) y + x^3 + (e + \sqrt{e}) (u^5 v + u v^5) + e u^3 v^3\right\} \subset \mathbb{P}(1,1,2,3) \] where $e \in \kk \setminus \{0,1\}$.
If $G$ is conjugate to a subgroup of the de Jonqui\`eres subgroup, we can conjugate $G$ to $\langle (\left[\begin{smallmatrix}  1 & 0  \\
0 &   j
\end{smallmatrix} \right] ,I),(I,\left[\begin{smallmatrix}  1 & 0  \\
0 &    j
\end{smallmatrix} \right] ) \rangle$ using Theorem \ref{buda5}. This last group is conjugate to the $p$-torsion subgroup of the diagonal torus of $\PGL(3,\kk)$.
\\

Finally we assume $p =2$. According to Proposition \ref{prop1}, we can conjugate $G$ to a subgroup of the de Jonqui\`eres group, to a subgroup of the automorphism group of a del Pezzo surface of degree $1$, $2$, $4$ or $\mathbb{F}_1$ with Picard rank equal to $1$, or to a subgroup of $\Aut(\mathbb{P}^2)$ or $\Aut(\mathbb{P}^1 \times \mathbb{P}^1)$.
If $G$ is conjugate to a subgroup of $\Aut(\mathbb{F}_1)$, then $G$ is conjugate to a subgroup of $\Aut(\mathbb{P}^2)$.
If $G$ is conjugate to a subgroup of $\Aut(\mathbb{P}^2)$ then according to Proposition \ref{prop3} it is conjugate to the subgroup
$\left\{ [X:Y:Z] \mapsto [\pm X: \pm Y: \pm Z] \right\}$. This group is conjugate by a birational map to the group $\langle (x,y) \mapsto (-x,y), (x,y) \mapsto (x,-y) \rangle$ acting on $\mathbb{A}^2$ which is a subgroup of the de Jonqui\`eres group.
If $G$ is conjugate to a subgroup of the automorphism group of $\mathbb{P}^1 \times \mathbb{P}^1$, according to Propositions \ref{p11}, $G$ is conjugate to a subgroup of the de Jonqui\`eres group.
If $G$ is conjugate to a subgroup of the automorphism group of a del Pezzo surface of degree $4$ with Picard rank equal to $1$, by using Proposition \ref{quarticresult} we can conjugate $G$ to the subgroup: \begin{center}
    $\left\{ [X_0:X_1:X_2:X_3:X_4] \mapsto [ \pm X_0: \pm X_1: \pm X_2:X_3:X_4] \right\}$ (family \ref{K})
\end{center}
or to the subgroup:
\begin{center}
    $\left\{ [X_0:X_1:X_2:X_3:X_4] \mapsto [ \pm X_0: \pm X_1: \pm X_2:\pm X_3:\pm X_4] \right\}$ (family \ref{P})
\end{center}
or to the subgroup: \[G = \left\{  [X_0 :  X_1 : X_2 : X_3 : X_4] \mapsto [\pm X_0 : \pm X_1 : X_2 : X_3 : X_4 ] \right\} \] acting on the quartic del Pezzo surface: \[S =  \left\{ X_0^2 + X_1^2 + X_2^2 + X_3^2 + X_4^2 = a_0 X_0^2 + a_1 X_1^2 + a_2 X_2^2 + a_3 X_3^2 + a_4 X_4^2  = 0 \right\} \] where $a_0, \dots, a_4 \in \kk$ are pairwise distinct.
Using Proposition \ref{conjcubickleinblue} Assertion \ref{conjcubickleinblue1}, in this last subcase we can conjugate $G$ to a subgroup of the de Jonquières group, hence eliminating this subcase.
If $G$ is conjugate to a subgroup of the automorphism group of a del Pezzo surface of degree $2$ with Picard rank equal to $1$, by using Proposition \ref{2results} we can conjugate $G$ to the subgroup: \begin{center} $\left\{ [W:X:Y:Z] \mapsto  [\pm W: \pm X:Y:Z] \right\}$
(family \ref{G}) \end{center} acting on the surface $\{W^2 = X^4 + L_2(Y,Z) X^2 + L_4(Y,Z)\} \subset \mathbb{P}(2,1,1,1)$, where $L_2$ \textit{resp} $L_4$ is a form of degree $2$ \textit{resp} $4$, or to the subgroup: \begin{center} 
$\left\{ [W:X:Y:Z] \mapsto  [\pm W: \pm X: \pm Y:Z] \right\}$
 (family \ref{L}) \end{center} acting on the surface: 
\[ \left\{W^2 = X^4 + Y^4 + Z^4 + d X^2 Y^2 + e X^2 Z^2 + f Y^2 Z^2 \right\} \subset \mathbb{P}(2,1,1,1) \] where $d,e,f \in \kk \setminus \{-2,2\}$ such that $ 4 - f^2 - d^2 - e^2 + def \neq 0$.
If $G$ is conjugate to a subgroup of the automorphism group of a del Pezzo surface of degree $1$ with Picard rank equal to $1$, by using Proposition \ref{12result} we can conjugate $G$ to the subgroup: \begin{center} 
$\left\{ [W:X:Y:Z] \mapsto [\pm W: \pm X:Y:Z] \right\}$
(family \ref{H}) \end{center} acting on the surface: \[ \{W^2 = Z^3 + L_1(X^2,Y^2) Z^2 + L_2 (X^2,Y^2)Z + L_3 (X^2,Y^2)\} \subset \mathbb{P}(2,1,1,3) \] where $L_1$ \textit{resp} $L_2$ \textit{resp} $L_3$ is a form of degree 1 \textit{resp} 2 \textit{resp} 3.

If $G$ is conjugate to a subgroup of the de Jonqui\`eres subgroup with a non-trivial element fixing a non-rational curve, by using Proposition \ref{ultimateproposition}, we can conjugate $G$ to a group of the families  \ref{I},  \ref{J},  \ref{M},  \ref{N} or  \ref{Q}.

If $G$ is conjugate to a subgroup of the de Jonqui\`eres subgroup with no non-trivial element fixing a non-rational curve, we use the proof of \cite[Proposition 8.4.]{blancarticle2}. Indeed we observe that the proof requires only the field to be algebraically closed and of characteristic different than $2$. This shows that such a subgroup is conjugate to a subgroup of automorphisms of $\mathbb{P}^1 \times \mathbb{P}^1$. Then we use Proposition \ref{detailedclassificationsubgsF0} for the classification of the $2$-elementary non-cyclic subgroups of automorphisms of $\mathbb{P}^1 \times \mathbb{P}^1$. We get then the families \ref{Iprime}, \ref{Mprime}, \ref{Qprime}.
\qed

\bigskip
\bigskip
\section{A prerequisite for the proof of Theorem \ref{theorem2}} $ $

The following Proposition \ref{lemmetheorem2} will be used in the proof of Theorem \ref{theorem2}. We did not state it before because it requires Theorem \ref{theorem1}.

\bigskip
\begin{prop} \label{lemmetheorem2} $ $
Let $G, G'$ be two subgroups from the list of Theorem \ref{theorem1}. \newline
We assume $G$ is in the case $\mathbb{D}$ and $G$ and $G'$ are conjugate by a birational map. \newline Then $G'$ is in the case $\mathbb{D}$ and $G$ and $G'$ are conjugate by an isomorphism. 
\end{prop}

\begin{proof} $ $

If we consider a $G$ subgroup of the de Jonquières group, we assume implicitly that it acts on a rational surface $S$, such that the pair $(G,S)$ is minimal.
\begin{itemize}
\item If $p \geq 5$ then there exists exactly one representative in the list of Theorem \ref{theorem1}.

\item If $p = 3$ then according to Theorem \ref{theorem1}, $G'$ is in the case $\mathbb{D}$. Let $S, S'$ be the del Pezzo surfaces on which $G,G'$ act. According to Theorem \ref{theorem1} we have $K_S^2, K_{S'}^2 \in \{1,3,9\}$.
If $K_S^2 \neq 9$ or $K_{S'}^2 \neq 9$ then we can use Proposition \ref{conjdplowdeg} (recall that the subgroups from the list of Theorem \ref{theorem1} are minimal).
If $K_S^2 = K_{S'}^2 = 9$ then we can use Proposition \ref{p21}.

\item If $p = 2$ then according to Theorem \ref{theorem1} we have $K_S^2 \in \{1,2,4\}$. If $K_S \in  \{ 1,2\}$ then we can use Proposition \ref{conjdplowdeg} (recall that the subgroups from the list of Theorem \ref{theorem1} are minimal). If $K_S^2 = 4$, since the subgroups are from the list of Theorem \ref{theorem1}, it implies that $G$ is not a Klein subgroup. Hence we can use Proposition \ref{conjofbig2elemquartic} (recall that the subgroups from the list of Theorem \ref{theorem1} are minimal).
\end{itemize}
\end{proof}

\newpage

\section{Proof of Theorem \ref{theorem2}} $ $


 If we consider a subgroup $G$ of the de Jonquières group, we assume implicitly that it acts on a rational surface $S$, such that the pair $(G,S)$ is minimal.
\begin{enumerate}[label=\textbf{(\arabic*)}, leftmargin=2em, labelsep=1em]
\item Let $G, G'$ be two subgroups in two different families from the list of Theorem \ref{theorem1}. We assume they are conjugate by a birational map.

If $G$ or $G'$ is in the case $\mathbb{D}$ then by using Proposition \ref{lemmetheorem2} we get that they are both in the case $\mathbb{D}$, and by denoting $S, S'$ their corresponding del Pezzo surfaces on which they act, we have that $(G,S)$ and $(G',S')$ are isomorphic. It leaves then the following two cases:
\begin{itemize}
\item families \ref{A} and \ref{B}. We use Proposition \ref{p21} to get that they can't be conjugate.
\item families \ref{C} and \ref{D}. We use Proposition \ref{resultcubic} to get that they can't be conjugate.
\end{itemize}

Now we assume that $G$ and $G'$ are both in the case $\mathbb{C}$. According to Theorem \ref{theorem1}, this implies that $p=2$.
Observe that every family in the case $\mathbb{C} \mathbb{A}$ has at least one non-trivial element fixing a non-rational curve whereas every family in the case $\mathbb{C} \mathbb{P}$ has no non-trivial element fixing a non-rational curve (Theorem \ref{theorem1}). This implies that no subgroup in the case $\mathbb{C} \mathbb{A}$ can be conjugate to a subgroup in the case $\mathbb{C} \mathbb{P}$ (Lemma \ref{ultimatelemma3}).
Since there is exactly one family in the case $\mathbb{C} \mathbb{P}$ for each isomorphism class, this implies that $G, G'$ are both in the case $\mathbb{C} \mathbb{A}$.
It leaves two possibilities:
\begin{itemize}
\item families \ref{I} and \ref{J}.
\item families \ref{M} and \ref{N}.
\end{itemize}
Hence $G, G'$ are not conjugate in the de Jonquières group (Lemma \ref{ultimatelemma3}). 
They also both have at least one non-trivial element fixing a non-rational curve because they are in the case $\mathbb{C} \mathbb{A}$ (see Theorem \ref{theorem1}).
Hence Proposition \ref{2elemconj2} applies and we get $G \simeq G' \simeq (\mathbb{Z}/2)^2$. Hence $G$ or $G'$ is in family \ref{I} and the other subgroup is in family \ref{J}. The second part of the assertion is a direct consequence of Proposition \ref{2elemconj2}.

\item 
This is a direct consequence of Proposition \ref{lemmetheorem2}.

\item According to Proposition \ref{ultimateproposition}, two such subgroups will have at least one non-trivial element fixing a non-rational curve. By contradiction let's assume that these two subgroups are not conjugate by a de Jonquières map. Then using Proposition \ref{2elemconj2}, one of the subgroup acts trivially on the base whereas the other doesn't. It means that they belong to the same family.

\item This is a direct consequence of Proposition \ref{detailedclassificationsubgsF0}.

\item This is a direct consequence of Proposition \ref{resultatgeneral}.
\end{enumerate}

\qed

\bigskip
\bigskip

\section{Proof of Theorem \ref{theorem3}} $ $

\begin{enumerate}[label=\textbf{(\arabic*)}, leftmargin=2em, labelsep=1em]
\item 
This is proven in Proposition \ref{theo2}.
\item 
This is proven in Proposition \ref{mitterandtheoremC}.
\end{enumerate}

\qed

\newpage

\part{Appendices} \label{part7} $ $

\bigskip\bigskip

\appendix

\section{The $2$-elementary subgroups of the group $\ZZ/q \rtimes \ZZ/q^\times$} \label{bratac2} \label{bratac0}  $ $

This appendix is devoted to group theory and is independent of algebraic geometry. It establishes a result on some finite groups that is required for our study of subgroups of automorphisms of quartic del Pezzo surfaces. 
The result of this section is Corollary \ref{pirouz3}. It is used in the proof of Proposition \ref{propquartic5points} Assertion \ref{propquartic5points2}.
\medskip

Let $q$ be an odd prime number. 
\begin{defi} $ $
We denote by $\ZZ/q \rtimes \ZZ/q^\times$ the holomorph of $\ZZ/q$, \textit{i.e.} the semi-direct product where $\ZZ/q^\times$ acts by multiplication on $\ZZ/q$. We identify it with the following group of affine transformations of the finite field $\mathbb{F}_q$:
$$\left\{\begin{array}{ccccc}
f_{a,b}& : & \mathbb{F}_q & \rightarrow &  \mathbb{F}_q \\
 & & x & \mapsto & a x + b  \\
\end{array} | \quad a \in \mathbb{F}_q^\times, b \in \mathbb{F}_q \right\}$$
\end{defi}

\noindent \makebox[0pt][l]{Now we take:} \hfill $K = \langle -1 \rangle \subset \mathbb{Z}/q^\times \text{ or } K = \mathbb{Z}/q^\times$, \hfill \mbox{}

\medskip

\noindent \makebox[0pt][l]{and:} \hfill $G = \mathbb{Z}/q \rtimes K \subset \mathbb{Z}/q \times \mathbb{Z}/q^\times$. \hfill \mbox{}

\begin{lemme} \label{pirouz1} $ $
The non-trivial involutions of $G$ are the elements $f_{-1,b}$ where $b \in \mathbb{F}_q$.
\end{lemme}

\begin{proof} $ $ 
Let $a \in \mathbb{F}_q^\times, b \in \mathbb{F}_q$. We have $f_{a,b}^2 = f_{a^2,ab+b}$. So $f_{a,b}^2 = id_{\mathbb{F}_q}$ if and only if $(a,b) = (1,0)$ or $a = -1$ (observe that $2 \in \mathbb{F}_q^\times$). Hence our result.
\end{proof}

\begin{prop} \label{pirouz2} $ $

\begin{enumerate}
\item  \label{pirouz21} All non-trivial involutions of $G$ are conjugate in $G$. 
\item  \label{pirouz22} Let $s$ be a non-trivial involution of $G$. The only non-trivial involution that commutes with $s$ is $s$ itself.
\end{enumerate}
\end{prop}

\begin{proof} $ $

\begin{enumerate} 
\item Let $b \in \mathbb{F}_q$. Let $c = \dfrac{b}{2}$ (observe that $2 \in \mathbb{F}_q^\times$). Then we have:
\[f_{1,c} \circ f_{-1,0} \circ f_{1,c}^{-1} = f_{-1,2c} = f_{-1,b}.\]
We conclude using Lemma \ref{pirouz1}.
\item Let $b, c \in  \mathbb{F}_q$. We have $f_{-1,b} \circ f_{-1,c} \circ f_{-1,b} \circ f_{-1,c} = f_{1,2b-2c}$. This element is the identity $\Leftrightarrow b=c$ (observe that $2 \in \mathbb{F}_q^\times$). We conclude using Lemma \ref{pirouz1}.
\end{enumerate}
\end{proof}

\begin{prop} \label{pirouz2bis} $ $

All $2$-elementary non-trivial subgroups of $G$ are cyclic and conjugate in $G$.
\end{prop}

\begin{proof} $ $
According to Proposition \ref{pirouz2} Assertion \ref{pirouz22}, any $2$-elementary subgroup of $P$ is cyclic. Then we use Proposition \ref{pirouz2} Assertion \ref{pirouz21} to ensure the conjugacy between these subgroups.
\end{proof}

\begin{corr} \label{pirouz3} $ $
Let $P$ be one of the three following groups:
$$\mathfrak{S}_3, \quad D_{10},  \quad \ZZ/5 \rtimes \ZZ/5^\times.$$
Then all $2$-elementary non-trivial subgroups of $P$ are cyclic and conjugate in $P$.
\end{corr}

\begin{proof} $ $
The hypotheses of Proposition \ref{pirouz2bis} are satisfied for these three groups.
\end{proof}

\newpage

\section{Results on automorphisms of del Pezzo surfaces} \label{appendixB} $ $

We show two important results in this appendix:
\begin{itemize}
\item We show that Bertini \textit{resp} Geiser involution commutes with the automorphism group of del Pezzo surfaces of degree $1$ \textit{resp} $2$ in Subsection \ref{ardoines5}.
\item
We also inject the automorphism group of del Pezzo surfaces of low degree (less or equal than three) in a bigger group in Subsection \ref{ardoines0}.
\end{itemize}

First we state Lemma \ref{Lem.Aaction_global_section_power_of_canonical} in Subsection \ref{ardoines1}. It will be used in all other proofs of this appendix.

\subsection{A preliminary result about the action of automorphisms on the sections of $-dK_X$} \label{ardoines1} $ $

The following Lemma \ref{Lem.Aaction_global_section_power_of_canonical} will be used in all other subsections of this appendix.

\begin{lemme} \label{Lem.Aaction_global_section_power_of_canonical} $ $
Let $d \geq 1$ and let $X$ be a smooth irreducible surface. Then every automorphism of $X$ induces naturally an automorphism of the space $H^0(X, -d K_X)$ of global sections of the line bundle $-d K_X$.
\end{lemme}

\begin{proof} $ $
Note that $-dK_X$ corresponds to the $d$-fold tensor product of $T_X \bigwedge T_X \to X$, where $T_X \to X$ denotes the tangent bundle of $X$. Let $\varphi$ be an automorphism of $X$. The claim follows from the following commutative diagram:
$$ 
\xymatrix{
T_X \bigwedge T_X \otimes \cdots \otimes T_X \bigwedge T_X \ar[d] \ar[rrrr]^{{\mathrm{d}\varphi \bigwedge \mathrm{d}\varphi \otimes \dots \otimes \mathrm{d}\varphi \bigwedge \mathrm{d}\varphi}} &&&& T_X \bigwedge T_X \otimes \cdots \otimes T_X \bigwedge T_X \ar[d] \\
X \ar[rrrr]^{\varphi} &&&& X
}
 $$
\end{proof}

\subsection{The Geiser and Bertini involutions} \label{ardoines5}

\begin{lemme} \label{bertinigeiser} $ $
Let $S$ be a del Pezzo surface of degree $1$ \textit{resp} of degree $2$.

\noindent Then the Bertini \textit{resp} the Geiser involution commutes with the automorphism group.
\end{lemme}

\begin{proof} $ $
Let $D = \left\{
    \begin{array}{rl}
        -K_S  & \mbox{if $S$ is of degree $2$} \\
- 2 K_S & \mbox{if $S$ is of degree $1$}
    \end{array}
\right.$

 Let $Z = \left\{
    \begin{array}{ll}
        \mathbb{P}^2  & \mbox{if $S$ is of degree $2$} \\
Q \subset \mathbb{P}^3 & \mbox{if $S$ is of degree $1$, where $Q$ is a quadric cone}
    \end{array}
\right.$

Let $\sigma_0$ be the Geiser involution if $S$ is of degree $2$, \textit{resp} the Bertini involution if $S$ is of degree $1$.
We know that the relation $|D| \circ \sigma = |D|$ is true if and only if $\sigma \in \{ id , \sigma_0\}$ by definition of the Geiser \textit{resp} the Bertini involution.
We also know from Lemma \ref{Lem.Aaction_global_section_power_of_canonical} that for $\sigma \in \Aut(S)$, there exists $\overline{\sigma} \in \Aut(Z)$
such that the following diagram commutes:
\begin{center}\begin{tikzcd}
S \arrow[r, "\sigma"] \arrow[d, "|D|"] & S \arrow[d, "|D|"] \\
Z \arrow[r, "\overline{\sigma}"] & Z
\end{tikzcd} \end{center}

We have the following:

$
\begin{array}{lllr}
|D| \circ (\sigma \circ \sigma_0 \circ \sigma^{-1})& = (|D| \circ \sigma) \circ \sigma_0 \circ \sigma^{-1}  \\ & = (\overline{\sigma} \circ |D| ) \circ \sigma_0 \circ \sigma^{-1}  \\ &  = \overline{\sigma} \circ (|D|  \circ \sigma_0) \circ \sigma^{-1} \\ & = \overline{\sigma} \circ |D| \circ \sigma^{-1}  \\ & = (\overline{\sigma} \circ |D|) \circ \sigma^{-1} \\ & = (|D| \circ \sigma) \circ \sigma^{-1}  \\ & = |D| \end{array}
$

Hence we have $\sigma \circ \sigma_0 \circ \sigma^{-1} \in \{id,\sigma_0\}$. Since $\sigma_0 \neq id$, we get $\sigma \circ \sigma_0 \circ \sigma^{-1} = \sigma_0$.

Hence the result.
\end{proof}

\subsection{Embedding of the automorphism group of low degree del Pezzo surfaces into a linear group} \label{ardoines0}

\begin{lemme} \label{2lemconj40} $ $

The automorphism group of any cubic del Pezzo surface is a subgroup of $\GL(4,\kk)$.
\end{lemme}

\begin{proof}
We use the arguments of the proof of \cite[Ch. III, Theorem~3.5]{Kol}. The space $H^0(X, -K_X)$ is of dimension $4$. Since every automorphism induces naturally an automorphism of  $H^0(X, -K_X)$, see Lemma~\ref{Lem.Aaction_global_section_power_of_canonical}, we get that every automorphism of $X$ induces an automorphism of $H^0(X, -K_X)$, and that this association is an injective group morphism, hence our result.
\end{proof}

\begin{lemme} \label{autodp2} $ $
Let $X$ be a del Pezzo surface of degree $2$.

Let $\varphi \colon X \to \mathbb{P}(1,1,1,2)$ be the closed embedding coming from the proof of  \cite[Ch. III, Theorem~3.5]{Kol}. Then every automorphism of  $X$ extends to an automorphism of $\mathbb{P}(1,1,1,2)$ via $\varphi$ of the form $[x:y:z:w] \mapsto [\eta(x, y, z):\lambda w]$, where $\eta$ is a linear automorphism of $\kk^3$ and $\lambda \in \kk^\times$.
\end{lemme}

\begin{proof} $ $
We use the arguments of the proof of \cite[Ch. III, Theorem~3.5]{Kol}.

The space $H^0(X, -K_X)$ is of dimension $3$. The image of the natural product map: $$f : H^0(X, -K_X) \oplus H^0(X, -K_X) \to H^0(X, -2K_X)$$ is of dimension $6$.
The space $H^0(X, -2K_X)$ is of dimension $7$. Since every automorphism induces naturally an automorphism of the space $H^0(X, -d K_X)$ for every $d \geq 1$, see Lemma~\ref{Lem.Aaction_global_section_power_of_canonical}, we get that every automorphism of $X$ induces an automorphism of:
$$H^0(X, -K_X) \oplus H^0(X, -2K_X)/ Spann_{\kk}(\Img(f)) \, $$
Taking the geometric $\mathbb{G}_m$-quotient with respect to the weights $1,1,1,2$ yields an automorphism of $\mathbb{P}(1,1,1,2)$ that extends $\varphi$ via $X$ and it has the claimed form.
\end{proof}

\begin{lemme} \label{autodp1} $ $
Let $X$ be a del Pezzo surface of degree $1$.

Let $\varphi \colon X \to \mathbb{P}(1,1,2,3)$ be the closed embedding coming from the proof of  \cite[Ch. III, Theorem~3.5]{Kol}. Then every automorphism of $X$ extends to an automorphism of $\mathbb{P}(1,1,2,3)$ via $\varphi$ of the form $[x:y:u:v] \mapsto [\eta(x, y):\lambda u : \mu v]$, where $\eta$ is a linear automorphism of $\kk^2$ and $\lambda, \mu \in \kk^\times$.
\end{lemme}

\begin{proof} $ $
We use the arguments of the proof of \cite[Ch. III, Theorem~3.5]{Kol}.

 The space $H^0(X, -K_X)$ is of dimension $2$.
The image of the natural product map: $$f : H^0(X, -K_X) \oplus H^0(X, -K_X) \to H^0(X, -2K_X)$$ spans a $3$-dimensional subspace of $H^0(X, -2K_X)$. Furthermore, the image of the natural product map:  $$g : H^0(X, -K_X) \oplus H^0(X, -2K_X) \to H^0(X, -3K_X)$$ spans the full $7$-dimensional space $H^0(X, -3K_X)$.
Since every automorphism induces naturally an automorphism of  $H^0(X, -d K_X)$ for every $d \geq 1$, see Lemma~\ref{Lem.Aaction_global_section_power_of_canonical}, we get that every automorphism of $X$ induces an automorphism of:
$$H^0(X, -K_X) \oplus H^0(X, -2K_X)/ Spann_{\kk} (\Img(f)) \oplus H^0(X, -3K_X)/ Spann_{\kk} (\Img(g)) \, $$
Taking the geometric $\mathbb{G}_m$-quotient with respect to the weights $1,1,2,3$ yields an automorphism of $\mathbb{P}(1,1,2,3)$ that extends $\varphi$ via $X$ and it has the claimed form.
\end{proof}

\newpage

\section{A formula for the rank of the invariant Picard group \\ \normalfont{(by Immanuel van Santen)}} \label{appendixA} $ $

Let $X$ be a smooth rational projective surface over an algebraically closed field $\kk$ equipped with a faithful algebraic action of a finite group $G$ such that the characteristic of $\kk$ and $|G|$ are coprime. 
The goal of this appendix is to give a formula of the rank of the invariant picard group $\Pic(X)^G$ in terms
of the topological Euler-characteristic of the fixed point varieties $X^g$ for $g \in G$.

We fix a prime $\ell$ different from the characteristic of $\kk$, denote 
by $\ZZ_{\ell}$ the $\ell$-adic completion of $\ZZ$ and by $\QQ_{\ell}$ the quotient field of $\ZZ_{\ell}$. 
For every integer $i \geq 0$ 
and every projective variety $Y$ we denote by $H^i(Y_{\et}, \QQ_{\ell})$ the $\ell$-adic \'etale cohomology of $Y$, see e.g. 
\cite[\S19, p. 126]{Mi2013Lectures-on-Etale-}. In fact, $H^i(Y_{\et}, \QQ_{\ell})$ is a
finite dimensional $\QQ_{\ell}$-vector space that vanishes for $i > 2 \dim Y$, see 
\cite[Theorems~15.1, 19.1, 19.2]{Mi2013Lectures-on-Etale-}. If $Y$ is smooth, projective and connected, 
then there is a unique isomorphism $H^{2 \dim Y}(Y, \QQ_{\ell}) \to \QQ_{\ell}$ that sends the class of a point to $1$
in $\QQ_{\ell}$, see~\cite[Theorem~24.1]{Mi2013Lectures-on-Etale-}. 
We identify in the following $H^{2 \dim Y}(Y, \QQ_{\ell})$ with $\QQ_{\ell}$ via this isomorphism.
The topological Euler characteristic of a projective
variety $Y$ is defined by
\[
	\chi(Y) = \sum_{i=0}^{2 \dim Y} (-1)^i \dim_{\QQ_{\ell}} H^i(Y_{\et}, \QQ_{\ell}) \, .
\]

\begin{example}
	Let $Y$ be a smooth projective curve of genus $g$. Then $\chi(Y) = 2-2g$ by~\cite[Proposition 14.2]{Mi2013Lectures-on-Etale-}. If $Y$ is just a point, then $\chi(Y) = 1$.
\end{example}

\begin{prop} \label{rankformula}
	\label{prop.Inv_Picarad_rank}
	Let $X$ be a smooth projective rational surface equipped with a faithful algebraic action of a finite group $G$
	such that the characteristic of $\kk$ and $|G|$ are coprime. Then 
	\[
		\rank \Pic(X)^G + 2 = \frac{1}{|G|} \sum_{g \in G} \chi(X^g) \, ,
	\]
	where $X^g \subseteq X$ denotes the closed subvariety of those points that are fixed by $g$.
\end{prop}

For the proof we use the following formula, known as the Lefschetz fixed point formula in $\ell$-adic 
\'etale cohomology:

\begin{theo}[{\cite[Theorem~25.1]{Mi2013Lectures-on-Etale-} and \cite[Proposition~1.3.6]{Kl1968Algebraic-cycles-a}}]
	\label{Thm.Lefschetz}
	Let $\varphi$ be an endomorphism of a smooth projective connected variety $Y$ and let $\Delta_Y$, $\Gamma_{\varphi}$
	be the diagonal and the graph of $\varphi$ in $Y \times Y$, respectively. Then the intersection product 
	of $\Delta_Y$ and $\Gamma_\varphi$ can be written as
	\[
		\Delta_Y \cdot \Gamma_{\varphi } = \sum_{i = 0}^{2 \dim Y} (-1)^i \Tr(\varphi^\ast \colon 
		H^i(Y_{\et}, \QQ_{\ell}) \to H^i(Y_{\et}, \QQ_{\ell})) \, ,
	\]
	where $\Tr(\varphi^\ast \colon H^i(Y_{\et}, \QQ_{\ell}) \to H^i(Y_{\et}, \QQ_{\ell}))$ denotes the trace of the 
	$\QQ_{\ell}$-linear automorphism $\varphi^\ast$
	induced by $\varphi$ on $H^i(Y_{\et}, \QQ_{\ell})$.
\end{theo}

Moreover, we need the following proposition that relates the intersection of the diagonal and the graph of 
a certain automorphism with the topological Euler characteristic of the fixed point variety:

\begin{prop}
	\label{Prop.Intersection}
	Let $Y$ be a smooth projective variety and let $\varphi$ be an automorphism of $Y$ such that the order
	of $\varphi$ is finite and coprime to the characteristic of $\kk$. Then we have
	\[
			\Delta_Y \cdot \Gamma_{\varphi } = \chi(Y^{\varphi}) \, .
	\]
\end{prop}

\begin{proof}
	As the order of $\varphi$ is coprime to the characteristic of $\kk$, the subgroup of
	$\Aut(Y)$ generated by $\varphi$ is linearly reductive and thus the fixed point variety 
	$Y^{\varphi}$ is smooth, see \cite[Proposition~A.8.10]{CoGaPr2010Pseudo-reductive-g}. 
	Let $C$ be any connected component of $Y^{\varphi}$ and consider the open subvariety
	$Y_C = Y \setminus (Y^\varphi \setminus C)$ of $Y$. As the points of $Y^\varphi$ are fixed by $\varphi$,
	it follows that $\varphi$ restricts to an automorphism of $Y_C$ and we denote this restriction by $\varphi_C$.
	Again using~\cite[Proposition~A.8.10]{CoGaPr2010Pseudo-reductive-g}, we get a pull-back diagram
	\[
		\xymatrix@R=10pt{
			C \ar[r]^-{c \mapsto (c, c)} \ar[d]_-{c \mapsto (c, c)}  & \Gamma_{\varphi_C} \ar@{}[d] |-{\bigcap}  \\
			\Delta_{Y_C}  \ar@{}[r] |-{\subset} & Y_C \times Y_C
		}
	\]
	where the horizontal morphisms are regular embeddings. By the excess intersection formula 
	(\cite[Theorem~6.3]{Fu1998Intersection-theor}) applied to the above diagram  and by \cite[Exp. VII, Corollaire~4.9]{1977Cohomologie-l-adiq} 
	we get $\Delta_{Y_C} \cdot \Gamma_{\varphi_C} = \chi(C)$.
	In summary
	\[
		\Delta_Y \cdot \Gamma_{\varphi } = \sum_C 
		\Delta_{Y_C} \cdot \Gamma_{\varphi_C} = \sum_C \chi(C) = \chi(Y^\varphi) \, .
	\]
\end{proof}

\begin{proof}[Proof of Proposition~\ref{prop.Inv_Picarad_rank}]
	Let $V = H^2(X_{\et}, \QQ_{\ell})$. As $X$ is a smooth rational projective surface, numerical equivalence and linear 
	equivalence of divisors are the same and hence the cycle map $\Pic(X) \otimes_{\ZZ} \QQ_{\ell} \to V$
	is an isomorphism, see \cite[Example~1.2.6]{Kl1968Algebraic-cycles-a} and also \cite[Proposition~1.17]{EiHa20163264-and-all-that-}.
	Since $\Pic(X)$ is a finitely generated free abelian group, the same holds for $\Pic(X)^G$ and thus
	we have 
	\begin{equation}
		\label{Eq.rank}
		\rank \Pic(X)^G = \dim_{\QQ_{\ell}} V^G \, .
	\end{equation}
	Consider $P \colon V \to V$ which is given by
	\[
		P(v) = \frac{1}{|G|} \sum_{g \in G} g v \, .
	\]
	Then $P \colon V \to V$ is $\QQ_{\ell}$-linear, $G$-invariant, the image is $V^G$ and it restricts to the identity on $V^G$.
	In particular we get a decomposition into $G$-subrepresentations $V = \ker(P) \oplus V^G$ and thus
	\begin{equation}
		\label{Eq.Rep}
		\dim_{\QQ_{\ell}} V^G = \Tr(P) = 
		\frac{1}{|G|} \sum_{g \in G} \Tr(V \xrightarrow{v \mapsto gv} V) \, .
	\end{equation}
	
	Note that $H^1(X_{\et}, \QQ_{\ell})$ vanishes, since $H^1(\PP^2_{\et}, \QQ_{\ell})$ vanishes 
	(see \cite[Example 16.3]{Mi2013Lectures-on-Etale-}) and since
	$H^1(\cdot, \QQ_{\ell})$ is invariant under blow-ups, see 
	\cite[Exp. VII, Proposition~8.5b)]{1977Cohomologie-l-adiq}. By Poincaré duality (see \cite[Theorem~24.1]{Mi2013Lectures-on-Etale-}) we get thus $H^0(X_{\et}, \QQ_{\ell}) \simeq H^4(X_{\et}, \QQ_{\ell}) = \QQ_{\ell}$ and $H^3(X_{\et}, \QQ_{\ell}) = H^1(X_{\et}, \QQ_{\ell}) = 0$.

	In summary
	\begin{eqnarray*}
			\rank \Pic(X)^G = \dim_{\QQ_{\ell}} V^G &\stackrel{\eqref{Eq.Rep}}{=}& 
			\frac{1}{|G|} \sum_{g \in G} \Tr(g^\ast \colon  H^2(X_{\et}, \QQ_\ell) \to H^2(X_{\et}, \QQ_\ell)) \\
			&\stackrel{\textrm{Thm.}\ref{Thm.Lefschetz}}{=}& \frac{1}{|G|} \sum_{g \in G} (\Delta_X \cdot \Gamma_g - 2) \\
			&\stackrel{\textrm{Prop.}\ref{Prop.Intersection}}{=} & -2 + \frac{1}{|G|} \sum_{g \in G} \chi(X^g) \, .
	\end{eqnarray*}
\end{proof}

\begin{remark}
The proof above shows that $\chi(X) = 2 + \rank \Pic(X)$.
\end{remark}

\par\bigskip
\renewcommand{\MR}[1]{}

\newpage

\bibliographystyle{alpha}
\bibliography{bib}

\end{document}